\documentclass[microtype]{gtpart}     
\usepackage{graphicx}
\usepackage{placeins}
\usepackage{float}
\usepackage{caption}
\usepackage{subcaption}

\DeclareMathOperator{\Rm}{\textup{Rm}}
\DeclareMathOperator{\R}{\textup{R}}
\DeclareMathOperator{\vol}{\textup{Vol}}
\DeclareMathOperator{\Ric}{\textup{Ric}}

\usepackage{subcaption}

\title{Noncollapsed degeneration of Einstein $4$-manifolds, II}

\author{Tristan Ozuch}
\givenname{Tristan}
\surname{Ozuch}
\address{MIT, Dept. of Math., 77 Massachusetts Avenue, Cambridge, MA 02139-4307.}
\email{ozuch@mit.edu}
\urladdr{https://tristanozuch.github.io/}


\arxivreference{}

\newtheorem{thm}{Theorem}[section]    
\newtheorem{lem}[thm]{Lemma}    

\newtheorem{conj}[thm]{Conjecture}  
\theoremstyle{definition}
\newtheorem{defn}[thm]{Definition}    
\newtheorem{rem}{Remark}             

\newtheorem{exmp}{Example}

\newtheorem{prop}{Proposition}

\newtheorem{cor}{Corollary}
\begin{document}

\begin{abstract}
    In this second article, we prove that any desingularization in the Gromov-Hausdorff sense of an Einstein orbifold by smooth Einstein metrics is the result of a gluing-perturbation procedure that we develop. This builds on our first paper where we proved that a Gromov-Hausdorff convergence implied a much stronger convergence in suitable weighted Hölder spaces, in which the analysis of the present paper takes place.
    
     The description of Einstein metrics as the result of a gluing-perturbation procedure sheds light on the local structure of the moduli space of Einstein metrics near its boundary. More importantly here, we extend the obstruction to the desingularization of Einstein orbifolds found by Biquard, and prove that it holds for any desingularization by trees of quotients of gravitational instantons only assuming a mere Gromov-Hausdorff convergence instead of specific weighted Hölder spaces. This is conjecturally the general case, and can at least be ensured by topological assumptions such as a spin structure on the degenerating manifolds. We also identify an obstruction to desingularizing spherical and hyperbolic orbifolds by general Ricci-flat ALE spaces.
\end{abstract}

\maketitle


\tableofcontents

\section*{Introduction}

An Einstein metric, $g$ satisfies, for some real $\Lambda$, the equation
$$\Ric(g)=\Lambda g.$$
In dimension $4$, they are considered optimal for the homogeneity of their Ricci curvature, as critical points of the Einstein-Hilbert functional with fixed volume, $g\mapsto\int_M \R_g dvol_g$, and more importantly as minimizers of the $L^2$-norm of Riemann curvature tensor, $g\mapsto \int_M |\Rm_g|^2dvol_g$. 

From dimension $4$, even under natural assumptions of bounded diameter (compactness) and lower bound on the volume (noncollapsing) Einstein metrics can develop singularities. One major goal for $4$-dimensional geometry is therefore to understand the moduli space of Einstein metrics on a differentiable manifold $M^4$ defined as
\begin{equation}
    \mathbf{E}(M^4) := \left\{(M^4,g)\;|\;\exists \Lambda\in \mathbb{R},\; \Ric(g)=\Lambda g,\; \vol(M^4,g)= 1\right\}\slash\mathcal{D}(M^4).\label{def moduli space}
\end{equation}
and to compactify it with a useful structure. This has been done in an $L^2$ and then Gromov-Hausdorff (GH) sense in \cite{andL2,ct06}. More precisely, if we denote by $\overline{\mathbf{E}(M^4)}_{GH}$ the compactification of the moduli space $\mathbf{E}(M^4)$ for the (pointed) Gromov-Hausdorff distance, $d_{GH}$, we have a decomposition
\begin{equation}
    \overline{\mathbf{E}(M^4)}_{GH} = \mathbf{E}(M^4)\cup\partial_o\mathbf{E}(M^4)\cup\partial_\infty\mathbf{E}(M^4),\label{compactification moduli space}
\end{equation}
where $\partial_\infty\mathbf{E}(M^4)$ consists in limits with infinite diameter, and $\partial_o\mathbf{E}(M^4)$ consists in singular limits with bounded diameter. 

We will focus on local questions and for simplicity assume most of the time that we work on spaces with uniformly bounded diameter and therefore study the Gromov-Hausdorff neighborhood of $\partial_o\mathbf{E}(M^4)$. We therefore work on the $d_{GH}$-\emph{completion} of $\mathbf{E}(M^4)$, which is $\mathbf{E}(M^4)\cup\partial_o\mathbf{E}(M^4)$. The metric spaces in $\partial_o\mathbf{E}(M^4)$ and the associated singularity blow-ups in the Gromov-Hausdorff sense have been understood for a long time in \cite{and,bkn}: they are respectively \emph{Einstein orbifolds} and \emph{Ricci-flat ALE orbifolds}. 

Anderson then asked the converse question for instance in \cite{andsurv}, namely, are all Einstein orbifolds limits of smooth Einstein ? To answer this question one has to understand if the reverse of the degeneration, the \emph{desingularization}, of Einstein orbifolds in $\partial_o\mathbf{E}(M^4)$ is possible. A natural way to desingularize an orbifold is by a gluing-perturbation technique. 

The goal of the present paper is to develop a gluing-perturbation procedure which attains \emph{any} noncollapsed Einstein $4$-manifold which is sufficiently close to an Einstein orbifold in the Gromov-Hausdorff sense. This in particular elucidates the $d_{GH}$-neighborhood of the boundary $\partial_o\mathbf{E}(M^4)$ in $\mathbf{E}(M^4)$, and we will use this description in future works. In this paper, we will use it to prove that not all Einstein orbifolds can be desingularized by Einstein metrics in the Gromov-Hausdorff sense with an expected topology which partially answers the above question of Anderson.

\subsection*{Desingularization of Einstein $4$-orbifolds and obstructions}

A natural technique to desingularize orbifolds is the following procedure: we glue Ricci-flat ALE manifolds to the singularities of the orbifold to obtain an approximate Einstein metrics, and then try to perturb it into an actual Einstein metric. We will call such gluings, naïve desingularizations of the orbifold and often denote them $g^D_t$ (see Definition \ref{def naive desing}), where $t$ is the set of gluing parameters which are small positive real numbers. The main result of \cite{ozu1} is that the Gromov-Hausdorff proximity of an Einstein metric to an Einstein orbifold implies the proximity to a naïve desingularization $g^D_t$ in the sense of a weighted Hölder norm denoted $C^{2,\alpha}_{\beta,*}(g^D_t)$. This norm is bounded on symmetric $2$-tensors decaying in the neck regions where the gluing takes place.

In the present paper, we will propose a partial converse by proving that any naïve desingularization can be perturbed to a metric which is Einstein modulo some \emph{obstructions}, which are elements of an approximate cokernel of the linearization of the Einstein operator. We will call such a metric an \emph{Einstein modulo obstructions} metric.

\begin{thm}[{Theorem \ref{fcts inv einst général}}]\label{gluing pert}
    Let $g^D_t$ be a naïve desingularization of an Einstein orbifold $(M_o,g_o)$ with small enough gluing parameters.
    
    Then, there exists a small $C^{2,\alpha}_{\beta,*}(g^D_t)$-neighborhood of $g^D_t$ in which there exists a unique metric $\hat{g}_t$ which is Einstein \emph{modulo obstructions} while satisfying some gauge conditions with respect to $g^D_t$.
\end{thm}
The proof relies on an inverse function theorem applied to the Einstein operator in well chosen coordinates.

These Einstein modulo obstructions metrics $\hat{g}_t$ are not interesting for themselves when they aren't Einstein as they are not geometrically motivated. Their purpose is for instance different from the metrics of \cite{gv} which are critical for some geometric functionals obtained by perturbing a connected sum of Einstein metrics.

Let us note that the Ricci flow starting at Einstein modulo obstructions metrics however exhibits interesting behaviors with respect to the Ricci flow. Indeed, in \cite{bk}, an obstruction is identified to a particular desingularization of $\mathbb{T}^4\slash\mathbb{Z}_2$ and an ancient solution to the Ricci flow smoothing out the orbifold $\mathbb{T}^4\slash\mathbb{Z}_2$ is constructed thanks to it.
\\

Our construction however produces \emph{every} smooth Einstein desingularization in a Gromov-Hausdorff sense. Indeed, together with the convergence of \cite{ozu1}, as a direct consequence, we have the following complete description of the Einstein metrics in a Gromov-Hausdorff neighborhood of an Einstein $4$-orbifold. 
\begin{cor}[{Corollary \ref{proximité désing naive}}]\label{prox GH gluing pert}
    Let $(M_o,g_o)$ be an Einstein $4$-orbifold. Then, there exists $\delta >0$ such that if $(M,g^\mathcal{E})$ is an Einstein manifold satisfying
	$$d_{GH}\big((M,g^\mathcal{E}),(M_o,g_o)\big)\leqslant \delta,$$
	then, $(M,g^\mathcal{E})$ is isometric to a result of the gluing-perturbation procedure of Theorem \ref{gluing pert}.
\end{cor}

\subsection*{A premoduli space in the neighborhood of a singular metric}

Classically, studying a moduli space requires understanding its compactification with a useful structure. The compactification \eqref{compactification moduli space} a priori does not carry a useful structure as it comes from the rough Gromov-Hausdorff distance. The moduli space $\mathbf{E}(M^4)$ however admits a real-analytic structure around smooth metrics.
\begin{thm}[{\cite{koi}}]\label{koiso structure}
    Let $(M,g_0)$ be an Einstein manifold. Then, there exists a $d_{GH}$-neighborhood of $g_0$ in $\mathbf{E}(M)$ which is the quotient by the isometry group of $g_0$ of a real-analytic subvariety of a finite dimensional real-analytic submanifold, $W$, of the space of smooth metrics on $M$.
\end{thm}
The finite dimensional real-analytic submanifold, $W$, consists in metrics which are Einstein modulo the cokernel of the linearization of the Einstein equation at $g_0$ as is usually obtained by Lyapunov-Schmidt reduction. The Einstein metrics are exactly the metrics for which these obstructions vanish. Our description extends this local description of $\mathbf{E}(M^4)$ to  the boundary $ \partial_o\mathbf{E}(M^4)$, and the set of Einstein modulo obstructions metrics $\hat{g}_t$ of Theorem \ref{gluing pert} is the analogue of the ambient space $W$ of Theorem \ref{koiso structure}.

Theorem \ref{koiso structure} is an important local result which implies for instance that $\mathbf{E}(M)$ is locally finite. Anderson asked in \cite{andsurv} whether this structure extends to $\mathbf{E}(M^4)\cup \partial_o\mathbf{E}(M^4)$. The new description of the neighborhood of $\partial_o\mathbf{E}(M)$ in $(\overline{\mathbf{E}(M)}_{GH},d_{GH})$ of Corollary \ref{prox GH gluing pert} provides a promising setting in which one can tackle this question. In particular, in Section \ref{section premoduli}, we provide an adaptation to the singular setting of Koiso's premoduli space around metrics of $\partial_o\mathbf{E}(M)$.

\subsection*{Degeneration of Kähler-Einstein manifolds}

Even if our purpose here is to study the \emph{real} Einstein equation and not Kähler-Einstein metrics, our analysis in weighted Hölder spaces extends the analysis leading to the gluing-perturbation theorems of \cite{ban,spo,br15,hv20} in the Kähler setting. Indeed, it allows us to glue and perturb multiple trees of singularities with arbitrary scales and Einstein deformations. It would therefore be interesting to extend the constructions of \cite{spo,br15} to remove the ``generic'' (\cite{spo}) or ``non degenerate'' (\cite{br15}) conditions which correspond to restricting the gluing scales depending on the size of the Einstein deformation. We should also be able to allow general degenerations forming trees of singularities.

For instance, in Section \ref{arbre kahleriens}, we precise the construction of \cite{hv20} in the case of Kronheimer's gravitational instantons and prove that any tree of Kähler Ricci-flat ALE spaces can be glued and perturbed to a single Kähler Ricci-flat ALE metric with uniform controls (in our weighted Hölder norms) only depending on the group at infinity.

\subsection*{Obstructions to the Gromov-Hausdorff desingularization of Einstein orbifolds}

Our main application in this paper is a nonexistence result: there exists Einstein orbifolds which cannot be approached by smooth Einstein metrics with specific topologies in the Gromov-Hausdorff sense. For this, it is enough to prove that the obstructions of Theorem \ref{gluing pert} do not vanish.

The hyperkähler ALE spaces which are called \emph{gravitational instantons} have been classified in \cite{kro} and their Kähler quotients have been classified in \cite{suv}. It is a famous conjecture, \cite{bkn}, that all Ricci-flat ALE spaces are Kähler.

Our first goal here is to prove that an obstruction holds for any Gromov-Hausdorff desingularization by trees of Kähler Ricci-flat ALE orbifolds, which are conjecturally the only possibilities. The obstruction to satisfy is $\det \mathbf{R} =0$ at a singular point of the orbifold metric, where $\mathbf{R}$ is the Riemannian curvature seen as an endomorphism on the space of $2$-forms.

\begin{thm}[{Theorem \ref{obst arbre kahler}}]
	Let $(M_i,g_i)_i$ be a sequence of Einstein manifolds converging in the Gromov-Hausdorff sense to an Einstein orbifold $(M_o,g_o)$, and assume that there exists a subsequence $(M_i,g_i)_i$ whose possible blow-up limits are Kähler Ricci-flat ALE orbifolds glued in the same orientation. 
	
	Then, at every singular point $p$ of $(M_o,g_o)$, we have
	$$\det \mathbf{R}_{g_o}(p) =0. $$
\end{thm}
This answers positively a question from \cite{biq1} and extends it to the case of several singularities and allows the formation of trees of singularities. It more precisely states that the obstruction of \cite{biq1} holds for any known possible configuration of singularity models, and that it holds even assuming the weakest possible convergence instead of a convergence in particular weighted Hölder spaces.

\begin{rem}
    Note that this obstruction is very different in nature from the result of \cite{oss} which shows that most compact Kähler-Einstein $4$-orbifolds with positive Ricci-curvature cannot be limits of Kähler-Einstein manifolds. Indeed, our obstruction $\det \mathbf{R} =0$ is always satisfied in this situation, and it remains unknown if these metrics can be desingularized by \emph{real} Einstein metrics. It is also different from the obstruction found in \cite{bk} where the obstruction $\det \mathbf{R}=0$ is also satisfied by the orbifold $\mathbb{T}^4\slash\mathbb{Z}_2$.
\end{rem}

Under topological assumptions, it is known that the singularity models appearing are Kähler and glued in the same orientation, in particular we have the following illustration of our obstruction.
\begin{exmp}[Corollary \ref{pinching}]\label{ex S4 quotient}
    Consider $\mathbb{S}^4\subset \mathbb{R}^5$ and the quotient by $\mathbb{Z}_2$ given by $(x_1,x_2,x_3,x_4,x_5) \sim (x_1,-x_2,-x_3,-x_4,-x_5)$. We will denote this space $\mathbb{S}^4\slash\mathbb{Z}_2$ which is an Einstein orbifold with two $\mathbb{R}^4\slash\mathbb{Z}_2$ singularities. The minimal resolutions of the two singularities $\mathbb{R}^4\slash\mathbb{Z}_2\approx\mathbb{C}^2\slash\mathbb{Z}_2$ has the topology $M:= \mathbb{S}^4\slash\mathbb{Z}_2\#T^*\mathbb{S}^2\#T^*\mathbb{S}^2$, where $\#$ denotes the gluing of an ALE space to an orbifold along their asymptotic cone.\\
    
\begin{minipage}{0.68\textwidth}
Then, for any $1\leqslant p<\infty$, there exists a sequence of metrics $(M,g_i)_i$ with both $$\|\Ric(g_i)-3g_i\|_{L^p(g_i)}\to 0 \text{ and }\Ric(g_i)\geqslant 3g_i$$ while $$(M,g_i)\xrightarrow[]{GH} (\mathbb{S}^4\slash\mathbb{Z}_2,g_{\mathbb{S}^4\slash\mathbb{Z}_2}),$$
    but there \emph{does not} exist any sequence of Einstein metrics satisfying $$\Ric(g_i)=3g_i,$$ and $$(M,g_i)\xrightarrow[]{GH} (\mathbb{S}^4\slash\mathbb{Z}_2,g_{\mathbb{S}^4\slash\mathbb{Z}_2}).$$
\end{minipage}
    \begin{minipage}{0.29\textwidth}
\includegraphics[width=\textwidth]{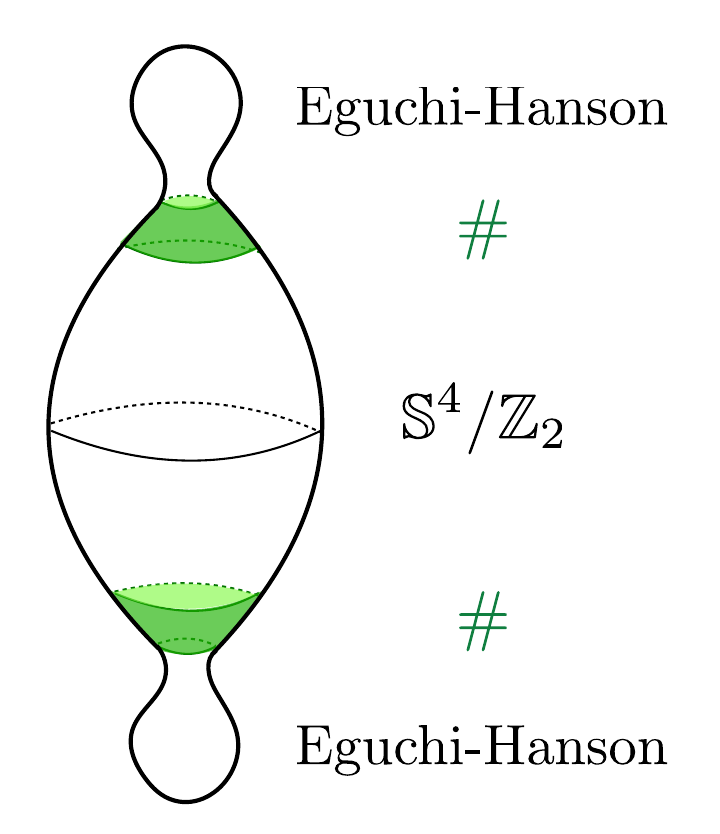}
\end{minipage}

\vspace{.5pt}
    In the same fashion, a conjecture of Anderson states that there is no sequence of asymptotically hyperbolic Einstein metrics on $T^*\mathbb{S}^2$ desingularizing the hyperbolic orbifold $\mathbb{H}^4\slash\mathbb{Z}_2$ obtained by antipodal identification in a global geodesic chart. It was proven in \cite{biq1} assuming among other things a convergence speed in weighted spaces towards the orbifold depending on the maximum of the curvature. We can prove it assuming a pointed Gromov-Hausdorff convergence together with a suitable control in weighted spaces at infinity, this time independent on the maximum of the curvature. It is again possible to desingularize $\mathbb{H}^4\slash\mathbb{Z}_2$ with Ricci pinched in any $L^p$, space for $1\leqslant p<\infty$ or with $\Ric$ bounded above or below by $-3$.
\end{exmp}

\subsubsection*{Hitchin-Thorpe inequality and degeneration of Einstein manifolds}

The Hitchin-Thorpe inequality provides a topological obstruction to the existence of Einstein metrics on a given $4$-dimensional differentiable manifold $M$,
$$2\chi(M)\geqslant 3|\tau(M)|,$$
where $\chi$ is the Euler characteristic, and $\tau$ the signature. These topological invariants have definitions adapted to orbifolds and ALE spaces which we will denote $\tilde{\chi}$ and $\tilde{\tau}$, and any orbifold $M_o$ admitting an Einstein metric satisfies
$$2\tilde{\chi}(M_o)\geqslant 3|\tilde{\tau}(M_o)|.$$
Any Gromov-Hausdorff desingularization damages this inequality, and the equality case implies the obstruction.

\begin{thm}[Theorem \ref{obst hitchin thorpe}]
    Let $(M_o,g_o)$ be an Einstein orbifold, and assume that $(M,g_i)_i$ is a sequence of Einstein metrics converging to $(M_o,g_o)$ in the Gromov-Hausdorff sense.
    
    Then, we have the following inequality,
    $$2\chi(M)-3|\tau(M)|\geqslant 2\tilde{\chi}(M_o)-3|\tilde{\tau}(M_o)|.$$
    Moreover, there is equality if and only if $M$ is a desingularization of $M_o$ by gluing of trees of Kähler Ricci-flat ALE orbifolds in the same orientation (with the same sign for $\tilde{\tau}$). 
    In this equality case, we have the condition $$\det\mathbf{R}(g_o)=0$$ at every singular point.
\end{thm}

\subsubsection*{Degeneration of Einstein metrics on spin manifolds}

Another large class of manifolds on which we can prove our obstruction is the class of $4$-manifolds admitting a spin structure.

\begin{thm}[Theorem \ref{obst spin}]
    Let $(M_i,g_i)_i$ be a sequence of \emph{spin} Einstein $4$-manifolds  converging to an Einstein orbifold $(M_o,g_o)$ in the Gromov-Hausdorff sense. Then, $(M_o,g_o)$ is spin, and at any of its singular points whose group is in $SU(2)$, we have the obstruction $$\det \mathbf{R}_{g_o} = 0.$$
\end{thm}

\paragraph{General obstructions for spherical and hyperbolic orbifolds.}

Our Theorem \ref{proximité désing naive} holds for any singularity model which might be non-Kähler. We will use it lastly to identify an obstruction to desingularizing spherical or hyperbolic orbifolds by \emph{any} Ricci-flat ALE manifold in Theorem \ref{obst generale bulle}. This provides an obstruction to any standard gluing-perturbation technique but will only imply an actual obstruction to the Gromov-Hausdorff desingularization by Ricci-flat ALE manifolds whose deformations are \emph{integrable} (this is the case of all known examples).
	\begin{thm}[{Corollary \ref{obst integrable}}]
	Spherical and hyperbolic orbifolds cannot be desingularized in the Gromov-Hausdorff sense by Ricci-flat ALE spaces which are integrable (see Definition \ref{definition integrable}).
	\end{thm}

\subsection*{Outline of the paper}
In Section 1, we give the principal definitions, and in Section 2, we introduce and motivate the function spaces we will use throughout the paper, and moreover restate the results of \cite{ozu1} thanks to them.

In Section 3, we prove that we can always pull-back an Einstein metric which is Gromov-Hausdorff close to an orbifold by a small diffeomorphism to ensure that it satisfies some gauge condition with respect to a naïve desingularization. The proof consists in a Lyapunov-Schmidt reduction in our weighted norms where the relevant operators are proven to be Fredholm.

In Section 4, we prove that any naïve desingularization can be perturbed to a metric which is Einstein modulo some obstruction, that is, an approximate cokernel of the linearization of the gauged Einstein operator. The point is that every possible Einstein metric is produced this way, and that whenever the obstructions do not vanish, it is impossible to perturb the naïve desingularization to an Einstein metric. The proof again relies on a Lyapunov-Schmidt reduction in our weighted Hölder spaces. We then extend Koiso's definition of a premoduli space in the neighborhood of a singular metric.

In Section 5, we estimate the obstructions to the above Einstein desingularization modulo obstructions. To obtain such an obstruction at \emph{all} singular points, we need to use an analysis on partial desingularizations and produce better approximations of Einstein modulo obstructions metrics.

In Section 6, we test the above obstructions on degenerations of Einstein manifolds forming trees of Kähler Ricci-flat ALE orbifolds. By developing our analysis on trees of singularities, we prove that the obstruction of \cite{biq1} for the Eguchi-Hanson metric extends to any tree of quotients of gravitational instantons and holds under a mere Gromov-Hausdorff convergence. An important step is to prove that a gluing of gravitational instantons in the same orientation can be uniformly perturbed to an Einstein metric in our norms.

In Section 7, we investigate topological conditions which ensure that a sequence of Einstein manifolds degenerating will only produce trees of Kähler Ricci-flat ALE spaces. We mainly use the result of \cite{nak} and consider the behavior of the Hitchin-Thorpe inequality as well as the degeneration of Einstein metrics on a spin manifold.

In Section 8, building on the notion of maximal volume for Ricci-flat ALE spaces of \cite{bh}, we prove that even without assuming that the trees of singularities are Kähler, there is a non vanishing obstruction to the desingularization of spherical and hyperbolic orbifolds. We can however only prove that this is a Gromov-Hausdorff obstruction under the technical assumption that the Ricci-flat ALE spaces have integrable deformations.

\subsection*{Acknowledgements}

I would like to thank my PhD advisor, Olivier Biquard, for his constant support as well as his mathematical and writing advice. I would also like to thank Aaron Naber for inviting me to Northwestern University in Spring 2017, and for pointing out and discussing the issue addressed in Section 8. I would also like to thank Gilles Carron and Michael Singer for their comments and suggestions when reviewing this work as a part of my PhD thesis. I would finally like to thank Aude Genevay for her help with the graphics.

\section{Orbifolds, ALE spaces and naïve desingularizations}

Let us start by defining the objects we will use throughout this article.

\subsection{Einstein orbifolds and ALE spaces}

For $\Gamma$ a finite subgroup of $SO(4)$ acting freely on $\mathbb{S}^3$, let us denote $(\mathbb{R}^4\slash\Gamma,g_e)$ the flat orbifold obtained by the quotient by the action of $\Gamma$, and $r_e:= d_e(.,0)$.

\begin{defn}[Orbifold (with isolated singularities)]\label{orb Ein}
    We will say that a metric space $(M_o,g_o)$ is an orbifold of dimension $n\in\mathbb{N}$ if there exists $\epsilon_0>0$ and a finite number of points $(p_k)_k$ of $M_o$ called \emph{singular} such that we have the following properties:
    \begin{enumerate}
        \item the space $(M_o\backslash\{p_k\}_k,g_o)$ is a manifold of dimension $n$,
        \item for each singular point $p_k$ of $M_o$, there exists a neighborhood of $p_k$, $ U_k\subset M_o$, a finite subgroup acting freely on $\mathbb{S}^{n-1}$, $\Gamma_k\subset SO(n)$, and a diffeomorphism $ \Phi_k: B_e(0,\epsilon_0)\subset\mathbb{R}^n\slash\Gamma_k \to U_k\subset M_o $ for which, the pull-back of $\Phi_k^*g_o$  on the covering $\mathbb{R}^n$ is smooth.
    \end{enumerate}
\end{defn}
\begin{rem}\label{analysis orbifold}
    Consequently, the analysis on an orbifold is exactly the same as the analysis on a manifold up to using finite local coverings at the singular points.
\end{rem}

\begin{defn}[The function $r_o$ on an orbifold]\label{ro}
    We define $r_o$, a smooth function on $M_o$ satisfying $r_o:= (\Phi_k)_* r_e$ on each $U_k$, and such that on $M_o\backslash U_k$, we have $\epsilon_0\leqslant r_o<1$ (the different choices will be equivalent for our applications).
    
    We will denote, for $0<\epsilon\leqslant\epsilon_0$, $$M_o(\epsilon):= \{r_o>\epsilon\} = M_o\backslash  \Big(\bigcup_k \Phi_k\big(\overline{B_e(0,\epsilon)}\big) \Big).$$
\end{defn}

Let us now turn to ALE Ricci-flat metrics.

\begin{defn}[ALE orbifold (with isolated singularities)]\label{def orb ale}
    An ALE orbifold of dimension $n\in\mathbb{N}$, $(N,g_{b})$ is a metric space for which there exists $\epsilon_0>0$, singular points $(p_k)_k$ and a compact $K\subset N$ for which we have:
    \begin{enumerate}
        \item $(N,g_b)$ is a orbifold of dimension $n$,
        \item there exists a compact subset $K\subset N$ and a diffeomorphism $\Psi_\infty: (\mathbb{R}^n\slash\Gamma_\infty)\backslash \overline{B_e(0,\epsilon_0^{-1})} \to N\backslash K$ such that we have $$r_e^l|\nabla^l(\Psi_\infty^* g_b - g_e)|_{g_e}\leqslant C_l r_e^{-n}.$$
    \end{enumerate}
\end{defn}

\begin{defn}[The function $r_{b}$ on an ALE orbifold]\label{rb}
We define $r_{b}$ a smooth function on $N$ satisfying $r_{b}:= (\Psi_k)_* r_e$ on each $U_k$, and $r_{b}:= (\Psi_\infty)_* r_e$ on $U_\infty$, and such that $\epsilon_0\leqslant  r_{b}\leqslant \epsilon_0^{-1}$ on the rest of $N$ (the different choices are equivalent for our applications).

For $0<\epsilon\leqslant\epsilon_0$, we will denote $$N(\epsilon):= \{\epsilon<r_b<\epsilon^{-1}\} = N\backslash  \Big(\bigcup_k \Psi_k\big(\overline{B_e(0,\epsilon)}\big) \cup \Psi_\infty \big((\mathbb{R}^4\slash\Gamma_\infty)\backslash B_e(0,\epsilon^{-1})\big)\Big).$$
\end{defn}
Now, consider a subset $S_o$ of the singular points of $M_o$ (respectively $S$ of $N$).
\begin{defn}[Functionals $r_{o,S_o}$ and $r_{b,S}$]\label{def roSo rbS}
    We define the functional $r_{o,S_o}$ (respectively $r_{b,S}$) exactly like in Definitions \ref{ro} (respectively \ref{rb}) by only considering the sets $U_k$ containing points of $S_o$ (respectively $S$).
\end{defn}

\subsection{Naïve desingularizations}\label{reecriture controle}

Let us now recall the definition of a naïve desingularization of an orbifold from \cite{ozu1}.

\paragraph{Gluing of ALE spaces to orbifold singularities.}

    Let $0<2\epsilon<\epsilon_0$ be a fixed constant, $t>0$, $(M_o,g_o)$ an orbifold and $\Phi: B_e(0,\epsilon_0)\subset\mathbb{R}^4\slash\Gamma \to U$ a local chart of Definition \ref{orb Ein} around a singular point $p\in M_o$. Let also $(N,g_b)$ be an ALE orbifold asymptotic to $\mathbb{R}^4\slash\Gamma$, and $\Psi_\infty: (\mathbb{R}^4\slash\Gamma)\backslash B_e(0,\epsilon_0^{-1}) \to N\backslash K$ a chart at infinity of Definition \ref{def orb ale}.
    
    Define $s>0$, $\phi_s: x\in \mathbb{R}^4\slash\Gamma\to sx \in \mathbb{R}^4\slash\Gamma$. For $t<\epsilon_0^4$, we define $M_o\#N$ as $N$ glued to $M_o$ thanks to the diffeomorphism $$ \Phi\circ\phi_{\sqrt{t}}\circ\Psi^{-1} : \Psi(A_e(\epsilon_0^{-1},\epsilon_0t^{-1/2}))\to \Phi(A_e(\epsilon_0^{-1}\sqrt{t},\epsilon_0)).$$ Consider moreover $\chi:\mathbb{R}^+\to\mathbb{R}^+$, a $C^\infty$ cut-off function supported on $[0,2]$ and equal to $1$ on $[0,1]$.
    
\begin{defn}[Naïve gluing of an ALE space to an orbifold]\label{def naive desing}
    We define a \emph{naïve gluing} of $(N,g_b)$ at scale $0<t<\epsilon^4$ to $(M_o,g_o)$ at the singular point $p$, which we will denote $(M_o\#N,g_o\#_{p,t}g_b)$ by putting $g_o\#_{p,t}g_b=g_o$ on $M\backslash U$, $g_o\#_{p,t}g_b=tg_b$ on $K$, and 
    $$g_o\#_{p,t}g_b =  \chi(t^{-\frac{1}{4}}r_e)\Psi_\infty^*g_b + \big(1-\chi(t^{-\frac{1}{4}}r_e)\big)\Phi^*g_o$$
    on $\mathcal{A}(t,\epsilon):=A_e(\epsilon^{-1}\sqrt{t},2\epsilon)$.
\end{defn}
\begin{figure}[h]
\centering\includegraphics[width=8cm]{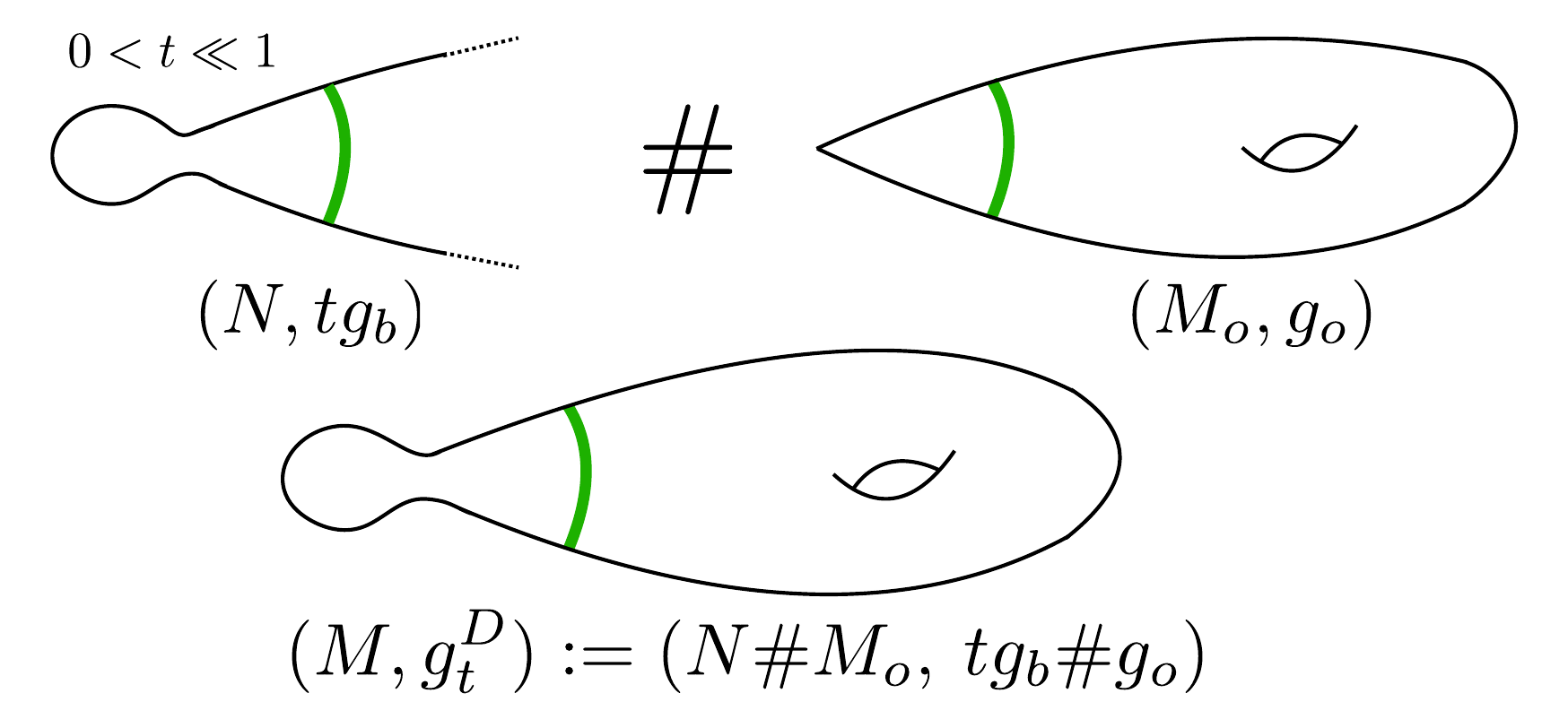}
\end{figure}
\begin{rem}\label{jauge recollement}
    It is possible to compose $\phi_{\sqrt{t}}$ with any isometry of $\mathbb{R}^4\slash\Gamma$. This is equivalent to gluing a different Ricci-flat ALE metric.
\end{rem}

More generally, it is possible to desingularize iteratively by trees of Ricci-flat ALE orbifolds. Consider $(M_o,g_o)$ an Einstein orbifold (the index $o$ stands for orbifold), and $S_o$ a subset of its singular points and $(N_j,g_{b_j})_j$ (the index $b_j$ stands for $j$-th bubble) a family of Ricci-flat ALE spaces asymptotic at infinity to $\mathbb{R}^4\slash\Gamma_j$ and $(S_{b_j})_j$ a subset of their singular points. Let us finally assume that there is a one to one map  $p:j\mapsto p_j\in S_o\cup \bigcup_k S_{b_k}$, where the singularity at $p_j$ is $\mathbb{R}^4\slash\Gamma_j$. We will call $D:= \big((M_o,g_o,S_o),(N_j,g_{b_j},S_{b_j})_j, p\big)$ a \emph{desingularization pattern}.

\begin{defn}[Naïve desingularization by a tree of singularities]\label{def total desing}
    Let $0<2\epsilon<\epsilon_0$, $D$ be a desingularization pattern for $(M_o,g_o)$, and let $0<t_j<\epsilon^4$ be relative gluing scales. The metric $g^D_t$ is then the result of the following finite iteration: 
    \begin{enumerate}
        \item start with a deepest bubble $(N_j,g_{b_j})$, that is, $j$ such that $S_j= \emptyset$,
        \item  if $p_j\in N_k$, replace $(N_k,g_{b_k},S_j)$ and $(N_j,g_{b_j},\emptyset)$ in $D$ by $(N_k\#N_j,g_{b_k}\#_{p_j,t_j}g_{b_j},S_k\backslash\{p_j\})$ and restrict $p$ as $l\to p_l$ for $l\neq j$ in $D$ and consider another deepest bubble. The same works if $p_j\in M_o$.
        \item choose another deepest bubble and do the same.
    \end{enumerate}
    
    For $t = (t_j)_j$, if $N_j$ is glued to $p_j\in N_{j_1}$, and $N_{j_1}$ is glued to $p_{j_1}\in N_{j_2}$, ..., $N_{j_{k-1}}$ is glued to $N_{j_k}$, which is glued to $M_o$, we define $T_j:= t_{j_1}t_{j_2}...t_{j_k}$. This way, on each $N_j^{16t}$, the metric is $T_jg_{b_j}$.
\end{defn}
\begin{figure}[h]
\centering\includegraphics[width=8cm]{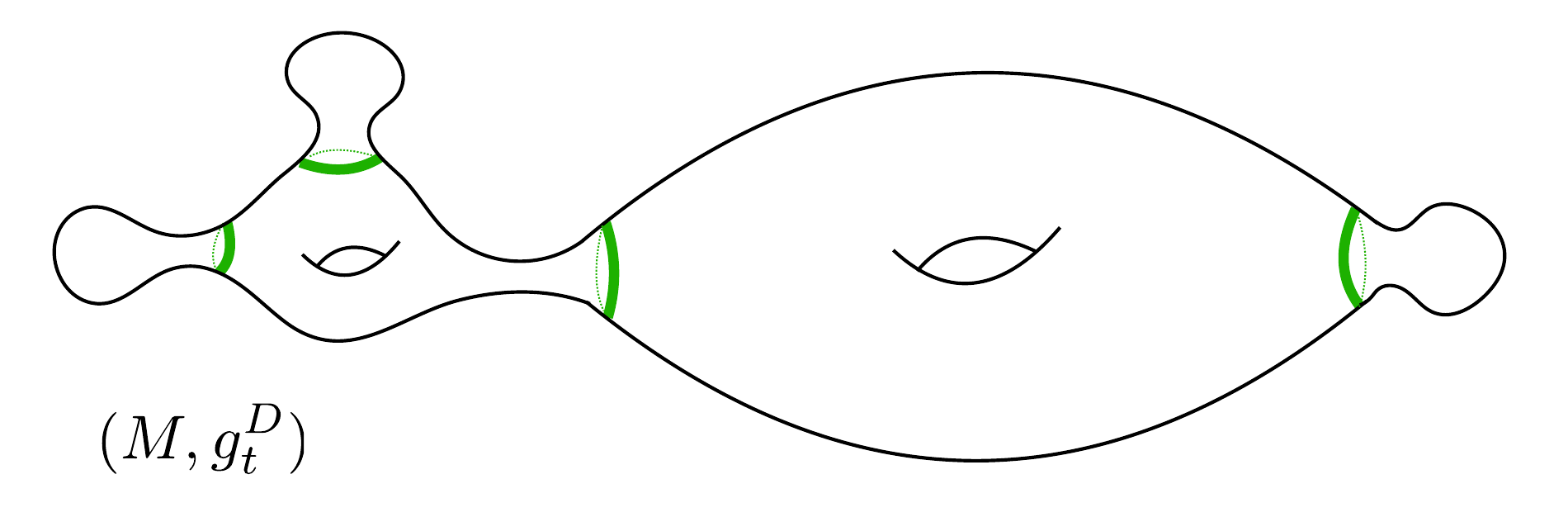}
\end{figure}
 Let $(M_o,g_o)$ be an Einstein orbifold, and $(M,g^D)$ a naïve desingularization of $(M_o,g_o)$ by a tree of ALE Ricci-flat orbifolds $(N_j,g_{b_j})$ glued at scales $T_j>0$. 

Here, the manifold $M$ is also covered as $M = M_o^t\cup \bigcup_jN_j^t$, where
\begin{equation}
    M_o^t: = M_o\backslash  \big(\bigcup_k \Phi_k(B_e(0,t_k^\frac{1}{4})) \big),\label{def Mot}
\end{equation}
 where $t_k>0$ is the relative gluing scale of $N_k$ at the singular point $p_k\in M_o$, and where 
 \begin{equation}
     N_j^t:= \big(N_j\backslash \Psi_\infty \big((\mathbb{R}^4\slash \Gamma_\infty)\backslash B_e(0, 2t_j^{-\frac{1}{4}}) \big)\big) \backslash \big(\bigcup_k \Psi_k(B_e(0,t_k^\frac{1}{4}) \big).\label{def Njt}
 \end{equation}
On $M_o^{16t}\subset M_o^t$, we have $g^D = g_o$ and on each $N_j^{16t}\subset N_j^t$, we have $g^D = T_j g_{b_j}$. We also define $t_{\max}:= \max_j t_j$. By Definition \ref{def naive desing}, on the intersection $N_j^t\cap M_o^t$ we then have $\sqrt{T_j}r_{b_j} = r_o $, and on the intersection $N_j^t\cap N_k^t$, we have $\sqrt{T_j}r_{b_j} = \sqrt{T_k}r_{b_k}$.

\begin{defn}[Function $r_D$ on a naïve desingularization]
    On a naïve desingularization $(M,g^D)$, we define a function $r_D$ in the following way:
    \begin{enumerate}
        \item $r_D = r_o$ on $M_o^t$,
        \item $r_D = \sqrt{T_j}r_{b_j}$ on each $N_j^t$.
    \end{enumerate}
    The function $r_D$ is smooth on $M$.
\end{defn}

\begin{defn}[{Neck regions, $\mathcal{A}_k(t,\epsilon)$}]\label{def neck region}
    Let $(N_k,g_{b_k})$ be a Ricci-flat ALE orbifold of the above tree of singularities. We define $\mathcal{A}_k(t,\epsilon)$ as the connected region with $\epsilon^{-1}\sqrt{T_k}<\sqrt{T_k}r_{b_k}=r_D<\epsilon t_k^{-\frac{1}{2}}\sqrt{T_k}$ with a nonempty intersection with $N_k^t$.
\end{defn}

\begin{figure}[h]
\centering\includegraphics[width=8cm]{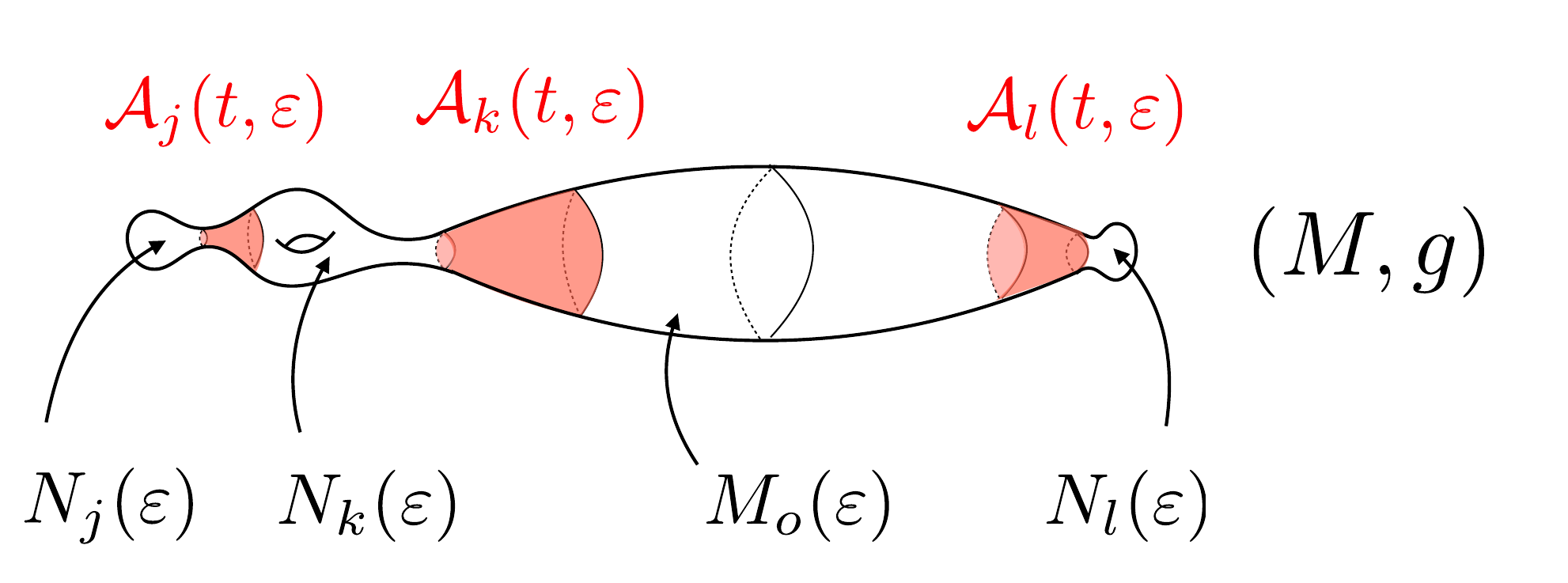}
\end{figure}

\begin{defn}[Cut-off functions $\chi_{M_o^t}$, $\chi_{N_j^t}$, $\chi_{\mathcal{A}_k(t,\epsilon)}$ and $\chi_{B(p_k,\epsilon)}$]\label{def cutoffs all}
    We define the following cut-off functions thanks to the cut-off function $\chi$ used in Definition \ref{def naive desing}.
    \begin{itemize}
        \item $\chi_{M_o^t}$, equal to $1$ on $M_o^{16t}$ and equal to $1-\chi(t_k^{-\frac{1}{4}}r_o)$ on each annulus $\mathcal{A}_k(t,\epsilon)$. It is supported on $M_o^t$.
        \item $\chi_{N_j^t}$, equal to $1$ on $N_j^{16t}$ and equal to $1-\chi(t_k^{-\frac{1}{4}}r_{b_j})$ on each annulus $\mathcal{A}_k(t,\epsilon)$ at its singular points and $\chi(t_j^{\frac{1}{4}}r_{b_j})$ in a neighborhood of infinity. It is supported on $N_j^t$.
        \item $\chi_{\mathcal{A}_k(t,\epsilon)}$ equal to $1$ on $\mathcal{A}_k(t,\frac{1}{2}\epsilon)$, and equal to $\chi(\epsilon^{-1}t_k^{\frac{1}{2}}r_{b_k}) - \chi(\epsilon r_{b_k})$. It is supported on $\mathcal{A}_k(t,\epsilon)$.
        \item $\chi_{B(p_k,\epsilon)}$ for $p_k\in M_o$ equal to $1$ on $r_o<\epsilon$ equal to $\chi(\epsilon^{-1}r_o)$ around $p_k$. It is supported in  supported in $r_o<2\epsilon$ around $p_k$.
    \end{itemize}
    Note that since $t_k^{\frac{1}{2}}r_{b_k} = r_o$ on the gluing region, we have $\chi_{B(p_k,\epsilon)}-\chi_{\mathcal{A}_k(t,\epsilon)} = \chi(\epsilon r_{b_k}) = \chi(\epsilon t_k^{-\frac{1}{2}} r_{o})$.
    
    The definition extends to deeper Ricci-flat ALE orbifolds thanks to the iteration of Definition \ref{def total desing}.
\end{defn}
This in particular yields a partition of unity,
\begin{equation}
    1 = \chi_{M_o^t} + \sum_j \chi_{N_{j}^t}.\label{cut off Mot Njt}
\end{equation}

\section{Weighted Hölder spaces and decoupling norms}

We now present the spaces in which the analysis of the rest of the article takes place.

\subsection{Weighted Hölder spaces}

Let us construct weighted Hölder spaces adapted to our situation. Let $(M,g^D_{t})$ be a naïve desingularization of an orbifold $(M_o,g_o)$ by Ricci-flat ALE orbifolds $(N_j,g_{b_j})$ at scales $T_j>0$. 

\subsubsection{Weighted Hölder spaces on orbifolds and ALE spaces}

Let us first define weighted spaces on manifolds asymptotic to cones or with conical singularities. For a tensor $s$, a point $x$, $\alpha>0$ and a metric $g$, if we denote $\exp_x$ the exponential map at $x$ whose injectivity radius is $\textup{inj}_g(x)$, we define the Hölder seminorm of $s$ on $M$ as
$$ [s]_{C^\alpha(g)}(x):= \sup_{\{y\in T_xM,|y|< \textup{inj}_g(x)\}} \Big| \frac{\exp_x^*s(0)-\exp_x^*s(y)}{|y|^\alpha} \Big|_{\exp_x^*g}.$$

For orbifolds, we will consider a norm which is bounded for tensors decaying at the singular points.
\begin{defn}[Weighted Hölder norms on an orbifold]\label{norm orbifold}
    Let $\beta\in \mathbb{R}$, $k\in\mathbb{N}$, $0<\alpha<1$ and $(M_o,g_o)$ an orbifold. Then, for all tensor $s$ on $M_o$, we define
    \begin{align*}
        \| s \|_{C^{k,\alpha}_{\beta}(g_o)} &:= \sup_{M_o}r_o^{-\beta}\Big(\sum_{i=0}^k r_o^{i}|\nabla_{g_o}^i s|_{g_o} + r_o^{k+\alpha}[\nabla_{g_o}^ks]_{C^\alpha(g_o)}\Big).
    \end{align*}
\end{defn}
\begin{rem}
    The injectivity radius at a point $x\in M_o$ is equivalent to $r_{o}$.
\end{rem}

For ALE orbifolds, we will consider a norm which is bounded for tensors decaying at infinity and at the singular points.
\begin{defn}[Weighted norm on ALE orbifolds]\label{norm ALE}
For $\beta\in \mathbb{R}$, $k\in\mathbb{N}$ and $0<\alpha<1$ on an orbifold ALE $(N,g_b)$, we define
   \begin{align*}
       \| s \|_{C^{k,\alpha}_{\beta}(g_b)}:= \sup_{N}\Big\{\max(r_b^\beta,r_b^{-\beta})\Big( \sum_{i=0}^kr_b^{i}|\nabla_{g_b}^i s|_{g_b} + r_b^{k+\alpha}[\nabla_{g_b}^ks]_{C^\alpha({g_b})}\Big)\Big\}.
   \end{align*}
\end{defn}
\begin{rem}
    The injectivity radius at a point $x\in N$ is equivalent to $r_{b}$.
\end{rem}

\begin{rem}\label{rem norm partial orb ALE}
    We similarly define the norms for $C^{k,\alpha}_{\beta}(g_o,S_o)$ and $C^{k,\alpha}_{\beta}(g_b,S)$ by replacing $r_o$ and $r_b$ by $r_{o,S_o}$ and $r_{b,S}$ of Definition \ref{def roSo rbS}.
\end{rem}

\subsubsection{Weighted Hölder spaces on trees of singularities}

Let us assume that $(M,g^D_t)$ is a naïve desingularization of $(M_o,g_o)$ by a tree of singularities $(N_j,g_{b_j})$. For $t_{\max}:= \max_j t_j<\epsilon_0^4$, for $\epsilon_0>0$ the constant of Subsection \ref{reecriture controle} only depending on $g_o$ and the $g_{b_j}$, we define the global weighted norm in the following way.

\begin{defn}[Weighted Hölder norm on a naïve desingularization]\label{norme a poids M}
		Let $\beta\in \mathbb{R}$ and $k\in\mathbb{N}$, $0<\alpha<1$. We define for $s\in TM^{\otimes l_+}\otimes T^*M^{\otimes l_-}$ a tensor of type $(l_+,l_-)\in \mathbb{N}^2$, with $l:= l_+-l_-$ the associated conformal weight.
		\begin{align}
		    \|s\|_{C^{k,\alpha}_{\beta}(g^D)}:=\| \chi_{M_o^t} s \|_{C^{k,\alpha}_{\beta}(g_o)} + \sum_j T_j^\frac{l}{2}\|\chi_{N_{j}^t}s\|_{C^{k,\alpha}_{\beta}(g_{b_j})}.\label{def norme poids avec g_o et g_b}
		\end{align}
\end{defn}

\begin{rem}
    The factor $T_j^\frac{l}{2}$ in \eqref{def norme poids avec g_o et g_b} comes from the fact that on $N_j^t$, the metric $g^D$ is close to $T_j g_{b_j}$. For a tensor $s$ of conformal weight $l$, we have $|s|_{T_jg_{b_j}} = T_j^{\frac{l}{2}}|s|_{g_{b_j}},$ and therefore $T_j^\frac{l}{2}\|\chi_{N_{j}^t}s\|_{C^{k,\alpha}_{0}(g_{b_j})}=\|\chi_{N_{j}^t}s\|_{C^{k,\alpha}_{0}(T_jg_{b_j})}$.
\end{rem}

\begin{rem}
For a function, being bounded for this norm means being bounded in $C^0_\beta$-norm means being bounded everywhere and having a particular decay in the neck regions.
\begin{figure}[h]
\centering\includegraphics[width=8cm]{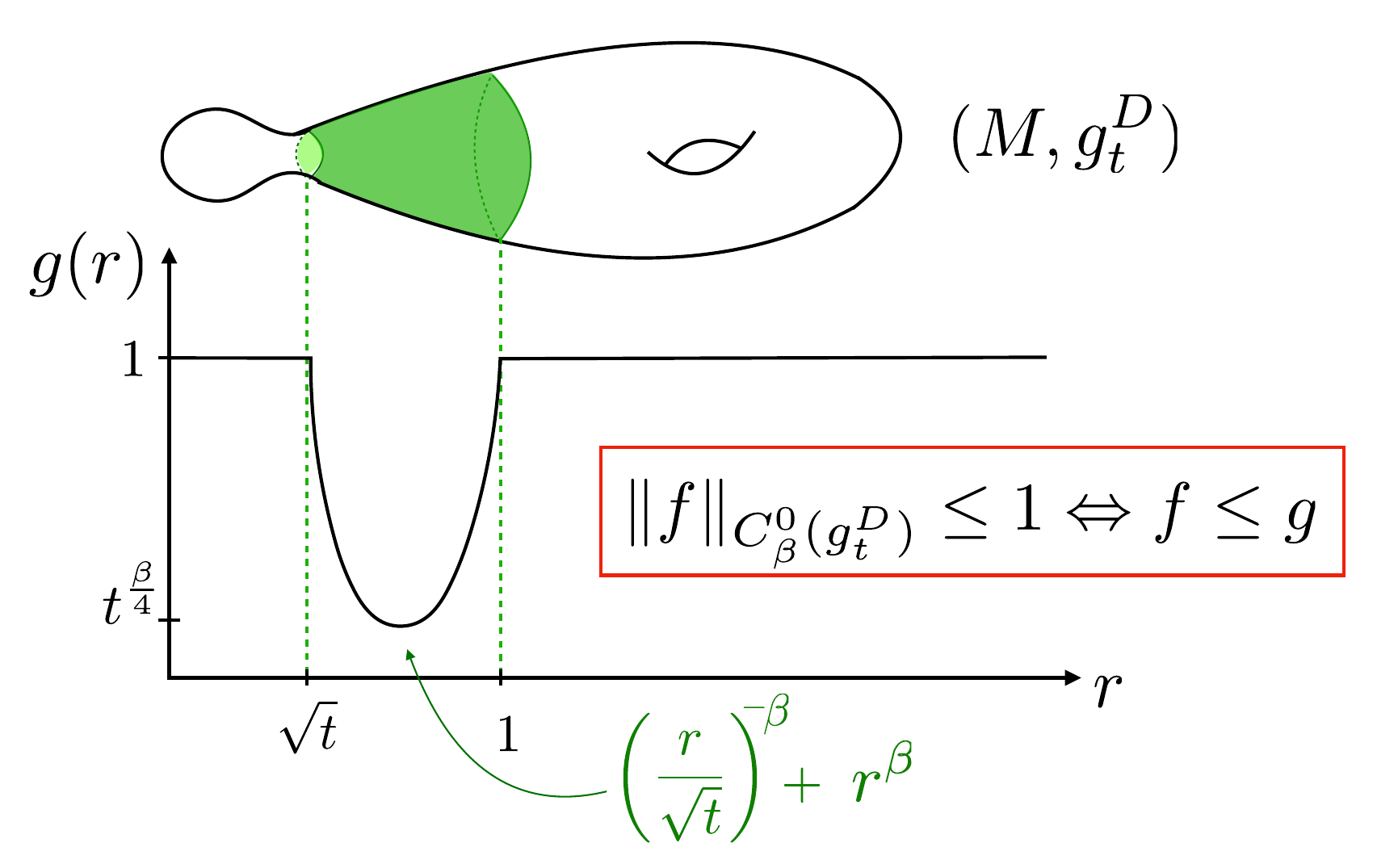}
\end{figure}
\\

Its main advantage is that it is totally adapted to trees of singularities.
\begin{figure}[h]
\centering\includegraphics[width=8cm]{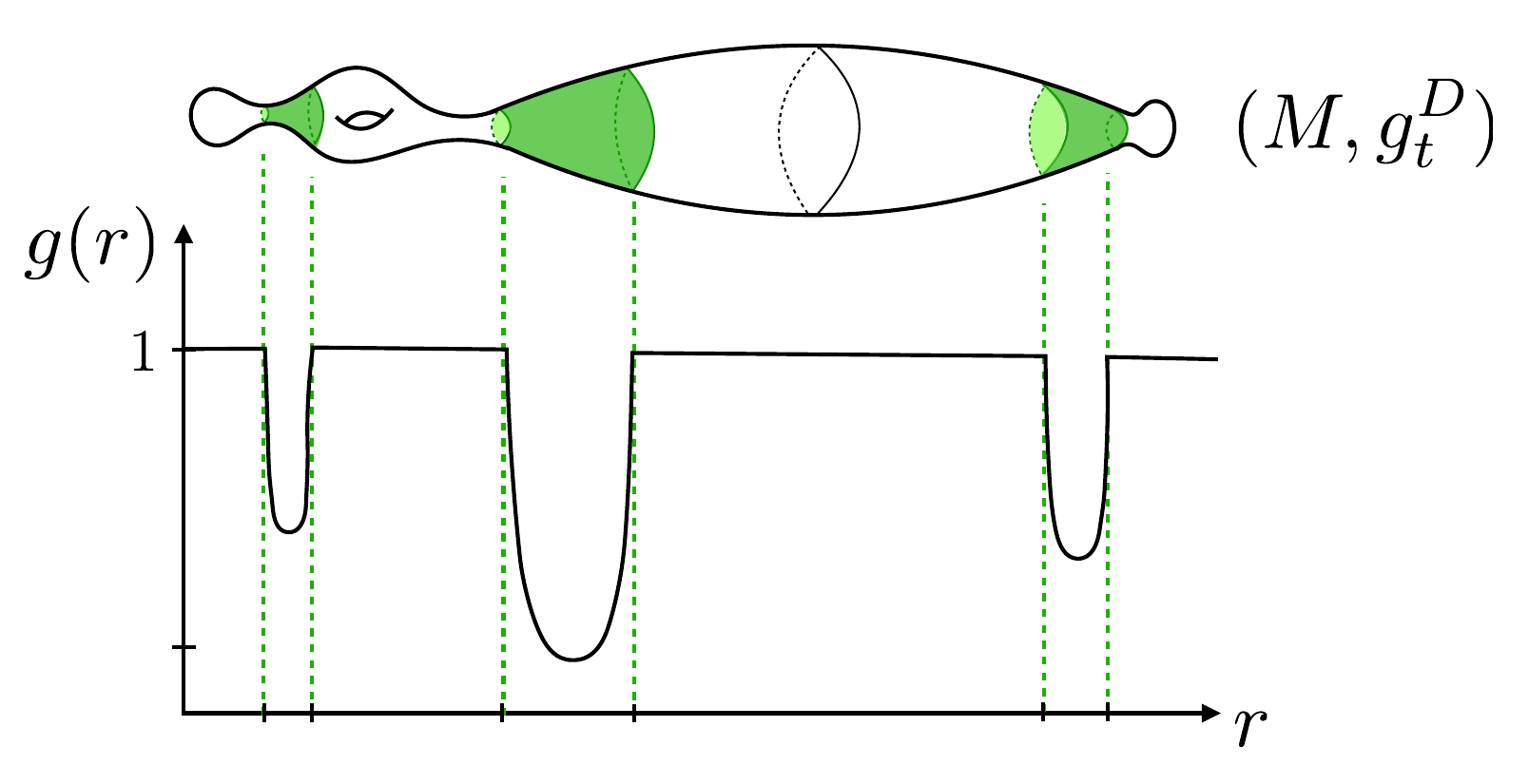}
\end{figure}

\end{rem}

\begin{rem}\label{rem partial desing}
    If $(M,g^D)$ is a partial desingularization at the singular points of $S_o$ and $S_j$, we define the $C^{k,\alpha}_{\beta}(g^D)$-norm similarly thanks to $C^{k,\alpha}_{\beta}(g_o,S_o)$ and $C^{k,\alpha}_{\beta}(g_b,S)$ of Remark \ref{rem norm partial orb ALE}.
\end{rem}

Thanks to this norm, we can for example rewrite and extend the statement of \cite[Theorem 6.4]{ozu1}.

\begin{cor}\label{GH to C3}
	Let $D_0,v_0>0$, $l\in \mathbb{N}$, and $\beta= \beta(v_0,D_0)>0$ obtained in \cite[Theorem 6.4]{ozu1}. Then, for all $\epsilon>0$, there exists $\delta = \delta(\epsilon,D_0,v_0,l) >0$ such that if $(M,g^\mathcal{E})$ is an Einstein manifold satisfying
	\begin{itemize}
		\item the volume is bounded below by $v_0>0$,
		\item the diameter is bounded above by $D_0$,
		\item the Ricci curvature is bounded $|\Ric|\leq 3$.
	\end{itemize}
	and for an Einstein orbifold $(M_o,g_o)$,
	$$d_{GH}\big((M,g^\mathcal{E}),(M_o,g_o)\big)\leqslant \delta,$$
	then, there exists a naïve desingularization $(M,g^D)$ of $(M_o,g_o)$ by a tree of singularities and a diffeomorphism $\phi: M\to M$ such that $$\big\|\phi^*g^\mathcal{E}-g^D\big\|_{C^l_{\beta}(g^D)}\leqslant \epsilon.$$
\end{cor}
\begin{proof}
    Let $l\in \mathbb{N}$. Let us give a proof by contradiction and consider a sequence of counter examples, that is a sequence of Einstein manifolds $(M_i,g_i)_i$ such that
    \begin{itemize}
        \item $\vol(g_i)\geqslant v_0>0$,
        \item $diam(g_i)<D_0$ and
        \item $|\Ric(g_i)|_{g_i} \leqslant 3$
    \end{itemize}
    converging in the Gromov-Hausdorff sense to an Einstein orbifold $(M_o,g_o)$, but such that there exists $\epsilon>0$ for which, for all $i\in \mathbb{N}$ and any naïve desingularization $(M_i,g^D_i)$ of $(M_o,g_o)$, and all diffeomorphism $\Phi_i:M_i \to M_i$, we have $\|\Phi_i^*g_i-g^D_i\|_{C^l_\beta(g^D)}>\epsilon$.
    
    According to \cite[Theorem 6.4]{ozu1}, this implies that there exists a subsequence $(M,g_i)_i$ with fixed topology, and a sequence $(M,g^D_i)_i$ contradicting the assumption for $i$ large enough by definition of the weighted norm.
\end{proof}

On the annuli of low curvature $\mathcal{A}_k(t,\epsilon_0)$ pulled back on flat annuli $A_e(\rho_1,\rho_2)\subset \mathbb{R}^4\slash\Gamma$, the weighted norm on $(M,g^D)$ is equivalent is equivalent to a particular norm which allows us to control \emph{independently of the radii} the sum of tensors decaying at the center of the annulus and of tensors decaying at infinity.

	\begin{defn}[Weighted norm adapted to an annulus]\label{def poids euclidien}
Let $0<\rho_1<\rho_2$, $\beta\in \mathbb{R}$, $k\in\mathbb{N}$, $0<\alpha<1$ and a tensor $s$ on $(A_e(\rho_1,\rho_2),g_e)$. We define
\begin{equation}
    \eta(r_e):=\max\Big(\big(\frac{\rho_1}{r_e}\big)^{\beta},\big(\frac{r_e}{\rho_2}\big)^{\beta}\Big)\leqslant 1,\label{def eta}
\end{equation}
 and the norm,
\begin{align*}
	 \|s\|_{C^{k,\alpha}_{\beta}(A_e(\rho_1,\rho_2))}:&= \sup_{A_e(\rho_1,\rho_2)}\Big[ \eta^{-1}(r_e)\Big( \sum_{i=0}^k r_e^{i}|\nabla_{g_e}^is|_{g_e} + r_e^{k+\alpha}[\nabla_{g_e}^ks]_{C^\alpha(g_e)}\Big)\Big].
		\end{align*}
\end{defn}

In the rest of this article, we will often use spaces denoted $fC^{k,\alpha}_{\beta}$ for a positive function $f$. They will always be equipped with the following norm
$$\|s\|_{fC^{k,\alpha}_\beta}:= \Big\|\frac{s}{f}\Big\|_{C^{k,\alpha}_\beta}.$$
\begin{rem}
    By definition of $r_D$, for all $m$, there exists a constant $C>0$ only depending on the cut-off functions of Definition \ref{def cutoffs all} such that
\begin{equation}
    \frac{1}{C}\|s\|_{r^m_DC^{k,\alpha}_{\beta}(g^D)} \leqslant \| \chi_{M_o^t} s \|_{r^m_oC^{k,\alpha}_{\beta}(g_o)} + \sum_j T_j^\frac{l-m}{2}\|\chi_{N_{j}^t}s\|_{r^m_{b_j}C^{k,\alpha}_{\beta}(g_{b_j})}\leqslant C\|s\|_{r^m_DC^{k,\alpha}_{\beta}(g^D)}.\label{Estimation rD ro rbj}
\end{equation}
\end{rem}

\begin{rem}
    The metric $g^D$ is equal to $g_o$ on $\{r_D>\epsilon\}\cap M_o^t$, and to $T_jg_{b_j}$ on $\{\sqrt{T_j}\epsilon<r_D<2\sqrt{T_j}\epsilon^{-1}\}\cap N_j^t$. Since on the $\mathcal{A}_k(t,\epsilon)$ between $N_k$ and $N_j$ (resp. $M_o$) identified with $A_e(\epsilon^{-1}\sqrt{t_kT_j},\epsilon\sqrt{T_j})$ (resp. $A_e(\epsilon^{-1}\sqrt{t_k},\epsilon)$), $g^D$ is arbitrarily close to $g_e$, we see that defining $\tilde{\eta} : M \mapsto \mathbb{R}^+$ a function equal to 
    \begin{itemize}
        \item $1$ on $\{r_D>\epsilon\}\cap M_o^t$ and on $\{\sqrt{T_j}\epsilon<r_D<2\sqrt{T_j}\epsilon^{-1}\}\cap N_j^t$ and
        \item equal to the function $\eta\leqslant 1$ defined in \eqref{def eta} on the associated euclidean annulus $A_e(\epsilon^{-1}\sqrt{t_kT_j},\epsilon\sqrt{T_j})$ (or $A_e(\epsilon^{-1}\sqrt{t_k},\epsilon)$).
    \end{itemize}
    Then, the norm $r_D^mC^{k,\alpha}_\beta(g^D)$ is equivalent (independently of $t$) to the norm which to a tensor $s$ associates
    \begin{equation}
        \|s\|=\sup_{M}r_D^{-m}\tilde{\eta}(r_D)^{-1}\Big(\sum_{i=0}^k r_D^{i}|\nabla_{g^D}^i s|_{g^D} + r_D^{k+\alpha}[\nabla_{g^D}^ks]_{C^\alpha(g^D)}\Big).\label{expression poids equivalent 1}
    \end{equation}
\end{rem}

\begin{rem}\label{comportement norme}
    Let $\beta\leqslant \beta'$, $k+\alpha\leqslant k'+\alpha'$, and $m,m'\in \mathbb{Z}$. 

    For all the previously mentioned weighted Hölder spaces generically denoted $r^m C^{k,\alpha}_\beta$, we have the following properties: for any tensors $s$ and $s'$
    \begin{itemize}
        \item $\|s\|_{r^mC^{k,\alpha}_\beta}\leqslant\|s\|_{r^mC^{k',\alpha'}_{\beta'}}$,
        \item $\|\nabla^ks\|_{r^mC^{k'-k,\alpha}_\beta}\leqslant \|s\|_{r^{m+k}C^{k',\alpha}_\beta}$
        \item if $*$ is a composition, a product of a contraction of tensors, there exists $C= C(*, k , \alpha)>0$ such that 
        \begin{equation}
            \| s*s' \|_{r^{m+m'}C^{k,\alpha}_{\beta+\beta'}}\leqslant C \|s\|_{r^{m}C^{k,\alpha}_{\beta}}\|s'\|_{r^{m'}C^{k,\alpha}_{\beta'}}.\label{behavior norm quadratic}
        \end{equation}
    \end{itemize}
    Let us give an explanation for that last inequality \eqref{behavior norm quadratic} for bilinear operations and assume first that $k=0,$ $\alpha=0$ and consider depending on the situation 
    \begin{enumerate}
        \item $(w,w') = (r_o^{-\beta},r_o^{-\beta'})$ on $(M_o,g_o)$,
        \item $(w,w') = \big(\max(r_b^\beta,r_b^{-\beta}),\max(r_b^{\beta'},r_b^{-\beta'})\big)$ on $(N,g_b)$
        \item or $(w,w')=({\tilde{\eta}}^{-1},\left.\tilde{\eta}'\right.^{-1})$ with the weights $\tilde{\eta}$ and $ {\tilde{\eta}'}$ used in \eqref{expression poids equivalent 1} above respectively associated to $\beta$ and $\beta'$ on $(M,g^D)$.
    \end{enumerate}
    The goal is to bound 
    $w.w'r^{-m-m'}|s*s'|$ uniformly and this is done using the definitions of the norms which yield
    $w r^{-m} |s|\leqslant \|s\|_{r^mC^0_\beta}$ and
    ${w'}r^{-m'}| s'|\leqslant \|s'\|_{r^{m'}C^0_{\beta'}}$ for any of the above spaces. The derivatives are treated thanks to the second above inequality and Leibniz rule.
\end{rem}

\subsubsection{Weighted Schauder estimates}

Weighted Schauder estimates hold in these norms for the operator $P:= \frac{1}{2}\nabla^*\nabla-\mathring{\R}$. 

\begin{prop}\label{Estimations elliptic P}
    For all $\beta>0$ and $0<\alpha<1$, there exists $C>0$ and $\epsilon>0$  such that if $h$ is a symmetric $2$-tensor on $(M,g^D)$, and $g$ a metric on $M$ satisfying $$\|g-g^D\|_{C^{2,\alpha}_\beta(g^D)}\leqslant \epsilon,$$
    then, we have
    $$\|h\|_{C^{2,\alpha}_{\beta}(g^D)}\leqslant C \big(\|P_gh\|_{r^{-2}_DC^{\alpha}_\beta(g^D)}+\|h\|_{C^0_{\beta}(g^D)}\big).$$
\end{prop}
\begin{proof}
    Let $g$ be a metric on $M$ satisfying $\|g-g^D\|_{C^{2,\alpha}_\beta(g^D)}\leqslant \epsilon$, for $\epsilon>0$ which we will choose small enough along the proof.
    On the compacts $M_o(\epsilon_0)$ of the orbifold and $N_j(\epsilon_0)$ of the ALE orbifolds minus their singular points, we have an elliptic estimate for the operators $P_{g_o}$ and $P_{g_{b_j}}$: there exists $C_1>0$ such that for any symmetric $2$-tensors $h_o$ on $M_o(\epsilon_0)$ and $h_j$ on $N_j(\epsilon_0)$, we have
    $$\|(h_o)_{|M_o(2\epsilon_0)}\|_{C^{2,\alpha}(g_o)}\leqslant C_1 \big(\|(P_{g_o}h_o)_{|M_o(\epsilon_0)}\|_{C^{\alpha}(g_o)}+\|(h_o)_{|M_o(\epsilon_0)}\|_{C^0(g_o)}\big),$$
    and
    $$\|(h_j)_{|N_j(2\epsilon_0)}\|_{C^{2,\alpha}(g_{b_j})}\leqslant C_1 \big(\|(P_{g_{b_j}}h_j)_{|N_j(\epsilon_0)}\|_{C^{\alpha}(g_{b_j})}+\|(h_j)_{|N_j(\epsilon_0)}\|_{C^0(g_{b_j})}\big).$$
    
    By assumption, there exists $C>0$ only depending on $g_o$ and $g_{b_j}$ such that $\|g-g_o\|_{C^{2,\alpha}(g_o)}\leqslant C\epsilon$ on $M_o(\epsilon_0)$ and $\big\|\frac{g}{T_j}-g_{b_j}\big\|_{C^{2,\alpha}(g_{b_j})}\leqslant C\epsilon$ on $N_j(\epsilon_0)$. We conclude that for $\epsilon$ small enough, the operators $P_g$ and $P_{\frac{g}{T_j}}$, which are close to the operators $P_{g_o}$ and $P_{g_{b_j}}$, satisfy for all $h_o$ on $M_o(\epsilon_0)$ and $h_j$ on $N_j(\epsilon_0)$,
    $$\|(h_o)_{|M_o(2\epsilon_0)}\|_{C^{2,\alpha}(g_o)}\leqslant 2C_1 \big(\|(P_{g}h_o)_{|M_o(\epsilon_0)}\|_{C^{\alpha}(g_o)}+\|(h_o)_{|M_o(\epsilon_0)}\|_{C^0(g_o)}\big),$$
    and
    $$\|(h_j)_{|N_j(2\epsilon_0)}\|_{C^{2,\alpha}(g_{b_j})}\leqslant 2C_1 \big(\|(P_{\frac{g}{T_j}}h_j)_{|N_j(\epsilon_0)}\|_{C^{\alpha}(g_{b_j})}+\|(h_j)_{|N_j(\epsilon_0)}\|_{C^0(g_{b_j})}\big).$$
    
    On each almost flat annulus $\mathcal{A}_k(t,\epsilon_0)$, let us denote $\mathcal{A}(\rho,\rho'):=\{\rho\leqslant r_D \leqslant \rho'\}$. There exists a diffeomorphism $\phi_\rho: A_e(1/2,4) \to \mathcal{A}(\rho/2,4\rho)$ such that 
    $$\Big\|\frac{\phi_\rho^*g}{\rho^2}-g_e\Big\|_{C^{1,\alpha}(A_e(1/2,4))}\leqslant C\eta(\rho)\epsilon.$$
    then, by ellipticity, for $\epsilon$ small enough, there exists $C_2>0$, such that for all symmetric $2$-tensor $h$ on $A_e(1/2,4)$ we have,
    $$ \|h\|_{C^{2,\alpha}(A_e(1,2))} \leqslant  2C_2\Big( \|h\|_{C^{0}(A_e(1/2,4))} + \big\|P_{\frac{\phi_\rho^*g}{\rho^2}} h\big\|_{C^{\alpha}(A_e(1/2,4))}\Big). $$
    Coming back to $(M,g)$, this implies that for $\epsilon$ small enough, we have for $h$ a symmetric $2$-tensor on $\mathcal{A}(\rho/2,4\rho)$,
    $$ \|h\|_{C^{2,\alpha}(\mathcal{A}(\rho,2\rho),\frac{g}{\rho^2})} \leqslant  4C_2\left( \|h\|_{C^{0}(\mathcal{A}(\rho/2,4\rho),\frac{g}{\rho^2})} + \big\|P_{\frac{g}{\rho^2}} h\big\|_{C^{\alpha}(\mathcal{A}(\rho/2,4\rho),\frac{g}{\rho^2})}\right). $$
    The norm of a symmetric $2$-tensor $s$ behaves in the following way by rescaling, for $t>0$
    $$|s|_{\frac{g}{t}} = t|s|_{g},$$
    and the operator $P$ behaves in the following way by rescaling, for $t>0$:
    $P_{\frac{g}{t}} = tP_g$. Multiplying both sides of the equality by $r_D^{-2}$, we get
    $$ \|h\|_{C^{2,\alpha}((\rho,2\rho),g)} \leqslant  4C_2\left( \|h\|_{C^{0}((\rho/2,4\rho),g)} + r_D^2\|P_g h\|_{C^{\alpha}((\rho/2,4\rho),g)}\right). $$
    Given the controls on the derivatives of $r_D$, we deduce the stated result by definition of the weighted norms by multiplying both sides of the inequality by the weight of the norm.
\end{proof}

Analogous estimates also hold for the elliptic operator $\delta\delta^*$ with the same proof.

\begin{prop}\label{Estimations elliptic delta delta}
    For all $\beta>0$ and $0<\alpha<1$ there exists $C>0$ and $\epsilon>0$ such that if $X$ is a vector field on $(M,g^D)$, and $g$ a metric on $M$ satisfying $$\|g-g^D\|_{C^{2,\alpha}_\beta(g^D)}\leqslant \epsilon,$$
    then, we have
    $$\|X\|_{r_DC^{3,\alpha}_{\beta}(g^D)}\leqslant C \big(\|\delta_g\delta^*_gX\|_{r^{-1}_DC^{1,\alpha}_{\beta}(g^D)}+\|X\|_{r_DC^0_{\beta}(g^D)}\big).$$
\end{prop}

\subsection{Decoupling norms}

We will see here that to expect good controls for the operators $P$ and $\delta\delta^*$ in the annular regions of our manifold, we need to consider separately the influence of traceless constant $2$-tensors for $P$ and linear vector fields of the kernel of $\delta_e\delta^*_e$ for $\delta\delta^*$.

\subsubsection{Estimates on annuli $A_e(\epsilon,\epsilon^{-1})$ of $(\mathbb{R}^4, g_e)$.}
Let us start by studying the situation on flat annuli to motivate our new norms.

\begin{prop}\label{inverse P on annulus}
    Let $0<\beta<1$, $0<\alpha<1$, and $P = \frac{1}{2}\nabla^*\nabla - \mathring{\R}$. There exists $C_e>0$, and $\epsilon_e>0$ such that for any symmetric $2$-tensor $h$ on an annulus of radii $0<\epsilon<\epsilon_e$ and $\frac{1}{\epsilon}$, there exists a constant symmetric $2$-tensor $H_0$ and a symmetric $2$-tensor $H_*$ satisfying 
         $$\nabla^*_e\nabla_e H_*=0,$$         
         \begin{equation}
        \|H_*\|_{C^{2,\alpha}_1(A_e(2\epsilon,(1/2)\epsilon^{-1}))} \leqslant C_e \|h-H_0\|_{C^{0}_{\beta}(A_e(\epsilon,\epsilon^{-1}))},\label{est reste harmonique}
    \end{equation}
    (notice the norm $C^{2,\alpha}_{\mathbf{1}}$ for the left hand side) and,
    \begin{equation}
        \|h-H_0-H_*\|_{C^{2,\alpha}_{\beta}(A_e(2\epsilon,(1/2)\epsilon^{-1}))}\leqslant C_e\|P_{g_e} h\|_{r^{-2}_eC^\alpha_\beta(A_e(\epsilon,\epsilon^{-1}))}.\label{proj harmonique annulus}
    \end{equation}
    This implies in particular the following control: for all $x \in A_e(1/2,2)$,
    \begin{align}
        |h-H_0(x)|_{g_e}+|&\nabla h(x)|_{g_e} +|\nabla^2 h(x)|_{g_e} + [\nabla^2 h]_{C^\alpha(g)}(x) \nonumber \\
        \leqslant& \; C_e \big( (2\epsilon)^\beta\|P_{g_e} h\|_{r^{-2}_eC^\alpha_\beta(A_e(\epsilon,\epsilon^{-1}))}+ 2\epsilon\|h-H_0\|_{C^{2,\alpha}_\beta(A_e(\epsilon,\epsilon^{-1}))}\big).\label{estimation inverse annulus h}
    \end{align}
\end{prop}

\begin{rem}
    This is a strictly better estimate than the elliptic estimates of Proposition \ref{Estimations elliptic delta delta} which would only have given
    \begin{align}
        |h-H_0(x)|_{g_e}+|&\nabla h(x)|_{g_e} +|\nabla^2 h(x)|_{g_e} + [\nabla^2 h]_{C^\alpha(g_e)}(x) \nonumber \\
        \leqslant& \; C_e \big( (2\epsilon)^\beta\|P_{g_e} h\|_{r^{-2}_eC^\alpha_\beta(A_e(\epsilon,\epsilon^{-1}))}+ (2\epsilon)^\beta\|h-H_0\|_{C^{2,\alpha}_\beta(A_e(\epsilon,\epsilon^{-1}))}\big).
    \end{align}
    The difference will be crucial in the proof of Proposition \ref{inversion with cste}.
\end{rem}
\begin{proof}
    Let us start by noting that \eqref{estimation inverse annulus h} is a consequence of \eqref{est reste harmonique} and \eqref{proj harmonique annulus}. Indeed, $h-H_0 = (h-H_0-H_*)+H_*$, and we have therefore, denoting $$\|s\|_{C^{2,\alpha}(1/2,2)}:= \sup_{x\in A_e(1/2,2)}|s(x)|_{g_e}+|\nabla s(x)|_{g_e} +|\nabla^2 s(x)|_{g_e} + [\nabla^2 s]_{C^\alpha(g_e)}(x), $$
    \begin{align*}
        \|h-H_0\|_{C^{2,\alpha}(1/2,2)} \leqslant& \;\|h-H_0-H_*\|_{C^{2,\alpha}(1/2,2)} + \|H_*\|_{C^{2,\alpha}(1/2,2)}\\
         \leqslant &\;(2\epsilon)^\beta \|h-H_0-H_*\|_{C^{2,\alpha}_{\beta}(A_e(\epsilon,\epsilon^{-1}))} + 2\epsilon \|H_*\|_{C^{2,\alpha}_1(A_e(\epsilon,\epsilon^{-1}))}\\
        \leqslant &\;C_e\big((2\epsilon)^\beta\|P_{g_e} h\|_{r^{-2}_eC^\alpha_\beta(A_e(\epsilon,\epsilon^{-1}))} + 2\epsilon \|h-H_0\|_{C^{2,\alpha}_{\beta}(A_e(\epsilon,\epsilon^{-1}))}\big),
    \end{align*}
    by definition of the weighted norms and assuming, for $C_e>0$, the inequalities \eqref{proj harmonique annulus} and \eqref{est reste harmonique}.

    On $\mathbb{R}^4\backslash \{0\}$, the harmonic symmetric $2$-tensors are sum of homogeneous harmonic symmetric $2$-tensors whose coefficients in the canonical basis of $\mathbb{R}^4$ are proportional to $r^j$ for $j\in \mathbb{Z}\backslash \{-1\}$.  These harmonic symmetric $2$-tensors are more precisely of the form $ r_e^k H_k $ or $r^{-2-k}_eH_k$ for $k\in \mathbb{N}$, where $H_k$ is a homogeneous symmetric $2$-tensor with $|H_k|_{g_e}\sim r^0$ whose coefficients, once restricted to the unit sphere are eigenfunctions of the spherical Laplacian with eigenvalue $-k(k+2)$.
    \\
    
    For any symmetric $2$-tensor $h$ on $A_e(\epsilon,\epsilon^{-1})$, let us define $ \tilde{H}$ the solution on $A_e(\epsilon,\epsilon^{-1})$ of the following Dirichlet problem, denoting for $r>0$, $S_e(r):= \{r_e = r\}$,
    \begin{equation*}
  \left\{
      \begin{aligned}
        &\nabla^*_e\nabla_e \tilde{H} =0, \\
        &\tilde{H} = h \text{ on } S_e(\epsilon)\cup S_e(\epsilon^{-1}).
      \end{aligned}
    \right.
\end{equation*}
    More precisely, $\tilde{H} = \sum_{k\geqslant 0}(\epsilon r_e)^k \tilde{H}_{k}^+    + (\epsilon^{-1} r_e)^{-2-k} \tilde{H}_{k}^-$ where the $\tilde{H}_{k}^\pm$ are homogeneous with $|\tilde{H}_{k}^+|_{g_e} \sim r_e^0$ and which, once restricted to the sphere are eigenvectors associated to $-k(k+2)$. If we decompose in spherical harmonics $ h_{|S_e(\epsilon)} =: \sum_{k} H_{k}(\epsilon)$ and $ h_{|S_e(\epsilon^{-1})} =: \sum_{k} H_{k}(\epsilon^{-1})$, we have the system
    \begin{equation}
  \left\{
      \begin{aligned}
        &H_{k}(\epsilon^{-1}) = \tilde{H}_{k}^+ + \epsilon^{4+2k} \tilde{H}_{k}^-, \\
        &H_{k}(\epsilon) = \epsilon^{2k}\tilde{H}_{k}^+ + \tilde{H}_{k}^-,
      \end{aligned}
    \right.\label{dvp h}
\end{equation}
and therefore,
\begin{equation}
  \left\{
      \begin{aligned}
        &\tilde{H}_{k}^+ = \frac{1}{1-\epsilon^{4+4k}}\big(H_{k}(\epsilon^{-1}) -\epsilon^{4+2k}H_{k}(\epsilon) \big), \\
        &\tilde{H}_{k}^- = \frac{1}{1-\epsilon^{4+4k}}\big(H_{k}(\epsilon) - \epsilon^{2k} H_{k}(\epsilon^{-1})\big),
      \end{aligned}
    \right.\label{système}
\end{equation}

Denote $\tilde{H}_*:= \tilde{H}-\tilde{H}_0^+$. Since $\nabla^*_e\nabla_e \tilde{H}_* = 0$, by elliptic regularity on the annulus $A(\rho/2,4\rho)\subset A_e(\epsilon,\epsilon^{-1})$, there exists a constant $C>0$ independent of $h$ such that we have,
\begin{equation}
    \|\tilde{H}_*\|_{C^0(A(\rho,2\rho))}\leqslant \frac{C}{\rho^2}\|\tilde{H}_*\|_{L^2(A_e(\rho/2,4\rho))},\label{est ellipt H*}
\end{equation}
so to control the norm $C^{0}_1(A_e(\epsilon,\epsilon^{-1}))$ of $\tilde{H}_*$, we just have to control the $L^2$-norm of $\tilde{H}_*$ on the different annuli $A_e(\rho,2\rho)\subset A(\epsilon,\epsilon^{-1})$. Since the harmonic decompositions are $L^2(S_e(1))$-orthogonal, we have for a constant $C>0$ which may change from line to line
\begin{align}
    \int_{A_e(\rho,2\rho)}|\tilde{H}_*|_{g_e}^2 dv_{g_e} =&\;\epsilon^4\int_\rho^{2\rho} \int_{S_e(1)}|\tilde{H}_0^-|_{g_e}^2r^{-4}dv_{S_e(1)}r^3dr\nonumber\\
    &+ \sum_{k\geqslant 1}  \int_\rho^{2\rho} \int_{S_e(1)} |\epsilon^kr^k\tilde{H}_k^{+}
+\epsilon^{2+k}r^{-2-k}\tilde{H}_k^{-}|_{g_e}^2 dv_{S_e(1)}r^3dr\nonumber\\
 \leqslant&\; C \epsilon^4 \int_{S_e(1)}|\tilde{H}_0^-|_{g_e}^2 dv_{S_e(1)} \nonumber\\
 &+C\sum_{k\geqslant 1}  \epsilon^{2k}\int_\rho^{2\rho} \int_{S_e(1)} \big|\tilde{H}_k^{+}\big|_{g_e}^2dv_{S_e(1)}r^{2k+3}dr\nonumber\\
&+\epsilon^{4+2k}\int_\rho^{2\rho} \int_{S_e(1)} \big|\tilde{H}_k^{-}\big|_{g_e}^2dv_{S_e(1)}r^{-1-2k}dr \nonumber\\
\leqslant&\; C \epsilon^4 \int_{S_e(1)}|\tilde{H}_0^-|_{g_e}^2 dv_{S_e(1)}\nonumber\\
&+C\sum_{k\geqslant 1}  \epsilon^{2k}\rho^{2k+4}
 \int_{S_e(1)}|\tilde{H}_k^{+}|_{g_e}^2 dv_{S_e(1)}+\epsilon^{4+2k}\rho^{-2k} \int_{S_e(1)}|\tilde{H}_k^{-}|_{g_e}^2 dv_{S_e(1)}\nonumber\\
\leqslant&\; C \epsilon^4\epsilon^{-3} \int_{S_e(\epsilon)}|\tilde{H}_0^-|_{g_e}^2 dv_{S_e(\epsilon)}\nonumber\\
&+C\sum_{k\geqslant 1}  \epsilon^{2k}\rho^{2k+4}
 \epsilon^{3}\int_{S_e(\epsilon^{-1})}|\tilde{H}_k^{+}|_{g_e}^2 dv_{S_e(\epsilon^{-1})}\nonumber\\
 &+\epsilon^{4+2k}\rho^{-2k} \epsilon^{-3}\int_{S_e(\epsilon)}|\tilde{H}_k^{-}|_{g_e}^2 dv_{S_e(\epsilon)}\label{calcul norme L2 H*}
\end{align}
Now, the equalities \eqref{système} and the fact that the decompositions in spherical harmonics are orthogonal imply that for a constant $C>0$ which may change from line to line we have 
\begin{align}
    \sum_{k\geqslant 0}\int_{S_e(\epsilon)}|\tilde{H}_k^{-}|_{g_e}^2 dv_{S_e(\epsilon)} &\leqslant C\sum_k 
    \int_{S_e(\epsilon)}|H_k(\epsilon)|_{g_e}^2 dv_{S_e(\epsilon)} + C\epsilon^{6+2k} \int_{S_e(\epsilon^{-1})}|H_k(\epsilon^{-1})|^2_{g_e} dv_{S_e(\epsilon^{-1})}\nonumber\\
    &\leqslant C 
    \int_{S_e(\epsilon)}|h-\tilde{H}_0^+|_{g_e}^2 dv_{S_e(\epsilon)} + C\epsilon^{6} \int_{S_e(\epsilon^{-1})}|h-\tilde{H}_0^+|_{g_e}^2 dv_{S_e(\epsilon^{-1})}\nonumber\\
    &\leqslant C \epsilon^{3} \|(h-\tilde{H}_0^+)_{|S_e(\epsilon)}\|^2_{C^0(g_e)}  + C\epsilon^3\|(h-\tilde{H}_0^+)_{|S_e(\epsilon^{-1})}\|^2_{C^0(g_e)}\nonumber\\
    &\leqslant C \epsilon^3 \|h-\tilde{H}_0^+\|^2_{C^0_\beta(A_e(\epsilon,\epsilon^{-1}))}\label{controle L2 C0beta -}
\end{align}
because $|(h-\tilde{H}_0^+)_{|S_e(\epsilon)}|_{g_e}\leqslant \|h-\tilde{H}_0^+\|_{C^0_\beta(A_e(\epsilon,\epsilon^{-1}))}$ and $|(h-\tilde{H}_0^+)_{|S_e(\epsilon^{-1})}|_{g_e}\leqslant \|h-\tilde{H}_0^+\|_{C^0_\beta(A_e(\epsilon,\epsilon^{-1}))}$ by definition of the norm, and similarly
\begin{align}
    \sum_{k\geqslant 1}\int_{S_e(\epsilon^{-1})}|\tilde{H}_k^{+}|_{g_e}^2 dv_{S_e(\epsilon^{-1})} \leqslant C \epsilon^{-3}\|h-\tilde{H}_0^+\|^2_{C^0_\beta(A_e(\epsilon,\epsilon^{-1}))}.\label{controle L2 C0beta +}
\end{align}
Together with \eqref{calcul norme L2 H*}, \eqref{controle L2 C0beta -} and \eqref{controle L2 C0beta +}, and since on $A_e(\epsilon,\frac{1}{\epsilon})$ for any $k\geqslant 1$, $ \epsilon(\rho+\rho^{-1})\geqslant \epsilon^k\rho^{\pm k}$, this yields the following estimate for $\epsilon$ small enough and a constant $C>0$,
\begin{align}
    \|\tilde{H}_*\|_{L^2(A_e(\rho,2\rho))}^2\leqslant&\; C \rho^4\epsilon^{2}\big(\rho^{1} +\rho^{-1}\big)^2\|h-\tilde{H}_0^+\|_{C^0_\beta(A_e(\epsilon,\epsilon^{-1}))}^2.\label{est L2 H*}
\end{align}
 Combining \eqref{est ellipt H*} and \eqref{est L2 H*}, we get
\begin{equation}
    \|\tilde{H}_*\|_{C^0_1(A_e(\epsilon,\epsilon^{-1}))} \leqslant C \|h-\tilde{H}_0^+\|_{C^0_\beta(A_e(\epsilon,\epsilon^{-1}))}.\label{est H* h-H0}
\end{equation}

Let us fix $x_0\in S_e(1)$ and modify our symmetric $2$-tensor $\tilde{H}$ to get a symmetric $2$-tensor $H$ such that $h-H$ vanishes at $x_0$ and on $S_e(\epsilon)$ while being constant on $S_e(\epsilon^{-1})$. The only possible choice with harmonic symmetric $2$-tensors is
$$H:= \Big(\tilde{H}_0^+-\frac{c_0}{1-\epsilon^2}\Big)+ \Big(\tilde{H}_{*}+\frac{\epsilon^2}{1-\epsilon^2}\frac{c_0}{r_e^2}\Big),$$
where $c_0 = (\tilde{H}-h)(x_0)$.
We will show that $h-H$ satisfies the estimate \eqref{proj harmonique annulus}, but let us start by proving the control \eqref{est reste harmonique} stated. For this, denote $H_0:=\tilde{H}_0^+-\frac{c_0}{1-\epsilon^2}$ the constant part of $H$, and $H_*:=\tilde{H}_{*}+\frac{1}{1-\epsilon^2}\frac{c_0}{(\epsilon^{-1}r_e)^2}$ its varying part.

According to \eqref{est L2 H*}, the part $\sum_{k\geqslant 1}(\epsilon r_e)^k \tilde{H}_{k}^+ + (\epsilon^{-1} r_e)^{-2-k} \tilde{H}_{k}^-$ is well controlled in $L^2$-norm by the varying parts of $h-H_0$ on $S_e(\epsilon)$ and $S_e(\epsilon^{-1})$: they are the same as the varying parts of $h-\tilde{H}_0$. There remains to control the part in $r_e^{-2}$, that is
$$ \frac{1}{1-\epsilon^2}\frac{(1-\epsilon^2)\tilde{H}_{0}^-+c_0}{\epsilon^{-2}r_e^2}.$$
In order to control this part, let us look at the mean values of $h-H_0$ on $S_e(\epsilon)$ and $S_e(\epsilon^{-1})$. On $S_e(\epsilon)$, we have $(h-H)_{|S_e(\epsilon)} = 0$, and therefore $$(h-H_0)_{|S_e(\epsilon)} = (h-H) + \tilde{H}_* + \frac{c_0}{1-\epsilon^2}\frac{\epsilon^2}{r_e^2}$$ and its mean value on $S_e(\epsilon)$ is then
\begin{equation}
    \tilde{H}_0^- + \frac{c_0}{1-\epsilon^2}.\label{mean value epsilon}
\end{equation}
Similarly, since $h-\tilde{H} = 0$ on $S_e(\epsilon^{-1})$ and
$$(h-H_0)_{|S_e(\epsilon^{-1})} = (h-\tilde{H}) + \tilde{H}_* + \frac{c_0}{1-\epsilon^2},$$
and its mean value is therefore
\begin{equation}
     \epsilon^4\tilde{H}_0^-+\frac{c_0}{1-\epsilon^2}.\label{mean value epsilon-1}
\end{equation}
By considering linear combinations of \eqref{mean value epsilon} and \eqref{mean value epsilon-1}, we control both $|c_0|_{g_e}$ and $|\tilde{H}_0|_{g_e}$ thanks to the mean values of $h-H_0$ on $S_e(\epsilon)$ and $S_e(\epsilon^{-1})$, and we consequently have for some constant $C>0$,
$$|c_0|_{g_e} + |\tilde{H}_0^-|_{g_e} \leqslant C\|h-H_0\|_{C^0_\beta(A_e(\epsilon,\epsilon^{-1}))}.$$

Hence we finally have the existence of a constant $C>0$ such that we have, going from $L^2$-controls to $C^0$-controls thanks to \eqref{est ellipt H*} applied to $H_*$,
$$\|H_*\|_{C^0_1}\leqslant C \|h-H_0\|_{C^0_\beta},$$
and therefore the stated inequality \eqref{est reste harmonique}.
\\

    Let us prove the estimate \eqref{proj harmonique annulus}, and assume towards a contradiction that there exists a sequence of positive numbers $\epsilon_i\to 0$, and a sequence of symmetric $2$-tensors $h_i$ on annuli $A_e(\epsilon_i,\epsilon_i^{-1})$ satisfying, $\|h_i-H_i\|_{C^{0}_{\beta}(A_e(\epsilon,\epsilon^{-1}))} = 1$, and
    $\|P_e h_i\|_{r^{-2}_eC^0_\beta(A_e(\epsilon_i,\epsilon_i^{-1}))}\leqslant \frac{1}{i}$.
    \begin{rem}
        The failure of these properties will indeed yield the estimate \eqref{proj harmonique annulus} since by elliptic regularity we will get higher order estimates on the smaller domain $A_e(2\epsilon,(2\epsilon)^{-1}))$.
    \end{rem}
    Let then $(x_i)_i$ be a sequence of points of $A_e(\epsilon_i,\epsilon_i^{-1})$ where the $C^{0}_\beta(A_e(\epsilon_i,\epsilon_i^{-1}))$-norm of $\bar{h}_i:=h_i-H_i$ is reached. We can extract a subsequence with one of the following behaviors:
    \begin{enumerate}
        \item $r_e(x_i)\to +\infty$, and $\epsilon_i r_e(x_i)\to 0$,
        \item $r_e(x_i)\to +\infty$, and $\epsilon_i r_e(x_i)\to c>0$,
        \item $r_e(x_i)\to 0$ and $\epsilon_i^{-1}r_e(x_i)\to +\infty$,
        \item $r_e(x_i)\to 0$ and $\epsilon_i^{-1}r_e(x_i)\to c>0$,
        \item $r_e(x_i)\to c>0$.
    \end{enumerate}
    In all cases, we rescale to fix $r_e(x_i)=1$ by defining, for all $x$,
    $$h_i'(x):= \frac{\bar{h}_i\big(r_e(x_i)x\big)}{\epsilon_i^{\beta}\big(r_e(x_i)^\beta+ r_e(x_i)^{-\beta}\big)},$$
    which satisfies
    $$ \big(\nabla_e^*\nabla_eh_i'\big)(x) = r_e(x_i)^2\big( \nabla_e^*\nabla_e\bar{h}_i\big)(r_e(x_i)x).$$
    Since we had by assumption the controls
    $$\bar{h}_i(x)\leqslant \epsilon_i^{\beta}\big(r_e(x)^\beta+ r_e(x)^{-\beta}\big),$$
    and
    $$|\nabla_e^*\nabla_e\bar{h}_i|_{g_e}(x)\leqslant \frac{1}{i}r_e(x)^{-2}\epsilon_i^\beta\big(r_e(x)^{\beta}+r_e(x)^{-\beta}\big),$$
    our new symmetric $2$-tensor $h'_i$ vanishes at $\frac{x_0}{r_e(x_i)}$ and on $S_e(\epsilon_i r_e(x_i)^{-1})$, and is constant on $S_e(\epsilon_i^{-1} r_e(x_i)^{-1})$. It moreover satisfies
    $$ |h_i'|_{g_e}(x)\leqslant \frac{\big((r_e(x_i)r_e(x))^{\beta}+(r_e(x_i)r_e(x))^{-\beta}\big)}{r_e(x_i)^\beta+ r_e(x_i)^{-\beta}} $$
    with equality at $x_i$ and
    $$|\nabla_e^*\nabla_eh_i'|_{g_e}(x)\leqslant \frac{1}{i}r_e(x)^{-2}\frac{\big((r_e(x_i)r_e(x))^{\beta}+(r_e(x_i)r_e(x))^{-\beta}\big)}{r_e(x_i)^\beta+ r_e(x_i)^{-\beta}}.$$
    
    In the different situations, up to extracting a subsequence, we finally get one of the following limits
        \begin{enumerate}
        \item on  $\mathbb{R}^4 \backslash \{0\}$, a solution $h'_\infty$ of $P_e h'_\infty = \frac{1}{2}\nabla_e^*\nabla_eh'_\infty =0$, and $\sup r^{-\beta}h'_\infty = 1$, but there does not exist such a solution because the harmonic symmetric $2$-tensors decay at least as $\mathcal{O}(r)$ at $0$ if they vanish at $0$ and must therefore grow at this rate at infinity. This is a contradiction.
        \item on  $B_e(1/c) \backslash \{0\}$, a solution $h'_\infty$ of $P_e h'_\infty =0$, and $\sup r^{-\beta}h'_\infty = 1$, and such that $(h_\infty')_{|S_e(1/c)}$ is constant. The unique solution to the Dirichlet problem with the zero condition at $0$ and a constant condition on $ S_e(1/c)$ is $h_\infty'=0$. This is a contradiction.
        \item on  $\mathbb{R}^4 \backslash \{0\}$, a solution $h'_\infty$ of $P_e h'_\infty =0$, and $\sup r^{\beta}h'_\infty = 1$, but there does not exist such a solution because the harmonic symmetric $2$-tensors decaying at infinity decay at least like $\mathcal{O}(r^{-2})$, and therefore blow up at least at this rate at $0$, and finally, $h_\infty'=0$. This is a contradiction.
        \item on  $\mathbb{R}^4 \backslash B_e(1/c)$, a solution $h'_\infty$ of $P_e h'_\infty =0$, and $\sup r^{\beta}h'_\infty = 1$ and $(h'_\infty)_{|S_e(1/c)}=0$. The unique solution to the Dirichlet problem on  $\mathbb{R}^4 \backslash B_e(1/c)$, decaying at infinity and vanishing on $S_e(1/c)$ being zero, we have $h_\infty =0$. This is a contradiction.
        \item on  $\mathbb{R}^4 \backslash \{0\}$, a solution $h'_\infty$ of $P_e h'_\infty =0$, and $\sup (r^\beta+r^{-\beta})h'_\infty = 1$ satisfying $h'_\infty\big(\frac{x_0}{c}\big)=0 $. The conditions $P_e h'_\infty =0$, and $\sup (r^\beta+r^{-\beta})h'_\infty =1 $ imply that $h'_\infty$ is constant, since $h'_\infty$ vanishes at $\frac{x_0}{c}$, we have $h'_\infty =0$. This is a contradiction.
    \end{enumerate}
    We therefore deduce that there exists $\epsilon_e>0$ and $C_e>0$ such that for all $0<\epsilon<\epsilon_e$ and all symmetric $2$-tensor $h$ on the annulus $A_e(\epsilon,\epsilon^{-1})$, we have
    $$\|h-H_0\|_{C^{0}_{\beta}(A_e(\epsilon,\epsilon^{-1}))}\leqslant C_e\|P_g h\|_{r^{-2}_eC^\alpha_\beta(A_e(\epsilon,\epsilon^{-1}))}.$$
    
    In order to prove the estimate \eqref{proj harmonique annulus} and go from a $C^0_\beta(A_e(\epsilon,\epsilon^{-1}))$-controls to $C^{2,\alpha}_\beta(A_e(2\epsilon,(2\epsilon)^{-1}))$-controls, we use elliptic estimates which are satisfied on the flat annuli according to the end of the proof of Proposition \ref{Estimations elliptic P}.
\end{proof}

With a completely analogous proof using  the harmonic decomposition of $1$-forms on a cone of \cite[(2.16)-(2.19)]{ct} (see also Section \ref{kernel of la linéarisation} for particular case of flat cones), we have the same result for vector fields, but this time, we treat the linear kernel of $\delta\delta^*$ on $\mathbb{R}^4\slash\Gamma$ separately. On $\mathbb{R}^4$, the elliptic operator
$$ \delta_e\delta_e^* = \nabla_e^* \nabla_e-\frac{1}{2} d_e^*d = dd_e^*+\frac{1}{2}d_e^*d, $$
has its kernel equal to the linear vector fields of the kernel of $\delta_e\delta^*_e$ among the vector fields of order $\mathcal{O}(r_e^{1-\beta}+r_e^{1+\beta})$ for $0<\beta<1$, see Lemma \ref{poids crit deltadelta} for a proof of this and Section \ref{kernel of la linéarisation} for a description of the kernel.

\begin{prop}\label{inverse delta delta on annulus}
Let $0<\beta<1$. There exists $C_e>0$, and $\epsilon_e>0$ such that for any vector field $X$ on an annulus of radii $0<\epsilon<\epsilon_e$ and $\frac{1}{\epsilon}$, there exists $Y_0$, a linear vector field of the kernel of $\delta_e\delta^*_e$, and an element $Y_*$ of the kernel of $\delta_e\delta_e^*$ satisfying 
$$ \|Y_*\|_{r_eC^{3,\alpha}_1(A_e(2\epsilon,(2\epsilon)^{-1}))} \leqslant C_e \|X -Y_0\|_{r_eC^{3,\alpha}_\beta(A_e(\epsilon,\epsilon^{-1}))}, $$ 
    $$ \|X-Y_0-Y_*\|_{r_eC^{3,\alpha}_\beta(A_e(2\epsilon,(2\epsilon)^{-1}))}\leqslant C_e \|\delta_{g_e}\delta^*_{g_e} X\|_{r^{-1}_eC^{1,\alpha}_\beta(A_e(\epsilon,\epsilon^{-1}))}.$$
    
    In particular, this implies the following control, for all $x \in A_e(1/2,2)$,
    \begin{align}
        |(X-Y_0)(x)|_{g_e}+|&\nabla (X-Y_0)(x)|_{g_e} +|\nabla^2 (X-Y_0)(x)|_{g_e} + [\nabla^2 (X-Y_0)]_{C^\alpha(g)}(x) \nonumber \\
        \leqslant& \;C_e \big( (2\epsilon)^\beta\|\delta_e\delta_e^*X\|_{r^{-1}_eC^{1,\alpha}_\beta(A_e(\epsilon,\epsilon^{-1}))}+ 2\epsilon\|X-Y_0\|_{r_eC^{3,\alpha}_\beta(A_e(\epsilon,\epsilon^{-1}))}\big).\label{estimation inverse annulus X}
    \end{align}
\end{prop}

\subsubsection{Approximate kernels}

Let $(M,g^D)$ be a naïve desingularization of an Einstein orbifold. For each annulus $\mathcal{A}_k(t,\epsilon)$ (see Definition \ref{def neck region}) between $N_k$ and $N_j$ or $N_k$ and $M_o$, by construction there exists a diffeomorphism
$$\Phi_k: A_e\big(\epsilon^{-1}\sqrt{T_j}\sqrt{t_k},\epsilon\sqrt{T_j}\big)\subset \mathbb{R}^4\slash\Gamma_k\to \mathcal{A}_k(t,\epsilon)\subset M,$$
such that there exists $C>0$ for which, for all $0<\beta<1$,
\begin{equation}
    \big\|\Phi^*_kg^D-g_e\big\|_{C^{2,\alpha}_\beta(A_e(\epsilon^{-1}\sqrt{T_j}\sqrt{t_k},\epsilon\sqrt{T_j}))}\leqslant C \epsilon^{2-\beta}.\label{controle gD-ge}
\end{equation}

Because of the above constant symmetric $2$-tensors and the linear vector fields, we cannot expect estimates independent of the gluing scales in the definition of $(M,g^D)$ of the type $\|h\|_{C^{2,\alpha}_\beta(g^D)} \leqslant C\|P_{g^D}h\|_{r_D^{-2}C^\alpha_\beta(g^D)}$ which are needed to apply an inverse function theorem. Indeed, we have the following estimates according to Proposition \ref{contrôles with perturbations triviales} (which is proven below). Recall that the cut-off functions are in Definition \ref{def cutoffs all}.

\begin{rem}
    In most of the rest of this article, we will often abusively forget the diffeomorphism $\Phi_k$ to simplify the notations. For instance, a symmetric $2$-tensor $\Phi_k^*\big(\chi_{\mathcal{A}_k(t,\epsilon)}H\big)$ will be denoted $\chi_{\mathcal{A}_k(t,\epsilon)}H$ on $M$.
\end{rem}

\begin{prop}\label{pas inverse P on annulus}
    On a naïve desingularization $(M,g^D_t)$, for all $0<\beta<1$, there exists $C>0$ such that for $H_k$ a constant symmetric $2$-tensor, and $\chi_{\mathcal{A}_k(t,\epsilon)}$ the cut-off function defined in Definition \ref{def cutoffs all},
    $$\|P_{g^D}\big(\chi_{\mathcal{A}_k(t,\epsilon)}H_k\big)\|_{r_D^{-2}C^\alpha_\beta(g^D)}\leqslant C |H_k|_{g_e},$$
    but $$ \|\chi_{\mathcal{A}_k(t,\epsilon)}H_k\|_{C^{0}_\beta(g^D)} \geqslant \frac{1}{2} t_{\max}^{-\frac{\beta}{4}}|H_k|_{g_e}.$$
\end{prop}
\begin{proof}[  ]
    
\end{proof}
Linear vector fields in the kernel of $\delta\delta^*$ also rule out the existence of estimates independent of $t$ for the operator $\delta \delta^*$ according to Proposition \ref{contrôles with perturbations triviales} proven below.

\begin{prop}\label{pas inverse delta delta on annulus}
    On a naïve desingularization $(M,g^D_t)$, for all $0<\beta<1$, there exists $C>0$ such that for $X_k$ a linear vector field in the kernel of $\delta_e\delta^*_e$, 
    $$\|\delta_{g^D}\delta_{g^D}^*\big(\chi_{\mathcal{A}_k(t,\epsilon)}X_k\big)\|_{r_D^{-1}C^{1,\alpha}_\beta(g^D)}\leqslant C \|X_k\|_{r_eC^0_0(g_e)},$$
    but $$ \|\chi_{\mathcal{A}_k(t,\epsilon)}X_k\|_{C^{0}_\beta(g^D)} \geqslant \frac{1}{2} t_{\max}^{-\frac{\beta}{4}}\|X_k\|_{r_eC^0_0(g_e)}.$$
\end{prop}
\begin{proof}[  ]
    
\end{proof}
    
\paragraph{Weighted decoupling norms.}
Propositions \ref{inverse P on annulus} and \ref{inverse delta delta on annulus} actually show that we can control the inverses of our operators once we solve our equations modulo constant symmetric $2$-tensors and the linear vector fields of the kernel of $\delta_e\delta^*_e$ on $\mathbb{R}^4$ and Propositions \ref{pas inverse P on annulus} and \ref{pas inverse delta delta on annulus} show that we cannot expect better. We therefore introduce new norms to reflect this. They are similar to the norms introduced in \cite{bam} for similar reasons.

\begin{defn}[Norm $\|.\|_{C^{k,\alpha}_{\beta,*}}$ on symmetric $2$-tensors]
    Let $ h $ be a symmetric $2$-tensor on $(M,g^D)$, (respectively $(M_o,g_o)$ or $(N,g_b)$). We define its $C^{k,\alpha}_{\beta,*}$-norm by
    $$\|h\|_{C^{k,\alpha}_{\beta,*}}:= \inf_{h_*,H_k} \|h_*\|_{C^{k,\alpha}_{\beta}} + \sum_k |H_k|_{g_e},$$
    where the infimum is taken among the couples $(h_*,H_k)$ satisfying $h= h_*+\sum_k \chi_{\mathcal{A}_k(t,\epsilon)}H_k$ (respectively $h= h_*+\sum_k \chi_{B_o(\epsilon)}H_k$ or $h= h_*+\sum_k \chi_{B_b(\epsilon)}H_k$), for $H_k$ a constant traceless symmetric $2$-tensor on $\mathbb{R}^4\slash\Gamma_k$.
\end{defn}

\begin{defn}[Norm $\|.\|_{rC^{k,\alpha}_{\beta,*}}$ on vector fields]
    Let $ X $ a vector field on $(M,g^D)$ (respectively $(M_o,g_o)$ or $(N,g_b)$). We define its $rC^{k,\alpha}_{\beta,*}$-norm, where $r$ is the function $r_D$ (respectively $r_o$ or $r_b$) by
    $$\|X\|_{rC^{k,\alpha}_{\beta,*}}:= \inf_{X_*,X_k}\|X_*\|_{rC^{k,\alpha}_{\beta}} + \sum_k \|X_k\|_{rC^0_0(g_e)},$$
    where the infimum is taken among the couples $(X_*,X_k)$ satisfying $X = X_* + \sum_k\chi_{\mathcal{A}_k(t,\epsilon)}X_k$ (respectively $X = X_* + \sum_k\chi_{B_o(\epsilon)}X_k$ or $X = X_* + \chi_{B_b(\epsilon)}X_k$).
\end{defn}

\begin{rem}\label{unicité decompositions}
    By definition of the weighted norms, on an orbifold or orbifold ALE, the decompositions $h = h_*+ \sum_k \chi_{B(\epsilon)}H_k$ and $X = X_* + \sum_k\chi_{B(\epsilon)}X_k$ are unique and determined respectively by the limits of $h$ and of $\frac{X}{r}$ when $r\to 0$ (where $r = r_o$ or $r=r_b$). Indeed, in other cases, the expression we minimize is infinite.
\end{rem}

\begin{rem}
    By definition, we have 
    $$\|.\|_{C^{k,\alpha}_{\beta,*}}\leqslant \|.\|_{C^{k,\alpha}_{\beta}}\text{, and }\; \|.\|_{rC^{k,\alpha}_{\beta,*}}\leqslant\|.\|_{rC^{k,\alpha}_{\beta}},$$
    and the spaces $(C^{k,\alpha}_{\beta,*},\|.\|_{C^{k,\alpha}_{\beta,*}})$ and $(rC^{k,\alpha}_{\beta,*},\|.\|_{rC^{k,\alpha}_{\beta,*}})$ are clearly Banach spaces.
\end{rem}

\subsubsection{Estimates in the decoupling norms}
Let us show that it is possible to control thanks to the $r^mC^{k,\alpha}_\beta$-norm the images by the operators $P$ and $\delta\delta^*$ of elements of $r^{m+2}C^{k+2,\alpha}_{\beta,*}$.
\begin{prop}\label{contrôles with perturbations triviales}  Let $0<\beta<1$, and $(M_\bullet,g_\bullet)$ one of the spaces $(M_o,g_o)$, $(N_j,g_{b_j})$ or $(M,g^D)$, $g$ a metric, $h$ a symmetric $2$-tensor, and $X$ a vector field on $M_\bullet$. We then have, the following controls:
    $$\|P_{g_\bullet}h\|_{r^{-2}_\bullet C^{\alpha}_\beta(g_\bullet)} \leqslant C\|h\|_{C^{2,\alpha}_{\beta,*}(g_\bullet)},$$
    $$\|\delta_{g_\bullet}\delta^*_{g_\bullet} X \|_{r^{-1}_\bullet C^{1,\alpha}_\beta(g_\bullet)}  \leqslant C\|X\|_{r_\bullet C^{3,\alpha}_{\beta,*}(g_\bullet)},$$
    $$\|P_g(h)-P_{g_\bullet}(h)\|_{r^{-2}_\bullet C^{\alpha}_\beta(g_\bullet)} \leqslant C\|g-g_\bullet\|_{C^{2,\alpha}_{\beta,*}(g_\bullet)}\|h\|_{C^{2,\alpha}_{\beta,*}(g_\bullet)},$$
    and
    $$\| \delta_g\delta^*_g (X) - \delta_{g_\bullet}\delta^*_{g_\bullet} (X) \|_{r^{-1}_\bullet C^{1,\alpha}_\beta(g_\bullet)}  \leqslant C\|g-g_\bullet\|_{C^{2,\alpha}_{\beta,*}(g_\bullet)}\|X\|_{r_\bullet C^{3,\alpha}_{\beta,*}(g_\bullet)}.$$
\end{prop}
\begin{proof}
    Let us show the result for $g^D$, the proof for other spaces is very similar. For the two first inequalities, consider $h$ a symmetric $2$-tensor and $X$ a vector field on $M$, and some decompositions $h = h_*+\sum_k\chi_{\mathcal{A}_k(t,\epsilon)}H_k$ and $X = X_*+\sum_k\chi_{\mathcal{A}_k(t,\epsilon)}X_k$. Remark \ref{comportement norme} implies that we have the following controls for $h_*$ and $X_*$,
    \begin{equation}
        \|P_{g^D}h_*\|_{r^{-2}_D C^{\alpha}_\beta(g^D)} \leqslant C\|h_*\|_{C^{2,\alpha}_{\beta}(g^D)},\label{est PgD h*}
    \end{equation}
    and
    \begin{equation}
        \|\delta_{g^D}\delta^*_{g^D} X_* \|_{r^{-1}_D C^{1,\alpha}_\beta(g^D)}  \leqslant C\|X_*\|_{r_D C^{3,\alpha}_{\beta}(g^D)}.\label{est deltadelta X*}
    \end{equation}
    On $\mathbb{R}^4\slash\Gamma$, we have $P_eH_k = 0$ and $\delta_e\delta_e^*X_k=0$, hence, since for all $l\in \mathbb{N}$, we have 
    \begin{equation}
        |\nabla^l\chi_{k}|_{g^D} \leqslant C_lr_D^{-l},\label{est chik}
    \end{equation}
    and thanks to the control \eqref{controle gD-ge}, we have \begin{equation}
        \|P_{g^D}(\chi_{\mathcal{A}_k(t,\epsilon)}H_k)\|_{C^\alpha_\beta(g^D)}\leqslant C |H_k|_{g_e},\label{est PgD Hk}
    \end{equation}
    and
    \begin{equation}
        \|\delta_{g^D}\delta_{g^D}^*(\chi_{\mathcal{A}_k(t,\epsilon)}X_k)\|_{r_D^{-1}C^{1,\alpha}_\beta(g^D)}\leqslant C \|X_k\|_{r_eC^0_0(g_e)},\label{est deltadelta Xk}
    \end{equation}
    where we pulled-back thanks to the diffeomorphism $$\Phi_k: A_e(\epsilon^{-1}\sqrt{T_j}\sqrt{t_k},\epsilon\sqrt{T_j})\subset \mathbb{R}^4\slash\Gamma_k\to \mathcal{A}_k(t,\epsilon)\subset M.$$ Summing the controls \eqref{est PgD h*} and \eqref{est PgD Hk} on the one hand, and the controls \eqref{est deltadelta X*} and \eqref{est deltadelta Xk} on the other hand, yields the two fist inequalities stated.
    \\
    
    Let us now focus on the two last inequalities, which are more difficult to obtain. The control we want is local, let us therefore write down the expressions of our operators in local coordinates in an orthonormal basis $(e_i)$. For a symmetric $2$-tensor $h$, denoting $h_{ij} = h(e_i,e_j)$ and $R_{ijkl}$ the Riemannian curvature in coordinates, we have 
    \begin{equation}
        P_g(h)_{ij}= \frac{1}{2}\big(\nabla_g^*\nabla_g h\big)_{ij}-g^{kp}g^{lq} R_{ikjl} h_{pq}, \label{Expression en coordonnées P}
    \end{equation}
    where $\nabla_i$ is the covariant derivative for $g$ in the direction $e_i$.
    We directly see thanks to the estimates of Remark \ref{comportement norme} that we have the controls:
    $$\|P_g(h)-P_{g^D}(h)\|_{r^{-2}_DC^{\alpha}_\beta(g^D)} \leqslant C\|g-g^D\|_{C^{2,\alpha}_{\beta}(g^D)}\|h\|_{C^{2,\alpha}_{\beta}(g^D)}.$$
    
    Let us now consider $\sum_k \chi_{\mathcal{A}_k(t,\epsilon)}H_k $ and $\sum_k \chi_{\mathcal{A}_k(t,\epsilon)}H_k' $. These tensors being all supported in the annuli $\mathcal{A}_k(t,\epsilon)$, we just need to restrict our attention to them. The crucial remark is that in \eqref{Expression en coordonnées P}, every term involves at least a derivative of $h$ or of $g-g^D$. Hence, we have a more precise control on $h$ a symmetric $2$-tensor supported in $\mathcal{A}_k(t,\epsilon_0)$
    \begin{align}
        \big\|P_g(h)-P_{g^D}(h)\big\|_{r^{-2}_DC^{\alpha}_\beta(g^D)} \leqslant& C\Big(\|g-g^D\|_{C^{2,\alpha}_{0}(g^D)}\|\nabla^2 h\|_{r^{-2}_DC^{\alpha}_{\beta}(g^D)}\nonumber\\
        &+\|\nabla(g-g^D)\|_{r^{-1}_DC^{1,\alpha}_{\beta}(g^D)}\|\nabla h\|_{r^{-1}_DC^{1,\alpha}_{\beta}(g^D)}\nonumber\\ &+\|\nabla^2(g-g^D)\|_{r^{-2}_DC^{\alpha}_{\beta}(g^D)}\| h\|_{C^{2,\alpha}_{0}(g^D)} \Big),\label{control PgPgD}
    \end{align}
    (notice the norms $C^{2,\alpha}_{0}(g^D)$ in which we have $\|\chi_{\mathcal{A}_k(t,\epsilon)}H_k\|_{C^{2,\alpha}_{0}(g^D)}\leqslant C|H_k|_{g_e}$ and $\|\chi_{\mathcal{A}_k(t,\epsilon)}H'_k\|_{C^{2,\alpha}_{0}(g^D)}\leqslant C|H_k'|_{g_e}$). There remains to control the derivatives of the tensors $ \chi_{\mathcal{A}_k(t,\epsilon)}H_k $ and $\chi_{\mathcal{A}_k(t,\epsilon)}H_k'$. Since the $H_k$ and $H'_k$ are constant on $\mathbb{R}^4$, and since the cut off functions are bounded in $C^2_0(g^D)$ by \eqref{est chik}, for $i \in \{1,2\}$, we have
    $$\|\nabla^i\big(\chi_{\mathcal{A}_k(t,\epsilon)}H_k\big)\|_{r^{-i}_DC^{2-i,\alpha}_\beta(g^D)}\leqslant C |H_k|_{g_e},$$
    and
    $$\|\nabla^i\big(\chi_{\mathcal{A}_k(t,\epsilon)}H_k'\big)\|_{r^{-i}_DC^{2-i,\alpha}_\beta(g^D)}\leqslant C |H_k'|_{g_e},$$ 
    which together with \eqref{control PgPgD} let us conclude that the third estimate holds.
    \\
    
    For the vector fields, we have the following rewriting for $X$ a vector field supported in $ \mathcal{A}_{k}(t,\epsilon_0)$
    \begin{align*}
        \big\| \delta_g\delta^*_g (X) - \delta_{g^D}\delta^*_{g^D} (X) \big\|_{r^{-1}_DC^{1,\alpha}_\beta(g^D)}  \leqslant&\; C\big(\big\|(\delta_{g}-\delta_{g^D})(\mathcal{L}_Xg) \big\|_{r^{-1}_DC^{1,\alpha}_\beta(g^D)} \\ &+\big\|\delta_{g^D}(\mathcal{L}_{X}(g^D-g)) \big\|_{r^{-1}_DC^{1,\alpha}_\beta(g^D)} \big).
    \end{align*}
    We moreover know that for $X_k$ a linear vector field in the kernel of $\delta_e\delta_e^*$, then the symmetric $2$-tensor $\mathcal{L}_{X_k}g_e$ is constant, and more generally, for $H_k$ a constant symmetric $2$-tensor, we have $\delta_{g_e}(\mathcal{L}_{X_k}H_k)=0$ on $\mathbb{R}^4$. Using these two facts and the controls of the cut-off functions, we conclude that the last estimate of the statement holds by an argument similar to the above one for $2$-tensors and $P$.
\end{proof}

\subsubsection{Elliptic estimates for the decoupling norms}
Some elliptic estimates are still satisfied in these norms.

\begin{prop}\label{est ellipt with cstes}
    Let $0<\beta<1$, $g$ a metric, $h$ a symmetric $2$-tensor and $X$ a vector field  on $M_o$ (respectively $N_j$, or $M$). Then, there exists $\epsilon_*=\epsilon_*(g_o,g_{b_j},g^D, \beta)>0$ and $C>0$ such that if we have $\|g-g_\bullet\|_{C^{2,\alpha}_{\beta,*}(g_\bullet)}\leqslant \epsilon_*$, where $g_\bullet$ is one of the norms $g_o$, $g_{b_j}$ or $g^D$, then,
    $$\|h\|_{C^{2,\alpha}_{\beta,*}(g_\bullet)}\leqslant C \big(\|P_gh\|_{r^{-2}_\bullet C^{\alpha}_\beta(g_\bullet)}+\|h\|_{C^0_{\beta,*}(g_\bullet)}\big),$$
    and
    $$\|X\|_{r_\bullet C^{3,\alpha}_{\beta,*}(g_\bullet)}\leqslant C \big(\|\delta_g\delta^*_gX\|_{r^{-1}\bullet C^{1,\alpha}_{\beta}(g_\bullet)}+\|X\|_{r_\bullet C^0_{\beta,*}(g_\bullet)}\big).$$
\end{prop}
\begin{proof}
    Let $g_\bullet$ be one of the metrics $g_o$, $g_{b_j}$ or $g^D$, and for all $k$, $H_k$ a traceless constant symmetric $2$-tensor on $\mathbb{R}^4\slash\Gamma_k$, and $X_k$ a Killing vector field on $\mathbb{R}^4\slash\Gamma_k$. Let moreover $h_*$ be a symmetric $2$-tensor of $C^{2,\alpha}_\beta(g_\bullet)$ and $X_*$ be a vector field of $r_\bullet C^{3,\alpha}_{\beta}(g_\bullet)$, and define $h = h_*+ \sum_k \chi_\bullet X_k$ and $X= X_*+\sum_k \chi_\bullet X_k$, where $\chi_\bullet$ is $\chi_{\mathcal{A}_k(t,\epsilon)}$ or $\chi_{B_{g_\bullet}(\epsilon)}$ (of Definition \ref{def cutoffs all}) depending on the metric.
    
    We then have the following controls:
    $$\|P_{g_\bullet}\chi_\bullet H_k\|_{r^{-2}_\bullet C^{\alpha}_\beta(g_\bullet)}\leqslant C |H_k|_{g_e},$$
    and
    $$\|\delta_{g_\bullet}\delta_{g_\bullet}^*\chi_\bullet X_k\|_{r^{-1}_\bullet C^{1,\alpha}_\beta(g_\bullet)}\leqslant C \|X_k\|_{r_e C^0_0(g_e)}.$$
    Hence, for $h_*$, we have
    $$\|P_{g_\bullet}h_*\|_{r^{-2}_\bullet C^{\alpha}_\beta}\leqslant C\big(\|P_{g_\bullet}h\|_{r^{-2}_\bullet C^{\alpha}_\beta(g_\bullet)} + \sum_k|H_k|_{g_e}\big),$$
    and the expected estimate for $g = g_\bullet$ is then a consequence of the elliptic estimates in the weighted spaces of Lemma \ref{Estimations elliptic P} which give 
    $$ \|h_*\|_{C^{2,\alpha}_{\beta}(g_\bullet)}\leqslant C\big(\|P_{g_\bullet}h_*\|_{r^{-2}_\bullet C^{\alpha}_\beta(g_\bullet)}+\|h_*\|_{C^0_{\beta}(g_\bullet)} \big),$$
    and imply therefore that
    $$\|h\|_{C^{2,\alpha}_{\beta,*}(g_\bullet)}\leqslant 2C^2 \big(\|P_gh\|_{r^{-2}_\bullet C^{\alpha}_\beta(g_\bullet)}+\|h\|_{C^0_{\beta,*}(g_\bullet)}\big).$$
    The same argument works for the operator $\delta\delta^*$ on the vector fields thanks to the elliptic estimates of Lemma \ref{Estimations elliptic delta delta}.

    Proposition \ref{contrôles with perturbations triviales} finally lets us go from the metric $g_\bullet$ to a metric $g$ satisfying $\|g-g_\bullet\|_{C^{2,\alpha}_{\beta,*}(g_\bullet)} \leqslant \epsilon_*$. 
\end{proof}

\section{Reduced divergence-free gauge}

When the Einstein orbifold which we approximate has nonpositive scalar curvature, we can always put our Einstein metrics in Bianchi gauge with respect to a naïve desingularization (see \cite[Lemme 8.2]{biq1} adapted to our norms). When the Ricci curvature of our Einstein manifolds is positive, this is not necessarily true, but we can still use the divergence-free gauge. This is the goal of this section whose main result is Proposition \ref{Mise en Reduced divergence-free}. To show this, we will use a Banach fixed point theorem approach which necessitates the study of the linearized equation:
$$ \delta\delta^*X = - \delta h, $$
where $X$ is a vector field, and $h$ a symmetric $2$-tensor.

In our degenerating situation, we want to obtain estimates in our weighted norms which are independent of the gluing scales. A difficulty is that our limit orbifold might have more symmetries than the Ricci-flat ALE spaces (for example, $\mathbb{S}^4\slash\mathbb{Z}_2$ desingularized by Eguchi-Hanson metrics). The associated Killing vector fields would give an approximate kernel for $\delta\delta^*$ which would not be an actual kernel or cokernel. We will need to define a reduced divergence-free gauge to obtain uniform estimates as the gluing scales go to zero.

\begin{rem}
    All along this section, if nothing is precised, an Einstein orbifold $(M_o,g_o)$ will be either compact or ALE.
\end{rem}

\subsection{Kernel of the linearization}\label{kernel of la linéarisation}
Let us focus on the operator $\delta\delta^*$  on a flat cone $(\mathbb{R}^4\slash\Gamma,g_e)$, on an orbifold $(M_o,g_o)$, and on Ricci-flat ALE orbifolds $(N_j,g_{b_j})$.

\paragraph{On a flat cone.}
On the flat cone $(\mathbb{R}^4\slash\Gamma,g_e)= (\mathbb{R}^+\times \mathbb{S}^3\slash\Gamma, dr^2+r^2 g_{\mathbb{S}^3\slash\Gamma})$, according to \cite[Section 2]{ct}, any $1$-form on $\mathbb{R}^4\slash\Gamma$ is a countable sum of $1$-forms of one of the following types which are preserved by $\delta\delta^*$:
\begin{enumerate}
    \item $p(r)\psi$, where $\delta_{\mathbb{S}^3\slash\Gamma}\psi = 0$, and $\psi$ is eigenvector of the Hodge Laplacian of $\mathbb{S}^3\slash\Gamma$,
    \item $r^{-1}l(r)\phi dr + u(r)r d_{\mathbb{S}^3\slash\Gamma} \phi$, and $\phi$ is eigenfunction of the Hodge Laplacian of $\mathbb{S}^3\slash\Gamma$,
\end{enumerate}
where $p,l,u: \mathbb{R}^+\to \mathbb{R}$ and $\phi: \mathbb{S}^3\slash\Gamma\to \mathbb{R}$ are functions, and where $\psi$ is a $1$-form on $\mathbb{S}^3\slash\Gamma$.

According to \cite[Section 4.1]{av}, thanks to the computation of the eigenvalues of the Laplacian and of the Hodge Laplacian on the $1$-forms of the sphere \cite[Theorem C]{fol}, the solutions to $\delta_e\delta_e^* \omega = 0$ are countable sums of $1$-forms of the following types
\begin{enumerate}
    \item $r^{a^\pm_j}\psi$ with $a^\pm_j:= \pm (1+j)$, $j\in \mathbb{N}^*$, where $\psi$ is an eigenvector the Hodge Laplacian,
    \item $r^{b^{\pm}_j}d_{\mathbb{S}^3\slash\Gamma}\phi \;+\; b^\pm_j r^{b^\pm_j-1}\phi dr$, or $2r^{b^{\pm}_j+2}d_{\mathbb{S}^3\slash\Gamma}\phi \;+\; b^\mp_j r^{b^\pm_j+1}\phi dr$, with $b_j^\pm = -1 \pm (1+j)$, $j\in \mathbb{N}$ and where $\phi$ is an eigenfunction the Hodge Laplacian.
\end{enumerate}

Since we are interested in solving an equation 
$$\delta\delta^*X = - \delta h, $$
where $X$ is a vector field, and $h$ a symmetric $2$-tensor is in $C^{2,\alpha}_{\beta,*}$, we are naturally looking for $X$ in $ r_DC^{3,\alpha}_{\beta,*}$. The \emph{exceptional values} of $\delta_e \delta_e^*$ are the values $\gamma\in \mathbb{R}$ such that there exists a homogeneous $1$-form whose norm is proportional to $r_e^\gamma$ in the kernel of $\delta_e \delta_e^*$. We are interested in the exceptional values around the exceptional value $1$ associated to the linear vector fields of the kernel of $\delta_e\delta^*_e$.

\begin{lem}\label{poids crit deltadelta}
    On $(\mathbb{R}^4\slash\Gamma)\backslash \{0\}$ for $\Gamma \neq \{e\}$, $1$ is the only exceptional value between $-3$ and $2$.
\end{lem}
\begin{proof}
According to the above discussion, the exceptional values are a priori of the form $ a^\pm_j-1 = -1\pm(1+j) $ for $j\in \mathbb{N}^*$, $b^\pm_j-1 = -2\pm(1+j)$ with $j\in \mathbb{N}$, or $ b^\pm_j+1=\pm(1+j) $ with $j\in \mathbb{N}$. Let us first note that $a^\pm_j-1\in (-3,2)$ for $j\in \mathbb{N}^*$ implies that $a^\pm_j-1 = 1$, and therefore that no other exceptional value between $-3$ and $2$ come from the first type of $1$-form. 

For $b^\pm_j+1$, the values $0$ and $-1$ are a priori possible, and for $b^\pm_j-1$, $-1$ and $-2$ are a priori possible. However, these values cannot appear on a flat cone $\mathbb{R}^4\slash\Gamma$ for $\Gamma\neq \{e\}$. Indeed, the values $b^\pm_j-1 =0$ and $b^\pm+1 = -2$ only appear if $-3$ is an eigenvalue of the Laplacian on the link of the cone, but this is not the case for $\mathbb{S}^3\slash\Gamma$ because there does not exist any non zero $\Gamma$-invariant linear function on $\mathbb{R}^4$. 

For the values $b^\pm_j-1=-1$ and $b^\pm_j+1=-1$, we use the form of the solutions. In the first case, $b^\pm_j=0$ gives
$$r^{b^{\pm}_j}d_{\mathbb{S}^3\slash\Gamma}\phi \;+\; b^\pm_j r^{b^\pm_j-1}\phi dr = d_{\mathbb{S}^3\slash\Gamma}\phi,$$
for $\Delta_{\mathbb{S}^3\slash\Gamma} \phi=0$, therefore $\phi$ is constant and finally $d_{\mathbb{S}^3\slash\Gamma}\phi =0$. In the second case, the equality $b^\pm_j=-2$, that is $b^\mp_j=0$, gives
$$2r^{b^{\pm}_j+2}d_{\mathbb{S}^3\slash\Gamma}\phi \;+\; b^\mp_j r^{b^\pm_j+1}\phi dr = d_{\mathbb{S}^3\slash\Gamma}\phi,$$
for $\Delta_{\mathbb{S}^3\slash\Gamma} \phi=0$, therefore $\phi$ is constant and finally $d_{\mathbb{S}^3\slash\Gamma}\phi =0$.
\end{proof}

The $1$-forms associated to the exceptional value $1$ are sum of $1$-forms of the three following types:
\begin{enumerate}
    \item $r^2\psi$, where $\psi$ is the dual of a Killing vector field of $\mathbb{S}^{3}$,
    \item $r dr$,
    \item $2r\phi d r + r^2d_{\mathbb{S}^3\slash\Gamma} \phi$.
\end{enumerate}

\paragraph{On an orbifold or an ALE space.}
Since there is no exceptional value other than $1$ in $(-3,2)$, we have the following result on an orbifold ALE. 

\begin{prop}\label{delta inj fredholm}\label{controle deltadelta sur les espaces modèles}
Let $(N_j,g_{b_j})$ be a Ricci-flat ALE orbifold. For $0<\beta<1$, the operator $$\delta_{g_{b_j}}\delta_{g_{b_j}}^*: r_{b_j}C^{3,\alpha}_{\beta,*}\to r_{b_j}^{-1}C^{1,\alpha}_\beta$$
is bijective.

Let $(M_o,g_o)$ be a compact Einstein orbifold. For $0<\beta<1$, the operator $$\delta_{g_o}\delta_{g_o}^*: r_oC^{3,\alpha}_{\beta,*}\to r_o^{-1}C^{1,\alpha}_\beta$$ is Fredholm and both its kernel and its cokernel are equal to $\mathbf{K}_o$, the set of Killing vector fields of $(M_o,g_o)$.

As a consequence, there exist $C_o>0$ and $\epsilon_o>0$ depending on $g_o$ 
    such that if $\|g-g_o\|_{C^{2,\alpha}_{\beta,*}(M_o)}\leqslant \epsilon_o$, then we have for any vector field $X \in \mathbf{K}_o^\perp$ on $M_o$
    $$\|X\|_{r_oC^{3,\alpha}_{\beta,*}(g_o)}\leqslant C_o\|\delta_g\delta_g^*X\|_{r_o^{-1}C^{1,\alpha}_{\beta}(g_o)}.$$
    
    There also exists $C_j>0$ and $\epsilon_{j}>0$ depending on $g_{b_j}$
    such that if $\|g-g_{b_j}\|_{C^{2,\alpha}_{\beta,*}(N_j)}\leqslant \epsilon_{j}$ then we have for any vector field $X$ on $N_j$,
    $$\|X\|_{r_{b_j}C^{3,\alpha}_{\beta,*}(g_{b_j})}\leqslant C_j\|\delta_g\delta_g^*X\|_{r_{b_j}^{-1}C^{1,\alpha}_{\beta}(g_{b_j})}.$$
\end{prop}
\begin{proof}
    For orbifold singularities, we will first authorize our tensors to behave like $r^{1-\beta}$ for $0<\beta<1$ at the singularities, instead of being in $rC^{3,\alpha}_{\beta,*}$ to use the theory of elliptic operators in weighted Hölder spaces, see for instance \cite[Chapter 2]{pr} of \cite{lm}. Let us start by considering an Einstein orbifold $(M_o,g_o)$ and the operator $\delta_{g_o}\delta_{g_o}^*: r_oC^{3,\alpha}_{-\beta}\to r_o^{-1}C^{1,\alpha}_{-\beta}$ (notice the $-\beta$). Its kernel is composed of Killing vector fields of $g_o$. Indeed, if for $X\in r_oC^{3,\alpha}_{-\beta}$ we have  $\delta_{g_o}\delta_{g_o}^* X=0$, integrating by parts yields,
    \begin{align*}
        0=&\;\int_{M_o} \langle\delta_{g_o}\delta_{g_o}^*X,X\rangle dv_o\\
        &=\int_{M_o} |\delta_{g_o}^*X|^2_{g_o}dv_o + \lim_{r\to 0}\int_{\{r_o = r\}} \delta^*_{g_o}X(n,X)\\
        &=\int_{M_o} |\delta_{g_o}^*X|^2_{g_o}dv_o,
    \end{align*}
    where $n=\frac{\nabla r_o}{|\nabla r_o|}$, because the boundary term which is schematically $\lim_{r\to 0}(\mathcal{O}(r^{-\beta+1-\beta+3}))$ vanishes. Similarly, its cokernel is equal to the kernel of $\delta_{g_o}\delta_{g_o}^*$ on $r_o^{-3}C^{1,\alpha}_{\beta}$ which is also reduced to $\mathbf{K}_o$ because there is no exceptional value between $-3$ and $1$. 
    
    On an ALE orbifold $(N,g_b)$, let us assume that a vector field $X\in r_b^{1-\beta}C^{3,\alpha}_{0}$ satisfies $\delta_{g_b}\delta_{g_b}^*X = 0$. Since there is no exceptional value between $-3$ and $1$, we actually have $X = \mathcal{O}(r_b)$ when $r_b\to 0$ at the singular points of $(N,g_b)$, and $X = \mathcal{O}(r_b^{-3})$ at infinity. Let us then consider the following integration by parts,
    \begin{align*}
        0=&\;\int_N \langle\delta_{g_b}\delta_{g_b}^*X,X\rangle dv_b\\
        &=\int_N |\delta_{g_b}^*X|^2_{g_b}dv_b - \lim_{\rho\to \infty}\int_{\{r_b = \rho\}} \delta^*_{g_b}X(n,X)+ \lim_{r\to 0}\int_{\{r_b = r\}} \delta^*_{g_b}X(n,X)\\
        &=\int_N |\delta_{g_b}^*X|^2_{g_b}dv_b,
    \end{align*}
    where the boundary term vanishes because it is the sum of the limit for $r_b\to \infty$ of a $\mathcal{O}(|X|_{g_b}|\nabla X|_{g_b}r_b^3) = \mathcal{O}(r_b^{-4})$ and of the limit when $r_b\to 0$ of a $\mathcal{O}(|X|_{g_b}|\nabla X|_{g_b}r_b^3) = \mathcal{O}(r_b^{4})$.    Hence, we have $\delta_{g_b}^*X = 0$, and since $g_b$ is Ricci-flat, $\big(\delta_b+\frac{1}{2}d \textup{tr}_b\big)\delta_{g_b}^*X = \nabla_b^*\nabla_b X =0$, which implies that $\nabla_b X=0$ by integration by parts against $X$, and finally, that $X$ is parallel on $N$. Since $X$ tends to $0$ at infinity, we have $X=0$. The operator $\delta_{g_b}\delta_{g_b}^*: r_b^{1-\beta}C^{3,\alpha}_0\to r_b^{-1-\beta}C^{1,\alpha}_0$ is therefore injective. 
    
    The cokernel of the self adjoint operator $\delta_{g_b}\delta_{g_b}^*: r_b^{1-\beta}C^{3,\alpha}_0\to r_b^{-1-\beta}C^{1,\alpha}_0$ is equal to the kernel of $\delta_{g_b}\delta_{g_b}^*$ on $r_b^{-3+\beta}C^{1,\alpha}_{0}$ which is also reduced to $\{0\}$ because there is no exceptional values between $-3+\beta$ and $1-\beta$. The operator $\delta_{g_b}\delta_{g_b}^*: r_b^{1-\beta}C^{3,\alpha}_0\to r_b^{-1-\beta}C^{1,\alpha}_0$ is therefore bijective.
    \\
    
    Let us finally work in the norms we are interested in and study the operators $\delta_{g_o}\delta_{g_o}^*: r_oC^{3,\alpha}_{\beta,*}\to r_o^{-1}C^{1,\alpha}_\beta$ and $\delta_{g_b}\delta_{g_b}^*: r_bC^{3,\alpha}_{\beta,*}\to r_b^{-1}C^{1,\alpha}_\beta$.
    Since the spaces $r_bC^{3,\alpha}_{\beta,*}$ and $r_oC^{3,\alpha}_{\beta,*}$ are respectively only the direct sum of $r_bC^{3,\alpha}_{\beta}$ and $r_oC^{3,\alpha}_{\beta}$ with a space of finite dimension composed of cut-off of linear vector fields, the image remains closed and of finite codimension. We can be more precise by noticing that
    $$\delta_{g_o}\delta_{g_o}^* \big(r_oC^{3,\alpha}_{\beta,*}\big) = \delta_{g_o}\delta_{g_o}^* \big(r_oC^{3,\alpha}_{-\beta}\big)\cap r_o^{-1}C^{1,\alpha}_\beta.$$
    Indeed, we have $\delta_{g_o}\delta_{g_o}^* \big(r_oC^{3,\alpha}_{\beta,*}\big) \subset \delta_{g_o}\delta_{g_o}^* \big(r_oC^{3,\alpha}_{-\beta}\big)\cap r_o^{-1}C^{1,\alpha}_\beta$ because $r_oC^{3,\alpha}_{\beta,*}\subset r_oC^{3,\alpha}_{-\beta}$ and thanks to Proposition \ref{contrôles with perturbations triviales}. Conversely, if for $X\in r_oC^{3,\alpha}_{-\beta}$ we have $\delta_{g_o}\delta_{g_o}^*X\in r_o^{-1}C^{1,\alpha}_\beta$, then, since the only exceptional value between $1-\beta$ and $1+\beta$ is $1$ and corresponds to the linear kernel of $\delta_e\delta_e^*$, we have $X\in r_oC^{3,\alpha}_{\beta,*}$.
    Similarly, we conclude that
    $$\delta_{g_b}\delta_{g_b}^* \big(r_bC^{3,\alpha}_{\beta,*}\big) = \delta_{g_b}\delta_{g_b}^* \big(r_b^{1-\beta}C^{3,\alpha}_{0}\big)\cap r_b^{-1}C^{1,\alpha}_\beta,$$
    and finally, 
    $\delta_{g_o}\delta_{g_o}^*: r_oC^{3,\alpha}_{\beta,*}\to r_o^{-1}C^{1,\alpha}_\beta$ is Fredholm with $\mathbf{K}_o$ as kernel and cokernel, and $\delta_{g_b}\delta_{g_b}^*: r_bC^{3,\alpha}_{\beta,*}\to r_b^{-1}C^{1,\alpha}_\beta$ is bijective. We finally conclude by the open mapping theorem between Banach spaces which is stable by small perturbation of the operator.
\end{proof}

\subsection{Controls on the inverse of the linearization}

These controls will help us treat the case of trees of singularities with small enough gluing parameters.

For this, we approximate the kernel $\mathbf{K}_o$ on our naïve desingularization $(M,g^D)$ in the following way. Note that $\mathbf{K}_o=0$ for an ALE Ricci-flat orbifold $(M_o,g_o)$ by Proposition \ref{controle deltadelta sur les espaces modèles}. For all $\mathbf{X}_o\in \mathbf{K}_o$, according to Remark \ref{unicité decompositions}, on an orbifold, there exists a unique decomposition
$$\mathbf{X}_o = \mathbf{X}_{o,*} + \sum_k\chi_{B(p_k,\epsilon_0)}\mathbf{X}_{o,k},$$
such that $\|\mathbf{X}_o\|_{r_oC^{3,\alpha}_{\beta,*}} = \|\mathbf{X}_{o,*}\|_{r_oC^{3,\alpha}_\beta} + \sum_k\|\mathbf{X}_{o,k}\|_{r_eC^{0}_0}$ (other decompositions make the value infinite). We then define $\tilde{\mathbf{K}}_o$ as the space of the following vector fields on $M$
$$\tilde{\mathbf{X}}_{o,t}:= \chi_{M_o^{t}}\mathbf{X}_{o,*} + \sum_k\chi_{\mathcal{A}_k(t,\epsilon_ 0)}\mathbf{X}_{o,k},$$ for $\mathbf{X}_o\in \mathbf{K}_o$. Note that we therefore have $\tilde{\mathbf{X}}_{o,t} = \mathbf{X}_o$ on $M_o^{16t}$.

\begin{rem}\label{eq L2 Ck Ko}
    By elliptic regularity on $M_o$, the norms $L^2(g_o)$, $r_DC^{3,\alpha}_{\beta,*}(g_o)$ and $r_D^{-1}C^{1,\alpha}_\beta(g_o)$ are equivalent on the finite-dimensional space $\mathbf{K}_o$. Since the $C^{4}_0$-norms of the cut-off functions are bounded, we conclude that for $\epsilon$ and $t_{\max}$ small enough, the norms $L^2(g^D)$, $r_DC^{3,\alpha}_{\beta,*}(g^D)$ and $r_D^{-1}C^{1,\alpha}_\beta(g^D)$ are equivalent on $\tilde{\mathbf{K}}_o$.
\end{rem}

\begin{defn}[Reduced divergence-free gauge]
    We define the \emph{reduced divergence} operator, $\tilde{\delta}_g:=\pi_{\tilde{\mathbf{K}}_o^\perp} \delta_g$, where $\pi_{\tilde{\mathbf{K}}_o^\perp}$ is the  $L^2(g^D)$-orthogonal projection on $\tilde{\mathbf{K}}_o^\perp$. We will say that a metric $g_1$ is in \emph{reduced divergence-free} gauge with respect to a metric $g_2$ if $\tilde{\delta}_{g_2}g_1 = 0$.
\end{defn}
 Let us start by noticing that the operator $\tilde{\delta}_{g^D}$ is actually very close to $\delta_{g^D}$ for a naïve desingularization $g^D$ with small enough gluing parameters.

\begin{lem}
    There exists $C>0$ such that for any symmetric $2$-tensor $h\in C^{2,\alpha}_{\beta,*}(g^D)$, we have,
    \begin{equation}
        \big\|(\tilde{\delta}_{g^D}-\delta_{g^D}) h \big\|_{r_D^{-1}C^{1,\alpha}_{\beta}(g^D)}\leqslant C t_{\max} \|h\|_{C^{2,\alpha}_{\beta,*}(g^D)}.\label{controle delta tilde delta}
    \end{equation}
\end{lem}
\begin{proof}
    If $(M_o,g_o)$ is ALE, then, one has $\mathbf{K}_o = \{0\}$ and therefore $ \tilde{\delta}_{g^D}=\delta_{g^D} $. Let us focus on the case when $M_o$ is compact.
    
    Thanks to the equivalence of the different norms, see Remark \ref{eq L2 Ck Ko}, it is enough to show that the $L^2(g^D)$-projection on $\tilde{\mathbf{K}}_o$ of $\delta_{g^D} h$ is small to show the result. We naturally proceed by integration by parts. Let $\tilde{\mathbf{X}}_{o,t}\in \tilde{\mathbf{K}}_o$ for $\mathbf{X}_o\in \mathbf{K}_o$ be an approximate Killing vector field as above. We have,
    \begin{align*}
        \Big|\int_M(\delta_{g^D}h)\big(\tilde{\mathbf{X}}_{o,t}\big) dv_{g^D}\Big| =&\; \Big|\int_M\langle h, \delta_{g^D}^*(\tilde{\mathbf{X}}_{o,t}) \rangle_{g^D} dv_{g^D}\Big|,
    \end{align*}
    and,
    \begin{align*}
        \delta_{g^D}^*(\tilde{\mathbf{X}}_{o,t}) =&\; \delta^*_{g_o}\mathbf{X}_o +\delta^*_{g_o}((\chi_{M_o^t}-1)\mathbf{X}_{o,*}) + \delta^*_{g_o}\Big((\chi_{\mathcal{A}_k(t,\epsilon_0)}- \chi_{B_o(p_k,\epsilon_0)})\sum_k\mathbf{X}_{o,k}\Big)\\
        &+ \big(\delta_{g_o}^*-\delta_{g^D}^*\big)\Big(\chi_{\mathcal{A}_k(t,\epsilon_0)} \sum_k\mathbf{X}_{o,k}\Big),
    \end{align*}
    where by definition $\delta^*_{g_o}\mathbf{X}_o =0$.
    Thanks to Definition \ref{def cutoffs all}, on $M_o^t,$ $\chi_{M^{t}_o}\mathbf{X}_{o,*}$ is equal to $\mathbf{X}_{o,*}$ except on the annuli of radii $t_k^\frac{1}{4}$ and $2t_k^\frac{1}{4}$, and $\chi_{\mathcal{A}_k(t,\epsilon_0)}- \chi_{B_o(p_k,\epsilon_0)}$ is supported in $\epsilon_0^{-1}\sqrt{t}_k<r_D<2\epsilon_0^{-1}\sqrt{t}_k$ and well-defined on $M$. For any $l\in \mathbb{N}$, the cut off functions are moreover uniformly bounded in $C^l_0$. If we denote $\mathbb{1}_A$ the indicator function of $A$,
    for the vector fields, we therefore have
    $$r_D^{l}\Big(|\nabla^l\mathbf{X}_{o,*}|_{g_o}+\Big|\nabla^l\sum_k\mathbf{X}_{o,k}\Big|_{g_o}\Big)\leqslant C_l \|\mathbf{X}_o\|_{r_oC^l_{\beta,*}} r_D,$$
    and
    $$r_D^{l}\big|\nabla^l\big(g^D-g_o\big)\big|_{g_o}\leqslant C_l \mathbb{1}_{\{r_D<t_k^\frac{1}{4}\}} \big(r_D^2+t_k^2r_D^{-4}\big).$$
    
    As a consequence, because of the properties of the norms detailed in Remark \ref{comportement norme}, we have
    $$|\delta^*_{g_o}((\chi_{M_o^t}-1)\mathbf{X}_{o,*})|_{g_o}
    \leqslant \; C\mathbb{1}_{\{t_k^\frac{1}{4}<r_D<2t_k^\frac{1}{4}\}}\|\mathbf{X}_o\|_{r_oC^1_0(g_o)},$$
    on the annulus of radii $\epsilon_0^{-1}\sqrt{t}_k<r_D<2\epsilon_0^{-1}\sqrt{t}_k$, we have
    $$\Big|\delta^*_{g_o}\Big(\big(\chi_{\mathcal{A}_k(t,\epsilon_0)}- \chi_{B_o(p_k,\epsilon_0)}\big)\sum_k\mathbf{X}_{o,k}\Big)\Big|\leqslant C\mathbb{1}_{\{\epsilon_0^{-1}\sqrt{t}_k<r_D<2\epsilon_0^{-1}\sqrt{t}_k\}}\|\mathbf{X}_o\|_{r_oC^1_0(g_o)},$$
    and
    $$ \Big|\big(\delta_{g_o}^*-\delta_{g^D}^*\big)\Big(\chi_{\mathcal{A}_k(t,\epsilon_0)} \sum_k\mathbf{X}_{o,k}\Big)\Big|\leqslant C \mathbb{1}_{\{\epsilon_0^{-1}\sqrt{t}_k<r_D<t_k^\frac{1}{4}\}}\big(r_D^2+t_k^2r_D^{-4}\big)\|\mathbf{X}_o\|_{r_oC^1_0(g_o)}.$$
    Finally, since we have $\vol\big(A_e\big(t_k^\frac{1}{4},2t_k^\frac{1}{4}\big)\big)\approx t_{k},$ $\vol\big(A_e\big(\epsilon_0^{-1}\sqrt{t_k},2\epsilon_0^{-1}\sqrt{t_k}\big)\big)\approx t_k^2$, and also $\int_{\{\epsilon_0^{-1}\sqrt{t_k}<r_D<t_k^\frac{1}{4}\}}\big(r_D^2+t_k^2r_D^{-4}\big) dv_{g^D}\approx t_k^{\frac{3}{2}}+t_k^2|\log t_k|$, we have
    $$\Big|\int_M\langle \delta_{g^D}h, \tilde{\mathbf{X}}_{o,t} \rangle_{g^D} dv_{g^D}\Big|\leqslant C t_{\max} \|h\|_{C^0_0(g^D)}\|\mathbf{X}_o\|_{r_oC^1_0(g_o)}.$$
    Finally, let us denote $\tilde{\mathbf{Y}}_{o,t} = \pi_{\mathbf{K}_o}\delta_{g^D}h$, we have $\tilde{\delta}_{g^D}h = \pi_{\tilde{\mathbf{K}}_o^\perp}\delta_{g^D}h= \delta_{g^D}h - \tilde{\mathbf{Y}}_{o,t}$ with $$\|\tilde{\mathbf{Y}}_{o,t}\|_{r_D^{-1}C^{1,\alpha}_\beta(g^D)}\leqslant C t_{\max} \|h\|_{C^0_0(g^D)},$$
    by the equivalence of the norms of Remark \ref{eq L2 Ck Ko}.
\end{proof}

\begin{lem}\label{Mise en Reduced divergence-free linear}
    Let $0<\beta<1$, $0<\alpha<1$ and $(M,g^D)$ a naïve desingularization of a compact or ALE Einstein orbifold by a tree of singularities. Then, there exists $\tau_D >0$ and $\epsilon_D>0$ and $C_D>0$, only depending on $\beta$ and the constants of Proposition \ref{controle deltadelta sur les espaces modèles}, such that for $t_{\max}<\tau_D$, and any metric $g$ satisfying $\|g-g^D_{t}\|_{C^{2,\alpha}_{\beta,*}(g^D)}\leqslant \epsilon_D$, the operator $$\tilde{\delta}_g\delta_g^*: \tilde{\mathbf{K}}_o^\perp\cap r_DC^{3,\alpha}_{\beta,*}(g^D) \to \tilde{\mathbf{K}}_o^\perp\cap r^{-1}_DC^{1,\alpha}_{\beta}(g^D)$$ is invertible and we have for any vector field $X\perp \tilde{\mathbf{K}}_o$ on $M$,
    $$\|X\|_{r_DC^{3,\alpha}_{\beta,*}(g^D)}\leqslant C_D\|\tilde{\delta}_g\delta_g^*X\|_{r^{-1}_DC^{1,\alpha}_{\beta}(g^D)}.$$
\end{lem}
\begin{proof}
    Let $0<\epsilon<\epsilon_D^{\frac{1}{2-\beta}}<\epsilon_0$ for $\epsilon_D$ and $\epsilon$ which we will choose small enough along the proof, and assume that $t_{\max}< \epsilon^4$. Therefore, by construction, on each annulus $\mathcal{A}_k:=\mathcal{A}_k(t,\epsilon)$ between $N_k$ and $N_j$ or $N_k$ and $M_o$ (in which case, we will fix $T_o=1$), we have a diffeomorphism $$\Phi_k: A_e\big(\epsilon^{-1}\sqrt{T_j}\sqrt{t_k},\epsilon\sqrt{T_j}\big)\subset \mathbb{R}^4\slash\Gamma_k\to \mathcal{A}_k(t,\epsilon)\subset M,$$
    such that for all $0<\beta<1$, there exists $C>0$, for which we have
    \begin{equation}
        \big\|\Phi^*_kg^D-g_e\big\|_{C^{2,\alpha}_\beta(A_e(\epsilon^{-1}\sqrt{T_j}\sqrt{t_k},\epsilon\sqrt{T_j}))}\leqslant C \epsilon^{2-\beta}<C\epsilon_D.\label{controle coord annulus gD}
    \end{equation}
   Until the end of the proof, we will use the notation
    $$A_k:= A_e(\epsilon^{-1}\sqrt{T_j}\sqrt{t_k},\epsilon\sqrt{T_j}).$$
    
    According to the estimate \eqref{controle delta tilde delta}, for $t_{\max}$ small enough, it is enough to have
    $$\|X\|_{r_DC^{3,\alpha}_{\beta,*}(g^D)}\leqslant \frac{C_D}{2}\|\delta_{g^D}\delta_{g^D}^*X\|_{r^{-1}_DC^{1,\alpha}_{\beta}(g^D)}$$
    to obtain the stated result.
    
    The diffeomorphisms $\Phi_k: A_k\to \mathcal{A}_k$ allow us to pull the situation back on $\mathbb{R}^4$, where the ratio of the annuli $A_k$ is $\epsilon^2t_k^{-1/2}$ which is arbitrarily large for $t_{\max}$ arbitrarily small. According to the estimate \eqref{estimation inverse annulus X} of Proposition \ref{inverse delta delta on annulus} and thanks to the controls of Proposition \ref{contrôles with perturbations triviales}, for $t_{\max}$ and $\epsilon_D$ small enough, then, there exist linear vector fields $X_k$ of the kernel of $\delta_e\delta^*_e$ such that the vector fields $\chi_{\mathcal{A}_k}X_k$ in the annuli $\mathcal{A}_k$ satisfy
    \begin{align}
        \|\Phi_k^*X-X_k\|_{r_eC^{3,\alpha}_0(A(T_j^{1/2}t_k^{1/4}))} \leqslant& \; C_e T_j^{\frac{1}{2}}t_k^{\frac{1}{4}}\Big(\epsilon^{-\beta}t_k^{\frac{\beta}{4}}\big\|\delta_{g_e}\delta_{g_e}^*\Phi_k^*X\big\|_{r^{-1}_eC^{1,\alpha}_{\beta}(A_k)} \nonumber\\
        &+ 2 \epsilon^{-1}t_k^{\frac{1}{4}} \|\Phi_k^*X-X_k\|_{r_eC^{3,\alpha}_{\beta}(A_k)} \Big)\nonumber  \\
        \leqslant& \; 2 C_eT_j^{\frac{1}{2}}t_k^{\frac{1}{4}}\Big(\epsilon^{-\beta}t_k^{\frac{\beta}{4}}\big\|\big(\delta_{g^D}\delta_{g^D}^*X\big)_{|\mathcal{A}_k}\big\|_{r^{-1}_DC^{1,\alpha}_{\beta}(g^D)}\nonumber\\
        &+ 2 \epsilon^{-1}t_k^{\frac{1}{4}} \|(X-\chi_{\mathcal{A}_k}X_k)_{|\mathcal{A}_k}\|_{r_DC^{3,\alpha}_\beta(g^D)}\Big)\label{controle annulus} 
    \end{align}
     on $ A(\sqrt{T_j}t_k^{1/4}):= A_e((1/2)\sqrt{T_j}t_k^{1/4},4\sqrt{T_j}t_k^{1/4})$. Let us then consider the decomposition 
    \begin{equation}
        X = X_*+\sum_k\chi_{\mathcal{A}_k}X_k\label{decomposition chp vect}
    \end{equation}
    with the above linear vector fields $X_k$ in the kernel of $\delta_e\delta_e^*$ for the rest of the proof.
    
    The objective is now to show that there exists a constant $C_D>0$ such that $$\|X_*\|_{r_DC^{3,\alpha}_{\beta}(g^D)}+ \sum_k \|X_k\|_{r_eC^{0}_0(g_e)}\leqslant C_D \|\delta_{g^D}\delta_{g^D}^*X\|_{r^{-1}_DC^{1,\alpha}_{\beta}(g^D)}.$$
    In order to do this, we will reduce our situation to $M_o$ and to the $N_j$ where such controls have been shown in Proposition \ref{controle deltadelta sur les espaces modèles}. 
    \\
    
    On $M_o^{t/16}$, $g^D-g_o$ is supported in $M_o^{t/16}\backslash M_o^{16t}$, that is where $\frac{1}{2} t_k^{1/4}\leqslant r_D< 2 t_k^{1/4}$ on each annulus $\mathcal{A}_k$, and for all $l\in \mathbb{N}$, there exists $C_l>0$ such that in these regions, we have
    \begin{equation}
        t_k^{\frac{l}{4}}|\nabla^l (g^D-g_o)|_{g_o}\leqslant C_l t_k^\frac{1}{2}.\label{gD go recollement}
    \end{equation}
    
    Consider the cut-off function $\chi_{M_o^{t/16}}$ of Definition \ref{def cutoffs all} supported in $M_o^{t/16}$ such that $\chi_{M_o^{t/16}}\equiv 1$ on $M_o^{t}$ and such that for all $l\in \mathbb{N}$, there exists $C_l>0$ for which in each $\mathcal{A}_k$, 
    \begin{equation}
        t_k^\frac{l}{4}|\nabla^l\chi_{M_o^{t/16}}|_{g^D}\leqslant C_l. \label{est chiMo}
    \end{equation}
    We then define a vector field $X_o$ on $M_o$ by
    $$X_o:= \chi_{M_o^{t/16}}X_* + \sum_{k\in K_o} \chi_{B_o(p_k,\epsilon)}X_k,$$
    where $K_o$ is the set of $k$ such that the annulus $\mathcal{A}_k$ has a nonempty intersection with $M_o^t$. 
    
    By construction, $X_*=X_{o,*}$ on $M_o^t$ and we therefore have the following obvious control. Denoting $X_{o,*}:=  \chi_{M_o^{t/16}}X_*$, we have
    \begin{equation}
        \|X_{o,*}\|_{r_oC^{3,\alpha}_{\beta}(g_o)}\geqslant \|X_{o,*}\|_{r_oC^{0}_{\beta}(g_o)}\geqslant \|(X_*)_{|M_o^t}\|_{r_DC^{0}_{\beta}(g^D)}.\label{contrôle Xo*}
    \end{equation}
    
    On $M_o^{\frac{t}{16}}$, we have,
    \begin{align}
        \delta_{g^D}\delta_{g^D}^*X= \delta_{g_o}\delta_{g_o}^*X_o+\delta_{g_o}\delta_{g_o}^*(X-X_o)+\big(\delta_{g^D}\delta_{g^D}^*-\delta_{g_o}\delta_{g_o}^*\big)X.\label{decomposition deltadeltagD X}
    \end{align}
    Since the cut off functions are bounded in norm $C^{3,\alpha}_0(g^D)$ and $C^{3,\alpha}_0(g_o)$ by \eqref{est chiMo}, and since their derivatives are supported in $M_o^{t/16}\backslash M_o^t$, we have the following bound on the last two terms of \eqref{decomposition deltadeltagD X}: for $C>0$ depending on the cut off function, we have
    \begin{itemize}
        \item $\|\delta_{g_o}\delta_{g_o}^*(X-X_o)\|_{r^{-1}_DC^{1,\alpha}_{\beta}(g^D)}\leqslant C\big\|(X_*)_{|M_o^{t/16}\backslash M_o^t}\big\|_{r_DC^{3,\alpha}_{\beta}(g^D)} $ since the difference between $X$ and $X_o$ on $M_o^{16t}$ only comes from the cut-off on $X_*$, and
        \item $\big\|\big(\big(\delta_{g^D}\delta_{g^D}^*-\delta_{g_o}\delta_{g_o}^*\big)X\big)_{|M_o^{\frac{t}{16}}}\big\|_{r^{-1}_DC^{1,\alpha}_{\beta}(g^D)}\leqslant C\sum_{k\in K_o} t_k^\frac{1}{2} \|X\|_{r_DC^{3,\alpha}_{\beta,*}} $ thanks to \eqref{gD go recollement}.
    \end{itemize}
    Consequently, by \eqref{decomposition deltadeltagD X}, and using \eqref{controle annulus}, for $C>0$ depending on the above constants, we have
    \begin{align}
        \big\|\big(\delta_{g^D}\delta_{g^D}^*X\big)_{|M_o^{t/16}}\big\|_{r^{-1}_DC^{1,\alpha}_{\beta}(g^D)}\geqslant &\; \|\delta_{g_o}\delta_{g_o}^*X_o\|_{r^{-1}_oC^{1,\alpha}_{\beta}(g_o)}\nonumber\\
        &- C\big\|(X_*)_{|M_o^{t/16}\backslash M_o^t}\big\|_{r_DC^{3,\alpha}_{\beta}(g^D)}\nonumber\\
        &-C\sum_{k\in K_o} t_k^\frac{1}{2} \|X\|_{r_DC^{3,\alpha}_{\beta,*}}\nonumber\\
        \geqslant &\; \|\delta_{g_o}\delta_{g_o}^*X_o\|_{r^{-1}_oC^{1,\alpha}_{\beta}(g_o)}\nonumber\\
        &- 2C\Big( \big\|\big(\delta_{g^D}\delta_{g^D}^*X\big)_{|\mathcal{A}_k}\big\|_{r^{-1}_DC^{1,\alpha}_{\beta}(g^D)}\nonumber\\
        &+\sum_{k\in K_o} \epsilon^{\beta-1}t_k^{\frac{1-\beta}{4}}\|(X-\chi_{\mathcal{A}_k(t,\epsilon)}X_k)_{\mathcal{A}_k}\|_{r_DC^{3,\alpha}_\beta(g^D)}\Big)\nonumber
        \\
        &-C\sum_{k\in K_o} t_k^\frac{1}{2} \|X\|_{r_DC^{3,\alpha}_{\beta,*}}.\label{contrôle deltadelta X Mo}
    \end{align}
    
    Now, when $t_{\max}\to 0$, we have $$\frac{\|\pi_{\mathbf{K}_o^\perp}X_o\|_{r_oC^{3,\alpha}_{\beta,*}(g_o)}}{\|X_o\|_{r_oC^{3,\alpha}_{\beta,*}(g_o)}} \to 0$$
    because $X\perp \tilde{\mathbf{K}}_o$. Proposition \ref{controle deltadelta sur les espaces modèles} therefore yields, for $t_{\max}$ small enough,
    \begin{equation}
        \|X_o\|_{r_oC^{3,\alpha}_{\beta,*}(g_o)}\leqslant 2C_o\|\delta_{g_o}\delta_{g_o}^*X_o\|_{r^{-1}_oC^{1,\alpha}_{\beta}(g_o)}.\label{contr Xo}
    \end{equation}
    Therefore, thanks to \eqref{contrôle Xo*} and \eqref{contrôle deltadelta X Mo}, for $t_{\max}$ small enough, and denoting by $C'>0$ a constant that may change from line to line but only depending on the previous ones of this proof, and therefore only on $g_o$ and the $g_{b_j}$ and $\gamma(t_{\max}):= \sum_{k}t_{\max}^\frac{1}{2} + t_{\max}^\frac{1-\beta}{4} $, we have
    \begin{align}
        \|(X_*)_{|M_o^t}\|_{r_DC^{0}_{\beta}(g^D)} + \sum_{k\in K_o}& \|X_k\|_{r_eC^0_0(g_e)} -C' \gamma(t_{\max})\|X\|_{r_DC^{0}_{\beta,*}(g^D)}\nonumber \\
        \leqslant & \|X_o\|_{r_oC^{3,\alpha}_{\beta,*}(g_o)}- C'\gamma(t_{\max})\|X\|_{r_DC^{0}_{\beta,*}(g^D)} \nonumber\\
        \leqslant& 2C_o\|\delta_{g_o}\delta_{g_o}^*X_o\|_{r^{-1}_oC^{1,\alpha}_{\beta}(g_o)}- C'\gamma(t_{\max})\|X\|_{r_DC^{0}_{\beta,*}(g^D)} \nonumber\\
        \leqslant& C'\big\|\delta_{g^D}\delta_{g^D}^*X\big\|_{r^{-1}_DC^{1,\alpha}_{\beta}(g^D)},\label{control csts}
    \end{align}
    where we successively used \eqref{contrôle Xo*}, \eqref{contr Xo} and \eqref{contrôle deltadelta X Mo}.
    Indeed, on an orbifold $(M_o,g_o)$, the vector fields $X_k$ of the decomposition \eqref{decomposition chp vect} reaching the infimum of the definition of the norm $\|.\|_{r_oC^{3,\alpha}_{\beta,*}}$ are determined by the limit of $r_o^{-1}X_o$ at each singular point according to Remark \ref{unicité decompositions}. Here, the infimum is therefore reached with the $X_k$ of the decomposition \eqref{decomposition chp vect}.
    \\
    
    We next consider the vector field $X_1:= X- \sum_{k\in K_o} \chi_{\mathcal{A}_k}X_k$ which satisfies for a constant $C>0$,
    \begin{align}
        \| \delta_{g^D}\delta_{g^D}^*X_1 \|_{r_D^{-1}C^{1,\alpha}_\beta(g^D)}\leqslant&\; C\big( \| \delta_{g^D}\delta_{g^D}^*X \|_{r_D^{-1}C^{1,\alpha}_\beta(g^D)}+\gamma(t_{\max})\|X\|_{r_DC^{0}_{\beta,*}(g^D)}\big)\label{est X1}
    \end{align}
    thanks to the control \eqref{control csts} of $\sum_{k\in K_o} \|X_k\|_{r_eC^0_0(g_e)}$.
    
    Given $j\in K_o$, the Ricci-flat ALE orbifold $(N_j,g_{b_j})$ is glued to $M_o$ and we can extend the vector field $X_1 = X_{*} + \sum_{k \notin K_o} \chi_{\mathcal{A}_l}X_l$ to $N_j$ by 
    $$X_{j}:= \chi_{N_j^{t/16}}X_* + \sum_{l\in K_j} \chi_{B_j(p_l,\epsilon)}X_l,$$
    where $K_j$ is the set of $k\neq j$ such that $\mathcal{A}_k$ has a nonempty intersection with $N_j^t$. 
    
    \begin{rem}
        By considering $X_1$ instead of $X$, we do not have a linear vector field of the kernel of $\delta_e\delta^*_e$ to extend at at infinity of $N_j$. The vector field $X_j$ is therefore well controlled in $r_{b_j}C^{3,\alpha}_{\beta,*}(g_{b_j})$.
    \end{rem}
    
    The difference $\frac{g^D}{T_j}-g_{b_j}$ is supported in $N_j^{t/16}\backslash N_j^{16t}$ and there exists for all $l\in \mathbb{N}$, $C_l>0$ such that we have the following controls. Around the singular points where $$\frac{1}{2}\sqrt{T_j}t_k^{\frac{1}{4}}<r_D = \sqrt{T_j}r_{b_j}<2\sqrt{T_j}t_k^{\frac{1}{4}},$$ 
    we have
    \begin{equation}
        t_k^{\frac{l}{4}}\Big|\nabla^l \Big(\frac{g^D}{T_j}-g_{b_j}\Big)\Big|_{g_{b_j}}\leqslant C_l t_k^\frac{1}{2}.\label{gDgbj tk}
    \end{equation}
    and at infinity, where $$\frac{1}{2}\sqrt{T_j}t_j^{-\frac{1}{4}}<r_D=\sqrt{T_j}r_{b_j}<2\sqrt{T_j}t_j^{-\frac{1}{4}}$$ we have
    \begin{equation}
        t_j^{-\frac{l}{4}}\Big|\nabla^l \Big(\frac{g^D}{T_j}-g_{b_j}\Big)\Big|_{g_{b_j}}\leqslant C_l t_j^\frac{1}{2}.\label{gDgbj tj}
    \end{equation}
    
    Denoting $X_{j,*}:= \chi_{N_j^{t/16}}X_*$, we have
    \begin{equation}
        \|X_{j,*}\|_{r_{b_j}C^{3,\alpha}_{\beta}(g_{b_j})}\geqslant \|X_{j,*}\|_{r_{b_j}C^{0}_{\beta}(g_{b_j})}\geqslant \|(X_*)_{N^t_j}\|_{r_DC^{0}_{\beta}(g^D)},\label{contrôle Xj*}
    \end{equation}
    and thanks to \eqref{controle annulus} and the inequalities \eqref{gDgbj tk} and \eqref{gDgbj tj}, we have
    \begin{align*}
        \delta_{g^D}\delta_{g^D}^*X_1:= \delta_{g_{b_j}}\delta_{g_{b_j}}^*X_j+\delta_{g_{b_j}}\delta_{g_{b_j}}^*(X_1-X_j)+\big(\delta_{g^D}\delta_{g^D}^*-\delta_{g_{b_j}}\delta_{g_{b_j}}^*\big)X_1,
    \end{align*}
    analogously to \eqref{contrôle deltadelta X Mo}, we find for $C>0$ depending on the above constants such that
    \begin{align}
        \big\|\big(\delta_{g^D}\delta_{g^D}^*X_1\big)_{|N_j^{t/16}}\big\|_{r^{-1}_DC^{1,\alpha}_{\beta}(g^D)}\geqslant &\; \|\delta_{g_{b_j}}\delta_{g_{b_j}}^*X_{j}\|_{r^{-1}_{b_j}C^{1,\alpha}_{\beta}(g_{b_j})}\nonumber\\
        &- C\Big( \|\big(\delta_{g^D}\delta_{g^D}^*X_1\big)_{|\mathcal{A}_k}\|_{r^{-1}_DC^{1,\alpha}_{\beta}(g^D)}\nonumber\\
        &+ \sum_{k\in K_j} \epsilon^{\beta-1}t_k^{\frac{1-\beta}{4}}\|(X_*)_{\mathcal{A}_k}\|_{r_DC^{3,\alpha}_\beta(g^D)} \Big)\nonumber\\
        &-C t_k^\frac{1}{2} \|X_1\|_{r_DC^{3,\alpha}_{\beta,*}}- C t_j^\frac{1}{2}\|X_1\|_{r_DC^{3,\alpha}_{\beta,*}},\label{contrôle deltadelta Xo}
    \end{align}
    where we remark that $\|X_1\|_{r_DC^{3,\alpha}_{\beta,*}}\leqslant \|X_*\|_{r_DC^{3,\alpha}_{\beta}} + \sum_{k\in K_j} \|X_k\|_{r_eC^0_0(g_e)}$.
    
    Proposition \ref{controle deltadelta sur les espaces modèles} then yields
    $$ \|X_j\|_{r_jC^{3,\alpha}_{\beta,*}(g_{b_j})}\leqslant C_j\|\delta_{g_{b_j}}\delta_{g_{b_j}}^*X_j\|_{r^{-1}_jC^{1,\alpha}_{\beta}(g_{b_j})}, $$
    and thanks to the control \eqref{est X1}, we then have for a constant $C>0$ only depending on the constants of Propositions \ref{controle deltadelta sur les espaces modèles} and  \ref{inverse delta delta on annulus}, the control
    \begin{align}
        \|(X_*)_{|N_j^t}\|_{r_DC^{0}_{\beta}(g^D)} + \sum_{k\in K_j} \|X_k\|_{r_eC^0_0(g_e)} - C\gamma&(t_{\max})\|X\|_{r_DC^0_{\beta,*}(g^D)}\nonumber\\
        \leqslant &\; C\big\|\delta_{g^D}\delta_{g^D}^*X\big\|_{r^{-1}_DC^{1,\alpha}_{\beta}(g^D)},\label{control csts j}
    \end{align}
    similar to \eqref{control csts}.
    
    Iterating this to the other Ricci-flat ALE orbifolds of the tree of singularities, we get controls similar to \eqref{control csts j} on all the $N_j$ which, with \eqref{control csts} on $M_o$, give the following control on the whole manifold only depending on $g_o$ and the $g_{b_j}$ 
    \begin{align*}
        \|X_*\|_{r_DC^{0}_{\beta}(g^D)} + \sum_{k} &\|X_k\|_{r_eC^0_0(g_e)} - C\gamma(t_{\max}) \|X\|_{r_DC^{0}_{\beta,*}(g^D)}\\
        &\leqslant  C\big\|\delta_{g^D}\delta_{g^D}^*X\big\|_{r^{-1}_DC^{1,\alpha}_{\beta}(g^D)}
    \end{align*}
    and for $t_{\max}$ small enough. Together with the elliptic estimates of Proposition \ref{est ellipt with cstes}, this shows the stated result for $g=g^D$ because $M = M_0^t\cup \bigcup_jN_j^t$.
    \\
    
    To get the estimate for another metric $g$ close to $g^D$, we just use Proposition \ref{contrôles with perturbations triviales} to ensure that for $\|g-g^D\|_{C^{2,\alpha}_{\beta,*}(g^D)}$ arbitrarily small, $\tilde{\delta}_g\delta^*_g$ is arbitrarily close to $\delta_{g^D}\delta^*_{g^D}$ for the operator norm on $r_DC^{3,\alpha}_{\beta,*}(g^D)$.
    
    Finally, notice that $(\tilde{\delta}_{g^D})^* = \delta^*_{g^D} \pi_{\tilde{\mathbf{K}}_o^\perp}$, and therefore that $\tilde{\delta}_{g^D}\delta^*_{g^D}$ is self adjoint on $\tilde{\mathbf{K}}_o^\perp$. Its injectivity implies its surjectivity by integration by parts on the compact manifold $(M,g^D)$.  
\end{proof}

We can finally prove the main result of the section by fixed point theorem.

\begin{prop}\label{Mise en Reduced divergence-free}
    Let $0<\beta<\frac{1}{2}$, and $(M,g^D)=(M,g^D_{t})$ be a naïve desingularization of a compact Einstein orbifold, $(M_o,g_o)$. Then, there exist $\epsilon_D,\tau_D,C_D>0$ which only depends on the constants of Lemma \ref{Contrôle P espace modèles} such that for $t_{\max}\leqslant\tau_D$ and for any metric $g$ satisfying $\|g-g^D\|_{C^{2,\alpha}_{\beta,*}(g^D)}\leqslant \epsilon_D$, there exists a unique vector field $X\perp \tilde{\mathbf{K}}_o$ on $M$ for which, 
    $$\tilde{\delta}_{g^D}(\exp_X^*g)=0,$$
    where $\exp_X$ is the diffeomorphism $\exp_X: x\in M \mapsto \exp_x^{g^D}(X(x))$.
    
    We moreover have, $\|X\|_{r_DC^{3,\alpha}_{\beta,*}(g^D)}\leqslant C_D \|\tilde{\delta}_{g^D}(g-g^D)\|_{r^{-1}_DC^{1,\alpha}_{\beta}(g^D)}$, and therefore, there exists $\eta: \mathbb{R}^+\to \mathbb{R}^+$ with $\lim_0\eta=0$ such that we have 
    $$ \|\exp_X^*g-g\|_{C^{2,\alpha}_{\beta,*}(g^D)}\leqslant \eta\big(\|g-g^D\|_{C^{2,\alpha}_{\beta,*}(g^D)}\big).$$
\end{prop}
\begin{proof}
    Let us fix $g$ a metric on $M$, such that $\|g-g^D\|_{C^{2,\alpha}_{\beta,*}}\leqslant \epsilon$ for $\epsilon>0$ which we will choose small enough along the proof and define the operator $F_g: r_DC^{3,\alpha}_{\beta,*}(g^D)\to r^{-1}_DC^{1,\alpha}_{\beta}(g^D)$ which to a vector field $X$ associates $$F_g(X):=\tilde{\delta}_{(\exp_{g^D}X)_*g^D}g.$$
    
    The objective is therefore to find $X$ such that $F_g(X)=0$, which will imply that $\tilde{\delta}_{g^D}(\exp_{g^D}X)^*g = 0$ because for any diffeomorphism $\phi$, $\phi_*(\delta_{g^D}\phi^* g) = \delta_{\phi_*g^D}g$ (by applying $\phi_*$ to $g^D$ which is $C^\infty$, we do not loose regularity). The map $g\mapsto F_g$ is linear, and the linearization of the operator $F_{g^D}$ around zero is $\tilde{\delta}_{g^D}\delta^*_{g^D}$ which is invertible between the orthogonals of $\tilde{\mathbf{K}}_o$ according to Lemma \ref{Mise en Reduced divergence-free linear}. 
    
    There remains to control the nonlinear terms in our norms. Let us denote them $Q = F_{g^D}-\tilde{\delta}_{g^D}\delta^*_{g^D}$. We schematically have that
    $Q(X)$ is a converging sum of terms of the form $(\nabla_{g^D}^2X*X)*X*...*X$, $(\nabla_{g^D}X*\nabla_{g^D}X)*X*...*X$, $(\Rm(g^D)* X*X)*...*X $ which are at least quadratic in $X,$ where $*$ denotes various multilinear operations or contractions. 
    
    Using \eqref{behavior norm quadratic} of Remark \ref{comportement norme}, since on a compact manifold $\|r_D\|_{C^{3,\alpha}_0}$ is bounded, there exists $C>0$ such that 
    \begin{equation}
        \|.\|_{r_D^{-1}C^{1,\alpha}_\beta(g^D)} \leqslant C\|.\|_{C^{1,\alpha}_\beta(g^D)},\label{control rD-1 r0}
    \end{equation}
    and for any vector field $X$, assuming $\beta<\frac{1}{2}$, we have
    \begin{equation}
        \|\nabla_{g^D} X\|_{r_D^{-1/2}C^{2,\alpha}_0(g^D)}\leqslant\|\nabla_{g^D} X\|_{r_D^{-1/2}C^{2,\alpha}_\beta(g^D)} \leqslant \| X\|_{r_D^{1/2}C^{3,\alpha}_\beta(g^D)} \leqslant C\|X\|_{r_DC^{3,\alpha}_{\beta,*}(g^D)}.\label{control rD-1/2 r0}
    \end{equation}
    We therefore find
    \begin{align*}
        \|Q(X)- &Q(X')\|_{r_D^{-1}C^{1,\alpha}_\beta(g^D)}\\
        \leqslant&\; C\Big(\|X-X'\|_{r_DC^{1,\alpha}_{0}(g^D)}\big(\|\nabla^2 X\|_{r^{-1}_DC^{1,\alpha}_\beta(g^D)}+\|\nabla^2 X'\|_{r^{-1}_DC^{1,\alpha}_\beta(g^D)}\big)\\
        &+\big(\|X\|_{r_DC^{1,\alpha}_{0}(g^D)}+\|X'\|_{r_DC^{1,\alpha}_{0}(g^D)}\big)\|\nabla^2 (X-X')\|_{r^{-1}_DC^{1,\alpha}_\beta(g^D)}\\
        &+\big(\|X\|_{r_DC^{1,\alpha}_{0}(g^D)}+\|X'\|_{r_DC^{1,\alpha}_{0}(g^D)}\big)\big(\|X-X'\|_{r_DC^{1,\alpha}_{0}(g^D)}\big)\|\Rm(g^D)\|_{r^{-2}_DC^{1,\alpha}_\beta(g^D)}\\
        &+\|\nabla (X-X')\|_{r_D^{-1/2}C^{2,\alpha}_0(g^D)}\big(\|\nabla X\|_{r_D^{-1/2}C^{2,\alpha}_\beta(g^D)}+\|\nabla X'\|_{r_D^{-1/2}C^{2,\alpha}_\beta(g^D)}\big)\\
        \leqslant&\; 3C \big(\|X\|_{r_DC^{3,\alpha}_{\beta,*}(g^D)}+\|X'\|_{r_DC^{3,\alpha}_{\beta,*}(g^D)}\big)\|X-X'\|_{r_DC^{3,\alpha}_{\beta,*}(g^D)},
    \end{align*}
    notice the different norms with $\beta$ and $0$ for the weight power. We controlled the
    $C^{1,\alpha}_\beta(g^D)$-norm (which is larger than the $r_D^{-1}C^{1,\alpha}_\beta(g^D)$-norm by \eqref{control rD-1 r0}) in the first three lines and the $r_D^{-1}C^{1,\alpha}_\beta(g^D)$-norm in the last one (notice the $r_D^{-1/2}$-norms controlled by \eqref{control rD-1/2 r0}).
    \begin{rem}
        Using $r_D^{-1/2}$-norms was necessary because the first derivatives of the linear element of the kernel of $\delta_{g_e}\delta^*_{g_e}$ do not decay in the neck regions, that is, $\nabla_{g^D}\chi_{\mathcal{A}_k(t,\epsilon)}X_k\notin C^{2,\alpha}_\beta(g^D)$.
    \end{rem}
    
    The crucial reason for such a control of the nonlinear terms, already noted in \cite[Proof of Lemma 8.2]{biq1}, is that our norm is equivalent to a norm $C^{k,\alpha}(g^D)$ weighted by a function uniformly bounded below by $1$ independently on $t$, see \eqref{expression poids equivalent 1}. We can therefore finally put our metrics in gauge with respect to each other thanks to a fixed point theorem with explicit constant below, Lemma \ref{fonctions inverses}.
\begin{lem}\label{fonctions inverses}
				Let $\Phi: E\to F$, be a smooth map between Banach spaces and let $Q:= \Phi - \Phi(0)- d_0\Phi$.
				
				Assume that there exist $q>0$, $r_0>0$ and $c>0$ such that:
				\begin{enumerate}
				\item for all $x$ and $y$ in $B(0,r_0)$, we have the following control on the nonlinear terms
				$$\|Q(x)-Q(y)\|\leqslant q (\|x\|+\|y\|)\|x-y\|.$$
				\item the linearization $d_0\Phi$ is an isomorphism, and more precisely, we have
				$$\|(d_0\Phi)^{-1}\|\leq c.$$
				\end{enumerate}
				
				If $r\leqslant \min\Big(r_0,\frac{1}{2qc}\Big)$ and $\|\Phi(0)\| \leqslant \frac{r}{2c}$, then, the equation $\Phi(x) =0$ admits a unique solution in $B(0,r)$.
\end{lem}
    
    Let us finally remark that for a linear vector field $X_k$ in the kernel of $\delta_e\delta_e^*$, the symmetric $2$-tensor $\delta_e^* X_k$ is constant. This lets us define for any metric $g$ on $M$, a continuous map $\psi_g : r_DC^{3,\alpha}_{\beta,*}\to C^{2,\alpha}_{\beta,*}$ by
    $$ X\mapsto \psi_g(X):=\exp^*_X g.$$
    It is indeed continuous since for any diffeomorphism $\phi : M \to M$, we have
    $$ (\phi^* g)(x) _{kl} = g(\phi(x))_{ij}\frac{\partial_i\phi}{\partial x^k}\frac{\partial_j\phi}{\partial x^l}$$
    in local coordinates and therefore, for any vector field $X_*\in r_DC^{3,\alpha}_{\beta}$, the symmetric $2$-tensor $\exp^*_X g-g$ is arbitrarily small for the $C^{2,\alpha}_{\beta,*}$-norm. For the constant part, it is enough to note that for a linear vector field $X_k$ in the kernel of $\delta_e\delta^*_e$, and for a constant symmetric $2$-tensor $H_k$, the symmetric $2$-tensor $\exp^*_{X_k}H_k$ is also constant and controlled $|\exp^*_{X_k}H_k|_{g_e} \leqslant C\|X_k\|_{r_eC^0}|H_k|_{g_e}$.
\end{proof}

\subsection{Einstein metrics in gauge}\label{Métriques d'Einstein en jauge}

Let us now come back to Einstein metrics which can be characterized in dimension $n$ thanks to the Bianchi identity as the zero set of
$$E(g):= \Ric(g)-\frac{\textup{R}(g)}{2}g+\frac{n-2}{2n}\overline{R}(g)g$$
on a compact manifold $M$, where $\overline{R}:= \frac{1}{\vol(M,g)}\int_M \R(g)dv_g$. Notice that $\delta_gE(g)=0$, again by the Bianchi identity.

The equation $E(g)=0$ is invariant by the action of diffeomorphisms and by scaling, we will therefore restrict our attention to deformations which are transverse to these actions in order to obtain an operator whose linearization is elliptic. More precisely, we will fix the volume and fix a gauge thanks to the reduced divergence-free condition.

It turns out that we can characterize the zeros of $E$ in reduced divergence-free gauge as the zeros of a single operator $\mathbf{\Phi}_{g^D}$ defined by
$$\mathbf{\Phi}_{g^D}(g):=E(g) + \delta_g^*\tilde{\delta}_{g^D}g.$$
Indeed, if we have $E(g)=0$ and $\tilde{\delta}_{g^D} g=0$, then we have $\mathbf{\Phi}_{g^D}(g) = 0$. And conversely, if $\mathbf{\Phi}_{g^D}(g) = 0$, then since $E(g)$ is divergence-free (for $g$) by the Bianchi identity, by taking the reduced divergence of $\mathbf{\Phi}_{g^D}(g) = 0$, we get $$\tilde{\delta}_g\mathbf{\Phi}_{g^D}(g) = (\tilde{\delta}_g\delta_g^*)\tilde{\delta}_{g^D}g.$$
Since for $g$ close enough to $g^D$, $(\tilde{\delta}_g\delta_g^*)$ is invertible on the image of $\tilde{\delta}_{g^D}$ by Lemma \ref{Mise en Reduced divergence-free linear}, we finally have $\tilde{\delta}_{g^D}g = 0$ and $E(g)=0$. In a $C^{2,\alpha}_{\beta,*}(g^D)$-neighborhood of $g^D$ the zero set of $\mathbf{\Phi}_{g^D}$ is exactly the set of Einstein metrics in reduced divergence-free gauge with respect to $g^D$.

\begin{cor}\label{mise en Reduced divergence-free Einstein}
Let $D_0,v_0>0$, $l\in \mathbb{N}$, and $\beta= \beta(v_0,D_0)>0$ obtained in Corollary \ref{GH to C3}. Then, for all $\epsilon>0$, there exists $\delta = \delta(\epsilon,D_0,v_0,l) >0$ such that if $(M,g^\mathcal{E})$ is an Einstein manifold satisfying
	\begin{itemize}
		\item the volume is bounded below by $v_0>0$,
		\item the diameter is bounded above by $D_0$,
		\item the Ricci curvature is bounded $|\Ric|\leq 3$.
	\end{itemize}
	and such that for an Einstein orbifold $(M_o,g_o)$,
	$$d_{GH}\big((M,g^\mathcal{E}),(M_o,g_o)\big)\leqslant \delta,$$
	then, there exists a naïve desingularization $(M,g^D)$ of $(M_o,g_o)$ by a tree of singularities, and a diffeomorphism $\psi: M\to M$ such that $$\big\|\psi^*g^\mathcal{E}-g^D\big\|_{C^l_{\beta,*}(g^D)}\leqslant \epsilon,$$
	and
	$$ \tilde{\delta}_{g^D}(\psi^*g^\mathcal{E})=0. $$
	In particular, we have $$\mathbf{\Phi}_{g^D}(\psi^*g^\mathcal{E})=0.$$
\end{cor}

\section{Resolution of the Einstein equation modulo obstructions}

We will now show that it is always possible to produce metrics which are Einstein modulo some obstructions (which are elements of the cokernel of the linearization of the Einstein operator) in our weighted Hölder spaces. The main result of the section is Theorem \ref{fcts inv einst général} which allows us to perturb any naïve desingularization $g^D$ to an Einstein modulo obstructions metric and in particular according to \cite{ozu1} we produce all Einstein metrics close to an Einstein orbifold in the Gromov-Hausdorff sense by this procedure. 

We have seen in Corollary \ref{mise en Reduced divergence-free Einstein} that up to a diffeomorphism, any Einstein metric $g$ close to $(M_o,g_o)$ in the Gromov-Hausdorff sense is a solution of
$$\mathbf{\Phi}_{g^D}(g):= E(g) + \delta_g^*\tilde{\delta}_{g^D}g =0.$$ 
To study this equation, we will naturally start by studying its linearization on volume preserving deformations at $g^D$, that is, on symmetric $2$-tensors $h$ satisfying $\int_M\textup{tr}_{g^D} h dv_{g^D} = 0$, for which we have the formula
\begin{align}
        \Bar{P}_{{g^D}}(h):= \;d_{g^D}\mathbf{\Phi}_{g^D}(h) =& \frac{1}{2}\Big(\nabla^*_{g^D}\nabla_{g^D} h-2\mathring{\R}_{g^D}(h)\nonumber\\
        &-2 \delta_{g^D}^*\delta_{g^D}h+2\delta^*_{{g^D}}\tilde{\delta}_{g^D} h-(\delta_{g^D}\delta_{g^D} h) {g^D} \nonumber\\ 
        &+ (\Delta_{g^D} \textup{tr}_{g^D}h) {g^D} - \nabla_{g^D}^2 \textup{tr}_{g^D} h\nonumber\\
    &+\Ric_{g^D}\circ h+h\circ \Ric_{g^D}-\R_{g^D} h+ \langle \Ric_{g^D}, h \rangle_{g^D} {g^D} \nonumber\\
    &+\frac{1}{2}\overline{\R}(g) h - \frac{1}{2\vol(g)}\int_M\Big\langle\Ric({g^D})-\frac{\R({g^D})}{2}, h\Big\rangle_{g^D} \textup{d}v_{g^D} \Big),\label{expression barP}
\end{align}
in dimension $4$. If $g^D$ were an Einstein metric and $h$ a divergence-free symmetric $2$-tensor, then the operator $\bar{P}_{g^D}$ would reduce to $$P_{g^D}:= \frac{1}{2}\nabla_{g^D}^*\nabla_{g^D} - \mathring{\R}_{g^D},$$
which is simpler to study. Since $g^D$ is almost Einstein and $h$ will be almost divergence-free, we will mostly study the operator $P_{g^D}$, and we will obtain results for $\bar{P}_{g^D}$ by approximation.

\subsection{Kernel and cokernel of the linearization on model spaces}

\paragraph{Exceptional values for $P_e:= \frac{1}{2}\nabla_e^*\nabla_e$ on $(\mathbb{R}^4\slash\Gamma,g_e)$ and gauge constraints.}

As described in the proof of Proposition \ref{inverse P on annulus}, the elements of the kernel of $P_e$ on $\mathbb{R}^4\slash\Gamma$ are sums of homogeneous symmetric $2$-tensors whose coefficients in an orthonormal basis of $\mathbb{R}^4$ are homogeneous of order $k$ or $-2-k$ for $k\in \mathbb{N}$. 

However, some of these tensors cannot appear in our developments because they are not trace-free or in divergence-free gauge on our nontrivial quotient of $\mathbb{R}^4\slash\Gamma$.
        \begin{lem}[{\cite[Proposition 4.65]{ct}}]\label{divergence-free gauge harmoniques}
            On $\mathbb{R}^4\slash\Gamma$ for $\Gamma\neq \{e\}$, there is no harmonic homogeneous symmetric $2$-tensor whose coefficients are of order $1$, $-2$ or $-3$ in divergence-free gauge.
        \end{lem}

	\paragraph{Kernel of the operator $P$ on the model spaces.}\label{Résolution d'equations on une bulle}
	Let us start by describing the kernel of $P$ on our model spaces. 
	
	\begin{lem}\label{kernel L2}
		Let $(N,g_b)$ be a Ricci-flat ALE orbifold, and denote $P_b:= \frac{1}{2}\nabla^*_b\nabla_b - \mathring{\R}_b$, and $\mathbf{O}(g_b)$, the kernel of $P_{g_b}$ on $C^{2,\alpha}_{\beta,*}(g_b)$.
		
		The elements of $\mathbf{O}(g_b)$ decay at least like $r_b^{-4}$ at infinity, and for all $\mathbf{o}_{b}\in \mathbf{O}(g_b)$, we have the following development coordinates at infinity,
		$$\mathbf{o}_{b} = O^4 + \mathcal{O}(r_b^{-5}),$$
		with $O^4\sim r_b^{-4}$ a harmonic homogeneous symmetric $2$-tensor. 
		
		Let also $(M_o,g_o)$ be a compact Einstein orbifold, we denote $\mathbf{O}(g_o)$ the kernel of $P_o$ on $C^{2,\alpha}_{\beta,*}$ for all $0<\beta<1$. An element $\mathbf{o}_o\in \mathbf{O}(g_o)$ has a development 
		$$ \mathbf{o}_o = O_{0}+ O_2+\mathcal{O}(r_o^3), $$
		for harmonic homogeneous symmetric $2$-tensors $O_i\sim r_o^i$.
	\end{lem}
	\begin{proof}
	Let us consider $\mathbf{o}\in \mathbf{O}(g_b)$, for which $P_{b}\mathbf{o}=0$, and $\mathbf{o}= \mathcal{O}(r_b^{-\beta})$ for some $\beta>0$. Such a symmetric $2$-tensor is actually necessarily traceless and divergence-free. Indeed, we have
	$\delta_{g_b}P_{g_b} =  \frac{1}{2}\nabla_{g_b}^*\nabla_{g_b} \delta_{g_b},$
	and
    $\textup{tr}_{g_b}P_{g_b} =  \frac{1}{2}\nabla_{g_b}^*\nabla_{g_b} \textup{tr}_{g_b}.$
    Therefore, if $P_{g_b}h = 0$ for $h=\mathcal{O}(r_b^{-\delta})$ for some $\delta>0$, then $\delta_{g_b}h=0$, and $\textup{tr}_{g_b}h = 0$ by the maximum principle. We deduce from Lemma \ref{divergence-free gauge harmoniques} that $\mathbf{o}$ decays at least like $r_b^{-4}$ and its principal term is a harmonic symmetric $2$-tensor.

	In the same way in the neighborhood of a singularity of an orbifold $(M_o,g_o)$ or of a Ricci-flat ALE orbifold $(N,g_b)$, since there is no harmonic symmetric $2$-tensor with linear growth because of the action of the nontrivial group $\Gamma$, an element of the kernel admits a development
	$$\mathbf{o} = O_0+O_2+\mathcal{O}(r_b^3),$$
	where $O_0$ and $O_2$ are harmonic homogeneous symmetric $2$-tensors in $r_o^0$ and $r_o^2$ respectively.
	\end{proof}

	\paragraph{Estimates on the inverses.}
	
	Just like for the operator $\delta \delta^*$ in Proposition \ref{controle deltadelta sur les espaces modèles}, the operators $P_{g_o}$ and $P_{g_{b_j}}$ are injective on the orthogonal of their respective kernels.
	
	\begin{lem}\label{Contrôle P espace modèles}
	    Let $(N,g_b)$ be a Ricci-flat ALE orbifold and $(M_o,g_o)$ an Einstein orbifold, and $0<\beta<1$.
	    
	    Then, the operators $$P_b: \mathbf{O}(g_b)^{\perp}\cap C^{2,\alpha}_{\beta,*}(g_{b}) \to  r^{-2}_{b}C^{\alpha}_\beta(g_{b}),$$
	    and
	    $$P_o: \mathbf{O}(g_o)^{\perp}\cap C^{2,\alpha}_{\beta,*}(g_o) \to  r^{-2}_{o}C^{\alpha}_\beta(g_{o})$$
	    are injective and there exist $C_o>0$ and $C_b>0$, such that we have for any symmetric $2$-tensor $h_b\perp \mathbf{O}(g_b)$ on $N$ and $h_o\perp \mathbf{O}(g_o)$ on $M_o$,
	    \begin{equation}
	        C_b^{-1} \|P_bh_b\|_{r^{-2}_{b}C^{\alpha}_\beta(g_{b})}\leqslant\|h_b\|_{C^{2,\alpha}_{\beta,*}(g_{b})}\leqslant C_b \|P_bh_b\|_{r^{-2}_{b}C^{\alpha}_\beta(g_{b})},\label{norme inverse Pb}
	    \end{equation}	    
	    and
	    \begin{equation}
	        C_o^{-1} \|P_oh_o\|_{r^{-2}_{o}C^{\alpha}_\beta(g_{o})}\leqslant\|h_o\|_{C^{2,\alpha}_{\beta,*}(g_o)}\leqslant C_o \|P_oh_o\|_{r^{-2}_{o}C^{\alpha}_\beta(g_{o})}.\label{norme inverse Po}
	    \end{equation}
	    Moreover, their respective cokernels are $\mathbf{O}(g_b)$ and $\mathbf{O}(g_o)$.
	 \end{lem}
	 \begin{proof}
	     By standard theory of elliptic operators between weighted Hölder spaces (see for instance \cite[Chapter 2]{pr}), the operators
	     $$P_o: C^{2,\alpha}_{-\beta}(g_o) \to r^{-2}_{o}C^{\alpha}_{-\beta}(g_{o}),$$
	     and
	     $$P_b: r_b^{-\beta}C^{2,\alpha}_{0}(g_{b}) \to r^{-2-\beta}_{b}C^{\alpha}_0(g_{b})$$
	     are Fredholm for $0<\beta<1$ because we avoid the exceptional values close to zero: $-2$ and $1$. Let us study their kernels and cokernels. 
	     
	     Let us start by the case of an Einstein orbifold $(M_o,g_o)$ and notice that $\ker_{ C^{2,\alpha}_{-\beta}(g_o)}P_{g_o}\subset \mathbf{O}(g_o)$ because there is no exceptional value between $0$ and $-\beta$. The kernel of $P_o: C^{2,\alpha}_{-\beta}(g_o) \to r^{-2}_{o}C^{\alpha}_{-\beta}(g_{o})$ is therefore equal to $\mathbf{O}(g_o)$. Since $P_o$ is self adjoint and since we are strictly between two exceptional values, its cokernel is the kernel of $P_o$ on $r^{-2}_{o}C^{\alpha}_{\beta}(g_{o})$ which is also reduced to $\mathbf{O}(g_o)$ by a similar argument.
	     
        Similarly, the kernel of  $P_b: r_b^{-\beta}C^{2,\alpha}_{0}(g_{b}) \to r^{-2-\beta}_{b}C^{\alpha}_0(g_{b})$ is $\mathbf{O}(g_b)$, and its cokernel is the kernel of $P_b$ on $r^{-2+\beta}_{b}C^{\alpha}_0(g_{b})$ which is reduced to $\mathbf{O}(g_b)$.
	    
	    For the cokernels of $$P_b: \mathbf{O}(g_b)^{\perp}\cap C^{2,\alpha}_{\beta,*}(g_{b}) \to  r^{-2}_{b}C^{\alpha}_\beta(g_{b}),$$
	    and
	    $$P_o: \mathbf{O}(g_o)^{\perp}\cap C^{2,\alpha}_{\beta,*}(g_o) \to  r^{-2}_{o}C^{\alpha}_\beta(g_{o}),$$ like in the end of the proof of Proposition \ref{delta inj fredholm}, we use the fact that $$ P_o\big(C^{2,\alpha}_{\beta,*}(g_o)\big) =P_o\big( C^{2,\alpha}_{-\beta}(g_o)\big)\cap r^{-2}_{o}C^{\alpha}_\beta(g_{o}), $$
	    and
	    $$ P_b\big(C^{2,\alpha}_{\beta,*}(g_{b})\big)=P_b\big( r_{b}^{-\beta}C^{2,\alpha}_{0}(g_{b_j})\big)\cap  r^{-2}_{b}C^{\alpha}_\beta(g_{b}).$$
	 \end{proof}

\paragraph{Approximation of kernels and cokernels on a naïve desingularization.}

We wish to solve the equation $\Ric(g) = \Lambda g$ for a metric $g$ close to $g^D$ modulo the kernel and the cokernel of the linearization of the Einstein operator. We will use approximate kernels and cokernels defined as the truncated infinitesimal deformations of each model space on the tree of singularities in order to obtain uniform controls as the singularities form, that is as the gluing parameters $t$ tend to $0$.

Let $\mathbf{o}_o\in \mathbf{O}(g_o)$ and $\mathbf{o}_j\in \mathbf{O}(g_{b_j})$, and define $\mathbf{o}_o = \mathbf{o}_{o,*} + \sum_k \chi_{B_o(p_k,\epsilon_0)}\mathbf{o}_{o,k}$ and $\mathbf{o}_j = \mathbf{o}_{j,*} + \sum_k \chi_{B_j(p_k,\epsilon_0)}\mathbf{o}_{j,k}$ their respective decompositions as a symmetric $2$-tensor of $C^{2,\alpha}_\beta$ and constant symmetric $2$-tensors truncated in the neighborhoods of the singular points. Thanks to the cut-off functions of Definition \ref{cut off Mot Njt}, we define on $M$ the following symmetric $2$-tensors
$$\tilde{\mathbf{o}}_{o,t}:=\chi_{M_o^{t}}\mathbf{o}_{o,*} + \sum_k \chi_{\mathcal{A}_k(t,\epsilon_0)}\mathbf{o}_{o,k},$$
and
$$\tilde{\mathbf{o}}_{j,t}:=\chi_{N_j^{t}}\mathbf{o}_{j,*} + \sum_k \chi_{\mathcal{A}_k}\mathbf{o}_{j,k}.$$
\begin{rem}
    We have $\tilde{\mathbf{o}}_{o,t} = \mathbf{o}_o$ on $M_o^{16t}$, and $\tilde{\mathbf{o}}_{j,t} =\mathbf{o}_j$ on $N_j^{16t}$.
\end{rem}

\begin{defn}[Space of truncated obstructions]\label{def obst tronq}
    Let $(M,g^D_{t})$ be a naïve desingularization of a Einstein orbifold $(M_o,g_o)$. On $M$, we will denote $$\tilde{\mathbf{O}}(g^D):=\Big\{\tilde{\mathbf{o}}_{o,t} + \sum_j\tilde{\mathbf{o}}_{j,t}, \;\mathbf{o}_o\in \mathbf{O}(g_o), \; \mathbf{o}_j\in \mathbf{O}(g_{b_j})\Big\},$$ the space of \emph{truncated obstructions}.
\end{defn}
\begin{rem}\label{eq normes O(gDt)}
    For $0<\beta<2$, by elliptic regularity for the elements of $\mathbf{O}(g_o)$ and the $\mathbf{O}(g_{b_j})$, and the $C^4_0$-control of the cut-off functions  we have
    $$\Big\|\tilde{\mathbf{o}}_{o,t} + \sum_jT_j\tilde{\mathbf{o}}_{j,t}\Big\|_{C^0(g^D)}\approx\Big\|\tilde{\mathbf{o}}_{o,t} + \sum_jT_j\tilde{\mathbf{o}}_{j,t}\Big\|_{C^{2,\alpha}_{\beta,*}(g^D)} \approx \sup \big(\|\mathbf{o}_o\|_{C^0(g_o)}, \|\mathbf{o}_j\|_{C^0(g_{b_j})}\big),$$
    and
    $$\Big\|\tilde{\mathbf{o}}_{o,t} + \sum_j\tilde{\mathbf{o}}_{j,t}\Big\|_{L^2(g^D)} \approx\Big\|\tilde{\mathbf{o}}_{o,t} + \sum_j\tilde{\mathbf{o}}_{j,t}\Big\|_{r_D^{-2}C^\alpha_\beta(g^D)} \approx \sup \big(\|\mathbf{o}_o\|_{L^2(g_o)}, \|\mathbf{o}_j\|_{L^2(g_{b_j})}\big).$$
\end{rem}

We would like to produce Einstein metrics in reduced divergence-free gauge with respect to $g^D$. But the point is that it is not always possible because the space $\tilde{\mathbf{O}}(g^D)$ is an (approximate) obstruction space. We will show that we can perturb $g^D+v$ for parameters $t>0$ and $v\in \tilde{\mathbf{O}}(g^D)$ small enough to obtain a metric $\Hat{g}_v=\hat{g}_{t,v}$ which will be in gauge with respect to $g^D + v$ and solution of:
$$\mathbf{\Phi}_{g^D}(\hat{g}_v) \in \tilde{\mathbf{O}}(g^D),$$
hence the term Einstein \emph{modulo obstructions}.

\paragraph{Control of the inverse of the linearization.}

We can first show that the linearization is invertible and that we can control its inverse \emph{independently of the gluing scales} thanks to Lemma \ref{Contrôle P espace modèles}. 

Let us start by showing that the operator $\pi_{\tilde{\mathbf{O}}(g^D)^{\perp}}P_{g^D}$ is close to $P_{g^D}$ for a sufficiently degenerate tree of singularities.

\begin{lem}\label{controle P tilde P}
    There exists $C>0$ such that for any symmetric $2$-tensor $h\in C^{2,\alpha}_{\beta,*}(g^D)$, we have
    $$\big\|\big(\pi_{\tilde{\mathbf{O}}(g^D)^\perp}P_{g^D}-P_{g^D}\big)h\big\|_{r_D^{-2}C^\alpha_\beta(g^D)}\leqslant C t_{\max}^\frac{1}{2} \|h\|_{C^{2,\alpha}_{\beta,*}(g^D)}.$$
\end{lem}
\begin{rem}
    Here, the constant $C$ only depends on the constants $C_o$ and $C_{b_j}$ of Lemma \ref{Contrôle P espace modèles}.
\end{rem}
\begin{proof}
    The proof is similar to that of the estimate \eqref{controle delta tilde delta}. Thanks to the equivalence of the norms, see Remark \ref{eq normes O(gDt)}, we only have to control the $L^2(g^D)$-norm of the projection on $\tilde{\mathbf{O}}(g^D)$.

    On $M_o$, since $P_{g^D}$ is self adjoint, an integration by parts yields 
    \begin{align*}
        \Big|\int_M\langle P_{g^D}h, \tilde{\mathbf{o}}_{o,t} \rangle_{g^D} \textup{d}v_{g^D}\Big| =&\; \Big|\int_M\langle h,P_{g^D}(\tilde{\mathbf{o}}_{o,t}) \rangle_{g^D} \textup{d}v_{g^D}\Big|,
    \end{align*}
    and since $g^D=g_o$ on $M_o^{16t}$ and $P_{g_o}(\mathbf{o}_o)=0$ we use the decomposition
    \begin{align}
        P_{g^D}(\tilde{\mathbf{o}}_{o,t}) =&\;P_{g_o}((\chi_{M_o^{16t}}-1)\mathbf{o}_{o,*})\nonumber\\
        &+ P_{g_o} \Big(\sum_k (\chi_{ \mathcal{A}_{k}(t,\epsilon_0)}- \chi_{B_o(p_k,\epsilon_0)})\mathbf{o}_{o,k}\Big)\Big) \nonumber\\
        &+ \big(P_{g_o}-P_{g^D}\big)\Big(\sum_k \chi_{ \mathcal{A}_{k}(t,\epsilon_0)}\mathbf{o}_{o,k}\Big)\label{dvp PgD oot}
    \end{align}
    in order to obtain the following  estimates (compare with the proof of \eqref{controle delta tilde delta}) thanks to the control of the cut-off functions of Definition \ref{def cutoffs all} and to Remark \ref{eq normes O(gDt)} 
    \begin{equation}
        |P_{g_o}((\chi_{M_o^{16t}}-1)\mathbf{o}_{o,*})|_{g_o}
    \leqslant \; C\|\mathbf{o}_o\|_{L^2(g_o)}\mathbb{1}_{\{2t_k^\frac{1}{4}<r_D<4t_k^\frac{1}{4}\}} t_k^{-\frac{1}{2}},\label{estimation 1}
    \end{equation}
    \begin{equation}
        \Big|P_{g_o} \Big(\sum_k (\chi_{ \mathcal{A}_{k}(t,\epsilon_0)}- \chi_{B_o(p_k,\epsilon_0)})\mathbf{o}_{o,k}\Big)\Big)\Big|\leqslant C\|\mathbf{o}_o\|_{L^2(g_o)}\mathbb{1}_{\{\epsilon_0^{-1}t_k^\frac{1}{2}<r_D<2\epsilon_0^{-1}t_k^\frac{1}{2}\}} t_k^{-1},\label{estimation 2}
    \end{equation}
    and
    \begin{equation}
        \Big|\big(P_{g_o}-P_{g^D}\big)\Big(\sum_k \chi_{ \mathcal{A}_{k}(t,\epsilon_0)}\mathbf{o}_{o,k}\Big)\Big|\leqslant C\|\mathbf{o}_o\|_{L^2(g_o)}\mathbb{1}_{\{\epsilon_0^{-1}t_k^\frac{1}{2}<r_D<2t_k^\frac{1}{4}\}} (1+t_k^2r_D^{-6}).\label{estimation 3}
    \end{equation}
    Since for any $C>0$ independent on $t$, we have
    \begin{itemize}
        \item $\vol_{g_e}\big(\big\{2t_k^\frac{1}{4}<r^D<4t_k^\frac{1}{4}\big\}\big)\leqslant Ct_k$,
        \item $\vol_{g_e}\big(\big\{\epsilon_0^{-1}t_k^\frac{1}{2}<r_D<2\epsilon_0^{-1}t_k^\frac{1}{2}\big\}\big)\leqslant C\epsilon_0^{-4}t_k^2$, and
        \item $\int_{\{\epsilon_0^{-1}t_k^\frac{1}{2}<r_D<2t_k^\frac{1}{4}\}} (1+t_k^2r_D^{-6})dv_{g^D}\leqslant Ct_{\max},$
    \end{itemize} 
    integrating the controls \eqref{estimation 1}, \eqref{estimation 2} and \eqref{estimation 3} of the terms of \eqref{dvp PgD oot} yields
    \begin{equation}
        \Big|\int_M\langle P_{g^D}h, \tilde{\mathbf{o}}_{o,t} \rangle_{g^D} \textup{d}v_{g^D}\Big| \leqslant C t_k^\frac{1}{2} \|h\|_{C^{2,\alpha}_{\beta,*}(g^D)}\|\mathbf{o}_o\|_{L^2(g_o)}.\label{Pu oo}
    \end{equation}
    Similarly, for the $N_j$, consider $\mathbf{o}_j\in \mathbf{O}(g_{b_j})$. By invariance of the $L^2$-norm of $2$-tensors in dimension $4$ and since $P_{\frac{g}{t}}=tP_g$ for any metric $g$ and $t>0$, we have
    \begin{align*}
        \int_M\langle P_{g^D}h,\tilde{\mathbf{o}}_{j,t} \rangle_{g^D} \textup{d}v_{g^D} &= \int_M\Big\langle  \frac{h}{T_j},P_{\frac{g^D}{T_j}}\tilde{\mathbf{o}}_{j,t}\Big\rangle_{\frac{g^D}{T_j}} \textup{d}v_{\frac{g^D}{T_j}}.
    \end{align*}
    The control at the singular points is the same as in the case of $M_o$ and at infinity we have
    \begin{align*}
        P_{\frac{g^D}{T_j}}\tilde{\mathbf{o}}_{j,t} = (P_{\frac{g^D}{T_j}}-P_{g_{b_j}}) \tilde{\mathbf{o}}_{j,t} +P_{g_{b_j}}\mathbf{o}_j+ P_{g_{b_j}}((\chi_{N_j^{16t}}-1)\mathbf{o}_{j,*}),
    \end{align*}
    and therefore, since $\mathbf{o}_j = \mathcal{O}(r_{b_j}^{-4})$, we have
    $$\big|P_{\frac{g^D}{T_j}}\tilde{\mathbf{o}}_{j,t}\big|_{g_{b_j}}\leqslant C\|\mathbf{o}_j\|_{L^2(g_{b_j})}\mathbb{1}_{\{\frac{1}{2}t_j^{-1/4}<r_{b_j}<t_j^{-1/4}\}}r_{b_j}^{-6} \approx C\|\mathbf{o}_j\|_{L^2(g_{b_j})}\mathbb{1}_{\{\frac{1}{2}t_j^{-1/4}<r_{b_j}<t_j^{-1/4}\}}t_j^\frac{3}{2},$$
    and since $\vol\big(A_e\big(\frac{1}{2}t_j^{-\frac{1}{4}},t_j^{-\frac{1}{4}}\big)\big)\leqslant Ct_j^{-1}$ and at every point $\big|\chi_{N^{16t}_j}\frac{h}{T_j}\big|_{g_{b_j}}\leqslant \|h\|_{C^{2,\alpha}_{\beta,*}(g^D)}$, we have
    \begin{equation}
        \Big|\int_M\langle P_{g^D}h, \tilde{\mathbf{o}}_{j,t} \rangle_{g^D} \textup{d}v_{g^D}\Big| \leqslant C t_{\max}^\frac{1}{2} \|h\|_{C^{2,\alpha}_{\beta,*}(g^D)}\|\mathbf{o}_j\|_{L^2(g_{b_j})},\label{Pu oj}
    \end{equation}
    and finally,
    $$\big\|\big(\pi_{\tilde{\mathbf{O}}(g^D)^\perp}P_g^D-P_g^D\big)h\big\|_{r_D^{-2}C^\alpha_\beta(g^D)}\leqslant C t_{\max}^\frac{1}{2} \|h\|_{C^{2,\alpha}_{\beta,*}(g^D)}.$$
\end{proof}
\begin{prop}\label{inversion with cste}
    Let $0<\beta<1$, $k\in \mathbb{N}$, $0<\alpha<1$ and let $(M,g^D)$ be a naïve desingularization of a compact or ALE Einstein orbifold by a tree of singularities. Then, there exists $\tau_D >0$ and $\epsilon_D>0$ and $C_D>0$ only depending on $\beta$, and the constants of Lemma \ref{Contrôle P espace modèles} and of Proposition \ref{inverse P on annulus} such that for $t_{\max}<\tau_D$ and for any metric $g$ such that $\|g-g^D_{t}\|_{C^{2,\alpha}_{\beta,*}(g^D)}\leqslant \epsilon_D$, the operator $$\pi_{\tilde{\mathbf{O}}(g^D)^{\perp}}P_g: \tilde{\mathbf{O}}(g^D)^{\perp}\cap C^{2,\alpha}_{\beta,*}(g^D) \to \tilde{\mathbf{O}}(g^D)^{\perp}\cap r^{-2}_DC^{\alpha}_{\beta}(g^D),$$ where $\pi_{\tilde{\mathbf{O}}(g^D)^{\perp}}$ is the orthogonal projection for $g^D$ on $\tilde{\mathbf{O}}(g^D)^{\perp}$, is invertible and we have for any symmetric $2$-tensor $h\perp \tilde{\mathbf{O}}(g^D)$ on $M$,
    $$\|h\|_{C^{2,\alpha}_{\beta,*}(g^D)}\leqslant C_D\|\pi_{\tilde{\mathbf{O}}(g^D)^{\perp}}P_gh\|_{r^{-2}_DC^{\alpha}_{\beta}(g^D)}.$$
\end{prop}

    \begin{proof}
        The proof is similar to that of Lemma \ref{Mise en Reduced divergence-free linear}. The idea is again to extend the symmetric $2$-tensors on the model spaces and to deduce a control on the whole tree of singularities. 
        
Let $0<\epsilon<\epsilon_D^{\frac{1}{2-\beta}}<\epsilon_0$ for $\epsilon_D$ and $\epsilon$ which we will choose small enough along the proof, and assume that $t_{\max}< \epsilon^4$ in order to have on each annulus $\mathcal{A}_k:= \mathcal{A}_k(t,\epsilon)$ between $N_k$ and $N_j$, the existence of a diffeomorphism $$\Phi_k: A_e(\epsilon^{-1}\sqrt{T_j}\sqrt{t_k},\epsilon\sqrt{T_j})\subset \mathbb{R}^4\slash\Gamma_k\to \mathcal{A}_k\subset M,$$
        such that for any $0<\beta<1$, there exists $C>0$, $$\big\|\Phi^*_kg^D-g_e\big\|_{C^{2,\alpha}_\beta(A_e(\epsilon^{-1}\sqrt{T_j}\sqrt{t_k},\epsilon\sqrt{T_j}))}\leqslant C \epsilon^{2-\beta}<C\epsilon_D,$$
        by definition of $g^D$. Until the end of the proof, we will denote
        $$A_k:= A_e(\epsilon^{-1}\sqrt{T_j}\sqrt{t_k},\epsilon\sqrt{T_j}).$$
        
        Let $h$ be a symmetric $2$-tensor on $M$. Thanks to the estimate \eqref{estimation inverse annulus h} of Proposition \ref{inverse P on annulus} and to its generalization to metrics close to $g_e$ by Proposition \ref{contrôles with perturbations triviales}, for $\epsilon_D$ and $t_{\max}$ small enough, we can choose constant and traceless symmetric $2$-tensors $H_k$ on $\mathbb{R}^4\slash\Gamma_k$ such that we have on $ A(\sqrt{T_j}t_k^{1/4}):=A_e((1/2)\sqrt{T_j}t_k^{1/4},4\sqrt{T_j}t_k^{1/4})$
        \begin{align}
        \|\Phi_k^*h-H_k\|_{C^{2,\alpha}_0(A(T_j^{1/2}t_k^{1/4}))} \leqslant& \; C_e \Big(\epsilon^{-\beta}t_k^{\frac{\beta}{4}}\big\|P_{g_e}\Phi_k^*h\big\|_{r^{-2}_eC^{\alpha}_{\beta}(A_k)} \nonumber\\
        &+ 4 \epsilon^{-1}t_k^{\frac{1}{4}} \|\Phi_k^*h-H_k\|_{C^{2,\alpha}_{\beta}(A_k)} \Big)\nonumber  \\
        \leqslant&\;2 C_e\Big(\epsilon^{-\beta}t_k^{\frac{\beta}{4}}\big\|\big(P_{g^D}h\big)_{|\mathcal{A}_k}\big\|_{r^{-2}_DC^{\alpha}_{\beta}(g^D)}\nonumber\\
        &+ 4 \epsilon^{-1}t_k^{\frac{1}{4}} \|(h-\chi_{\mathcal{A}_k}H_k)\|_{C^{2,\alpha}_\beta(g^D)}\Big).\label{controle annulus P h} 
    \end{align}
    Let us then consider the decomposition $$h = h_* + \sum_k \chi_{\mathcal{A}_k}H_k.$$
    
    We define a symmetric $2$-tensor $h_o$ extending $h$ to $M_o$ in the following way:
    $$h_o:= \chi_{M_o^{t/16}} h_* +\sum_{k\in K_o}\chi_{B_o(p_k,\epsilon)}H_k,$$
    where $K_o$ is the set of $k$ such that $\mathcal{A}_k\cap M_o^t\neq \emptyset$. Denoting $h_{o,*}:=\chi_{M_o^{t/16}} h_*$, we have 
    $$ \|h_{o,*}\|_{C^{2,\alpha}_\beta(g_o)}\geqslant \|h_{o,*}\|_{C^{0}_\beta(g_o)}\geqslant \|(h_{*})_{|M_o^t}\|_{C^{0}_\beta(g^D)}, $$
    because $M_o^t $ is outside of the region damaged by the gluing.
    
    Since for $t_{\max}$ small enough, the metric $g^D$ is close to the metric $g_o$ on $M_o^{t/16}$ by \eqref{gD go recollement} and
    $$ P_{g^D}h = P_{g_o}h_o + P_{g_o}(h-h_o) + (P_{g^D}-P_{g_o})h,$$
    we moreover have the following control thanks to \eqref{controle annulus P h},
    \begin{align}
        \|(P_{g^D}h)_{|M_o^{t/16}}\|_{r_D^{-2}C^{\alpha}_{\beta}(g^D)}\geqslant& \; \|P_{g_o}h_o\|_{r_o^{-2}C^{\alpha}_{\beta}(g_o)}\nonumber\\
        &- CC_e\sum_{k\in K_o} \|(P_{g^D}h)_{|\mathcal{A}_k}\|_{r_D^{-2}C^{\alpha}_\beta(g^D)}\nonumber\\
        &-\gamma(t_{\max}) \|h\|_{C^{2,\alpha}_{\beta,*}(g^D)},\label{est P }
    \end{align}
    where $\gamma(t_{\max})\to 0$ when $t_{\max}\to 0$ (compare with \eqref{contrôle deltadelta X Mo} for the vector fields).
    
    Since $h\perp \tilde{\mathbf{O}}(g^D)$, we have $\frac{\|\pi_{\mathbf{O}(g_o)}h_o\|_{C^{2,\alpha}_{\beta,*}}}{\|h_o\|_{C^{2,\alpha}_{\beta,*}}}\to 0$, and by Lemma \ref{Contrôle P espace modèles}, this implies, for $t_{\max}$ small enough, the control
    $$ \|h_o\|_{C^{2,\alpha}_{\beta,*}(g_o)} \leqslant 2C_o \|P_o h_o\|_{r_o^{-2}C^\alpha_\beta(g_o)}, $$
    and the estimate \eqref{est P } and Lemma \ref{controle P tilde P} imply that for $t_{\max}$ small enough, there exists $C>0$ only depending on the previous constants such that we have
    \begin{equation}
        \|(h_*)_{|M_o^t}\|_{C^{0}_\beta(g^D)} + \sum_{k\in K_o} |H_k|_{g_e} -C\gamma(t_{\max})\|h\|_{C^{2,\alpha}_{\beta,*}(g^D)}\leqslant C \|\pi_{\tilde{\mathbf{O}}(g^D)^{\perp}}P_{g^D}h\|_{r_D^{-2}C^{\alpha}_{\beta}(g^D)}.\label{controle on Mo}
    \end{equation}
    Indeed, on an orbifold $(M_o,g_o)$, the constant symmetric $2$-tensors of the decomposition in the definition of the norm $\|.\|_{C^{2,\alpha}_{\beta,*}}$ are determined, see Remark \ref{unicité decompositions}.
    \\
    
    Let us then consider the symmetric $2$-tensor $h_1:= h- \sum_{k\in K_o} H_k$ which satisfies for $C>0$ depending on the previous constants, 
    \begin{equation}
        \| \pi_{\tilde{\mathbf{O}}(g^D)^{\perp}}P_{g^D}h_1 \|_{r_D^{-2}C^{\alpha}_\beta(g^D)}\leqslant C \| \pi_{\tilde{\mathbf{O}}(g^D)^{\perp}}P_{g^D}h \|_{r_D^{-2}C^{\alpha}_\beta(g^D)}+C\gamma(t_{\max}) \|h\|_{C^{2,\alpha}_{\beta,*}(g^D)}\label{contrôle Ph1}
    \end{equation}
    thanks to the control \eqref{controle on Mo} of $\sum_{k\in K_o} |H_k|_{g_e}$.
    
    Given $j\in K_o$, the Ricci-flat ALE orbifold $(N_j,g_{b_j})$ is glued to $M_o$ and we can extend the symmetric $2$-tensor $h_1 = h_{*} + \sum_{l \notin K_o} \chi_{\mathcal{A}_l}H_l$ to $N_j$ by $$h_{j}:= \chi_{N_j^{t/16}}h_* + \sum_{l\in K_j} \chi_{B_j(p_l,\epsilon)}H_l,$$ where $K_j$ is the set of $l\notin K_o$ such that $\mathcal{A}_l$ has a nonempty intersection with the neighborhood of a $N_j^t$. 

    Denoting $h_{j,*}:= \chi_{N_j^{t/16}}h_*$, we have
    \begin{equation}
        \|h_{j,*}\|_{C^{2,\alpha}_{\beta}(g_{b_j})}\geqslant \|h_{j,*}\|_{C^{0}_{\beta}(g_{b_j})}\geqslant \|(h_*)_{|N_j^t}\|_{C^{0}_{\beta}(g^D)},\label{contrôle hj*}
    \end{equation}
    and by \eqref{controle annulus P h} and since $\frac{g^D}{T_j}$ is close to $g_{b_j}$ on $N_j^{t/16}$ for $t_{\max}$ small enough, we moreover have the following control thanks to \eqref{controle annulus P h},
    \begin{align}
        \big\|\big(P_{g^D}h_1\big)_{|N_j^{t/16}}\big\|_{r^{-2}_DC^{\alpha}_{\beta}(g^D)}\geqslant &\; \|P_{g_{b_j}}h_{j}\|_{r^{-2}_{b_j}C^{\alpha}_{\beta}(g_{b_j})}\nonumber\\
        &- CC_e \sum_{k\in K_j}\|(P_{g^D}h_1)_{|\mathcal{A}_k}\|_{r_D^{-2}C^{\alpha}_\beta(g^D)}\nonumber\\
        &-\gamma(t_{\max}) \|h\|_{C^{2,\alpha}_{\beta,*}(g^D)},\label{contrôle P h Nj}
    \end{align}
    where $\gamma(t_{\max})\to 0$ when $t_{\max}\to 0$. We then have a control on $(h_*)_{|N_j^t}$ and on the $H_k$, for $k\in K_j$ thanks to Lemma \ref{Contrôle P espace modèles} by using again the fact that $h\perp \tilde{\mathbf{O}}(g^D)$ which implies that for $t_{\max}$ small enough, we have
    $$ \|h_{j}\|_{r_jC^{3,\alpha}_{\beta,*}(g_{b_j})}\leqslant 2C_j\|P_{g_{b_j}}h_{j}\|_{r^{-2}_jC^{\alpha}_{\beta}(g_{b_j})}. $$
    The estimates \eqref{contrôle P h Nj}, \eqref{controle annulus P h} and Lemma \ref{controle P tilde P} then yield
    \begin{equation}
        \|(h_*)_{|N_j^t}\|_{C^0_\beta}+ \sum_{k\in K_j}|H_k|_{g_e}-C\gamma(t_{\max}) \|h\|_{C^{2,\alpha}_{\beta,*}(g^D)}\leqslant C \| \pi_{\tilde{\mathbf{O}}(g^D)^{\perp}}P_{g^D}h \|_{r_D^{-2}C^{\alpha}_\beta(g^D)}.\label{controle on Nj}
    \end{equation}
    
    Iterating the above controls to the other Ricci-flat ALE orbifolds of the tree of singularities, we obtain controls similar to \eqref{controle on Nj} on all of the $N_j$ and with the control \eqref{controle on Mo}, we finally find
    $$\big(1-C\gamma(t_{\max})\big) \|h\|_{C^{2,\alpha}_{\beta,*}(g^D)} \leqslant C \| \pi_{\tilde{\mathbf{O}}(g^D)^{\perp}}P_{g^D}h \|_{r_D^{-2}C^{\alpha}_\beta(g^D)}, $$
    and therefore the stated result for $g=g^D$ for $t_{\max}$ small enough.
    
    To obtain the same result for a metric $g$ close to $g^D$, we simply apply Proposition \ref{contrôles with perturbations triviales} to ensure that for $\|g-g^D\|_{C^{2,\alpha}_{\beta,*}(g^D)}$ arbitrarily small, $P_g$ is arbitrarily close to $P_{g^D}$ for the operator norm on $C^{2,\alpha}_{\beta,*}(g^D)$.
    
    The operator $P$ being self adjoint on a compact manifold, its injectivity implies its surjectivity by integration by parts.
    \end{proof}

\subsection{Resolution modulo obstructions of the Einstein equation}

Let us now show that we can always solve the Einstein equation \emph{modulo obstructions}. Let us recall that being Einstein and in reduced divergence-free with respect to a naïve desingularization $g^D$ is equivalent to being a zero of the operator
$$\mathbf{\Phi}_{g^D}: g\mapsto \Ric(g)-\frac{\textup{R}(g)}{2}g+\lambda g + \delta_g^*\tilde{\delta}_{g^D}g.$$

\begin{thm}\label{fcts inv einst général}
    Let $(M_o,g_o)$ be a compact or ALE Einstein orbifold of dimension $4$ such that $\Ric(g_o) = \Lambda g_o$, for $\Lambda \in \mathbb{R}$, and let $(N_j,g_{b_j})_{j}$ be a tree of singularities desingularizing $(M_o,g_o)$ with pattern $D$, and $0<\beta<1$. 
    
    Then, there exists $\tau>0$, $\epsilon>0$ only depending on $\beta$ and the constants of Lemma \ref{Contrôle P espace modèles} and of Proposition \ref{inverse P on annulus} such that for any naïve desingularization $g^D:=g^D_{t}$, with $t_{\max}< \tau$, and for all $v\in \tilde{\mathbf{O}}(g^D)$ satisfying $\|v\|_{C^0_{\beta,*}(g^D)}<\epsilon$, there exists a \emph{unique} solution $\hat{g}_v=\hat{g}_{t,v}$ to the equation 
    $$\mathbf{\Phi}_{g^D}(\hat{g}_v)\in \tilde{\mathbf{O}}(g^D),$$
    satisfying the following conditions: 
    \begin{enumerate}
        \item $\|\hat{g}_v-g^D\|_{C^{2,\alpha}_{\beta,*}}\leqslant 2\epsilon$,
        \item $\hat{g}_v-(g^D+v)$ is $L^2(g^D)$-orthogonal to $\tilde{\mathbf{O}}(g^D)$.
    \end{enumerate}
\end{thm}
\begin{proof}
    Let $(M,g^D_{t})=(M,g^D)$ be a naïve desingularization of an orbifold $(M_o,g_o)$ by a tree of singularities $(N_j,g_{b_j})_j$.

    Define the operator $\Psi: g^D + C^{2,\alpha}_{\beta,*}(g^D)\to \tilde{\mathbf{O}}(g^D)^\perp \cap r^{-2}_DC^{\alpha}_\beta(g^D)$ by
    $$\Psi(g):= \pi_{\tilde{\mathbf{O}}(g^D)^\perp}\mathbf{\Phi}_{g^D}(g) = \pi_{\tilde{\mathbf{O}}(g^D)^\perp}\Big(\Ric(g) -\frac{\textup{R}(g)}{2}g + \lambda g + \delta^*_{g}\tilde{\delta}_{g^D}(g)\Big),$$
where $\pi_{\tilde{\mathbf{O}}(g^D)^\perp}$ is the $L^2(g^D)$-orthogonal projection on $\tilde{\mathbf{O}}(g^D)^\perp$. The conclusion of the theorem for $v=0$ then rewrites: there exists a unique solution $g\in g^D +\big(\tilde{\mathbf{O}}(g^D)^\perp \cap C^{2,\alpha}_{\beta,*}\big)$ to the equation $\Psi(g) = 0$.

Let us apply the inverse function theorem, Lemma \ref{fonctions inverses}, to $\Psi$. The linearization of the operator $\Psi$ at $g^D$ for any symmetric $2$-tensor $h$ satisfying $\int_ M \textup{tr}_{g^D}h dv_{g^D} = 0 $ is
    \begin{align*}
        d_{g^D}\Psi(h)=&\; \pi_{\tilde{\mathbf{O}}(g^D)^\perp}\Bar{P}_{g^D}(h),
    \end{align*}
    where $\Bar{P}_{g^D}$ is explicited in \eqref{expression barP}. Let us show that this linearization is invertible at $g^D$ and has an inverse which is uniformly bounded as $t\to 0$. We want to go back to the operator $P_g = \frac{1}{2}\nabla^*_g\nabla_g h-\mathring{\R}_g(h)$ for which the invertibility has been shown in Proposition \ref{inversion with cste}.
    
    First, by the estimate \eqref{controle delta tilde delta}, we have
    \begin{equation*}
    \|\tilde{\delta}_{g^D}h-\delta_{g^D} h\|_{r_D^{-1}C^{1,\alpha}_{\beta,*}(g^D)}\leqslant Ct_{\max}^\frac{1}{2}\|h\|_{C^{2,\alpha}_{\beta,*}(g^D)},
    \end{equation*}
    and therefore the term $-2 \delta_g^*\delta_gh+2\delta^*_{g}\tilde{\delta}_{g^D} h$ of \eqref{expression barP} is controlled in the following way 
    \begin{equation}
        \|-2 \delta_g^*\delta_gh+2\delta^*_{g}\tilde{\delta}_{g^D} h \|_{r_D^{-2}C^{\alpha}_\beta(g^D)}\leqslant Ct_{\max}^\frac{1}{2}\|h\|_{C^{2,\alpha}_{\beta,*}(g^D)}.\label{divergence}
    \end{equation}
    
    Next, notice that the Ricci curvature of $g^D$ is almost constant: 
    \begin{equation}
        \|\Ric(g^D)-\Lambda g^D\|_{r^{-2}_DC^\alpha_\beta} \leqslant C t_{\max}^\frac{2-\beta}{4}, \label{Ricci}
    \end{equation}
    because $\Ric(g^D)-\Lambda g^D =0$ for $r_D>2t_{\max}^{1/4}$, and on the rest of the manifold, $|\Ric(g^D)-\Lambda g^D| \leqslant C$. Therefore, for $t_{\max}$ arbitrarily small, $d_{g^D}\Psi$ is close (as an operator from $C^{2,\alpha}_{\beta,*}$ to $r_D^{-2}C^\alpha_\beta$) up a power of $t_{\max}$ to the operator $\pi_{\tilde{\mathbf{O}}(g^D)^\perp}\hat{\hat{P}}_{g^D}$, where for a symmetric $2$-tensor $h$,
    \begin{align*}
        \hat{\hat{P}}_{g^D}(h):=&\frac{1}{2}\Big(\nabla^*_{g^D}\nabla_{g^D} h -2\mathring{\R}_{g^D}(h) \\
        &-(\delta_{g^D}\delta_{g^D}h){g^D}\\
        &- \nabla^2_{g^D} \textup{tr}_{g^D} h+ (\Delta_{g^D} \textup{tr}_{g^D}h) g^D+ \frac{\R_{g^D}}{4}(\textup{tr}_{g^D}h)g^D\Big),
    \end{align*}
    where we neglected the difference of the divergence terms by \eqref{divergence}, and simplified the terms involving the Ricci curvature by \eqref{Ricci}.
    
    Let us use the following decomposition of a $2$-tensor on $(M,g^D_t)$: for any symmetric $2$-tensor $h\in C^{2,\alpha}_{\beta,*}(g^D)$, there exists a unique decomposition 
    \begin{equation}
        h = h_{T} + \delta^*_{g^D}X\label{decomposition t lie C2alpha}
    \end{equation}
    with 
    \begin{enumerate}
        \item a symmetric $2$-tensor $h_{T}\in C^{2,\alpha}_{\beta,*}(g^D)$ satisfying $\tilde{\delta}_{g^D} h_{T}=0$, 
        \item a vector field $X\in r_DC^{3,\alpha}_{\beta,*}(g^D)\cap  \tilde{\mathbf{K}}_o^\perp$.
    \end{enumerate}
    Indeed, according to Lemma \ref{Mise en Reduced divergence-free linear}, for any $2$-tensor $h\in C^{2,\alpha}_{\beta,*}(g^D)$, there exists a unique $X\in r_DC^{3,\alpha}_{\beta,*}(g^D)\cap \tilde{\mathbf{K}}_o^\perp$ such that 
    $$\tilde{\delta}_{g^D}\delta_{g^D}^*X = \tilde{\delta}_{g^D}h. $$
    The decomposition \eqref{decomposition t lie C2alpha} is then $h =(h-\delta_{g^D}^*X) + \delta_{g^D}^*X$ and the sum is $L^2(g^D)$-orthogonal. 

    \begin{rem}
        A simple integration by parts shows that a differentiable symmetric $2$-tensor $h$ on $M$ satisfies $\tilde{\delta}_{g^D}$ if and only if it is $L^2(g^D)$-orthogonal to $\delta^*_{g^D}(C^\infty(TM)\cap\tilde{\mathbf{K}}_o^\perp)$.
    \end{rem}
    
    Similarly, for a symmetric $2$-tensor $v\in r_D^{-2}C^\alpha_\beta$ we have a unique decomposition
    \begin{equation}
        v= v_T+\delta^*_{g^D}Y\label{decomposition t lie r-2Calpha}
    \end{equation}
    with a symmetric $2$-tensor $v_{T}$ orthogonal to $\delta^*_{g^D}(C^\infty(TM)\cap\tilde{\mathbf{K}}_o^\perp)$ and a vector field $X\perp \tilde{\mathbf{K}}_o$.
    
    Now, we have the following properties:
    \begin{enumerate}
        \item our metric has almost constant Ricci curvature by \eqref{Ricci},
        \item the divergence of the obstructions is arbitrarily small, that is for $ \tilde{\mathbf{o}}\in \tilde{\mathbf{O}}(g^D)$, we have $$\frac{\|\delta_{g^D}\tilde{\mathbf{o}}\|_{r_D^{-1}C^{1,\alpha}_{\beta}}}{\|\tilde{\mathbf{o}}\|_{C^{2,\alpha}_{\beta,*}}}\to 0$$
        and
        $$\frac{\|\delta_{g^D}\tilde{\mathbf{o}}\|_{r_D^{-3}C^{0}_{\beta}}}{\|\tilde{\mathbf{o}}\|_{r_D^{-2}C^{\alpha}_{\beta}}}\to 0$$
        as $t_{\max}\to 0$ by elliptic regularity .
    \end{enumerate}
    These imply that for $X\in r_D C^{3,\alpha}_{\beta,*}\cap \tilde{\mathbf{K}}_o^\perp$, one has
    \begin{equation}
        \|\pi_{\tilde{\mathbf{O}}(g^D)^\perp}\hat{\hat{P}}_{g^D} (\delta_{g^D}^*X) - \delta^*_{g^D}\tilde{\delta}_{g^D}\delta^*_{g^D}X\|_{r_D^{-2}C^{\alpha}_\beta}\leqslant \gamma(t_{\max})\|X\|_{r_D C^{3,\alpha}_{\beta,*}},\label{controle P de lie}
    \end{equation}
    for $\gamma : \mathbb{R}^+\to \mathbb{R}^+$ with $\lim_{0}\gamma = 0$, and that 
    
    Consider the decompositions \eqref{decomposition t lie C2alpha} and \eqref{decomposition t lie r-2Calpha} where we identify the symmetric $2$-tensor of the form $\delta^*_{g^D}X$ with the vector field $X$ since $$\delta^*_{g^D}:\tilde{\mathbf{K}}_o^\perp\cap r_DC^{3,\alpha}_{\beta,*}\to C^{2,\alpha}_{\beta,*}$$
    and
    $$\delta^*_{g^D}:\tilde{\mathbf{K}}_o^\perp\cap r_D^{-1}C^{1,\alpha}_{\beta,*}\to r_D^{-2}C^{\alpha}_{\beta,*}$$
    are injective. Thanks to \eqref{controle P de lie}, the operator $\pi_{\tilde{\mathbf{O}}(g^D)^\perp}\hat{\hat{P}}_{g^D}: C^{2,\alpha}_{\beta,*}\to r_D^{-2}C^\alpha_\beta$ decomposes blockwise
    \begin{equation}
        \pi_{\tilde{\mathbf{O}}(g^D)^\perp}\hat{\hat{P}}_{g^D} = \begin{bmatrix}
\tilde{\delta}_{g^D}\delta^*_{g^D} & 0  \\
 0 & \pi_{\tilde{\mathbf{O}}(g^D)^\perp}\hat{P}_{g^D} 
\end{bmatrix} + A_{g^D}\label{decomp blocks divergence}
    \end{equation}
    where $A_{g^D}$ is an arbitrarily small (for $t_{\max}$ arbitrarily small) operator between the same spaces, and
    \begin{align*}
        \hat{P}_{g^D}(h):=&\frac{1}{2}\Big(\nabla^*_{g^D}\nabla_{g^D} h -2\mathring{\R}_{g^D}(h) \\
        &- \nabla^2_{g^D} \textup{tr}_{g^D} h+ (\Delta_{g^D} \textup{tr}_{g^D}h) g^D+ \frac{\R_{g^D}}{4}(\textup{tr}_{g^D}h)g^D\Big).
    \end{align*}
    Given the shape of the matrix in \eqref{decomp blocks divergence}, it is enough to show that $$\tilde{\delta}_{g^D}\delta^*_{g^D} : \tilde{\mathbf{K}}_o^\perp\cap r_DC^{3,\alpha}_{\beta,*}\to\tilde{\mathbf{K}}_o^\perp\cap r_D^{-1}C^{1,\alpha}_{\beta,*}$$
    and 
    $$\pi_{\tilde{\mathbf{O}}(g^D)^\perp}\hat{P}_{g^D}:\tilde{\mathbf{O}}(g^D)^\perp\cap C^{2,\alpha}_{\beta,*}\to \tilde{\mathbf{O}}(g^D)^\perp\cap r_D^{-2}C^\alpha_\beta $$
    are invertible with inverse bounded independently of $t_{\max}$. This is already the case for $\tilde{\delta}_{g^D}\delta^*_{g^D}$ thanks to Lemma \ref{Mise en Reduced divergence-free linear}.
    \\
    
    Let us now focus on $\pi_{\tilde{\mathbf{O}}(g^D)^\perp}\hat{P}_{g^D}$. For any symmetric $2$-tensor $v$, we have an $L^2(g^D)$-orthogonal decomposition into a conformal and traceless part 
    \begin{equation}
        v = \frac{tr_{g^D}v}{4}g^D + \big(v - \frac{tr_{g^D}v}{4}g^D\big)\label{decomp conf tracefree}.
    \end{equation}
    Now, we have
    \begin{align*}
        \textup{tr}_{g^D}\Big(\pi_{\tilde{\mathbf{O}}(g^D)^\perp}\hat{P}_{g^D}(h)\Big) = \textup{tr}_{g^D}\hat{P}_{g^D}(h) - \textup{tr}_{g^D}\Big(\pi_{\tilde{\mathbf{O}}(g^D)}\hat{P}_{g^D}(h)\Big)
    \end{align*}
    and since for an element $\tilde{\mathbf{o}}$ of $\tilde{\mathbf{O}}(g^D)$, we have
    \begin{equation}
        \frac{\|(\textup{tr}_{g^D}\tilde{\mathbf{o}})g^D\|_{r_D^{-2}C^\alpha_\beta}}{\|\tilde{\mathbf{o}}\|_{r_D^{-2}C^\alpha_\beta}}\to 0\label{trace O(gD)}
    \end{equation}
    and
    \begin{equation}
        \frac{\|(\textup{tr}_{g^D}\tilde{\mathbf{o}})g^D\|_{C^{2,\alpha}_{\beta,*}}}{\|\tilde{\mathbf{o}}\|_{C^{2,\alpha}_{\beta,*}}}\to 0\label{trace O(gD) bis}
    \end{equation}
as $t_{\max}\to 0$, the operator $\textup{tr}_{g^D}(\pi_{\tilde{\mathbf{O}}(g^D)^\perp}\hat{P}_{g^D})$ is arbitrarily close to the operator $\textup{tr}_{g^D}\hat{P}_{g^D}$ (as operators from $C^{2,\alpha}_{\beta,*}$ to $r_D^{-2}C^\alpha_\beta$) which is itself arbitrarily close to $\Delta_{g^D}+\Lambda$ for $t_{\max}$ arbitrarily small because our metric is almost Einstein according to \eqref{Ricci}. 
    
    For the traceless part of $\pi_{\tilde{\mathbf{O}}(g^D)^\perp}\hat{P}_{g^D}(h)$, recall that for any symmetric $2$-tensor $h$, we have the following decomposition
    \begin{align}
        \pi_{\tilde{\mathbf{O}}(g^D)^\perp}\hat{P}_{g^D}(h) - tr_{g^D}\big(\pi_{\tilde{\mathbf{O}}(g^D)^\perp}\hat{P}_{g^D}(h)\big)\frac{g^D}{4}=&\; \pi_{\tilde{\mathbf{O}}(g^D)^\perp}\Big(\hat{P}_{g^D}(h) - tr_{g^D}\big(\hat{P}_{g^D}(h)\big)\frac{g^D}{4} \Big) \nonumber\\
        &-  tr_{g^D}\big(\hat{P}_{g^D}(h)\big)\frac{\pi_{\tilde{\mathbf{O}}(g^D)}g^D}{4}\nonumber\\
        &+  \Big(tr_{g^D}\big(\pi_{\tilde{\mathbf{O}}(g^D)}\hat{P}_{g^D}(h)\big)\Big)\frac{g^D}{4}.\label{dvt tr pi O P}
    \end{align}
    Using the control \eqref{Ricci} once more together with \eqref{trace O(gD)} and \eqref{trace O(gD) bis}, we neglect the last two terms of \eqref{dvt tr pi O P} and see that the traceless part of $\pi_{\tilde{\mathbf{O}}(g^D)^\perp}\hat{P}_{g^D}(h)$ is arbitrarily close to the firsts term $\pi_{\tilde{\mathbf{O}}(g^D)^\perp}\Big(\hat{P}_{g^D}(h) - tr_{g^D}\big(\hat{P}_{g^D}(h)\big)\frac{g^D}{4} \Big)$, that is, to
    $$\pi_{\tilde{\mathbf{O}}(g^D)^\perp}\big(P_{g^D}(h)-\big(\nabla^2_{g^D}tr_{g^D}h-\Delta_{g^D}(tr_{g^D}h)\frac{g^D}{4}\big)\big).$$
 
Let us therefore decompose $h$ and the operator $\pi_{\tilde{\mathbf{O}}(g^D)}\hat{P}_{g^D}(h)$ in conformal and traceless parts. Block-wise, we obtain the following operator:
    $$\pi_{\tilde{\mathbf{O}}(g^D)^\perp}\hat{P}_{g^D} = \begin{bmatrix}
\Delta_{g^D} + \Lambda  & 0 \\
 \pi_{\tilde{\mathbf{O}}(g^D)^\perp}\big(\Delta_{g^D}tr_{g^D}h\frac{g^D}{4}-\nabla^2_{g^D}tr_{g^D}h\big) & \pi_{\tilde{\mathbf{O}}(g^D)^\perp}P_{g^D} 
\end{bmatrix} + A'_{g^D}$$
    where $A'_{g^D} : C^{2,\alpha}_{\beta,*}\to r_D^{-2} C^\alpha_\beta$ is an arbitrarily small operator for $ t_{\max} $ arbitrarily small. There remains to show that the operator
    $$\begin{bmatrix}
\Delta_{g^D} + \Lambda  & 0 \\
 \pi_{\tilde{\mathbf{O}}(g^D)^\perp}\big(\Delta_{g^D}tr_{g^D}h\frac{g^D}{4}-\nabla^2_{g^D}tr_{g^D}h\big) & \pi_{\tilde{\mathbf{O}}(g^D)^\perp}P_{g^D} 
\end{bmatrix}$$
from $C^{2,\alpha}_{\beta,*}$ to $ r_D^{-2} C^\alpha_\beta$ is invertible with a uniformly bounded inverse (independently of $t$) for $t_{\max}$ small enough. Given the shape of the matrix, it is enough to show it for the two diagonal blocks since the operator at the bottom left is uniformly bounded. The operator $\Delta_{g^D} + \Lambda$ is invertible with a uniformly bounded inverse thanks to Lichnerowicz eigenvalue estimate, see \cite[Section 5]{andsurv} for instance, and the estimate \eqref{Ricci}. This is also the case for $\pi_{\tilde{\mathbf{O}}(g^D)^\perp}P_{g^D}$ thanks to Proposition \ref{inversion with cste}.
    
    We conclude that the linearization of $\Psi$ at $g^D$ is invertible with a bounded inverse as it is arbitrarily close to an invertible operator.

    \begin{rem}
        The operator $\Bar{P}_{g^D}$ itself is not self adjoint because the metric $g^D$ is not Einstein. Indeed, all of the terms are self adjoint except $- \nabla_g^2 \textup{tr}_g h$, $-(\delta_g\delta_g h) g$, and $\langle \Ric_g, h \rangle_g g$, and we remark that the adjoint of $h\mapsto \nabla_g^2 \textup{tr}_g h$ is $h\mapsto (\delta_g\delta_g h) g$. However, the term $h\mapsto\langle \Ric_g, h \rangle_g g$ whose adjoint is $h\mapsto ( \textup{tr}_g h )\Ric_g$ prevents $\Bar{P}_{g^D}$ to be self adjoint when $g^D$ is not Einstein.
    \end{rem}
    
    To apply the inverse function theorem, Lemma \ref{fonctions inverses} to the operator $\Psi$, there remains to control the non-linear terms. Since the variations of the Ricci curvature for a variation $h$ of a metric $g$, are schematically, 
    $$\Ric(g+h) = \Ric(g)+(g+h)^{-1}*\Rm(g)+(g+h)^{-2}*\nabla^2h+(g+h)^{-3}*\nabla h *\nabla h,$$
    where $*$ refers to diverse multilinear operations and by Remark \ref{comportement norme}, the non-linear terms $Q_{g^D}(h):=\Psi(g^D+h)-\Psi(g^D)-d_{g^D}\Psi(h)$ satisfy the control
    \begin{align*}
        \|Q_{g^D}(h)- &Q_{g^D}(h')\|_{r^{-2}_DC^{\alpha}_\beta(g^D)}\\
        \leqslant&\; C\Big(\big(\|h\|_{C^\alpha_{0}(g^D)}+\|h'\|_{C^\alpha_{0}(g^D)}\big)\big(\|h-h'\|^2_{C^\alpha_{0}(g^D)}\big)\|\Rm(g^D)\|_{r^{-2}_DC^{\alpha}_\beta(g^D)}\\
        &+\|h-h'\|_{C^\alpha_{0}(g^D)}\big(\|\nabla^2 h\|_{r^{-2}_DC^\alpha_{\beta}(g^D)}+\|\nabla^2 h'\|_{r^{-2}_DC^\alpha_{\beta}(g^D)}\big)\\
        +&\big(\|h\|_{C^\alpha_{0}(g^D)}+\|h'\|_{C^\alpha_{0}(g^D)}\big)\|\nabla^2 (h-h')\|_{r^{-2}_DC^\alpha_{\beta}(g^D)}\\
        &+2\|\nabla (h-h')\|_{r^{-1}_DC^\alpha_\beta(g^D)}\big(\|\nabla h\|_{r^{-1}_DC^\alpha_\beta(g^D)}+\|\nabla (h')\|_{r^{-1}_DC^\alpha_\beta(g^D)}\big)\Big)\\
        \leqslant&\; 3C \big(\|h\|_{C^{2,\alpha}_{\beta,*}(g^D)}+\|h'\|_{C^{2,\alpha}_{\beta,*}(g^D)}\big)\|h-h'\|_{C^{2,\alpha}_{\beta,*}(g^D)}.
    \end{align*}
    
    We moreover have the control 
    $$ \|\Psi(g^D)\|_{r^{-2}_DC^{\alpha}_\beta(g^D)} \leqslant Ct_{\max}^\frac{2-\beta}{4}.$$
    Hence, according to the inverse function theorem, Lemma \ref{fonctions inverses},  for $t_{\max}$ small enough, there exists a unique solution $\hat{g}$ with $\hat{g}-g^D\perp \tilde{\mathbf{O}}(g^D)$, to the equation $\Psi(\hat{g})=0$ satisfying moreover $\|g^D-\hat{g}\|_{C^{2,\alpha}_{\beta,*}(g^D)}\leqslant \|\Psi(g^D)\|_{r^{-2}_DC^{\alpha}_\beta(g^D)} \leqslant Ct_{\max}^\frac{2-\beta}{4}$.
    
    Now, we have only solved the equation in the neighborhood of $g^D$ and on the orthogonal of $\tilde{\mathbf{O}}(g^D)$. For $v\in \tilde{\mathbf{O}}(g^D)$ we study the operator $g\mapsto \Psi(g+v)$. The control of the non-linear terms is exactly the same for this operator, and for $v$ arbitrarily small, its linearization at $g^D$ is arbitrarily close to $d_{g^D}\Psi$ which is invertible. As a consequence, there exists $\epsilon>0$ such that for all $\|v\|_{C^0_{\beta}(g^D)}<\epsilon$, there exists a unique solution $\hat{g}_v$ of $\Psi(\hat{g}_v)=0$ with $\hat{g}_v-(g^D+v)\perp \tilde{\mathbf{O}}(g^D)$. 
    \begin{rem}
    By adding $v$, we however damage the estimate on $\Psi(g^D+v)$ which becomes
    \begin{equation}
        \|\Psi(g^D+v)\|_{r^{-2}_DC^\alpha_\beta(g^D)}\leqslant C\big(\|v\|^2_{C^{2,\alpha}_{\beta}(g^D)} + t_{\max}^{\frac{2-\beta}{4}}\big).\label{control operator gD+v}
    \end{equation}
    We will see later in Section \ref{an integrability issue} that without an integrability assumption, we cannot hope for a better estimate.
    \end{rem}
\end{proof}

\begin{cor}\label{fcts implicites hat g v}
        With the notations of Theorem \ref{fcts inv einst général}, the map $v\mapsto \hat{g}_v$ is analytic.
   \end{cor}
   \begin{proof}
    This is a consequence of the implicit function theorem for analytic functions, see \cite{whi} for instance. Let us define the map
       $$ (v,h) \mapsto \Psi(\hat{g}_{0}+v+h) $$
       from $\tilde{\mathbf{O}}(g^D)\times \big(\tilde{\mathbf{O}}(g^D)^\perp \cap C^{2,\alpha}_{\beta,*}\big)$ to $ \tilde{\mathbf{O}}(g^D)^\perp \cap r^{-2}_DC^{\alpha}_\beta(g^D)$ where we denote
    $$\Psi(g) := \pi_{\tilde{\mathbf{O}}(g^D)^\perp}\mathbf{\Phi}_{g^D}(g) = \pi_{\tilde{\mathbf{O}}(g^D)^\perp}\big(E(g) + \delta^*_{g}\tilde{\delta}_{g^D}(g)\big),$$
    like in the proof of Theorem \ref{fcts inv einst général} and where $\hat{g}_{0}$ is the solution of Theorem \ref{fcts inv einst général} for $v=0$.
       
    The map $ (v,h) \mapsto \Psi(\hat{g}_{0}+v+h) $ is analytic since $E$ and $g\mapsto \delta_g^*$ are, and since $\tilde{\delta}_{g^D}$ and $\pi_{\tilde{\mathbf{O}}(g^D)^\perp}$ are linear. We know that $\|\hat{g}_0-g^D\|\leqslant C t_{\max}^\frac{2-\beta}{4}$, and that $d_{g^D} \Psi$ is invertible thanks to the proof of Theorem \ref{fcts inv einst général}, hence $d_{\hat{g}_0} \Psi$ is also invertible. 
    
    The implicit function theorem for analytic functions, see \cite{whi}, then implies, by the uniqueness of the solution $\hat{g}_{v}$ of Theorem \ref{fcts inv einst général} that for $v$ small, the map $v\mapsto\hat{g}_{v}-(g^D+v)$ is analytic and that $v\mapsto\hat{g}_{v}$ is analytic too.
   \end{proof}
   
   \begin{rem}\label{fcts inv partiel}
    The previous analysis of Theorem \ref{fcts inv einst général} extends to partial desingularizations. 
    
    More precisely, let $S_o$ be a subset of the singularities of $M_o$ and for each $j$, $S_j$ a subset of the singularities of $N_j$ and denote $S=(S_o,(S_j)_j)$ and $M_S:= M_o\#_jN_j$ ($\#$ means gluing the ALE spaces at orbifold singularities) where the gluings are given by some gluing pattern $D$. 
    
    For $t_{\max,S}$ and $v$ small enough, the metric $g^D_S=g^D_{S,t}$ iteratively defined just like in Definition \ref{def naive desing} can be perturbed to a \emph{unique} solution $\hat{g}_{S,v}=\hat{g}_{S,t,v}$ to the equation 
    \begin{equation}
        \mathbf{\Phi}_{g^D}(\hat{g}_{S,v})\in \tilde{\mathbf{O}}(g^D_S),\label{desing part mod obst}
    \end{equation}
    satisfying the following conditions: 
    \begin{enumerate}
        \item $\|\hat{g}_{S,v}-g^D_S\|_{C^{2,\alpha}_{\beta,*}(g^D_S)}\leqslant 2\epsilon$, where $C^{2,\alpha}_{\beta,*}(g^D_S)$ is the partial desingularization norm of Remark \ref{rem partial desing},
        \item $\hat{g}_{S,v}-g^D_{S,v}$ is $L^2(g^D_S)$-orthogonal.
    \end{enumerate}
   \end{rem}
   
 Thanks to Corollary \ref{mise en Reduced divergence-free Einstein}, we have the following result.

   \begin{cor}\label{proximité désing naive}
   Let $D_0,v_0>0$, $l\in \mathbb{N}$, and $\beta= \beta(v_0,D_0)>0$ obtained in Corollary \ref{GH to C3}. Then, for all $\epsilon>0$, there exists $\delta = \delta(\epsilon,D_0,v_0,l) >0$ such that if $(M,g^\mathcal{E})$ is an Einstein manifold satisfying
	\begin{itemize}
		\item the volume is bounded below by $v_0>0$,
		\item the diameter is bounded above by $D_0$,
		\item the Ricci curvature is bounded $|\Ric|\leq 3$.
	\end{itemize}
	and such that for an Einstein orbifold $(M_o,g_o)$, we have
	$$d_{GH}\big((M,g^\mathcal{E}),(M_o,g_o)\big)\leqslant \delta,$$
	then, there exists a naïve desingularization $(M,g^D_{t,v})$ of $(M_o,g_o)$ by a tree of singularities, and a diffeomorphism $\psi: M\to M$ such that $$\psi^*g^\mathcal{E} = \Hat{g}_{t,v},$$ where $\Hat{g}_{t,v}$ is the perturbation of $g^D_{t}+v$ of Theorem \ref{fcts inv einst général}. 
   \end{cor}
\subsection{Premoduli space of Einstein metrics around a singular one}\label{section premoduli}

Let us now explain how the Einstein modulo obstructions metrics of Theorem \ref{fcts inv einst général} are analogous to the metrics in the set $W$ of Theorem \ref{koiso structure}.

Let $M$ a differentiable $4$-manifold and consider an orbifold $(M_o,g_o)\in\partial_o\mathbf{E}(M)\subset \overline{\mathbf{E}(M)}_{GH}\backslash\mathbf{E}(M)$. According to Corollary \ref{proximité désing naive}, the Einstein metrics which are sufficiently close to $(M_o,g_o)$ in the Gromov-Hausdorff are results of the gluing-perturbation procedure of Theorem \ref{fcts inv einst général}.
    
    By analogy with the definition of \cite{koi} of the premoduli space of Einstein metrics around a smooth one, we define de \emph{directional} premoduli space in the neighborhood of a singular metric in the following way.
    \begin{defn}[Directional premoduli space] We define $\mathbb{E}^D_{g_o}(M)$, the \emph{premoduli space of Einstein metrics} on $M$ in the neighborhood of $(M_o,g_o)$ and \emph{in the direction} $D$, as the set of metrics $g$ on $M$ for which there exists $t$ such that we have $\|g-g^D_t\|_{C^{2,\alpha}_{\beta,*}(g^D_t)}<\epsilon$ for $\epsilon>0$ the constant of Theorem \ref{fcts inv einst général} and $v\in \tilde{\mathbf{O}}(g^D_t)$ for which we have:
    \begin{enumerate}
        \item $E(g) = 0$, $\vol(g)= \vol(g_o)$,
        \item $g-\big(g^D_{t}+v\big) \perp_{g^D_{t}} \tilde{\mathbf{O}}(g^D_{t})$,
        \item $\tilde{\delta}_{g^D_{t}}g=0$.
    \end{enumerate}
    \end{defn}
    
    \begin{rem}
        The above premoduli space is directional in the sense that it does not cover all of the Gromov-Hausdorff desingularizations of $(M_o,g_o)$, but only the ones whose tree of singularities is $D$. Given an arbitrary sequence $(g_i)_i$ of smooth Einstein metrics on $M$ $d_{GH}$-converging to $(M_o,g_o)$, then, up to considering a subsequence, all of the metrics $g_i$ are in a \emph{single} directional premoduli space $\mathbb{E}^D_{g_o}(M)$.
    \end{rem}
    The result of \cite{ozu1} together with Theorem \ref{fcts inv einst général} yield the following statement.
    \begin{cor}
        The directional premoduli space $\mathbb{E}^D_{g_o}(M)$ is the zero-set of $E$ on the set denoted $\hat{W}$ of metrics $\hat{g}_{t,v}$ in the $C^{2,\alpha}_{\beta,*}(g^D_t)$-neighborhoods of the metrics $g^D_t$ of Theorem \ref{fcts inv einst général}.
    \end{cor}
    Since $E$ is analytic, Koiso's proof of Theorem \ref{koiso structure} reduces to proving that $W$ is an analytic submanifold by implicit function theorem. Therefore the question of the regularity of the $d_{GH}$-completion of $\mathbf{E}(M^4)$, $\mathbf{E}(M^4)\cup\partial_o\mathbf{E}(M^4) $, reduces to the question of the regularity of the set $\hat{W}$ of the Einstein modulo obstructions metrics of Theorem \ref{fcts inv einst général}.

\section{Obstructions to the Gromov-Hausdorff desingularization}

Let us now come to the main application of this article, which is the obstruction to the desingularization of Einstein orbifolds. 

\subsection{Analysis of integrable Ricci-flat ALE spaces}\label{an integrability issue}

In order to obtain an obstruction result we will need to assume that the Ricci-flat ALE metrics in the trees of singularities have integrable Ricci-flat deformations.

\subsubsection{Integrable Ricci-flat ALE}

Since Corollary \ref{GH to C3} does not control the convergence speed towards the limit orbifold or the Ricci-flat ALE spaces, like in Theorem \ref{fcts inv einst général}, we have to fix a gauge $v\in\tilde{\mathbf{O}}(g^D)$ on the approximate kernel of the operator $\bar{P}$.

Not to damage our controls, we cannot simply use $g^D + v$ as an approximate metric. We need to find a better approximation to extend Proposition \ref{meilleure approx and obst} to the case when $v\neq 0$. It turns out that this will only be possible if we assume that the Ricci-flat ALE metrics are \emph{integrable}.

\begin{defn}[Integrable Ricci-flat ALE orbifold]\label{definition integrable}
    We will say that a Ricci-flat metric ALE $g_b$ is \emph{integrable} if for all $v\in \mathbf{O}(g_b)$ small enough, there exists a Ricci-flat metric ALE $\bar{g}_{b,v}$ satisfying $\bar{g}_{b,v}-(g_b+v)\perp \mathbf{O}(g_b)$ and $\|\Bar{g}_{b,v}-g_b\|_{L^2(g_b)}\leqslant 2 \|v\|_{L^2}$, and such that $\delta_{g_b}\bar{g}_{b,v}=0$.
\end{defn}

\begin{rem}\label{les def inft of RF ALE are en jauge}
All of the known examples of Ricci-flat ALE spaces are integrable since they are quotients of hyperkähler spaces. Moreover, any infinitesimal $L^2$-deformation of ALE Ricci-flat orbifolds is automatically divergence-free and trace-free by the proof of Lemma \ref{kernel L2}.
\end{rem}

\subsubsection{Weighted Hölder spaces and asymptotics of Ricci-flat ALE spaces}

Let us introduce yet another function space to control the asymptotics of our ALE metrics. This will be crucial to deduce obstructions in the following sections.

\begin{defn}[$C^{2,\alpha}_{\beta,**}$-norm on a ALE orbifold]
    Let $(N,g_b)$ be an ALE orbifold, and let $h$ be a symmetric $2$-tensor on $N$, and assume that $h = H^4 + \mathcal{O}(r_b^{-4-\beta})$ for $\beta>0$ and $H^4$ a homogeneous harmonic symmetric $2$-tensor with $|H^4|\sim r_e^{-4}$. We define its $C^{2,\alpha}_{\beta,**}$-norm by
    $$\|h\|_{C^{2,\alpha}_{\beta,**}}:= \sup r_e^{4}|H^4|_{g_e} + \big\|(1+r_b)^{4}(h-\chi(\epsilon r_b) H^4)\big\|_{C^{2,\alpha}_{\beta,*}}.$$
\end{defn}

\begin{defn}[$r_bC^{3,\alpha}_{\beta,**}$-norm on a ALE orbifold]
    Let $(N,g_b)$ be an ALE orbifold, and let $X$ be a vector field on $N$, and assume that $X = Y^3 + \mathcal{O}(r_b^{-3-\beta})$ for $\beta>0$ and $Y^3$ a homogeneous element of the kernel of $\delta_e\delta_e^*$ with $|Y^3|\sim r_e^{-3}$. We define its $r_bC^{3,\alpha}_{\beta,**}$-norm by
    $$\|X\|_{r_bC^{3,\alpha}_{\beta,**}}:= \sup r_e^{3}|Y^3|_{g_e} + \big\|(1+r_b)^{4}(X-\chi(\epsilon r_b) Y^3)\big\|_{r_bC^{3,\alpha}_{\beta,*}}.$$
\end{defn}

These norms are motivated by the following Lemmata.
\begin{lem}\label{inverse P ALE r-4}
    Let $(N,g_b)$ be an ALE orbifold. Then, there exists $C>0$ such that we have, for any $h\perp \mathbf{O}(g_b)$,
    \begin{equation}
        C^{-1} \|(1+r_b)^{4}\bar{P}_{g_b}h\|_{r_b^{-2}C^{\alpha}_\beta}\leqslant\|h\|_{C^{2,\alpha}_{\beta,**}}\leqslant C \|(1+r_b)^{4}\bar{P}_{g_b}h\|_{r_b^{-2}C^{\alpha}_\beta}.\label{inverse ALE term 4}
    \end{equation}
\end{lem}
\begin{proof}
    By the theory of elliptic operators in weighted Hölder spaces (see for instance \cite[Chapter 2]{pr}), the operator $\bar{P}_{g_b}: (1+r_b)^{-4}C^{2,\alpha}_{-\beta} \to (1+r_b)^{-4}r_b^{-2}C^{\alpha}_{-\beta}$ is Fredholm with kernel $\mathbf{O}(g_b)$ and cokernel $\mathbf{O}(g_b)$ because there is no other exceptional value than $0$ between $2$ and $-4$. This implies that $\bar{P}_{g_b}: (1+r_b)^{-4}C^{2,\alpha}_{-\beta}\cap \mathbf{O}(g_b)^\perp \to (1+r_b)^{-4}r_b^{-2}C^{\alpha}_{-\beta}\cap \mathbf{O}(g_b)^\perp $ is invertible with a bounded inverse. 
    
    Moreover, we have $\bar{P}_{g_b}^{-1}\big((1+r_b)^{-4}r_b^{-2}C^{\alpha}_{\beta}\big) = C^{2,\alpha}_{\beta,**}$ since $-4$ is the first negative exceptional value for $\bar{P}$, and the stated inequality comes from the fact that the inverse is bounded.
\end{proof}
Similarly for vector fields, we have the following result.

\begin{lem}\label{controle inverse mise jauge ALE}
    Let $(N,g_b)$ be an ALE orbifold. Then, there exists $C>0$ such that for any vector field $X$ on $N$, we have
    \begin{equation}
        \|X\|_{r_bC^{3,\alpha}_{\beta,**}}\leqslant C \|(1+r_b)^{4}\delta_{g_b}\delta_{g_b}^*X\|_{r_b^{-1}C^{1,\alpha}_\beta}.\label{inverse delta ALE term 4}
    \end{equation}
\end{lem}

In particular the analysis of Theorem \ref{fcts inv einst général} and Proposition \ref{Mise en Reduced divergence-free} extends to the case where $(M_o,g_o)$ is a Ricci-flat ALE orbifold and where the norm $C^{2,\alpha}_{\beta,*}(g_o)$ is replaced by $C^{2,\alpha}_{\beta,**}(g_o)$ thanks to Lemma \ref{inverse P ALE r-4} and $r_oC^{3,\alpha}_{\beta,*}(g_o)$ is replaced by $r_oC^{3,\alpha}_{\beta,**}(g_o)$ thanks to Lemma \ref{controle inverse mise jauge ALE}. Indeed, all of the controls are local around the singular points or coming from an estimate on the inverse on the rest of the orbifold exactly like \eqref{inverse ALE term 4} and \eqref{inverse delta ALE term 4}. For the operator $\bar{P}$, this yields the following control on the asymptotic terms on the ALE end.
\begin{cor}\label{control termes ordre 4 nouveau poids}
 Let $(N,g_b)$ be a Ricci-flat ALE orbifold, and denote $(N^B,g^B_t)$ a naïve desingularization of $(N,g_b)$ by a tree of Ricci-flat ALE orbifolds glued according to a pattern $B$ with relative scales $t$.
    
    For $v$ and $t$ small enough depending on the constants of Lemma \ref{Contrôle P espace modèles}, let $\bar{g}_{t,v}$ be the unique metric (according to Theorem \ref{fcts inv einst général}) satisfying for $\epsilon>0$ small enough:
    \begin{enumerate}
        \item $\|g^B_t- \Bar{g}_{t,v}\|_{C^{2,\alpha}_{\beta,*}(g^B_t)}<2\epsilon$,
        \item $ (g^B_t+v)- \Bar{g}_{t,v} $ is $L^2(g^B_t)$-orthogonal to $\tilde{\mathbf{O}}(g^B_t)$, and
        \item $\pi_{\tilde{\mathbf{O}}(g^B_t)^\perp}\mathbf{\Phi}_{g^B_t}(\Bar{g}_{t,v})= 0$.
    \end{enumerate}
    Then, for any $0<\beta<1$, we have $$\Bar{g}_{t,v} = g_e +\bar{H}^4_{t,v} + \mathcal{O}(r_B^{-4-\beta}),$$
    for $|\bar{H}^4_{t,v}|\sim r_B^{-4}$ and $\bar{H}^4_{t,v}\to \bar{H}^4$, the asymptotic terms of $(N,g_b)$ as $(t,v)\to(0,0)$.
\end{cor}
Similarly, using Lemma \ref{controle inverse mise jauge ALE}, we can put our ALE metrics in gauge with respect to each other. The nonlinear terms are taken care of like in Proposition \ref{Mise en Reduced divergence-free} by noting that the weight of our norm, $(1+r_b)^4r_b^{-1}$, is this time also larger than $1$ at infinity.
\begin{cor}\label{Mise en Reduced divergence-free ALE}
Let $0<\beta<1$, let $(N,g_b)$ be a Ricci-flat ALE orbifold, and denote $(N^B,g^B_t)$ a naïve desingularization of $(N,g_b)$ by a tree of Ricci-flat ALE orbifolds glued according to a pattern $B$ with relative scales $t$.

Then, there exist $\epsilon_B,\tau_B,C_B>0$ which only depend on the metrics $g_b$ and the elements of $B$ such that for $t_{\max}\leqslant\tau_B$ and for any metric $g$ satisfying $\|g-g^B_t\|_{C^{2,\alpha}_{\beta,**}(g^B_t)}\leqslant \epsilon_B$, there exists a unique vector field $X$ on $M$ for which, 
    $$\tilde{\delta}_{g^B_t}(\exp_X^*g)=0,$$
    where $\exp_X$ is the diffeomorphism $\exp_X: x\in M \mapsto \exp_x^{g^B_t}(X(x))$.
    
    We moreover have, $\|X\|_{r_BC^{3,\alpha}_{\beta,**}(g^B_t)}\leqslant C_B \|(1+r_B)^{4}\tilde{\delta}_{g^B_t}(g-g^B_t)\|_{r^{-1}_BC^{1,\alpha}_{\beta}(g^B_t)}$, and therefore, there exists $\eta: \mathbb{R}^+\to \mathbb{R}^+$ with $\lim_0\eta=0$ such that we have 
    $$ \|\exp_X^*g-g\|_{C^{2,\alpha}_{\beta,**}(g^B_t)}\leqslant \eta\big(\|g-g^B_t\|_{C^{2,\alpha}_{\beta,**}(g^B_t)}\big).$$
\end{cor}

\subsection{Approximate Einstein modulo obstructions metric}
Let $(M_o,g_o)$ be an Einstein orbifold and let $p$ be one of its singular points whose singularity model is $\mathbb{R}^4\slash\Gamma$.

Let us now construct a good approximation of the metric $\hat{g}_{t,v}$ of Theorem \ref{fcts inv einst général}, which we will denote $g^A_{t,v}$. This is a crucial step to understand  and approximate the obstructions coming from the Ricci-flat ALE spaces appearing at a singular point $p\in M_o$.

In all of this section, we will assume the following properties:
\begin{itemize}
    \item $(M,g^D_t)$ is a naïve desingularization of $(M_o,g_o)$,
    \item at the singular point $p\in M_o$, there is only one Ricci-flat ALE \textit{manifold} $(N,g_b)$ glued, and therefore no tree of singularity,
    \item $(N,g_b)$ is an \textit{integrable} Ricci-flat ALE manifold.
\end{itemize}

Consider $S$ the complement of $\{p\}$ among the singular points of $M_o$, let $(M_{S},g^D_{S} = g^D_{S,t_S})$ be a naïve partial desingularization of $(M_o,g_o)$ which only leaves the point $p$ singular and let $(M_{S},\hat{g}_{S} = \hat{g}_{S,t_S,v_S})$ be the perturbation of $(M_{S},g^D_{S}+v_{S})$ orthogonally to $\tilde{\mathbf{O}}(g^D_{S})$ of Theorem \ref{fcts inv einst général} satisfying 
$$\mathbf{\Phi}_{g^D_{S}}(\hat{g}_{S}) = \mathbf{o}_S\in \tilde{\mathbf{O}}(g^D_{S}).$$
At $p$, the metric $\hat{g}_{S}$ has the following development in local coordinates where it is in divergence-free gauge with respect to $g_e$,
\begin{equation}
    \hat{g}_{S} = g_e + \Hat{H}_{S} + \mathcal{O}(r_o^3)
\end{equation}
and we know that $\mathbf{\Phi}_{g^D_{S}}(\hat{g}_{S}) = \mathbf{o}_{S} = \mathbf{O}_{S} + \mathcal{O}(r_e^2)$ with $|\mathbf{O}_{S}|_{g_e}\sim r_e^0$, $\textup{tr}_{g_e}\mathbf{O}_{S} = 0$ and $\delta_{g_e}\mathbf{O}_{S} = 0$. We therefore have
\begin{align}
    0=&\;\mathbf{\Phi}_{g^D_{S}}(\hat{g}_{S})-\mathbf{O}_{S} = \lambda g_e + \bar{P}_e(\Hat{H}_{S})-\mathbf{O}_{S}+ \mathcal{O}(r_e^2),\label{dvp pt sing}
\end{align}
where $\lambda = \frac{n-2}{2n}\overline{\R}(\hat{g}_{S})$.
Consequently $\bar{P}_e\hat{H}_S + \lambda g_e = \mathbf{O}_{S}$, where $\mathbf{O}_{S}$ is the limit of $\mathbf{o}_{S}$ at $p$.

For some small $v_p\in \mathbf{O}(g_b)$, we will glue a Ricci-flat deformation $\bar{g}_{b,v_p}$ of $g_b$ at a scale $t_p$ to $\hat{g}_{S}$. To obtain a better estimate, we will extend the quadratic terms of $\hat{g}_S$ in order to minimize the error in the gluing.

\begin{prop}\label{terme quadratique with obst}
	    Let $(N,g_b)$ a Ricci-flat ALE orbifold asymptotic to $\mathbb{R}^4\slash\Gamma$, $\hat{H}_S$ a quadratic symmetric $2$-tensor on $\mathbb{R}^4\slash\Gamma$, $\lambda\in \mathbb{R}$ and $\mathbf{O}_{S}$ a constant symmetric $2$-tensor on $\mathbb{R}^4\slash\Gamma$ such that we have: 
		$$\bar{P}_{e}\hat{H}_S+\lambda g_e = \mathbf{O}_{S}.$$
		Then, there exists a 2-tensor $\hat{h}_S$, and for $(\mathbf{o}_j)_j$ an $L^2$-orthonormal basis of $\mathbf{O}(g_b)$, real numbers $\hat{\lambda}_j$ such that $(\hat{h}_S,\hat{\lambda}_j)$ is a solution of 
		\begin{equation}
		     \left\{
    \begin{array}{ll}
        \bar{P}_b\hat{h}_S+\lambda g_b =& \chi \mathbf{O}_{S} + \sum_j \hat{\lambda}_j \mathbf{o}_j, \\
        |\hat{h}_S - \chi \hat{H}_S|_{g_b} =&\mathcal{O}(r_b^{-2}),
    \end{array}
\right.\label{equation h2}
		\end{equation}
	    where $\chi$ is a cut-off function supported in a neighborhood of the infinity of $(N,g_b)$ where the ALE coordinates are defined, and where 
	    \begin{equation}
	        \hat{\lambda}_j:= -\int_{\mathbb{S}^3\slash\Gamma} \big(3\langle \hat{H}_S, O^4_j \rangle_{g_{e}} + O_j^4\big(\nabla_e\textup{tr}_e\hat{H}_S,\partial_{r_e}\big) \big) dv_{\mathbb{S}^3\slash\Gamma} + \int_N \chi\langle \mathbf{O}_{S}, \mathbf{o}_j \rangle_{g_b} dv_b.\label{hat lambda j}
	    \end{equation}
		The set of solutions to the above equation \eqref{equation h2} is $(\hat{h}_S + \mathbf{O}(g_b),\hat{\lambda}_j)$.
	\end{prop}
	\begin{proof}
    We have $\bar{P}_b(\chi \hat{H}_S)+\lambda g_b - \chi \mathbf{O}_{S}\in r_b^{-2}C^\alpha_\beta(g_b)$ for $0<\beta<1$ because in a neighborhood of infinity, $g_b-g_e = \mathcal{O}(r_b^{-4})$ together with the derivatives. Lemma \ref{Contrôle P espace modèles} also holds by replacing the operator $P_b$ by the operator $\bar{P}_b$ as a consequence of Theorem \ref{fcts inv einst général} in the case where $(M_o,g_o) = (N,g_b)$ is a Ricci-flat ALE manifold. Consequently, there exists $\hat{h}'\in C^{2,\alpha}_\beta(g_b)$ such that $\bar{P}_b(\chi\hat{H}_S + \hat{h}') + \lambda g_b = \chi \mathbf{O}_{S} +\sum_j\hat{\lambda}_j \mathbf{o}_j$ with
    \begin{equation}
        \hat{\lambda}_j = -\int_{\mathbb{S}^3\slash\Gamma} \big(3\langle \hat{H}_S, O^4_j \rangle_{g_{e}} +\frac{1}{2} O_j^4\big(\nabla tr\hat{H}_S,\partial_{r_e}\big) \big)dv_{\mathbb{S}^3\slash\Gamma} + \int_N \chi\langle \mathbf{O}_{S}, \mathbf{o}_j \rangle_{g_b} dv_b,\label{exp bar lambdai}
    \end{equation}
    where $O^4_j\sim r_b^{-4}$ is the first term of the development of $\mathbf{o}_j$ at infinity. Indeed, by integration by parts, and using the fact that $\textup{tr}_b \mathbf{o}_j = \langle g_b,\mathbf{o}_j \rangle_{g_b} = 0$ and $\delta_b\mathbf{o}_j=0$, we get,
\begin{align*}
		    \langle \bar{P}_b(\chi \hat{H}_S), \mathbf{o}_j \rangle_{L^2} =&\; \frac{1}{2}\lim_{\rho\to\infty} \int_{r_b\leqslant\rho}\big\langle \nabla^*_b\nabla_b (\chi \hat{H}_S)- \nabla_b^2 (\textup{tr}_b \chi \hat{H}_S)\;,\;\mathbf{o}_j \big\rangle dv_{g_b} \\
		    =&\; \frac{1}{2}\lim_{\rho\to \infty}\int_{r_b=\rho}\big(\big\langle \chi \hat{H}_S\;,\;\nabla_n \mathbf{o}_j \big\rangle-\big\langle \nabla_n (\chi \hat{H}_S)\;,\;\mathbf{o}_j \big\rangle\\
		    &+ \mathbf{o}_j(\nabla_b(\textup{tr}_b\chi \hat{H}_S)\;,\; \partial_{r_e} )\big)dS_{\rho}\\
		    =&\; -\int_{\mathbb{S}^3\slash\Gamma} \big(3\langle \hat{H}_S, O^4_j \rangle_{g_{e}} + \frac{1}{2}O_j^4\big(\nabla_e \textup{tr}_e\hat{H}_S,\partial_{r_e}\big) \big)dv_{\mathbb{S}^3\slash\Gamma}.
\end{align*}
    Now, the integral $\int_N \chi\langle \mathbf{O}_{S}, \mathbf{o}_j \rangle_{g_b} dv_b$ converges even if $\langle \mathbf{O}_{S}, \mathbf{o}_j \rangle_{g_b} = \mathcal{O}(r_b^{-4})$. Indeed, in ALE coordinates, $r_b^4\mathbf{o}_j = \phi_{ij}dx^idx^j + \mathcal{O}(r_b^{-1})$, where the $\phi_{ij}: \mathbb{S}^3\to \mathbb{R}$ are nonconstant eigenfunctions and therefore have zero mean values, hence, $\int_{\{r_b=\rho\}}\langle \mathbf{O}_{S}, \mathbf{o}_j \rangle_{g_b}dS_\rho = \mathcal{O}(\rho^{-2})$ and the integral converges.  The values of $\hat{\lambda}_j$ from \eqref{exp bar lambdai} therefore ensure that we have 
    $$\bar{P}_b(\chi \hat{H}_S)+\lambda g_b - \chi \mathbf{O}_{S} -\sum_j\hat{\lambda}_j \mathbf{o}_j \perp_{g_b} \mathbf{O}(g_b).$$
\end{proof}
Choosing $t_S,t_p>0$, $v_S\in\tilde{\mathbf{O}}(g^D_{S,t_S})$ and $v_p\in\mathbf{O}(g_b)$, by gluing  metric $\bar{g}_{b,v_p}$ at scale $t_p$ to $\hat{g}_{S,v_S,t_S}$, we reach all of the gauges of Theorem \ref{fcts inv einst général}, that is we attain some element of $(g^D_t+v)+\tilde{\mathbf{O}}(g^D_t)^\perp$ for any small $v$ and $t$.
\\

We define $\tilde{\mathbf{O}}_{t_p}(\bar{g}_{b,v_p})$ as the cut-off of the elements of $\mathbf{O}(\bar{g}_{b,v_p})$ as in Definition \ref{def obst tronq} at scale $t_p$. We denote $\hat{h}_S(t,v)$ the $2$-tensor satisfying $ \hat{h}_S(t,v)\perp_{\bar{g}_{b,v_p}}\tilde{\mathbf{O}}_{t_p}(\bar{g}_{b,v_p})$ obtained in Proposition \ref{terme quadratique with obst} with the Ricci-flat ALE metric $\bar{g}_{b,v_p}$ and $\hat{H}_S$ the quadratic terms of $\hat{g}_{S}=\hat{g}_{S,t_S,v_S}$.

\begin{defn}[Approximate metric $g_{t,v}^A$]
    Let $v\in \tilde{\mathbf{O}}(g^D_t)$ for $(M,g^D_t)$ a naïve desingularization of $(M_o,g_o)$.

    The Riemannian manifold $(M,g^A =g^A_{t,v})$ is obtained by naïve gluing (Definition \ref{def naive desing}) of $(N,\bar{g}_{b,v_p} + t_p\hat{h}_S(t,v))$ to $(M_{S},\hat{g}_{S,t_S,v_S})$ at scale $t_p>0$.
    
    The numbers $t_p,t_S>0$, $v_p\in \mathbf{O}(g_b)$ and $v_S\in\tilde{\mathbf{O}}(g^D_S)$ are chosen in order to have $g^A_{t,v}-(g^D_t+v)\perp \tilde{\mathbf{O}}(g^D_t)$.
\end{defn}

\subsection{Better approximation and obstructions}

The obstruction will come from the better controls of $g^A_{t,v}$ and the following proposition.

\begin{prop}\label{meilleure approx and obst}
    Let $0<\alpha<1$, and $(M,g^D_t)$ be a naïve desingularization, and assume that for $t_{\max}<\tau$, the metric $(M,\Hat{g} = \Hat{g}_{t,v})$ obtained by Theorem \ref{fcts inv einst général} is an \textit{Einstein metric} (without obstructions). Then, there exists $\epsilon>0$ and $C>0$ only depending on $g_o$ and the $g_{b_j}$, such that, denoting $\Psi (g^A_{t,v} ):=\pi_{\tilde{\mathbf{O}}(g^D)^\perp} \mathbf{\Phi}_{g^D_t}(g^A_{t,v})$ and $\mathbf{o}^A_{t,v} := \pi_{\tilde{\mathbf{O}}_{t_p}(\bar{g}_{t_p,v_p})} \mathbf{\Phi}_{g^D_t}(g^A_{t,v} )$, we have
    \begin{equation}
        \|\hat{g}_{t,v} -g^A_{t,v} \|_{C_{\beta,*}^{2,\alpha}(g^D_t)}\leqslant C\big\|\Psi (g^A_{t,v} ) \big\|_{r^{-2}_DC_{\beta}^{\alpha}(g^D_t)}, \label{Meilleure approximation}
    \end{equation}
    and
    \begin{align}
        \|\mathbf{o}^A_{t,v}\|_{L^2(g^D)} \leqslant \; \big(&\| \Psi (g^A_{t,v} ) \|_{r^{-2}_DC_{\beta}^{\alpha}(g^D)}+ t_{p}^{\frac{1}{2}+\frac{\beta}{4}}\big) \| \Psi (g^A_{t,v} ) \|_{r^{-2}_DC_{\beta}^{\alpha}(g^D)}.\label{obstruction explicite}
    \end{align}
\end{prop}
\begin{proof}
    Let us denote $g^A := g^A_{t,v}$, $\hat{g}:=\hat{g}_{t,v}$ and $h^A:= \hat{g}-g^A$ as well as $\bar{g}_b = \bar{g}_{b,v_p}$ for simplicity. The inequality \eqref{Meilleure approximation} is a direct consequence of the proof of Theorem \ref{fcts inv einst général} because the proof uses an inverse function theorem. Indeed, at the linear level, by Lemma \ref{Contrôle P espace modèles}, we have
    \begin{equation}
        \|\hat{g} -g^A \|_{C_{\beta,*}^{2,\alpha}(g^D)}\leqslant C\big\|\pi_{\tilde{\mathbf{O}}(g^D)^\perp}\bar{P}_{g^D}(\hat{g}-g^A) \big\|_{r^{-2}_DC_{\beta}^{\alpha}(g^D)}.\label{linear fcts inv}
    \end{equation}
    Since $ \Psi(\hat{g}) = 0 $, the controls of the nonlinear terms of $\Psi$ imply that
    \begin{align}
        \|\Psi(g^A)\|_{r^{-2}_DC_{\beta}^{\alpha}(g^D)}\geqslant&\; \big\|\pi_{\tilde{\mathbf{O}}(g^D)^\perp}\bar{P}_{g^D}(\hat{g}-g^A) \big\|_{r^{-2}_DC_{\beta}^{\alpha}(g^D)}\nonumber\\
        &- C\big(\|\hat{g}-g^D \|_{C_{\beta,*}^{2,\alpha}(g^D)} +\|g^A-g^D \|_{C_{\beta,*}^{2,\alpha}(g^D)}\big)\|\hat{g} -g^A \|_{C_{\beta,*}^{2,\alpha}(g^D)}.\label{nonlin fcts inv}
    \end{align}
    We moreover know that $\|g^A-g^D \|_{C_{\beta,*}^{2,\alpha}(g^D)}\leqslant2\epsilon $  and $ \|\hat{g}-g^D \|_{C_{\beta,*}^{2,\alpha}(g^D)} = \mathcal{O} (t_{\max}^{\frac{2-\beta}{4}})$ thanks to the proof of Theorem \ref{fcts inv einst général} and by assumption. Therefore, choosing $\epsilon$ and $t_{\max}$ small enough and putting the estimates \eqref{linear fcts inv} and \eqref{nonlin fcts inv} together yields the expected control \eqref{Meilleure approximation}.
    
    Let us focus on the zone $N^t$ where the elements of $\tilde{\mathbf{O}}_{t_p}(\bar{g}_{b,v_p})$ are supported. For $\tilde{\mathbf{o}}\in \tilde{\mathbf{O}}_{t_p}(\bar{g}_{b,v_p})$, we have
    \begin{align}
        \big|\big\langle\Bar{P}_{g^A}h^A , \tilde{\mathbf{o}}\big\rangle_{L^2(g^D)}\big|\nonumber
        &\leqslant \big|\big\langle(\Bar{P}_{g^A} - \Bar{P}_{t_p\bar{g}_b})h^A , \tilde{\mathbf{o}}\big\rangle_{L^2(t_p\bar{g}_b)}\big|+\big|\big\langle\Bar{P}_{t_p\bar{g}_b}h^A ,\tilde{\mathbf{o}}\big\rangle_{L^2(t_p\bar{g}_b)}\big|\nonumber\\
        &= \big|\big\langle(\Bar{P}_{g^A} - \Bar{P}_{t_p\bar{g}_b})h^A , \tilde{\mathbf{o}}\big\rangle_{L^2(t_p\bar{g}_b)}\big|+\big|\big\langle t_p^{-1}h^A ,\Bar{P}_{\bar{g}_b}\tilde{\mathbf{o}}\big\rangle_{L^2(\bar{g}_b)}\big|,\label{premiere partie est PgAhAo}
  \end{align}
  since $\bar{P}_{\bar{g}_{b}}$ is self dual and since the $L^2$-product of $2$-tensor is invariant by rescaling in dimension $4$ as well as the rescaling behavior of $\bar{P}$. Now, we have the following estimates on $N^{t} = \{r_b<2t_p^{-\frac{1}{4}}\} \subset \{r_D<2t_p^\frac{1}{4}\} $, for $k\in\{0,1,2\}$, there exists $C>0$ such that:
  \begin{enumerate}
      \item $r_D^k|\nabla^k_{t_p\bar{g}_b}(g^A-t_p\bar{g}_b)|_{t_p\bar{g}_b}\leqslant C r_D^2$ since $g^A=t_p\bar{g}_b+t_p^2\hat{h}_S$ on $N^{16t}$ by construction and because of the controls of the cut-off functions,
      \item $r_b^k|\nabla^k_{\bar{g}_b}\tilde{\mathbf{o}}|\leqslant C r_b^{-4}$ as well as $\bar{P}_{\bar{g}_b}\tilde{\mathbf{o}}= 0$ on $N^{16t} \subset \{r_D<t_p^\frac{1}{4}\}$, and
      \item $r_D^k|\nabla^k_{t_p\bar{g}_b} h_A|_{t_p\bar{g}_b}\leqslant C(t_p^{-\frac{1}{2}}r_D)^{-\beta} \|\Psi(g^A)\|_{r^{-2}_DC_{\beta}^{\alpha}(g^D)}$ by \eqref{Meilleure approximation}.
  \end{enumerate}
  We can therefore control the terms of \eqref{premiere partie est PgAhAo} in the following way:
  \begin{align}
        \big|\big\langle\Bar{P}_{g^A}h^A , \tilde{\mathbf{o}}\big\rangle_{L^2(g^D)}\big|&\leqslant C\|\Psi(g^A)\|_{r^{-2}_DC_{\beta}^{\alpha}(g^D)}\int_{t_p^\frac{1}{2}}^{t_p^\frac{1}{4}} \big(\frac{r}{\sqrt{t_p}}\big)^{-\beta}\big(\frac{r}{\sqrt{t_p}}\big)^{-4} r^3dr\nonumber\\
        &+ \|\Psi(g^A)\|_{r^{-2}_DC_{\beta}^{\alpha}(g^D)}\int_{t_p^{-\frac{1}{4}}}^{2t_p^{-\frac{1}{4}}}t_p^{-1} r^{-\beta} r^{-6} r^3dr\nonumber\\
        &\leqslant C\|\Psi(g^A)\|_{r^{-2}_DC_{\beta}^{\alpha}(g^D)} \big( t_p^{2+\frac{\beta}{4}}+ t_p^{\frac{1}{2}+\frac{\beta}{4}} \big)\nonumber\\
        &\leqslant C \|\Psi(g^A)\|_{r^{-2}_DC_{\beta}^{\alpha}(g^D)} t_p^{\frac{1}{2}+\frac{\beta}{4}}.\label{dernière partie est PgAhAo}
  \end{align}
    Let $\mathbf{o}^A= \mathbf{o}^A_{t,v}:=\pi_{\tilde{\mathbf{O}}_{t_p}(\bar{g}_{b,v_p})}\mathbf{\Phi}_{g^D}(g^A)$, since $\mathbf{\Phi}_{g^D}(\hat{g})=0$, and $d_{g^A}\mathbf{\Phi}_{g^D} = \Bar{P}_{g^A }$. We have
    \begin{align}
        -\|\mathbf{o}^A\|_{L^2(g^D)}^2 = \langle\mathbf{\Phi}_{g^D}(\Hat{g} )-\mathbf{o}^A,\mathbf{o}^A\rangle= \langle\Psi (g^A ) + \Bar{P}_{g^A }h^A +Q_{g^A }(h^A ), \mathbf{o}^A\rangle,\label{dvp gA hat g}
    \end{align}
    where the non-linear terms $Q_{g^A }(h^A )= \mathbf{\Phi}_{g^D}(\hat{g} )-\mathbf{\Phi}_{g^D}(g^A ) - \Bar{P}_{g^A }(h^A )  $ satisfy
    \begin{equation}
        \|Q_{g^A }(h^A )\|_{r_D^{-2}C^{0}(M)}\leqslant C \|h^A \|_{{C_{\beta,*}^{0}}}^2 \leqslant C \big\| \Psi (g^A ) \big\|_{r^{-2}_DC_{\beta}^{0}(g^D)}^2,\label{est quad gA}
    \end{equation}
    and by definition of the weighted norms, since $r_D^{-2}C^\alpha_\beta(g^D)\subset L^2(g^D)$, we therefore have for $\mathbf{o}\in \tilde{\mathbf{O}}(g^D)$, 
    $$\big|\big\langle \mathbf{o}, Q_{g^A }(h^A ) \big\rangle_{L^2(g^D)} \big|\leqslant C\|\mathbf{o}\|_{L^2}\big\|\Psi (g^A )\big\|_{r^{-2}_DC_{\beta}^{0}(g^D)}^2.$$
    Hence, since $\Psi (g^A )\perp_{g^D} \mathbf{o}^A$ by definition of $\Psi $, we have by \eqref{dvp gA hat g}, and thanks to \eqref{dernière partie est PgAhAo} and \eqref{est quad gA} we therefore have 
    \begin{align*}
        \|\mathbf{o}^A\|_{L^2(g^D)} \leqslant \; C\big(&\| \Psi (g^A ) \|_{r^{-2}_DC_{\beta}^{\alpha}(g^D)}
    + t_{p}^{\frac{1}{2}+\frac{\beta}{4}}\big) \| \Psi (g^A ) \|_{r^{-2}_DC_{\beta}^{\alpha}(g^D)}.
    \end{align*} 
\end{proof}

\begin{rem}
    The inequality \eqref{Meilleure approximation} means that if we are able to construct a metric $g^A $ such that $\Psi(g^A)$ is small, then $g^A $ is a good approximation of $\hat{g} $, the only zero of $\Psi$. This allows us to approximate the metrics $\hat{g}$, and therefore degenerating Einstein metrics with an arbitrarily good precision.
    
    The inequality \eqref{obstruction explicite} is an obstruction result. Indeed, if we construct a good approximation $g^A $, for which $ \Psi(g^A )$ is small, but without having $\mathbf{o}^A$ small, then $ \hat{g} $ cannot be Einstein, and the metric $g^A$ (and therefore $g^D$) cannot be perturbed to an Einstein metric orthogonally to $\tilde{\mathbf{O}}(g^D)$.
\end{rem}

Let us now control the above quantities of \eqref{obstruction explicite} for $g^A_{t,v}$.

\begin{prop}\label{controle approximation desing part}
    For $(t_{S},t_p)>0$ small enough and $k\in \mathbb{N}$, denoting $(\tilde{\mathbf{o}}_j)_j$ an orthonormal basis of $\tilde{\mathbf{O}}_{t_p}(\bar{g}_{t_p,v_p})$ there exist real numbers $(\hat{\lambda}_j = \hat{\lambda}_j(t,v))_j$ and $C_k>0$ such that we have
    \begin{equation}
        \big\|\pi_{\tilde{\mathbf{O}}(g^D)^\perp}\mathbf{\Phi}_{g^D}(g^A_{t,v}) \big\|_{r^{-2}_DC^k_{\beta}(g^D)}\leqslant C_k t_p^{\frac{3-\beta}{4}},\label{Controle operateur hat}
    \end{equation}
    and, for any $i_0$, we have
    \begin{equation}
        \Big\langle\mathbf{\Phi}_{g^D}(g^A_{t,v})-t_p \sum_{j} \hat{\lambda}_j \tilde{\mathbf{o}}_{i}, \tilde{\mathbf{o}}_{i_0} \Big\rangle_{L^2(g^D)}\leqslant C_0t_p^{\frac{5}{4}}.\label{Controle proj obst part}
    \end{equation}
\end{prop}
\begin{rem}
    If we did not use the partial Einstein modulo obstructions desingularization $\hat{g}_{S}$, we would only have a control with powers of $t_{\max}$ instead of $t_p$. In particular, we would not be able to later prove that an obstruction holds at \emph{all} of the singular points, but just at one of them.
\end{rem}
\begin{proof}
    Let us again use the following notations along this proof: $g^A := g^A_{t,v}$, $\bar{g}_b = \bar{g}_{b,v_p}$, $\hat{g}_{S}=\hat{g}_{S,t_S,v_S}$ for simplicity.

    On $M_{S}^{16t}$, we have by assumption
    \begin{equation}
        \mathbf{\Phi}_{g^D_{S}}(\hat{g}_{S}) = \mathbf{o}_S\in \tilde{\mathbf{O}}(g^D_{S}),\label{controle on MSo}
    \end{equation}
    and on $N^{16t}$, we have
    \begin{align}
        \mathbf{\Phi}_{t_p\bar{g}_b}(t_p(\bar{g}_{b,v_p} + t_p\hat{h}_S)) = \chi \mathbf{O}_{S} +t_p\sum_j\hat{\lambda}_j\mathbf{o}_j + \mathcal{O}(r_D^2).\label{contrôle on N Phi hat}
    \end{align}
    Now, since at the point $p$, the development of $\hat{g}_{S}$ in local coordinates where the metric is divergence-free gauge with respect to $g_e$ is
    \begin{equation}
        \hat{g}_{S} = g_e + \hat{H}_S + \mathcal{O}(r_o^3),\label{dv hat g}
    \end{equation} and since $t_p(\bar{g}_{b,v_p}+t_p\hat{h}_S)$ has the following development for $t_p^{-\frac{1}{4}}\leqslant r_b\leqslant 2t_p^{-\frac{1}{4}}$
    \begin{equation}
        t_p(\bar{g}_{b,v_p}+t_p\hat{h}_S) = g_e + \hat{H}_S + \mathcal{O}(t_p^2r_b^{-2} +t_pr_b^{-4}),\label{dvp hat gb+th}
    \end{equation}
    with corresponding controls for the derivatives up to order $2$. On the annulus of radii $r_D=t_p^{\frac{1}{4}}$ and $r_D=2t_p^\frac{1}{4}$, recalling that $r_D = r_o = \sqrt{t_p}r_b$, for $k\in \mathbb{N}$, we have 
    \begin{align*}
        r_D^{k}\big|\nabla_{g^D}^k\left(\hat{g}_S - (t_p\bar{g}_{b,v_p}+t_p\hat{h}_S)\right) \big|_{g^D} &= \mathcal{O}(r_D^3 + t_p^2(t_p^{-\frac{1}{2}}r_D)^{-2}+ t_p(t_p^{-\frac{1}{2}}r_D)^{-4})\\
        &= \mathcal{O}(t_p^\frac{3}{4} + t_p^\frac{3}{2} + t_p)\\
        &=\mathcal{O}(t_p^\frac{3}{4})
    \end{align*}
    thanks to \eqref{dv hat g} and \eqref{dvp hat gb+th}. By definition of the gluing, this yields
    \begin{align}
        r_D^{2+k}\big|\nabla^k_{g^D_t}\big(\pi_{\tilde{\mathbf{O}}(g^D_t)}\mathbf{\Phi}_{g^D_t}(g^A_{t,v})\big)\big| = \mathcal{O}\big(t_p^\frac{3}{4}\big),\label{contrôle recollement hat}
    \end{align}
    According to \eqref{controle on MSo}, \eqref{contrôle on N Phi hat} and \eqref{contrôle recollement hat}, we have the estimate \eqref{Controle operateur hat}. 
    
    Finally, we have the control \eqref{Controle proj obst part} thanks to \eqref{dvp hat gb+th}, \eqref{contrôle on N Phi hat} and \eqref{contrôle recollement hat}:
    \begin{equation*}
        t_{p}|\hat{\lambda}_j(t,v)|\leqslant C \Big(t_{p}^{\frac{3-\beta}{4}}+t_{p}^{\frac{2+\beta}{4}}\Big)t_{p}^{\frac{3-\beta}{4}} \leqslant C t_{p}^{\frac{5}{4}}
    \end{equation*}
\end{proof}

\begin{lem}\label{variations of bar lambdai}
    Let $\hat{\lambda}_j(t,v)$ be the real numbers of Proposition \ref{terme quadratique with obst} obtained by extending $\hat{H}_S$, the quadratic terms of $\hat{g}_{S} = \hat{g}_{S,t_S,v_S}$ on $(N,\bar{g}_{b,v_p})$, and let $\hat{\lambda}_j$ be the ones obtained by extending the divergence-free quadratic terms of $g_o$ on $(N,g_{b})$. Then as $t,v\to 0$, we have
    \begin{equation}
        |\hat{\lambda}_j(t,v)-\hat{\lambda}_j|\xrightarrow[(t,v)\to (0,0)]{} 0.\label{est bar lambda}
    \end{equation}
\end{lem}
\begin{proof}
    We have the expression
    \begin{align*}
        \hat{\lambda}_j(t,v) =&\; -\int_{\mathbb{S}^3\slash\Gamma} \big(3\langle \hat{H}_S(t_S,v_{S}), O^4_j(v_p) \rangle_{g_{e}} + O_j^4(v_p)\big(\nabla_e\textup{tr}_e\hat{H}_S(t_S,v_{S}),\partial_{r_e}\big) \big) dv_{\mathbb{S}^3\slash\Gamma}\\
    &+ \int_N \chi\langle \mathbf{O}_{S}(t_S,v_{S}), \mathbf{o}_j(v_p) \rangle dv_b
    \end{align*}
     thanks to Proposition \ref{terme quadratique with obst}. Since we want to show that it converges, as $(t,v)\to 0$, to 
    $$\hat{\lambda}_j:= -\int_{\mathbb{S}^3\slash\Gamma} \Big(3\langle \hat{H}_S, O^4_j \rangle_{g_{e}} +\frac{1}{2} O_j^4\big(\nabla_e \textup{tr}_e\hat{H}_S,\partial_{r_e}\big) \Big)dv_{\mathbb{S}^3\slash\Gamma}.$$
    We therefore just have to show that as $(t,v)\to 0$, we have
    \begin{enumerate}
        \item $\hat{H}_{S}(t_S,v_S)\to \hat{H}_{S}$ in $r_e^2C^1_0(g_e)$,
        \item $|\mathbf{O}_S(t_S,v_S)|_{g_e}\to 0 $, and
        \item $O_j^4(v_p)\to O_j^4$ in $r_e^{-4}C^0_0(g_e)$.
    \end{enumerate}
    
    Thanks to Theorem \ref{fcts inv einst général}, we know that $\hat{g}_{S,t_S,v_S}$ converges smoothly to $g_o$ on compacts of $M_o\backslash S$ as $(t_S,v_S)\to (0,0)$, in particular, we have smooth convergence on small neighborhoods of $p$. Therefore, the quadratic terms of the expansion of $\hat{g}_{S,t_S,v_S}$ converge to those of $g_o$ as $(t_S,v_S)\to (0,0)$. Consequently, the obstruction $\mathbf{O}_S(t_S,v_S) = \bar{P}_e\hat{H}_S(t_S,v_S)+\lambda(t_S,v_S) g_e$ also converges to $\bar{P}_e\hat{H}_S+\lambda g_e=0$.
    
    Similarly, thanks to Corollary \ref{control termes ordre 4 nouveau poids}, the asymptotic term of $\bar{g}_{b,v_p}$ converge to the asymptotic terms of $g_b$ as $v_p\to 0$. 
    \end{proof}
    
    \begin{rem}
        We needed to consider \emph{partial} desingularizations to obtain these controls. 
    \end{rem}
\begin{rem}
    Thanks to the computations of \cite[Proposition 4]{biq3}, it turns out that generically (when the self-dual part of the curvature at $p$, $\mathbf{R}_+$ is of rank $2$ and $\Lambda \neq 0$), the difference $\hat{\lambda}_j(t,v)-\hat{\lambda}_j= \langle v,\mathbf{o}_j\rangle\big(\mathbf{o}_j(\mathbf{R}_+)-\Lambda \mathbf{o}_j\big) + \mathcal{O}(\|v\|^2_{C^{2,\alpha}_{\beta,*}})$ does not vanish.
\end{rem}

\begin{rem}
    If there were non integrable infinitesimal deformations of $g_b$, we a priori could not expect to prove an obstruction result by the above techniques. Indeed, the metric has an expansion $\bar{g}_{b,v_p} = g_b+v_p+w + \mathcal{O}(|v_p|^3)$, where $w$ satisfies:
    $$\bar{Q}^{(2)}_{g_b}(v_p,v_p)+\bar{P}_{g_b}(w) = \pi_{\mathbf{O}(g_b)}\bar{Q}^{(2)}_{g_b}(v_p,v_p) \in \mathbf{O}(g_b),$$ and we potentially have $\pi_{\mathbf{O}(g_b)}\bar{Q}^{(2)}_{g_b}(v_p,v_p) = \mathcal{O}(|v_p|^2)$ if $v_p$ is not integrable. By considering the metric $\frac{1}{t_p}g^A_{t,v}$, we have the following development of $\mathbf{\Phi}_{g_b}$ on $N^{16t}$,
    $$\mathbf{\Phi}_{g_b}\big(\frac{1}{t_p}g^A_{t,v}\big) = \mathbf{\Phi}_{g_b}(\Bar{g}_{b,v_p}) +t_p\bar{P}_b(\hat{h}_S) +  t_pQ^{(2)}_{g_{b}}(v_p,\hat{h}_S) + \mathcal{O}(t_p^2).$$
    Up to the order $t_p^2$ there are three different sources of obstructions: 
    \begin{enumerate}
        \item the projection of $\mathbf{\Phi}_{g_b}(\Bar{g}_{b,v_p})=\mathcal{O}(|v_p|^2)$ on $\mathbf{O}(g_b)$,
        \item the projection of $t_pP_b(\hat{h}_S) = \mathcal{O}(t_p)$ on $\mathbf{O}(g_b)$, and
        \item the projection of $t_pQ^{(2)}_{g_b}(v_p,\hat{h}_S)=\mathcal{O}(t_p|v_p|)$ on $\mathbf{O}(g_b)$.
    \end{enumerate}
    Hence, we can only "see" the obstructions which are $\mathcal{O}(t_p)$ if $\Ric(\bar{g}_{b,v_p})=0$, or if $|v_p|^2\ll t_p$.
\end{rem}

By adapting the end of the proof of \cite[Proposition 3.1]{biq1}, we get the following useful result.
\begin{cor}\label{si einst retour a lobst of biq1}
    Let $H_2$ be a quadratic symmetric $2$-tensor satisfying $d_e\Ric(H_2)=\Lambda g_e$, and let $V\sim r_e^3$ be a homogeneous vector field which satisfies $\delta_e\delta^*_eV = -\delta_e H_2$, and define $\hat{H}_S:= H_2+\delta^*_eV$ which satisfies $\delta_e \hat{H}_S = 0$. Then, we have $$\hat{\lambda}_j = \lambda_j,$$ where
    $$\hat{\lambda}_j:= -\int_{\mathbb{S}^3\slash\Gamma} \Big(3\langle \hat{H}_S, O^4_j \rangle_{g_{e}} +\frac{1}{2} O_j^4\big(\nabla_e \textup{tr}_e\hat{H}_S,\partial_{r_e}\big) \Big)dv_{\mathbb{S}^3\slash\Gamma},$$
    and
    $$\lambda_j:= -\int_{\mathbb{S}^3\slash\Gamma} \big(3\langle H_2, O^4_j \rangle_{g_{e}} + O_j^4\big(B_eH_2,\partial_{r_e}\big)\big) dv_{\mathbb{S}^3\slash\Gamma}$$
\end{cor}

\subsection{Obstruction to the Gromov-Hausdorff desingularization}

We can finally conclude that there are obstructions to the desingularization of some Einstein orbifolds.

\begin{thm}\label{obstruction intégrable tout point}  Let $(M_o,g_o)$ be an Einstein orbifold, and $(M_i,g_i)_i$ a sequence of Einstein manifolds converging to $(M_o,g_o)$ in the Gromov-Hausdorff sense and assume that, at a singular point $p\in M_o$, the possible non-flat blow-up limits are \textit{integrable} Ricci-flat ALE \emph{manifolds} (which implies that there are no tree of singularities forming).

    Then, if we denote $H_2$ the quadratic terms of the development of $g_o$ in geodesic coordinates at $p$, and $(O_j^4)_j$ the $r_b^{-4}$-terms of a basis of $\mathbf{O}(g_b)$, we have:
    \begin{equation}
        \int_{\mathbb{S}^3} \big(3\langle H_2,O_j^4 \rangle+O_j^4(B_eH_2,\partial_{r_e})\big)dv_{\mathbb{S}^3} = 0.\label{equation obst geod}
    \end{equation}
    
\end{thm}
\begin{proof}
Let $(M_o,g_o)$ be an Einstein orbifold, and assume that there exists a sequence of Einstein metrics $(M_i,g_i)_i$ converging to $(M_o,g_o)$ in the Gromov-Hausdorff sense whose blow-ups at the singular point $p$ satisfy the assumptions of the theorem. According to Corollary \ref{mise en Reduced divergence-free Einstein}, for $i$ large enough, there exists a naïve desingularization of $(M_o,g_o)$, $(M,g
^D_{t_i})$, $v_i\in \tilde{\mathbf{O}}(g^D_{t_i})$, $t_i,v_i\to 0$ and a diffeomorphism $\phi_i:M\to M$ such that $\phi_i^*g_i =\Hat{g}_{t_i,v_i}$ is the Einstein modulo obstructions perturbation of $g^D_{t_i}+v_i$ of Theorem \ref{fcts inv einst général}. Let us fix $p$ a singular point of $M_o$ and denote $S$ the rest of the singularities of $M_o$. Assume that no tree of singularities forms at $p$ and denote $(N,g_b)$ the Ricci-flat ALE manifold limit of blow-ups at $p$.
    
    According to Proposition \ref{controle approximation desing part}, there exists an approximation $g^A_{t_i,v_i}$ satisfying
    $$ \Big\|\pi_{\tilde{\mathbf{O}(g^D_{t_i})}}\mathbf{\Phi}_{g^D_{t_i}}(g^A_{t_i,v_i})\Big\|_{r^{-2}_DC^\alpha_{\beta}(g^D_{t_i})}\leqslant C t_{p,i}^{\frac{3-\beta}{4}},$$
    and consequently
    \begin{equation}
        \big\|t_{p,i} \sum_{j} \hat{\lambda}_j(t_i,v_i) \tilde{\mathbf{o}}^{j}_{t_i,v_i}\big\|_{L^2}\leqslant t_{p,i}^\frac{5}{4}\label{est obst suite}
    \end{equation}
    where the $(\tilde{\mathbf{o}}^{j}_{t_i,v_i})_j$  form an $L^2$-orthonormal basis of elements of $\tilde{\mathbf{O}}_{t_p}(\bar{g}_{b,v_{p,i}})$, and 
    \begin{enumerate}
        \item $\|g^A_{t_i,v_i}-g^D_{t_i}\|_{C^{2,\alpha}_{\beta,*}(g^D_{t_i})}\leqslant 2\epsilon$,
        \item $g^A_{t_i,v_i}-(g^D_{t_i}+v_i)$ is $L^2(g^D_{t_i})$-orthogonal to $\tilde{\mathbf{O}}(g^D_{t_i})$.
    \end{enumerate}
    The estimate \eqref{est obst suite} implies that we have
    \begin{equation}
        t_{p,i}|\hat{\lambda}_j(t_i,v_i)|\leqslant C t_{p,i}^{\frac{5}{4}}\ll t_{p,i}.\label{ineg lambda t}
    \end{equation}
    
    Now, we know that $t_i,v_i\to 0$, and according to Lemma \ref{variations of bar lambdai}, this implies that the $\hat{\lambda}_j(t_i,v_i)$ converge to $\hat{\lambda}_j$. Since the $\hat{\lambda}_j$ are only constants depending on the geometry of $(M_o,g_o)$ and that of $(N,g_{b})$, they must necessarily vanish to satisfy the inequality \eqref{ineg lambda t} for $t_i$ arbitrarily small. By coming back to the expression of $\hat{\lambda}_j$ of \eqref{hat lambda j} we find the obstruction. We can finally extend it in geodesic coordinates (for example) to obtain \eqref{equation obst geod} thanks to Corollary \ref{si einst retour a lobst of biq1}.
\end{proof}

\section{Obstructions for known singularity models}\label{obstructions explicites}

The description of the previous section allowed us to find obstructions to the desingularization of Einstein orbifolds by smooth Einstein manifolds. We will now test them on the known examples and first show that the obstruction to the desingularization by gluing-perturbation of an Eguchi-Hanson metric of \cite{biq1} also holds for any Gromov-Hausdorff desingularization of a finite number of singularities by trees of Kähler ALE Ricci-flat orbifolds in  Theorem \ref{obst arbre kahler}. This is conjecturally the only possible way for Einstein metrics to degenerate in a noncollapsed setting. 

In dimension $4$, the $2$-forms decompose into self-dual and anti self-dual $2$-forms which are elements of the eigenspaces of Hodge star operator $*$ (which satisfies $*^2 = \textup{Id}$) respectively associated to the eigenvalues $1$ and $-1$.
Thanks to this direct sum, the symmetric endomorphism on $2$-forms, $\mathbf{R}$ given by the Riemannian curvature decomposes into blocks,
\[
\mathbf{R}=:
  \begin{bmatrix}
     \mathbf{R}^+ & \Ric^0 \\
     \Ric^0 & \mathbf{R}^-
  \end{bmatrix},
\]
 where the $\Ric^0$ is the traceless part of the Ricci curvature, and where $\mathbf{R}^\pm$ are the self-dual and anti self-dual parts of the curvature.

\subsection{Kähler Ricci-flat ALE metrics and obstructions}

The first obstructions to the desingularization of an Einstein orbifold $(M_o,g_o)$ by a Ricci-flat ALE manifold $(N,g_b)$ come from the infinitesimal deformations of $g_b$ decaying as $r_b^{-4}$ at infinity. We will show that for any Kähler Ricci-flat ALE orbifold, there is a common obstruction to the desingularization which was already found in the case of the gluing of an Eguchi-Hanson metric for a particular gluing-perturbation procedure in \cite{biq1}:
$$\det\mathbf{R}^+ =0,$$
at the singular point. We will moreover see that \emph{generically} (see Remark \ref{obstruction générique}), there are additional obstructions corresponding to $$\mathbf{R}^+ = 0.$$

\begin{rem}
 If we glue the Kähler Ricci-flat ALE metrics with the opposite orientation, that is with with a gluing parameter in $O(4)\backslash SO(4)$ the common obstruction becomes $\det\mathbf{R}^- =0$ and therefore in general, since the Einstein equation implies that the curvature is block diagonal ($\Ric^0=0$), the obstruction is $$\det\mathbf{R} =0.$$
\end{rem}

The only known examples of ALE Ricci-flat orbifolds are Kähler. They have been classified and we have a satisfactory parametrization of the moduli space of these quotients of hyperkähler (hence Ricci-flat) ALE metrics.

Let us precise what the deformations of these known Ricci-flat ALE orbifolds are, in order to extend the obstruction to the desingularization by any Kähler Ricci-flat ALE space.

    \begin{prop}\label{obst hyperkhaler}
        Let $\Gamma$ be a finite subgroup of $SU(2)$, $(N,g_b)$ a hyperkähler ALE manifold asymptotic to $\mathbb{R}^4\slash\Gamma$ and $(O_i^4)_i$ a basis of the $r_b^{-4}$-terms of the elements of $\mathbf{O}(g_b)$. 
        
        Then, the conditions $\int_{\mathbb{S}^3}\big(3\langle H_2,O_i^4\rangle_{g_e}+O_i^4(B_eH_2,\partial_{r_e})\big)dv_{\mathbb{S}^3} = 0$ for a quadratic symmetric $2$-tensor $H_2$ satisfying $d_e\Ric H_2 = \Lambda g_e$ imply that $$ \det\mathbf{R}_+(H_2) =0,$$
        where $\mathbf{R}_+(H_2)$ is the common selfdual part of the curvature of metrics with a development $g_e + H_2 + \mathcal{O}(r_e^3)$.
    \end{prop}
    \begin{proof}
     
Denote by $(x,y,z,t)$ the coordinates in an orthonormal basis of $\mathbb{R}^4$, and define a basis of invariant $1$-forms on the sphere $\mathbb{S}^3$, $(\alpha_1,\alpha_2,\alpha_3)$ by,
\begin{align*}
    \alpha_1:=\frac{1}{r_e^2}(xdy-ydx+zdt-tdz)\\
    \alpha_2:=\frac{1}{r_e^2}(xdz-zdx+ydt-tdy)\\
    \alpha_3:=\frac{1}{r_e^2}(xdt-tdx+ydz-zdy).
\end{align*}
    Manifestly, from Theorem \ref{obstruction intégrable tout point} these obstructions do not particularly depend on the Eguchi-Hanson metric, but on the $r_b^{-4}$-terms of the development of its deformations, $O_i^4$, which are, by \cite[(27)]{biq1}:
    \begin{enumerate}
        \item $O^4_1:= 2 \frac{dr_e^2 + r_e^2\alpha_1^2-r_e^2\alpha_2^2-r_e^2\alpha_3^2}{r_e^4}$,
        \item $O^4_2:= \frac{r_e^2\alpha_1.\alpha_2+r_e dr_e.\alpha_3}{r_e^4}$,
        \item $O^4_3:= \frac{r_e^2\alpha_1.\alpha_3-r_e dr_e.\alpha_2}{r_e^4}$.
    \end{enumerate}
    
        Let $\Gamma$ be a finite subgroup of $SU(2)$. Then, according to \cite{kro}, there exists $k_\Gamma\in \mathbb{N}^*$ and $D_\Gamma$, a finite union of vector subspaces of $ \mathbb{R}^{3k_\Gamma}$ of codimension at least $3$ containing $0$ such that the set of smooth hyperkähler metrics asymptotic to $\mathbb{R}^4\slash\Gamma$ is parametrized as $(X_\zeta,g_\zeta)_{\zeta\in\mathbb{R}^{3k_\Gamma}\backslash D_\Gamma}$. Moreover, by \cite[Theorem 2.1]{auv}, for each $\zeta = (\zeta_1,\zeta_2,\zeta_3) \in \mathbb{R}^{3k_\Gamma}\backslash D_\Gamma$, there exists a diffeomorphism $\Phi_{\zeta}$ from a neighborhood of the infinity of $\mathbb{R}^4\slash\Gamma$ to the infinity of $X_\zeta$ such that $\Phi_{\zeta}^*g_\zeta = g_e + h_\zeta +\mathcal{O}(r_e^{-6})$, where,
    \begin{align*}
           h_\zeta =& - \sum_{j,k,l}|\zeta_j|^2 \frac{ dr_e^2 + 
    r_e^2\alpha_j^2-r_e^2\alpha_k^2-r_e^2\alpha_l^2}{r_e^4}\\
           &-\langle \zeta_1, \zeta_2 \rangle \frac{r_e^2\alpha_1.\alpha_2-r_e dr_e.\alpha_3}{r_e^4}\\
           &-\langle \zeta_1, \zeta_3 \rangle \frac{r_e^2\alpha_1.\alpha_3+r_e dr_e.\alpha_2}{r_e^4} \\
           &-\langle \zeta_2, \zeta_3 \rangle \frac{r_e^2\alpha_2.\alpha_3-r_e dr_e.\alpha_1}{r_e^4},
    \end{align*}
    where the first sum is taken on the $(j,k,l)$ satisfying $l\equiv k+1 \equiv j+2 \mod 3$.
        \\
        
        Without loss of generality, we can assume that the first coordinates of $\zeta_1,\zeta_2$ and $\zeta_3 \in \mathbb{R}^{k_\Gamma}$ are $1$, $0$ and $0$. Indeed, there exists $l\in \{1,...,k_\Gamma\}$ such that the $l$-th coordinate of $(\zeta_1,\zeta_2,\zeta_3)$, $(\zeta^l_1,\zeta^l_2,\zeta^l_3)$ does not vanish. Just like for the homothetic deformations of the Eguchi-Hanson metric, thanks to an action of $SO(4)$ on the asymptotics of the metric and a homothetic transformation (which yields an action of $SO(3)$ and a common rescaling on all the $\zeta^l$), we are able to reach another metric $g_{\tilde{\zeta}}$ with $(\tilde{\zeta}^l_1,\tilde{\zeta}^l_2,\tilde{\zeta}^l_3)= (1,0,0)$.
    
        By differentiating the above expression of $h_\zeta$, we see that the infinitesimal variations associated to the variations of $(\tilde{\zeta}^l_1,\tilde{\zeta}^l_2,\tilde{\zeta}^l_3)$ are then asymptotic to $O^4_1$, $O^4_2$ and $O^4_3$. The obstructions they induce by Theorem \ref{obstruction intégrable tout point} are therefore the same as for the Eguchi-Hanson metric, and by \cite[Theorem 4.1]{biq1} they imply the condition $\det\mathbf{R}_+ =0$ which is independent of the above actions of $SO(4)$ and scaling.
    \end{proof}
    
    \begin{rem}\label{obstruction générique}
        The case of the Eguchi-Hanson metric, or when the $\zeta^l$ are parallel, is actually the least obstructed case, and by the formula \cite[(38)]{biq1} the obstruction condition is $\mathbf{R}_+ =0$ for $\zeta$ generic when $k_\Gamma>1$. 
    \end{rem}
    We find the same obstruction for Kähler Ricci-flat ALE orbifolds which are all asymptotic to $\mathbb{R}^4\slash\Gamma\sim\mathbb{C}^2\slash\Gamma$ for a group $\Gamma =\frac{1}{dn^2}(1,dnm-1)\subset U(2)$, that is the cyclic group generated by
\[
  \begin{bmatrix}
     e^\frac{i2\pi}{dn^2} & 0 \\
     0 & e^\frac{i2\pi (dnm-1)}{dn^2}
  \end{bmatrix},
\]
where $d\geqslant 1$, $n\geqslant 2$ and, $n$ and $m$ are mutually prime.
    \begin{cor}\label{obst kahler}
        Let $\Gamma$ be a group $\frac{1}{dn^2}(1,dnm-1)\subset U(2)$ for $d\geqslant 1$, $n\geqslant 2$ and $n$ and $m$ mutually prime, and let $(N,g_b)$ be a Kähler Ricci-flat ALE metric asymptotic to $\mathbb{R}^4\slash \Gamma$, and $(O_i^4)_i$ a basis of the $r_b^{-4}$-terms of the elements of $\mathbf{O}(g_b)$. 
        
        Then, for a quadratic symmetric $2$-tensor $H_2$ such that $d_e\Ric H_2 = \Lambda g_e$, for $i\in\{1,2,3\}$ we have the conditions
        $$\int_{\mathbb{S}^3}\big(3\langle H_2,O_i^4\rangle_{g_e}+O_i^4(B_eH_2,\partial_{r_e})\big)dv_{\mathbb{S}^3}= 0$$
        which imply that $$ \det\mathbf{R}_+(H_2)= 0 .$$
    \end{cor}
    \begin{proof}
    Let $\frac{1}{dn^2}(1,dnm-1)\subset U(2)$ be a finite subgroup of $U(2)$, and $(N,g_b)$ a non flat Kähler Ricci-flat ALE orbifold asymptotic to $\mathbb{R}^4\slash \Gamma$. According to \cite{suv}, $(\tilde{N},\tilde{g}_b)$ the universal cover of $(N,g_b)$ is a hyperkähler orbifold asymptotic to $\tilde{\Gamma} = \frac{1}{dn}(1,-1)\subset \frac{1}{dn^2}(1,dnm-1)$. Let $H_2$ be a quadratic symmetric $2$-tensor on $\mathbb{R}^4\slash\Gamma$, we can also lift it to $\mathbb{R}^4\slash\tilde{\Gamma}$ as $\tilde{H}_2$. 
    
    Let us come back to the origin of the obstruction in Proposition \ref{meilleure approx and obst}, and more precisely the existence of a symmetric $2$-tensor $h_2$ asymptotic to $H_2$ such that $d_{g_b}\Ric(h_2)=\Lambda g_b$. If such a symmetric $2$-tensor exists on $(N,g_b)$, we can lift it as a symmetric $2$-tensor $\tilde{h}_2$ on $(\tilde{N},\tilde{g}_b)$ asymptotic to $\tilde{H}_2$ and satisfying $d_{\tilde{g}_b}\Ric(\tilde{h}_2)=\Lambda \tilde{g}_b$ which implies, again according to Proposition \ref{meilleure approx and obst} and to Proposition \ref{obst hyperkhaler}, the condition $ \det\mathbf{R}_+(\tilde{H}_2)= 0$, and finally $ \det\mathbf{R}_+(H_2)= 0$.
    \end{proof}

\begin{rem}
    The above proof (or any proof in this article really) does not use the Kähler nature of the studied metric. It would also apply to any Ricci-flat ALE metric whose order $r_e^{-4}$ terms do not vanish \cite[Proposition 2.5]{bh}. It is however not clear if these terms can vanish (in well chosen coordinates) on a non flat Ricci-flat ALE metric.
\end{rem}

\subsection{Trees of Kähler Ricci-flat ALE orbifolds and obstructions}\label{arbre kahleriens}

Let us now treat the case of trees of ALE Kähler Ricci-flat orbifolds.

\subsubsection{Uniform controls of gravitational instantons}
Let us use the notations of the proof of Proposition \ref{obst hyperkhaler}. Let $\Gamma_1\in SU(2)$ and $\zeta_1\in D_{\Gamma_1}$ fixed such that the ALE orbifold $(X_{\zeta_1},g_{\zeta_1})$ has a singularity $\mathbb{R}^4\slash\Gamma$ for $\Gamma\subset SU(2)$. For $v\in \mathbf{O}(g_{\zeta_1})$ which we will choose small, let us then denote $\bar{g}_{\zeta_1,v}$ the Ricci-flat (and even hyperkähler) of $g_{\zeta_1}$ given by Definition \ref{definition integrable}. The goal of this section is to show that there exists a common scale $\tau>0$ such that for any $\zeta\in \mathbb{R}^{3k_\Gamma}$ with $|\zeta| = 1$, for any $t<\tau$, the naïve gluing of $g_\zeta$ to $\bar{g}_{\zeta_1,v}$ at scale $t$, denoted $g^B_{v,t,\zeta}$ can be perturbed into a Ricci-flat modulo obstructions metric which we will denote $\hat{g}_{v,t,\zeta}$.

The naïve gluing $g^B_{v,t,\zeta}$ is Ricci-flat everywhere except on the annulus where the gluing takes place. We therefore have the control
\begin{equation}
    \|\mathbf{\Phi}_{g^B_{v,t,\zeta}}(g^B_{v,t,\zeta})\|_{r_B^{-2}C^{\alpha}_{\beta}(g^B_{v,t,\zeta})}\leqslant Ct^\frac{2-\beta}{4}\label{Phi recollement naif kahler}
\end{equation}
where $C>0$ only depends on $\zeta_1$ and $v$ because $g_\zeta$ for $|\zeta|= 1$ is controlled at infinity in the ALE coordinates of \cite[Corollaire 3.14]{kro}.

Going back to the proof of Theorem \ref{fcts inv einst général} by inverse function theorem, in order to perturb $g^B_{v,t,\zeta}$ to a Ricci-flat modulo obstructions metric to to show the existence of $\hat{g}_{v,t,\zeta}$, it would be enough to show that the linearization of $\pi_{\tilde{\mathbf{O}}(g^B_{v,t,\zeta})^\perp}\mathbf{\Phi}_{g^B_{v,t,\zeta}}$, $$\pi_{\tilde{\mathbf{O}}(g^B_{v,t,\zeta})^\perp}P_{g^B_{v,t,\zeta}} : \tilde{\mathbf{O}}(g^B_{v,t,\zeta})^\perp\cap C^{2,\alpha}_{\beta,*}(g^B_{v,t,\zeta}) \to \tilde{\mathbf{O}}(g^B_{v,t,\zeta})^\perp \cap r_B^{-2}C^{\alpha}_{\beta}(g^B_{v,t,\zeta})$$
is uniformly bounded and has a uniformly bounded inverse for $|\zeta|=1$ and $t$ small enough and to control the nonlinear terms of $\mathbf{\Phi}_{g^B_{v,t,\zeta}}$.
Fixing $\zeta_1$ and $v$, It would therefore be enough thanks to Proposition \ref{inversion with cste} to show that $$P_{g_\zeta} : \mathbf{O}(g_\zeta)^\perp\cap C^{2,\alpha}_{\beta,*}(g_\zeta) \to \mathbf{O}(g_\zeta)^\perp \cap r_\zeta^{-2}C^{\alpha}_{\beta}(g_\zeta)$$
is invertible with bounded inverse. This is however not the case when $\zeta\to D_\Gamma$, that is when $g_\zeta$ degenerates to an orbifold. In this situation, we replace the norms with respect to $g_\zeta$ by norms with respect to a naïve desingularization $g^B_\zeta$ (with additional decay in the neck regions) close to $g_\zeta$ in order to keep a uniform gluing scale.

\begin{prop}\label{controle uniforme g zeta}
    For any $\Gamma\subset SU(2)$ and $\epsilon>0$, there exists $C = C(\Gamma,\epsilon)>1$ such that for any $\zeta\in \mathbb{R}^{3k_\Gamma}$ with $|\zeta|=1$,
    there exists a naïve desingularization (partial if $\zeta\in D_\Gamma$ and total if $\zeta\notin D_\Gamma$, see Definition \ref{def naive desing}) $g^B_\zeta$ of an orbifold $g_{\zeta_o}$ with $\zeta_o\in \mathbb{R}^{3k_\Gamma}$ and $|\zeta_o|= 1$ by hyperkähler orbifolds $g_{\zeta_j}$ for $|\zeta_j|= 1$ and $\zeta_j\in \mathbb{R}^{k_{\Gamma_j}}$ with $|\Gamma_j|<|\Gamma|$, for which, denoting
    $$ R(g^B_\zeta) := \sup_{h\in \tilde{\mathbf{O}}(g^B)^\perp\cap C^{2,\alpha}_{\beta,*}(g^B)} \frac{\|h\|_{C^{2,\alpha}_{\beta,*}(g^B)}}{ \|P_{g^B}h\|_{r_B^{-2}C^\alpha_\beta(g^B)}} + \sup_{h\in C^{2,\alpha}_{\beta,*}(g^B)} \frac{\|P_{g^B}h\|_{r_B^{-2}C^\alpha_\beta(g^B)}}{\|h\|_{C^{2,\alpha}_{\beta,*}(g^B)}} ,$$
    we have
    \begin{enumerate}
        \item $\|g^B_\zeta - g_\zeta\|_{C^{2,\alpha}_{\beta,*}(g^B_\zeta)}<\frac{\epsilon}{R(g^B_\zeta)}$, 
        \item $R(g^B_\zeta) \leqslant C$.
    \end{enumerate}
\end{prop}
\begin{rem}
    The norms in the above statement are the naïve desingularizations norms (Definition \ref{def naive desing}) with respect to $g^B_\zeta$. They are needed in order to obtain uniform estimates.
\end{rem}
\begin{rem}
    What is crucial in this statement is the fact that the constant $C$ is independent of the metric $g_\zeta$ and in particular allows $\zeta$ to approach $ D_\Gamma$ with uniform controls. The constant $\epsilon$ essentially decides from which proximity we replace the metric $g_\zeta$ by a naïve desingularization. Indeed, the trivial naïve desingularization $g^B_\zeta = g_\zeta$ always satisfies the first property of the conclusion.
\end{rem}
\begin{proof}
    Let $\epsilon >0 $ and let us show the result by induction on the order $k$ of the group at infinity. It holds for $\mathbb{Z}_2$ of order $2$ because all the metrics $g_\zeta$ for $\zeta\in\mathbb{R}^3$ are isometric for $|\zeta| =1$, and we can take $g^B_\zeta = g_\zeta$.
    \\
    
    Assume now that the conclusion is satisfied for any group of order less than or equal to $k-1\geqslant 2$ and consider  $\Gamma$ of order $k$.
    
    Let us again work by induction, this time on the value of the square of the $L^2$-norm of the curvature of the ALE orbifolds asymptotic to $\mathbb{R}^4\slash\Gamma$. This quantity is proportional to the dimension of the (orbifold) $L^2$ cohomology in degree $2$ for these gravitational instantons.
    \\
    
    Consider, the orbifolds $g_{\zeta}$, $\zeta\in \mathbb{R}^{3k_\Gamma}$, $|\zeta|= 1$ which have the smallest energy. Assume by contradiction that there exists a sequence of orbifolds $g_{\zeta_i}$, $|\zeta_i|= 1$ with the minimal energy such that for any naïve desingularization (Definition \ref{def naive desing}) $g^B$ of a hyperkähler orbifold asymptotic to $\mathbb{R}^4\slash\Gamma$, we have
    \begin{enumerate}
        \item either, $\|g_{\zeta_i}-g^B\|_{C^{2,\alpha}_{\beta,*}(g^B)}\geqslant \frac{\epsilon}{R(g^B)}$,
        \item either, $R(g^B)>i$.
    \end{enumerate} 
    In the coordinates of \cite[Proposition 3.14]{kro}, all of the metrics $g_{\zeta_i}$ admit coordinates of order $4$ with a uniform control in $r_{\zeta_i}^{-4}C^3(g_{\zeta_i})$ in a uniform neighborhood of infinity, where $r_{\zeta_i}$ is the function of Definition \ref{rb}. Therefore, there can only be $L^2$-concentration of the curvature in a compact of diameter uniformly bounded and volume uniformly bounded below by  Bishop-Gromov inequality. By minimality of the $L^2$-norm of the curvature, no tree of singularities can form with $|\zeta_i|= 1$ because the limit orbifold would then have a smaller $L^2$-norm of the curvature. Therefore, there exists  $g_{\zeta_\infty}$ such that the $g_{\zeta_i}$ converge smoothly (considering local covering at the singular points) to $g_{\zeta_\infty}$. Since the asymptotic terms converge by \cite[Proposition 3.14]{kro}, we have $\|g_{\zeta_i}-g_{\zeta_\infty}\|_{C^{2,\alpha}_{\beta,*}(g_{\zeta_\infty})}\to 0$ when $i\to \infty$ where we have $\|.\|_{C^{2,\alpha}_{\beta,*}(g_{\zeta_\infty})} \approx \|(1+r_{\zeta_\infty})^\beta.\|_{C^{2,\alpha}_0(g_{\zeta_\infty})}$
    here since we use local coverings at the singularities instead of weights at the singularities of $g_{\zeta_\infty}$. The trivial naïve desingularization $g^B =g_{\zeta_\infty}$ therefore contradicts the assumptions. Indeed, the convergence happens in $C^{2,\alpha}_{\beta,*}(g_{\zeta_\infty})$ and since $P_{g_{\zeta_\infty}}$ is elliptic, there exists $C>0$ such that for any $h\perp \mathbf{O}(g_{\zeta_\infty})$, we have
    $$ \|h\|_{C^{2,\alpha}_{\beta,*}(g_{\zeta_\infty})}\leqslant C\|P_{g_{\zeta_\infty}}h\|_{r_{\zeta_\infty}^{-2}C^{\alpha}_{\beta}(g_{\zeta_\infty})}.$$
    
    Consider then a higher value for the $L^2$-norm of the curvature $E>0$ and assume that our property holds for the hyperkähler orbifolds asymptotic to $\mathbb{R}^4\slash\Gamma$ whose $L^2$-norm of the curvature is strictly less than $E$. Let us assume by contradiction that there exists a sequence $\zeta_i\in \mathbb{R}^{3k_\Gamma}$ $|\zeta_i|=1$ satisfying $\|\Rm_{g_{\zeta_i}}\|_{L^2(g_{\zeta_i})}=E$ for which for any naïve desingularization $g^B$ of a hyperkähler orbifold asymptotic to $\mathbb{R}^4\slash\Gamma$, we have
    \begin{itemize}
        \item either, $\|g_{\zeta_i}-g^B\|_{C^{2,\alpha}_{\beta,*}(g^B)}\geqslant\frac{\epsilon}{R(g^B)}$,
        \item or, $R(g^B)>i$.
    \end{itemize} 
    Like in the above argument for the minimal energy, if no singularity was forming, up to taking a subsequence, we would have a $C^{2,\alpha}_{\beta,*}(g_{\zeta_\infty})$-convergence to a orbifold $g_{\zeta_\infty}$ with  $R(g_{\zeta_\infty})$ bounded. A tree of singularities is therefore forming. More precisely, up to taking a subsequence, there exists a subsequence naïve desingularizations $g^{B'}_{\zeta_i}$ of a hyperkähler orbifold asymptotic to $\mathbb{R}^4\slash\Gamma$, composed of orbifolds $g_{\zeta_o}$ with $|\zeta_o|= 0$ asymptotic to $\mathbb{R}^4\slash\Gamma$ and $g_{\zeta_j}$ with $|\zeta_j|=1$ (up to changing the scale of the gluing, we can always assume that since $g_{t\zeta}$ is isometric to $tg_\zeta$), $\zeta_j\in\mathbb{R}^{3k_{\Gamma_j}}$, at scales $t_{i,j}$, for which we have $\|g_{\zeta_i}-g^{B'}_{\zeta_i}\|_{C^{2,\alpha}_{\beta,*}}\to 0$ as $i\to \infty$. Moreover, all of the metrics $g_{\zeta_o}$ and $g_{\zeta_j}$ are hyperkähler according to \cite{ban}. The $L^2$-norm of the curvature of $g_{\zeta_o}$ is strictly inferior to $E$ since some of the total $L^2$-norm is lost in the singularities by \cite{and,bkn}, and we have $|\Gamma_j|<|\Gamma|=k $ by Bishop-Gromov inequality. Up to replacing $g_{\zeta_o}$ and the $g_{\zeta_j}$ by the naïve desingularizations $g^B_{\zeta_o}$ and $g^B_{\zeta_j}$ of the previous steps of our inductions, we obtain a naïve desingularization $g^B_{\zeta_i}$ for which we uniformly (depending on the constants of the previous steps of the induction only) control the operator $P$. Thanks to Lemma \ref{Contrôle P espace modèles}, we obtain a uniform control on the inverse of the operator $P_{g^B_{\zeta_i}}$.
    This contradicts the initial assumptions and proves the statement.
    \end{proof}
    
    The following lemma lets us approximate the kernel of the operator $P_{g_\zeta}$ thanks to the approximate kernel $\tilde{\mathbf{O}}(g^B_\zeta)$.
\begin{lem}\label{approx kernel} 
Let $P$ and $P'$ be two operators between Banach spaces $X$ and $Y$ for which there exists $C>0$, $\frac{1}{100 C}>\epsilon>0$, a finite dimensional space $K'\subset X$ and $S'$ a complement of $K'$ in $X$ such that we have:
    \begin{enumerate}
        \item for any $x\in X$, $$\|(P-P')x\|_Y\leqslant \epsilon\|x\|_X,$$
        \item for any $x \in S'$, $$\|x\|_X\leqslant C\|P'x\|_Y,$$
        \item for any $x\in K'$, 
        $$\|P'x\|_Y\leqslant \epsilon\|x\|_X,$$
        \item and $dim(\ker P) = dim (K')$.
    \end{enumerate}
    Then, for any $k\in\ker P$, there exists $k'\in K'$ such that
    \begin{equation}
        \|k-k'\|_X\leqslant \frac{2C\epsilon}{1-C\epsilon} \|k'\|_X.\label{approx elt kernel}
    \end{equation}
\end{lem}
\begin{proof}
    Let $k\in \ker P$, and consider its decomposition $k = k'+s'$ in the direct sum $X = K' \oplus S'$. Thanks to the first hypothesis, we have
    \begin{equation}
        \|P'k\|_Y\leqslant \epsilon\|k\|_X,\label{P'k leq epsilon k}
    \end{equation}
    thanks to the second,
    \begin{equation}
        \|s'\|_X\leqslant C\|P's'\|_Y\label{P's' }
    \end{equation}
    and thanks to the third one, we have
    \begin{equation}
        \|P'k'\|_Y\leqslant \epsilon\|k'\|_X.\label{P'k' }
    \end{equation}
    Putting \eqref{P'k leq epsilon k}, \eqref{P's' } and \eqref{P'k' } together, we find
    \begin{equation}
        \epsilon\|k\|_X\geqslant \|P's'\|_Y-\|P'k'\|_Y\geqslant \frac{\|s'\|_X}{C}-\epsilon \|k'\|_X,
    \end{equation}
    hence, since $k = k' + s'$,
    $$ \|s'\|_X\leqslant C\epsilon \big(\|k\|_X+\|k'\|_X\big)\leqslant 2C\epsilon \|k'\|_X+C\epsilon\|s'\|_X, $$
    and finally
    $$\|k-k'\|_X=\|s'\|_X\leqslant \frac{2C\epsilon}{1-C\epsilon} \|k'\|_X.$$
\end{proof}

\begin{rem}
    With the three first assumptions of Lemma \ref{approx kernel}, we still have $$ dim \ker P \leqslant dim K', $$
    and any element of $\ker P$ is close to an element of  $K'$ (its projection on $K'$ parallel to $S'$) in the sens of \eqref{approx elt kernel}.
\end{rem}

We conclude from Lemma \ref{approx kernel} and the estimates 1 and 2 of Proposition \ref{controle uniforme g zeta} that for any $\Gamma\subset SU(2)$ and any $\epsilon>0$ small enough depending on $\Gamma$ only, for any $\zeta\in \mathbb{R}^{3k_\Gamma}$ and for the metric $g^B_\zeta$ obtained by Proposition \ref{controle uniforme g zeta}, we have the following control for $C= C(\Gamma,\epsilon)>0$ : for any
    $h\perp \mathbf{O}(g_\zeta)$,
    \begin{equation}
        \frac{1}{C}\|h\|_{C^{2,\alpha}_{\beta,*}(g^B_\zeta)}<\|P_{g_\zeta}h\|_{r_B^{-2}C^{\alpha}_{\beta}(g^B_\zeta)}<C\|h\|_{C^{2,\alpha}_{\beta,*}(g^B_\zeta)}.\label{contrôle g zeta}
    \end{equation}

We therefore conclude that we can define Ricci-flat modulo obstructions perturbations of our naïve desingularizations.

\begin{cor}\label{desing mod obst ALE}
    Let $(X_{\zeta_1},g_{\zeta_1})$ be a Ricci-flat orbifold for $\zeta_1\in D_{\Gamma_1}$ having a singularity $\mathbb{R}^4\slash\Gamma$, $\Gamma\subset SU(2)$. Then, there exists $\epsilon>0$ and $\tau>0$ such that for any $0<t<\tau$, $\zeta\in \mathbb{R}^{3k_{\Gamma}}$ with $|\zeta| = 1$ and $v\in \mathbf{O}(g_{\zeta_1})$ with $\|v\|_{C^{2,\alpha}_{\beta,*}(g_{\zeta_1})}<\tau$, if we denote 
    \begin{itemize}
        \item $g^{BB}_{v,t,\zeta}$ the naïve gluing of the metric $(X_\zeta,g^B_\zeta)$, of Proposition \ref{controle uniforme g zeta} for the constant $\epsilon$, at the singularity $\mathbb{R}^4\slash\Gamma$ of $\bar{g}_{\zeta_1,v}$ (defined above), and
        \item $g^B_{v,t,\zeta}$ the naïve gluing of $(X_\zeta,g_\zeta)$ at the scale $t$,
    \end{itemize}
    there exists a unique metric $\hat{g}_{v,t,\zeta}$ satisfying
    \begin{itemize}
        \item $g^B_{v,t,\zeta}-\hat{g}_{v,t,\zeta}\perp \Tilde{\mathbf{O}}(g^B_{v,t,\zeta})$,
        \item $\|g^B_{v,t,\zeta}-\hat{g}_{v,t,\zeta}\|_{C^{2,\alpha}_{\beta,*}(g^{BB}_{v,t,\zeta})}\leqslant 2\tau$ (notice the naïve desingularization norm $g^{\mathbf{BB}}_{v,t,\zeta}$), and
        \item $\mathbf{\Phi}_{g^B_{v,t,\zeta}}(\hat{g}_{v,t,\zeta})\in\Tilde{\mathbf{O}}(g^B_{v,t,\zeta})$.
    \end{itemize}
Moreover, the metric  $\hat{g}_{v,t,\zeta}$ depends smoothly on $v$, $t$ and $\zeta$. 
\end{cor}
\begin{proof}
    As discussed at the beginning of this section, for the naïve gluing $g^B_{v,t,\zeta}$, we have a control
$$\|\mathbf{\Phi}_{g^B_{v,t,\zeta}}(g^B_{v,t,\zeta})\|_{r_B^{-2}C^{\alpha}_{\beta}(g^{BB}_{v,t,\zeta})}\leqslant Ct^\frac{2-\beta}{4}$$
where $C>0$ only depends on $\zeta_1$. Now, using the control \eqref{contrôle g zeta} to replace the controls of Lemma \ref{Contrôle P espace modèles}, we adapt the proof of Proposition \ref{inversion with cste} to show that for $\tau$ small enough depending on $\Gamma$ and $\zeta_1$ alone, the operator
$$\pi_{\tilde{\mathbf{O}}(g^B_{v,t,\zeta})^\perp}P_{g^B_{v,t,\zeta}} : \tilde{\mathbf{O}}(g^B_{v,t,\zeta})^\perp\cap C^{2,\alpha}_{\beta,*}(g^{BB}_{v,t,\zeta}) \to \tilde{\mathbf{O}}(g^{B}_{v,t,\zeta})^\perp \cap r_B^{-2}C^{\alpha}_{\beta}(g^{BB}_{v,t,\zeta}),$$
is uniformly bounded (independently of $\zeta$ and of $t$ and $v$ small enough) and invertible with a uniformly bounded inverse. Notice the norms $g^{BB}_{v,t,\zeta}$ again. The proof of Theorem \ref{fcts inv einst général} by inverse function theorem extends here, and there exists a unique solution $\hat{g}_{v,t,\zeta}$ to the equations of the statement.

Like in Corollary \ref{fcts implicites hat g v}, the smooth dependence of the metric is a consequence of the implicit function theorem applied to $$\Psi : (v,t,\zeta,h)\mapsto \pi_{\Tilde{\mathbf{O}}(g^B_{v,t,\zeta})^\perp}\mathbf{\Phi}_{g^B_{v,t,\zeta}}(g^B_{v,t,\zeta}+h)$$
for $h\perp\Tilde{\mathbf{O}}(g^B_{v,t,\zeta})$.
\end{proof}

\subsubsection{Gluing-perturbation of gravitational instantons}

Let us first make sure that all of the gluing gauges given by the isometries of $\mathbb{R}^4\slash\Gamma$ (see Remark \ref{jauge recollement}) are equivalent to gluing a gravitational instanton but with a different parameter $\zeta$.

\begin{lem}
    Let $(M_o,g_o)$ be an orbifold (compact or ALE) with a singularity $\mathbb{R}^4\slash\Gamma$ at $p$ with $\Gamma \subset SU(2)$, and let
$\zeta \in \mathbb{R}^{3k_\Gamma}$. Let moreover $\psi$ be an isometry of  $\mathbb{R}^4\slash\Gamma$ preserving the orientation.

Then, there exists $\zeta'\in \mathbb{R}^{3k_\Gamma}$ such that the naïve gluing of $(X_{\zeta},g_{\zeta})$ at $p$ with the ALE coordinates of \cite[Proposition 3.14]{kro} composed with $\psi$ when identifying with the coordinates of $(M_o,g_o)$ (see Remark \ref{jauge recollement})
is isometric to the gluing of $(X_{\zeta'},g_{\zeta'})$ to $(M_o,g_o)$ at $p$
without composing with an isometry.
\end{lem}
\begin{proof}
    Let ${\Gamma}\subset SU(2)$ be a finite subgroup. Having $\psi\in \textup{Isom}(\mathbb{R}^4\slash{\Gamma})$ implies that $\psi$ is in the normalizer of ${\Gamma}$ in $SO(4)$ and therefore that $\psi\in \phi(\mathbb{S}^3\times N_{\Gamma})$, where $N_{\Gamma}$ is the normalizer of ${\Gamma}$ in $SU(2)$ (see \cite{mcc} for an explicit expression) and where $\phi: \mathbb{S}^3\times \mathbb{S}^3\to SO(4)$ is the double covering of $SO(4)$ where the first $\mathbb{S}^3$ is the left multiplication by a unit quaternion and the second a right multiplication. In this identification, we have $SU(2) = \phi(\{1\}\times \mathbb{S}^3)$ and we will denote  $\overline{SU(2)} = \phi( \mathbb{S}^3\times \{1\})$.
    \\
    
    Let us now study the action of the normalizer of $\Gamma$, which is $\phi(\mathbb{S}^3\times N_{\Gamma})$, on a metric of \cite{kro} in the coordinates of \cite[Proposition 3.14]{kro}. Let us come back from their construction in \cite{kro} starting with $$P:=\mathbb{C}^2\otimes End(R),$$
    where $\mathbb{C}^2$ is the standard representation of $SU(2)$ and where $R$ is its regular representation $SU(2)$. Denote $P^\Gamma$ the set of $\Gamma$-invariant elements of $P$, $F$ the set of elements commuting with $\Gamma$ seen as a subset of the unitary transformations of $R$, and finally denote $T$ the subgroup of complex numbers of unit norm acting trivially. The gravitational instantons asymptotic to $\mathbb{C}^2\slash\Gamma$ are the hyperkähler quotients of $P^\Gamma$ by $F/T$. 
    
    By definition, the normalizer $\phi(\mathbb{S}^3\times N_{\Gamma})$ acts on $P^\Gamma$ and commutes to the action of $F$. It consequently acts on the set of solutions of the moment map $\mu : P^\Gamma\to (\mathfrak{f}/\mathfrak{t})^*\otimes \mathbb{R}^3$ where $\mathfrak{f}$ and $\mathfrak{t}$ are the Lie algebras of $F$ and $T$ respectively. More precisely, denoting $I_k$, $k\in\{1,2,3\}$ the 3 complex structures of $P^\Gamma$ given by the identification of $\mathbb{C}^2$ with the quaternions, for all $\xi\in \mathfrak{f}/\mathfrak{t}$, the three coordinates of $\mu$ satisfy the equations
    $$grad(\mu_k.\xi) = I_k(V_\xi),$$
    where $V_\xi$ is the vector field on $P^\Gamma$ generated by the action of $\xi$. The $SU(2)$ in which $\Gamma$ and $N_\Gamma$ commutes to the three complex structures since it is identified to the \emph{right} multiplication by a unit-norm quaternion. We conclude that $N_\Gamma$ acts by isometry on the hyperkähler metric $\mu^{-1}(\zeta)/(F/T)$ for all $\zeta\in (\mathfrak{f}/\mathfrak{t})^*\otimes \mathbb{R}^3$. The part $\overline{SU(2)}$ of the normalizer, acts by rotation on the three complex structures. And more precisely, an element $n_-\in \overline{SU(2)}$ sends $\mu^{-1}(\zeta)/(F/T)$ to $\mu^{-1}(Ad(n_-)\zeta)/(F/T)$ where $Ad(n_-)$ is the standard action of $SU(2)/\pm\approx SO(3)$ on the factor $\mathbb{R}^3$ of $(\mathfrak{f}/\mathfrak{t})^*\otimes \mathbb{R}^3$.
    \\
    
    To conclude there remains to ensure that this action of the normalizer is the standard action on the asymptotic cone $\mathbb{R}^4\slash\Gamma$ in the coordinates of \cite[Proposition 3.14]{kro}. To prove this, we use the identification $\mathbb{R}^4\slash\Gamma\approx \mu^{-1}(0)/(F/T)$ of \cite[Corollary 3.2]{kro}. The correspondence between the infinities of $\mu^{-1}(\zeta)/(F/T)$ and $\mu^{-1}(0)/(F/T)$ for the coordinates of \cite[Proposition 3.14]{kro} given by \cite[(3.13)]{kro} lets us conclude that the action of the normalizer is indeed the standard action on the asymptotic cone $\mathbb{R}^4\slash\Gamma$.
    
    Therefore, the gluing $(X_\zeta,g_\zeta)\approx \mu^{-1}(\zeta)/(F/T)$ for $\zeta\in \mathbb{R}^{3k_{\Gamma}}\approx (\mathfrak{f}/\mathfrak{t})^*\otimes \mathbb{R}^3$ composed with the isometry $\psi=\phi(\psi_-,\psi_+)$ of $\mathbb{R}^4\slash\Gamma$ is isometric to the gluing of $(X_{\zeta'},g_{\zeta'})$ with $\zeta' = Ad(\psi_-) \zeta$.
\end{proof}

Let us then remark that any ALE hyperkähler orbifold can be desingularized by a sequence of ALE hyperkähler manifolds.

\begin{lem}[{\cite{ban}}]\label{hk nonvide}
    Let $(X_{\zeta_1},g_{\zeta_1})$ be an ALE hyperkähler orbifold with a singularity $\mathbb{R}^4\slash\Gamma$. Then, for any $t>0$, there exists $\zeta_0^t\in \mathbb{R}^{3k_\Gamma}\backslash D_{\Gamma}$ and $v_t\in \mathbf{O}(g_{\zeta_1})$ satisfying $\lim_{t\to 0}v_t = 0$ and $ \lim_{t\to 0} X_{\zeta_0^t} = X_{\zeta_0}\in \mathbb{R}^{3k_\Gamma}\backslash D_{\Gamma}$, such that the Einstein modulo obstructions desingularization $(X_{\zeta_1}\#X_{\zeta_0},\hat{g}_{v_t,t,\zeta_0^t})$ of Corollary \ref{desing mod obst ALE} is hyperkähler.
\end{lem}
\begin{proof}
    The result of \cite[Theorem 4]{ban} shows that for $t$ small enough, there exists $\zeta_0$ for which there is a hyperkähler metric $ g_t $ satisfying 
    $$ \|g_t-g^B_{0,t,\zeta_0}\|_{(1+r_B)^{-\beta}C^{2,\alpha}_0(g^B_{0,t,\zeta_0})}\leqslant \gamma(t), $$
    where $\lim_{t\to 0}\gamma(t)=0$. We can then apply our construction of coordinates of \cite{ozu1} to $(X_{\zeta_1}\#X_{\zeta_0},g_t)$ in order to obtain a diffeomorphism $\phi : X_{\zeta_1}\#X_{\zeta_0}\to X_{\zeta_1}\#X_{\zeta_0}$ thanks to which we have the better control
    $$ \|\phi^*g_t-g^B_{0,t,\zeta_0}\|_{C^{2,\alpha}_{\beta,*}(g^B_{0,t,\zeta_0})}\leqslant C\gamma(t). $$
    Thanks to Proposition \ref{Mise en Reduced divergence-free}, there exists a small diffeomorphism $\psi : X_{\zeta_1}\#X_{\zeta_0}\to X_{\zeta_1}\#X_{\zeta_0}$ such that $\tilde{\delta}_{g^B_{0,t,\zeta_0}}\psi^*\phi^*g_t= 0$. The uniqueness of Theorem \ref{fcts inv einst général} and the fact that the infinitesimal deformations of $(X_{\zeta_0},g_{\zeta_o})$ integrate to curves $t\mapsto(X_{\zeta_0(t)},g_{\zeta_0(t)})$ by the classification of \cite{kro,kro2} ensures that there exists $v_t\in \Tilde{\mathbf{O}}(g^B_{0,t,\zeta_0})$ and $\zeta_0^t$ such that we have
    $$\psi^*(\phi^*g_t)=\hat{g}_{v_t,t,\zeta_0^t},$$
    with $\hat{g}_{v_t,t,\zeta_0^t}$ one of the metrics of Theorem \ref{fcts inv einst général}.
\end{proof}

\begin{rem}
    Note that \cite[Theorem 6.2]{hv20} provides such gluings-perturbations in all directions $\zeta\in \mathbb{R}^{3k_\Gamma}\backslash D_\Gamma$. The admissible scales are however not independent on $|\zeta|=1$, and arbitrary trees of singularities are not allowed.  
\end{rem}

\begin{prop}\label{Recollement arbre de'ALE khaleriens}
    There exists $\tau>0$ such that for all 
    $$(v,t,\zeta)\in \big(B_{C^{2,\alpha}_{\beta,*}(g_{\zeta_1})}(0,\tau)\cap \mathbf{O}(g_{\zeta_1})\big)\times (0,\tau)\times (\mathbb{S}^{3k_\Gamma-1}\backslash D_{\Gamma}),$$
    the metric $\hat{g}_{v,t,\zeta}$ of Proposition \ref{desing mod obst ALE} is isometric to a hyperkähler metric of \cite{kro}.
\end{prop}
\begin{proof}
    Let us start by noting that the set $$\big(B_{C^{2,\alpha}_{\beta,*}(g_{\zeta_1})}(0,\tau)\cap \mathbf{O}(g_{\zeta_1})\big)\times (0,\tau)\times (\mathbb{S}^{3k_\Gamma-1}\backslash D_{\Gamma})$$ is connected since $D_{\Gamma}$ is a finite union of spaces of codimension at least $2$ in $\mathbb{S}^{3k_\Gamma-1}$. We therefore just need to prove that the set
    $$E = \big\{(v,t,\zeta)\in \big(B_{C^{2,\alpha}_{\beta,*}(g_{\zeta_1})}(0,\tau)\cap \mathbf{O}(g_{\zeta_1})\big)\times (0,\tau)\times (\mathbb{S}^{3k_\Gamma-1}\backslash D_{\Gamma}),  \Ric(\hat{g}_{v,t,\zeta})=0\big\},$$
    which is non empty by Lemma \ref{hk nonvide} is open and closed. It is therefore isometric to $g_{\xi(v,t,\zeta)}$ for $\xi(v,t,\zeta)\in \mathbb{R}^{3k_{\Gamma_1}}\backslash D_{\Gamma_1}$.
    
    The set $E$ is closed by the continuity of $$(v,t,\zeta)\in \big(B_{C^{2,\alpha}_{\beta,*}(g_{\zeta_1})}(0,\tau)\cap \mathbf{O}(g_{\zeta_1})\big)\times (0,\tau)\times (\mathbb{S}^{3k_\Gamma-1}\backslash D_{\Gamma})\mapsto \Ric(\hat{g}_{v,t,\zeta})$$ proven in Corollary \ref{desing mod obst ALE}.
    
    For the openness, let us assume that for $(v,t,\zeta)$ fixed, we have $\Ric(\hat{g}_{v,t,\zeta}) = 0$. The metric $\hat{g}_{v,t,\zeta}$ is therefore one of the metrics of \cite{kro}. There exists a space of the same dimension as $\tilde{\mathbf{O}}(g^B_{v,t,\zeta})$ of hyperkähler deformations in the neighborhood of $ g_{\xi(v,t,\zeta)} $ and therefore the Ricci-flat modulo obstructions deformations of $\hat{g}_{v,t,\zeta}$ are hyperkähler. Indeed, since the metrics $\bar{g}_{\zeta_1,v}$, $g_{\zeta_2}$ and $ g_{\xi(v,t,\zeta)} $ are hyperkähler, they have a hyperkähler deformation space of dimension three times that of their $L^2$-cohomology in degree $2$, and the dimension of this cohomology is additive for our gluings.
    
    According to Lemma \ref{approx kernel} applied to $P = P_{\hat{g}_{v,t,\zeta}}$, $P' =P_{g^B_{v,t,\zeta}}$, $K' = \tilde{\mathbf{O}}(g^B_{v,t,\zeta})$ and $S' = \tilde{\mathbf{O}}(g^B_{v,t,\zeta})^\perp$, the Ricci-flat deformations of $\hat{g}_{v,t,\zeta}$ are arbitrarily close to elements of $\tilde{\mathbf{O}}(g^B_{v,t,\zeta})$. By Corollary \ref{control termes ordre 4 nouveau poids}, the metrics $g^B_{v',t',\zeta'}$ approximate the small Ricci-flat deformations $\hat{g}_{v',t',\zeta'}$ of $\hat{g}_{v,t,\zeta}$ staying in $B_{C^{2,\alpha}_{\beta,*}(g^B_{v,t,\zeta},2\epsilon)}$, for $\epsilon>0$ the constant of Theorem \ref{fcts inv einst général}. We therefore reach metrics isometric to all the $\hat{g}_{v',t',\zeta'}$ for $(v',t',\zeta')$ in a neighborhood of $(v,t,\zeta)$ in $\big(B_{C^{2,\alpha}_{\beta,*}(g_{\zeta_1})}(0,\tau)\cap \mathbf{O}(g_{\zeta_1})\big)\times (0,\tau)\times (\mathbb{S}^{3k_\Gamma-1}\backslash D_{\Gamma})$. The set $E$ is therefore open.
\end{proof}

\subsubsection{Obstruction for trees of Kähler Ricci-flat ALE orbifolds}

 Let us use the notations of the proof of Proposition \ref{obst hyperkhaler}, and parametrize the set of Kähler Ricci-flat manifolds asymptotic to $\mathbb{R}^4\slash\Gamma$ as $(X_{\zeta},g_\zeta)_{\zeta\in \mathbb{R}^{3k_\Gamma}\backslash D_\Gamma}$ in the following statement.

\begin{lem}\label{obstr o1o2o3 sur arbre}
Let $(X_{\zeta_0},g_{\zeta_0})$ be a Kähler Ricci-flat orbifold asymptotic to $\mathbb{R}^4\slash\Gamma$, and let $\zeta\in \mathbb{R}^{d_\Gamma}\backslash D_\Gamma$ be close to $\zeta_0$. Then, there exists a naïve desingularization $g^B_t$ of $(X_{\zeta_0},g_{\zeta_0})$ by Kähler Ricci-flat ALE orbifolds glued in the same orientation and $v\in \tilde{\mathbf{O}}(g^B_t)$ such that $(X_\zeta,g_\zeta) = (N,\bar{g}_{b,t,v})$ is the (iterated) perturbation of Lemma \ref{Recollement arbre de'ALE khaleriens} of $(N,g^B_t+v)$.
    
    Moreover, there exists $\epsilon>0$ such that for $\zeta$ close enough to $\zeta_0$, there exists a diffeomorphism $\Phi_\zeta$ between neighborhoods of the infinities of $(X_\zeta,g_\zeta)$ and of $\mathbb{R}^4\slash\Gamma$ such that there exists $\mathbf{o}_1(\zeta)$, $\mathbf{o}_2(\zeta)$ and $\mathbf{o}_3(\zeta)$ elements of $\mathbf{O}(g_\zeta)$ satisfying for all $i\in \{1,2,3\}$,
    $$\Phi_\zeta^*\mathbf{o}_i(\zeta) = O_i^4 +\mathcal{O}(r_{B}^{4+\beta}),$$
    where $O_i^4 = \mathcal{O}(r_B^{-4})$ is the homogeneous symmetric $2$-tensor used in the proof of Proposition \ref{obst hyperkhaler}, and with $\|\mathbf{o}_i(\zeta)\|_{L^2(g_{\zeta})}\geqslant \epsilon$.
\end{lem}
\begin{proof}
According to Corollary \ref{control termes ordre 4 nouveau poids}, $(X_\zeta,g_\zeta) = (N,\bar{g}_{b,t})$ is a Kähler Ricci-flat perturbation of a naïve desingularization $(N,g^B_t)$ for some small $t$ depending on $\zeta$. Moreover, the $r_e^{-4}$ terms of $(X_{\zeta},g_\zeta)$ converge to those of $(X_{\zeta_0},g_{\zeta_0})$ as $\zeta\to \zeta_0$ by Corollary \ref{control termes ordre 4 nouveau poids}. By assumption, there exists $l\in \{1,...,k_\Gamma\}$ such that $\zeta_0^l\neq 0$, hence, for $\zeta$ close enough to $\zeta_0$, $\zeta^l \neq 0$ by continuity. Just like in the proof of Proposition \ref{obst hyperkhaler}, this implies that there exists a diffeomorphism $\Phi_\zeta$ between neighborhoods of the infinities of $(X_\zeta,g_\zeta)$ and of $\mathbb{R}^4\slash\Gamma$ and infinitesimal deformations of $g_\zeta$, $\mathbf{o}_1(\zeta)$, $\mathbf{o}_2(\zeta)$ and $\mathbf{o}_3(\zeta)$ such that there exists $C>0$ independent of $\zeta$ for which we have for all $i\in \{1,2,3\}$,
$$\big|\Phi_\zeta^*\mathbf{o}_i(\zeta) - O_i^4\big| \leqslant C r_{B}^{-4-\beta}$$
by Corollary \ref{control termes ordre 4 nouveau poids}. In particular, since $O_i^4\neq 0$, there exists $\epsilon>0$ depending on $C$ and $\beta$, but independent of $\zeta$ such that we have $\|\mathbf{o}_i(\zeta)\|_{L^2(g_\zeta)}\geqslant \epsilon$.
\end{proof}

\begin{lem}\label{Résolution d'equations arbre}
    Let $g^B_t$ be a naïve gluing of Kähler Ricci-flat ALE orbifolds, and $\bar{g}_{b,t}$ its Kähler Ricci-flat pertubation of Lemma \ref{Recollement arbre de'ALE khaleriens}.

    Then, for any symmetric $2$-tensor $w\in r_{B}^{-2}C^\alpha_\beta(g^B_t)$, there exists a unique symmetric $2$-tensor $u\in \mathbf{O}(\bar{g}_{b,t})^{\perp_{\bar{g}_{b,t}}}\cap C^{2,\alpha}_{\beta,*}(g^B_t)$, such that 
    \begin{equation}
        \bar{P}_{\bar{g}_{b,t}} u = \pi_{\mathbf{O}(\bar{g}_{b,t})^\perp}w.
    \end{equation}
    
    We moreover have the following control for $C = C(g^B_t)>0$, 
    $$ \|u\|_{C^{2,\alpha}_{\beta,*}(g^B_t)}\leqslant C\|\pi_{\mathbf{O}(\bar{g}_{b,t})^\perp}w\|_{r_{B}^{-2}C^\alpha_\beta(g^B_t)}. $$
\end{lem}
 \begin{rem}
    The crucial part of this lemma is the fact that the solution is controlled in the tree of singularities norm $C^{2,\alpha}_{\beta,*}(g^B_t)$ which behaves well as $t\to 0$.
 \end{rem}
\begin{proof}
    According to Theorem \ref{fcts inv einst général}, we have
\begin{equation}
        \|\bar{g}_{b,t}-g^B_t\|_{C^{2,\alpha}_{\beta,*}(g^B_t)}\leqslant C t_{\max}^{\frac{2-\beta}{4}},\label{est gbt gD}
\end{equation}
which, combined with the proof of Theorem \ref{fcts inv einst général} implies that, for $t_{\max}$ small enough, the operator $\bar{P}_{\bar{g}_{b,t}}$ is injective on $\tilde{\mathbf{O}}(g^B_t)^\perp\cap C^{2,\alpha}_{\beta,*}(g^B_t)$. 

Moreover, for $0<\beta<1$, its cokernel on $r_D^{-2}C^\alpha_\beta(g^B_t)$ is equal to its kernel on $r_D^{-2}C^\alpha_{-\beta}(g^B_t)$ which is equal to $\mathbf{O}(\bar{g}_{b,t})$. Indeed, for any $g_b$ a Ricci-flat ALE metric, the kernel and the cokernel of $\bar{P}_{g_b}: C^{2,\alpha}_{\beta,*}(g_b)\to r_b^{-2}C^\alpha_\beta(g_b)$ are equal to $\mathbf{O}(g_b)$ because taking the divergence of $\bar{P}_{g_b}(h) = 0$ for $h\in C^{2,\alpha}_{\beta,*}(g_b)$, yields $\delta_{g_b}\delta_{g_b}^*(\delta_{g_b}h) = 0$, and finally $\delta_{g_b}h=0$ by Proposition \ref{controle deltadelta sur les espaces modèles}. By taking the trace of the remaining of the equation, we find that $\nabla^*_{g_b}\nabla_{g_b}(\textup{tr}_{g_b}h)=0$, and since $h$ decays at infinity, $\textup{tr}_{g_b}h = 0$. Finally $P_{g_b}(h)=0$, and we conclude that the kernel of $ \bar{P}_{g_b}: C^{2,\alpha}_{\beta,*}\to r_b^{-2}C^\alpha_\beta$ is $\mathbf{O}(g_b)$, and similarly, its cokernel is also $\mathbf{O}(g_b)$.
\end{proof}

\begin{cor}\label{h2 arbre kahlerien}
    Let $t=(t_1 =1,...,t_k)>0$, and let
    \begin{itemize}
        \item $(N_k,g_{b_k})_k$ be a tree of ALE Kähler Ricci-flat orbifolds desingularizing $\mathbb{R}^4\slash\Gamma$,
        \item $(N,g^B_t)$ the naïve gluing of the $(N_k,g_{b_k})$ at the relative scales $t_k$ to $(N_1,g_{b_1})$, small enough for $k\neq 1$, and
        \item $(N,\bar{g}_{b,t})$ be the Kähler Ricci-flat ALE perturbation of $(N,g^B_t)$ of Lemma \ref{Recollement arbre de'ALE khaleriens}.
    \end{itemize}
       Let us assume that $(N_1,g_{b_1})$ is asymptotic to $\mathbb{R}^4\slash\Gamma$, consider $\hat{H}_S$ a quadratic symmetric $2$-tensor on $\mathbb{R}^4\slash\Gamma$ such that $\bar{P}_e \hat{H}_S +\lambda g_e = \mathbf{O}_S$ for a constant symmetric $2$-tensor $\mathbf{O}_S$.
    
    Then, there exists $C>0$ independent of the $t_k$ and $\chi$, a cut-off function supported in a neighborhood of infinity of $(N,\bar{g}_{b,t})$ independent of the $t_k$, and there exists $\hat{h}_{2}$ a symmetric $2$-tensor on $N$ such that we have
    $$\Bar{P}_{\bar{g}_{b,t}}\hat{h}_{2} + \lambda \bar{g}_{b,t} - \chi \mathbf{O}_S =\sum_{i}\hat{\lambda}_{i}\mathbf{o}_{i} \in \mathbf{O}(\bar{g}_{b,t}),$$
    and $$\|\hat{h}_S-\chi \hat{H}_S\|_{C^{2,\alpha}_{\beta,*}(g^B_t)} \leqslant C \|\hat{H}_S\|_{r^2_eC^{0}(g_e)}.$$
\end{cor}
\begin{proof}
    Let us consider $\hat{H}_S$ a quadratic symmetric $2$-tensor on $\mathbb{R}^4\slash\Gamma$ such that $\bar{P}_e\hat{H}_S + \lambda g_e = \mathbf{O}_S$, and let $\chi$ be a cut-off function on $N_1$ supported in a neighborhood of infinity where $(N_1,g_{b_1})$ has ALE coordinates we will also denote $\chi$ on $N$ the cut-off function extended by $0$ on the deeper ALE orbifolds.
    
    We then have $$\|\Bar{P}_{\bar{g}_{b,t}}(\chi \hat{H}_S) + \lambda \bar{g}_{b,t} -\chi \mathbf{O}_S \|_{r_D^{-2}C^\alpha_\beta(g^B_t)}\leqslant C \|\hat{H}_S\|_{r_e^2C^0_0}.$$
    Indeed, in a neighborhood of infinity where $\chi\equiv 1$, since $\bar{g}_{b,t}-g_e = \mathcal{O}(r_{B}^{-4})$, we have $\Bar{P}_{\bar{g}_{b,t}}(\chi \hat{H}_S) + \lambda \bar{g}_{b,t} = \mathbf{O}_S+ \mathcal{O}(r_{B}^{-4})$, and on the rest of the manifold, we have the expected control by definition of the norm $r_D^{-2}C^\alpha_\beta(g^B_t)$. According to Lemma \ref{Résolution d'equations arbre} applied to $g = \bar{g}_{b,t}$, there exists a unique symmetric $2$-tensor $h'\in C^{2,\alpha}_{\beta,*}(g^B_t) \cap \mathbf{O}(\bar{g}_{b,t})^\perp$, such that we have
    $$\Bar{P}_{\bar{g}_{b,t}}(\chi \hat{H}_S + h') + \lambda \bar{g}_{b,t}- \chi \mathbf{O}_S \in \mathbf{O}(\bar{g}_{b,t}).$$
    Moreover, according to Proposition \ref{terme quadratique with obst}, the element of $\mathbf{O}(\bar{g}_{b,t})$ is explicit. More precisely, consider $(\mathbf{o}_{i})_i$ an orthonormal basis of $\mathbf{O}(\bar{g}_{b,t})$, and thanks to the diffeomorphism $\Phi_t$ of Lemma \ref{obstr o1o2o3 sur arbre}, let us assume that the three first elements are asymptotic to $c_i\Phi_{t,*}O_i^4$ for $c_i>\frac{1}{\epsilon}$. We have
    $$\Bar{P}_{\bar{g}_{b,t}}(\chi \hat{H}_S + h') + \lambda \bar{g}_{b,t}- \chi \mathbf{O}_S = \sum_i \hat{\lambda}_{i}\mathbf{o}_{i}\in \mathbf{O}(\bar{g}_{b,t}),$$
    where, for $i = 1,2,3$,
    $$\hat{\lambda}_i:= -\int_{\mathbb{S}^3\slash\Gamma} \big(3\langle \hat{H}_S, O_i^4 \rangle_{g_{e}} + O_i^4\big(\nabla_e\textup{tr}_e\hat{H}_S,\partial_{r_e}\big) \big) dv_{\mathbb{S}^3\slash\Gamma} + \int_N \chi\langle \mathbf{O}_S, \mathbf{o}_i \rangle_{g_b} dv_b.$$
\end{proof}

Let $(M_o,g_o)$ be an Einstein orbifold and $p$ one of its singular points of singularity $\mathbb{R}^4\slash\Gamma$, $S$ the set of singularities of $M_o\backslash\{p\}$, and let $(N_k,g_{b_k})_k$ be a tree of ALE Kähler Ricci-flat orbifolds desingularizing $\mathbb{R}^4\slash\Gamma$. Let moreover $\Hat{g}_{S}$ be a naïve desingularization modulo obstructions of $(M_o,g_o,S)$ and $\hat{H}_S$ the quadratic terms of a development in divergence-free gauge at $p\in M_o$, $t_1>0$, $\bar{g}_{b,t}$ a Kähler Ricci-flat gluing of the $(N_k,g_{b_k})_k$ at relative scales $t=(t_k)_k>0$ produced by Lemma \ref{Recollement arbre de'ALE khaleriens}, and a symmetric $2$-tensor $\Hat{h}_{2}$ on $N$ and the real numbers $\Hat{\lambda}_i$ of Lemma \ref{h2 arbre kahlerien}. 
\begin{defn}[Metric $\Hat{g}^A$]
    Let us define the approximate metric $\Hat{g}^A$ as the naïve gluing (Definition \ref{def naive desing}) of $\Hat{g}_{S}$ and $t_1\big(\bar{g}_{b,t}+t_1\hat{h}_S\big)$.
\end{defn}
We have the following control whose proof is the same as Theorem \ref{obstruction intégrable tout point}.

\begin{cor}\label{controle approximation arbres}
    Let $\beta>0$, and let us use the above notations.
    For $t_{\max}>0$ small enough we have the following controls: for all $k\in \mathbb{N}$ there exists $C_k>0$,
    \begin{equation}
        \big\|\pi_{\tilde{\mathbf{O}}(g^D)^\perp}\mathbf{\Phi}_{g^D}(\Hat{g}^A) \big\|_{r^{-2}_DC^k_{\beta}(g^D)}\leqslant C_k t_{1}^{\frac{3-\beta}{4}},\label{Controle operateur gAv bis}
    \end{equation}
    and for all $\tilde{\mathbf{o}}\in \tilde{\mathbf{O}}_{t_1}(\bar{g}_{b,t})$, and denoting $(\tilde{\mathbf{o}}_{i})_i$ and orthonormal basis of $ \tilde{\mathbf{O}}_{t_1}(\bar{g}_{b,t})$
    \begin{equation}
        \Big\langle\mathbf{\Phi}_{g^D}(\Hat{g}^A)-t_1 \sum_{i} \hat{\lambda}_i \tilde{\mathbf{o}}_{i}, \tilde{\mathbf{o}} \Big\rangle_{L^2(g^D)}\leqslant C_0\|\mathbf{o}\|_{L^2(g_b)}t_1^\frac{5}{4},\label{Controle proj obst gAv bis}
    \end{equation}
    while satisfying, 
    \begin{enumerate}
        \item $\|\Hat{g}^A-g^D\|_{C^{2,\alpha}_{\beta,*}(g^D)}\leqslant 2\epsilon$,
        \item $\Hat{g}^A-g^D$ is $L^2(g^D)$-orthogonal to $\tilde{\mathbf{O}}(g^D)$,
    \end{enumerate}
\end{cor}

\begin{rem}
    The crucial part here is that, by considering the right weighted spaces, $C^{2,\alpha}_{\beta,*}(g_t^B)$, and Kähler Ricci-flat perturbations of our tree of singularities, we obtain a control by powers of $t_1$ only.
\end{rem}

We then conclude, exactly like in Theorem \ref{obstruction intégrable tout point} that the obstruction is satisfied in the limit at every singular point of $(M_o,g_o)$ where the trees of singularities appearing are composed of Kähler Ricci-flat orbifolds ALE.

\begin{thm}\label{obst arbre kahler}
    Let $(M_o,g_o)$ an Einstein orbifold, and assume that there exists $(M_i,g_i)_i$ a sequence of Einstein manifolds such that $$(M_i,g_i)\xrightarrow{GH} (M_o,g_o).$$
    
    Then, $(M_o,g_o)$ satisfies $\det \mathbf{R}(g_o)=0$ at every singular point where the trees of singularities forming in the Gromov-Hausdorff sense according to Corollary \ref{GH to C3} are composed of ALE Kähler Ricci-flat orbifolds glued in the same orientation.
\end{thm}

\begin{rem}
    The result is optimal in the sense that it is the only local obstruction to the desingularization of a $\mathbb{R}^4\slash\mathbb{Z}_2$ singularity. Indeed, together with the existence of Einstein desingularizations of \cite{biq1}, proven in the case of rigid asymptotically hyperbolic Einstein metrics with a singularity $\mathbb{R}^4\slash\mathbb{Z}_2$ singularity, we see that there exists a desingularization in the Gromov-Hausdorff sense by Eguchi-Hanson metrics \emph{if and only if} the condition $\det \mathbf{R}(g_o)=0$ is satisfied.
\end{rem}

\begin{rem}
    For now, we cannot prove any obstruction result if trees of non Kähler Ricci-flat ALE orbifolds were to appear. The reason is that it might not be possible to glue and perturb them into a single Ricci-flat ALE manifold. The obstructions to such a gluing could possibly compensate the ones coming from the gluing to the orbifold.
\end{rem}

\begin{exmp}\label{quotient sphère}
    Like in Example \ref{ex S4 quotient}, let us consider the sphere $\mathbb{S}^4$ as $\mathbb{S}^4\subset \mathbb{R}^5 = \mathbb{R}\times \mathbb{R}^4$. We define $\mathbb{S}^4\slash\Gamma$, the orbifold obtained as the quotient of $\mathbb{S}^4$ by the action of $\Gamma$ for the first $4$ coordinates of $\mathbb{R}^5$. $\mathbb{S}^4\slash\Gamma$ has its sectional curvatures constant equal to $1$, and two singularities modeled on $\mathbb{R}^4\slash\Gamma$. The condition $\det \mathbf{R}= 0$ is therefore not satisfied for this orbifold.
\end{exmp}

\section{Obstructions under topological assumptions}

Let us now give topological conditions which will ensure that the Ricci-flat ALE orbifolds appearing as blow ups in our degenerations are Kähler and glued in the same orientation, and therefore that the obstruction $\det \mathbf{R}=0$ holds.

All of these topological conditions come from the topological characterization of \cite{nak}, see also \cite{lv} for a generalization. Basically, if a desingularization has the topology of a minimal resolution of a $SU(2)$-singularity (or a quotient for the $U(2)$ singularities) in a neighborhood of a singularity, then, all of the bubbles are Kähler and glued in the same orientation and we can apply Theorem \ref{obst arbre kahler}. We will state more global topological conditions based on Hitchin-Thorpe inequality in Theorem \ref{obst hitchin thorpe}, and a spin condition in Theorem \ref{obst spin}. We will then finally comment on the desingularization of Einstein orbifolds with various pinching conditions and bound on the Ricci curvature in Corollary \ref{pinching} and Remark \ref{lower bound}.

\subsection{Hitchin-Thorpe inequality and desingularization of Einstein orbifolds}
Let us first notice that desingularizing an Einstein orbifold by smooth Einstein manifolds necessarily damages the Hitchin-Thorpe inequality satisfied by the orbifold, see Theorem \ref{obst hitchin thorpe}. The equality case is exactly when all the Ricci-flat ALE orbifolds are Kähler and glued in a common orientation.

For an Einstein manifold of dimension $4$, Chern-Gauss-Bonnet formula implies,
\begin{equation}
    \chi(M)= \frac{1}{8\pi^2}\int_M |\Rm|^2dv = \frac{1}{8\pi^2}\int_M \Big(\frac{\R^2}{24}+|W^+|^2+|W^-|^2\Big)dv,\label{GaussBonnet}
\end{equation}
and Hirzebruch's signature formula gives,
\begin{equation}
    \tau(M)= \frac{1}{12\pi^2}\int_M \big(|W^+|-|W^-|^2\big)dv. \label{Hirzebruch}
\end{equation}
Simply because $\int_M |W^\pm|^2dv\geqslant 0$ and $\int_M \R^2dv\geqslant 0$, thanks to \eqref{GaussBonnet} and \eqref{Hirzebruch}, we have the following Hitchin-Thorpe inequality for Einstein $4$-manifolds,
$$2\chi(M)\geqslant 3|\tau(M)|.$$
with equality if and only if $(M,g)$ is a quotient of the flat torus or of a hyperkähler metric on the $K3$ surface.

In the case of orbifolds and ALE metrics, to be consistent with Chern-Gauss-Bonnet and Hirzebruch formulas, \eqref{GaussBonnet} and \eqref{Hirzebruch} for compact Einstein manifolds of dimension $4$, we have to modify the Euler characteristic and the signature thanks to a boundary term. The integral quantities \eqref{GaussBonnet} and \eqref{Hirzebruch} above are topological invariants for Einstein orbifolds and Ricci-flat ALE orbifolds. We will denote them $\tilde{\chi}$ and $\tilde{\tau}$.

For Ricci-flat ALE manifolds, Nakajima obtained an Hitchin-Thorpe inequality.
\begin{lem}[{\cite[Theorem 4.2]{nak}}]\label{HT Nakajima}
    Let $(N,g_b)$ be a Ricci-flat ALE manifold of dimension $4$. Then, we have the following inequality between the modified Euler characteristic and the modified signature of Ricci-flat ALE orbifolds,
    $$ 2\tilde{\chi}(N)\geqslant 3 |\tilde{\tau}(N)|, $$
    with equality if and only if $(N,g_b)$ is a Kähler Ricci-flat ALE orbifold.
\end{lem}

\begin{rem}
    In particular, the only Ricci-flat ALE manifolds diffeomorphic to a minimal resolution of a singularity $\mathbb{C}^2\slash\Gamma$ for $\Gamma\subset SU(2)$ or one of its quotients are Kähler.
\end{rem}

The topological invariants $\tilde{\tau}$ and $\tilde{\chi}$ are additive by gluing ALE spaces to orbifold singularities like in Definition \ref{def naive desing}. If $M = M_o\#_jN_j$, we then have
$$ \tau(M) = \tilde{\tau}(M_o) + \sum_j\tilde{\tau}(N_j),$$
and
$$ \chi(M) = \tilde{\chi}(M_o)+ \sum_j\tilde{\chi}(N_j).$$
This directly implies:
\begin{align*}
   2\chi(M) - 3|\tau(M)| =&\; 2\Big(\tilde{\chi}(M_o) + \sum_j\tilde{\chi}(N_j)\Big) - 3 \Big|\tilde{\tau}(M_o) + \sum_j\tilde{\tau}(N_j)\Big|\\
   \geqslant&\; 2\tilde{\chi}(M_o)- 3 |\tilde{\tau}(M_o)| + \sum_j \big(2\tilde{\chi}(N_j) - 3|\tilde{\tau}(N_j)|\big)\\
   \geqslant&\; 2\tilde{\chi}(M_o)- 3 |\tilde{\tau}(M_o)|.
\end{align*}
Since every term is nonnegative by Hitchin-Thorpe inequality and Lemma \ref{HT Nakajima}, we see that there is equality if and only if for all $j$ we have $2\tilde{\chi}(N_j) - 3|\tilde{\tau}(N_j)|= 0$ and that the gluings are done in the same orientation for which $\tilde{\tau}(M_o)$ and all the $\tilde{\tau}(N_j)$ have the same sign.

\begin{exmp}
    If $(M_o,g_o)$ is a hyperkähler orbifold, then the only Gromov-Hausdorff desingularizations preserving the inequality are hyperkähler and correspond to gluing hyperkähler ALE in the same orientation. 
\end{exmp}

\begin{exmp}
	For $\Gamma\subset SU(2)$, an Einstein desingularization of $\mathbb{S}^4\slash\Gamma$ preserving Hitchin-Thorpe inequality is diffeomorphic to $M = \mathbb{S}^4\slash \Gamma \#X_\Gamma\#X_\Gamma$ for $X_\Gamma$ a minimal resolution of the singularity $\mathbb{C}^2\slash\Gamma$. 
\end{exmp}

By studying the equality case in the previous inequalities, we get a quite restrictive situation.
\begin{thm}\label{obst hitchin thorpe}
    Let $(M_o,g_o)$ be an Einstein orbifold oriented so that $\tilde{\tau}(M_o)\geqslant 0$, and assume that $(M,g_i)_i$ is a sequence of Einstein metrics converging in the Gromov-Hausdorff sense to $(M_o,g_o)$.
    
    We then have the following inequality,
    $$2\chi(M)-3|\tau(M)|\geqslant 2\tilde{\chi}(M_o)-3\tilde{\tau}(M_o).$$
    Moreover, we have equality, if and only if
    $M$ is a desingularization of $M_o$ by gluing trees of  Kähler Ricci-flat ALE orbifolds in the same orientation (that is with gluing parameters in $SO(4)$), and in this equality case we have the condition $$\det\mathbf{R}_+(g_o)=0$$ at all of the singular points of $M_o$.
\end{thm}
\begin{rem}
    The equality condition limits the possible group actions of the singularities.
\end{rem}

This for example implies the following.
\begin{cor}
	Let $\Gamma\subset SU(2)$, $(M_i,g_i)_i$ be a sequence of Einstein manifolds converging in the Gromov-Hausdorff sense to the spherical orbifold $\mathbb{S}^4\slash \Gamma$. Then, for $i$ large enough, we have
	$$2\chi(M_i)-3|\tau(M_i)| > 2\tilde{\chi}(M_o)-3|\tilde{\tau}(M_o)|.$$
\end{cor}

\subsection{Spin manifolds}

Another way to ensure that the Ricci-flat ALE orbifolds appearing are Kähler and glued in the same orientation is to impose that the sequence of differentiable manifolds is spin. Our result is essentially an application of the following Lemma of Nakajima.
\begin{lem}[{\cite[Corollary 3.3]{nak}}]\label{ALE spin}
    Let $(N,g_b)$ be a Ricci-flat ALE metric on a spin manifold which is asymptotic to $\mathbb{R}^4\slash\Gamma$ for $\Gamma$ a finite subgroup of $SU(2)$, then, $(N,g_b)$ is a hyperkähler metric.
\end{lem}

As a consequence, there is also an obstruction to the desingularizations of Einstein orbifolds by smooth Einstein metrics on spin manifolds. The proof of Theorem 1.1 of \cite{kl} whose main tool is Lemma \ref{ALE spin} implies that the limit orbifold and the Ricci-flat ALE metrics are spin and glued in the same orientation for a degeneration of Einstein metrics on spin manifolds. If the group at infinity of the ALE spaces, which are also the groups of the singularities of the orbifold are in $SU(2)$, we use Lemma \ref{ALE spin} to get the following obstruction.

\begin{thm}\label{obst spin}
    Let $(M_i,g_i)_i$ be a sequence of Einstein spin manifolds of dimension $4$ converging to an Einstein orbifold $(M_o,g_o)$. Then, $(M_o,g_o)$ is spin and at its singular points whose groups are in $SU(2)$, we have $$\det \mathbf{R} = 0.$$
\end{thm}

\begin{rem}
    There is no restriction on the group singularities in \cite{kl}. This comes from their additional assumption on the kernel of the Dirac operator of the sequence which actually implies that all singularities have their group in $SU(2)$.
\end{rem}

\subsection{Pinched Ricci curvature and the Einstein condition}

Our result shows that there is a fundamental difference between the Einstein condition and some pinching conditions on the Ricci curvature once we authorize the formation of singularities. From Theorems \ref{obst hitchin thorpe} and \ref{obst spin}, we deduce that there exists an obstruction to the desingularization of Einstein orbifolds by smooth Einstein metrics which is not there if we consider pinching conditions on the Ricci curvature. Let us illustrate this with the simple example of a spherical orbifold, even though a similar result obviously holds for general orbifolds with singularity groups in $SU(2)$.

\begin{cor}\label{pinching}
    Let $\Gamma$ be a finite subgroup of $SU(2)$, and $M = \mathbb{S}^4\slash \Gamma \#X_\Gamma\#X_\Gamma$ ($\#$ means gluing at both orbifold singularities in an orientation), where $X_\Gamma$ is the minimal resolution of the singularity $\mathbb{C}^2\slash\Gamma$. Then, for all $1\leqslant p <+\infty$,
    \begin{enumerate}
        \item there exists a sequence of metrics $(M,g_i)_i$ such that we have 
        \begin{itemize}
            \item $\|\Ric(g_i)-3g_i\|_{L^p(g_i)}\leqslant \frac{1}{i}$
            , and
            \item $\big( M,g_i \big) \xrightarrow[GH]{ }  \big( \mathbb{S}^4\slash\Gamma, g_{\mathbb{S}^4\slash\Gamma}\big)$,
        \end{itemize}
        but,
        \item there  does \emph{not} exist any sequence of Einstein metrics $(M,g_i^\mathcal{E})$ such that 
        \begin{itemize}
            \item $\Ric(g^\mathcal{E}_i)=3g^\mathcal{E}_i$, and
            \item $\big( M,g_i^\mathcal{E} \big) \xrightarrow[GH]{ } \big(\mathbb{S}^4\slash\Gamma, g_{\mathbb{S}^4\slash\Gamma}\big)$.
        \end{itemize}
    \end{enumerate}
\end{cor}
\begin{proof}
    The second part is a consequence of Theorem \ref{obstruction intégrable tout point} because the curvature of the sphere never satisfies the condition $\det \mathbf{R} =0$ since $\mathbf{R}= \mathrm{Id}$ for such a metric.
    
    For the first part, we can just remark that our approximation metric $g^A_t$ with fixed Kähler Ricci-flat ALE metrics satisfies $\|\Ric(g^A_t)-\Lambda g^A_t\|_{L^\infty(g^A_t)}= \mathcal{O}(1)$ and that $\Ric(g^A_t)-\Lambda g^A_t$ is supported in regions with a volume of order $t$, therefore, if we choose $t$ small enough, we have the control in $L^p$-norm for $p<+\infty$.
\end{proof}

\paragraph{Question:} Can we desingularize $\mathbb{S}^4\slash\mathbb{Z}_2$ thanks to the Eguchi-Hanson metric by metrics with Ricci curvature converging to $3$ in the $L^\infty$-sense?

\begin{rem}
    By being more precise in the expression of the obstructions to the desingularization of $\mathbb{S}^4\slash\mathbb{Z}_2$ by two Eguchi-Hanson metrics, for $t_{\max}$ small enough, 
    \begin{equation}
        |\Ric(g^A_t)-3 g^A_t|_{g^A_t} \leqslant 1+\eta(t_{\max})\label{desing et Ric minre}
    \end{equation}
     where $\eta(t_{\max})\to 0$ when $t_{\max}\to 0$. 
\end{rem}

\begin{rem}\label{lower bound}
    It is possible to desingularize a spherical orbifold $\mathbb{S}^4\slash\Gamma$ for $\Gamma\subset SU(2)$ by metrics with $\Ric\geqslant 3$ (or $\Ric\leqslant 3$) while $\Ric$ is pinched in $L^p$. 
    
    Consider for $\epsilon>0$ and $b>1$, choose a cut-off function, $\chi_{b,\epsilon}$, supported on $[0,b\epsilon]$ and equal to $1$ on $[0,\epsilon]$ whose $k$-th derivatives are $\mathcal{O}\big(\frac{1}{\log(b)}\epsilon^{-k}\big)$, and define the metric
    $$g_{b,\epsilon}:= dr^2+ \sin\big((1+\chi_{b,\epsilon})r\big)g_{\mathbb{S}^{3}\slash\Gamma}.$$
    Assume that $\epsilon\to 0$, $b\to +\infty$ and $b\epsilon\to 0$, the orbifold metric therefore becomes arbitrarily close in the Gromov-Hausdorff sense to $\mathbb{S}^4\slash\Gamma$. Moreover, the sectional curvatures of $g_{b,\epsilon}$ are bounded below by $1- \frac{C}{\log b}\to 1$ for some uniform $C>0$. Let us finally glue $t(g_{EH}+2th_2)$, where $h_2$ is asymptotic to $-\frac{1}{3}r_e^2g_{\mathbb{S}^3\slash\Gamma}$ at the singular points for $t_{\max}\ll\epsilon^2$, so that the gluing happens in $0\leqslant r<\epsilon$ where the metric $g_{b,\epsilon}$ equals $dr^2+ \sin^2(2r)g_{\mathbb{S}^{3}\slash\Gamma}$ just like on the sphere of radius $\frac{1}{2}$ whose sectional curvatures are constant equal to $4$. For $r<\epsilon$, the controls are the same as on $\frac{g^A_t}{4}$, and therefore the metric satisfies $\Ric\geqslant 3$ by \eqref{desing et Ric minre} since $ 4\big(3-(1-\eta(t_{\max}))\big)>3$. Since the metric satisfies $\Ric\geqslant 3-\frac{C}{\log(b)}$ for larger $r$, we can simply rescale it a little to ensure that we have $\Ric>3$ everywhere.
\end{rem}

\section{A general obstruction for spherical and hyperbolic orbifolds}

Let us finally exhibit an obstruction to the desingularization of spherical and hyperbolic orbifolds by general Ricci-flat orbifolds (not necessarily Kähler) in Theorem \ref{obst generale bulle}. We will deduce from it that there does not exist any smooth desingularization of spherical or hyperbolic orbifolds whose blow ups are integrable Ricci-flat ALE spaces in Corollary \ref{obst integrable}.

\subsection{A general infinitesimal deformation for Ricci-flat ALE spaces}

On $(\mathbb{R}^4\slash\Gamma,g_e)$, the vector field $ 2r_e\partial_{r_e}$ is a conformal Killing vector field. It is moreover the gradient of the function $u:=r_e^2$ which is a solution to $-\nabla_e^*\nabla_e u = 8$, and we have $\mathcal{L}_{\nabla_e u}g_e = \textup{Hess}_{g_e}u = 4g_e$. On a Ricci-flat ALE we can approximate this by an infinitesimal deformation.

\begin{prop}\label{def by scaling}
    Let $(N,g_b)$ be a Ricci-flat ALE orbifold asymptotic to $\mathbb{R}^4\slash\Gamma$, and consider a diffeomorphism $\Phi$ between neighborhoods of the infinities of $N$ and $\mathbb{R}^4\slash\Gamma$.
    
    Then, there exists a unique vector field $X$ on $(N,g_b)$ such that $\Phi^{*}X = 2r_b\partial_{r_b} + o(r_b)$, and $\nabla^*\nabla X =0$. We actually have $X=\nabla u$, where $u$ is the unique solution of $-\nabla^*\nabla u = 8$, such that $u = r_b^2 + o(1).$
    
    Moreover, $(\mathcal{L}_Xg_b)^\circ = \mathcal{L}_Xg_b-4g_b$, the traceless part of $\mathcal{L}_Xg_b$ is an infinitesimal Ricci-flat deformation of $g_b$ which is trace-free and divergence-free.
\end{prop}
\begin{proof}
    The proof of the existence and the uniqueness of the function $u$ can be found in the proof of Theorem B of \cite{bh}. The symmetric $2$-tensor $(\mathcal{L}_{\nabla u}g_b)^\circ$ is indeed an infinitesimal deformation of $g_b$ because the equation $\Ric = 0$ is invariant by scaling and pull-back by diffeomorphisms, and the divergence and the trace of $(\mathcal{L}_{\nabla u}g_b)^\circ = 2\textup{Hess}_{g_b}u-4g_b$ vanish because $-\nabla^*\nabla u=8$. 
    
    Moreover, $(\mathcal{L}_{\nabla u}g_b)^\circ$ vanishes exactly for flat cones. Indeed, if it vanishes, then $\nabla u$ is a conformal Killing vector field and therefore generates a family of conformal diffeomorphisms. By considering the maximum of the pointwise norm of the curvature of $(N,g_b)$ which is preserved by this family of diffeomorphism, we see that it has to vanish.
\end{proof}
\begin{rem}
    This deformation is integrable because it simply comes from a rescaling and a change of coordinates. 
\end{rem}

\subsection{Obstructions to the desingularization of spherical and hyperbolic orbifolds}

Let us now use the above deformation $\mathbf{o}_1:= (2\textup{Hess}_{g_b}u-4g_b)$ in order to deduce some general obstructions to the desingularization of spherical and hyperbolic orbifolds.
\begin{thm}\label{obst generale bulle}
    Let $(N,g_b)$ be a Ricci-flat ALE orbifold and $H_2$ be the quadratic terms of a spherical or hyperbolic metric in geodesic coordinates, and $O_1^4$ terms of order $r_b^{-4}$ of the deformation $\mathbf{o}_1= (2\textup{Hess}_{g_b}u-4g_b)$.
    
    Then, we have
    $$\int_{\mathbb{S}^3}\big(3\langle H_2, O_1^4\rangle_{g_e} + O_1^4(B_eH_2,\partial_{r_e})\big)dv \neq 0,$$
    and therefore the perturbation of $g^D_t$ to an Einstein metric orthogonally to $\tilde{\mathbf{O}}(g^D_t)$ is \emph{always obstructed}.
\end{thm}
\begin{proof}

Let $(N,g_b)$ be a Ricci-flat ALE orbifold asymptotic to a flat cone $\mathbb{R}^4\slash\Gamma$, and let $\mathbf{o}_1= (2\textup{Hess}_{g_b}u-4g_b) = O^4_1 + \mathcal{O}(r_b^{-5})$ be the infinitesimal deformation of Proposition \ref{def by scaling}. Let us start by proving that $O_1^4(\partial_{r_e},\partial_{r_e})$ does not vanish. There exists a compact $K\subset M$ such that $M\backslash K$ is foliated by hypersurfaces $\Sigma_\rho$ whose mean curvature is constant equal to $\frac{3}{\rho}$. If we denote $\Omega_\rho$ the interior of $\Sigma_\rho$, then, by \cite[Theorem A]{bh} the following limit exists and is finite:
    \begin{equation}
        \mathcal{V}:= \lim_{\rho\to\infty}\big[\vol_{g_b}(\Omega_\rho)-\vol_{g_e}(B(0,\rho)\slash\Gamma) \big],
    \end{equation}
    and we actually have $\mathcal{V}\leqslant 0$, with equality if and only if $(N,g_b)=(\mathbb{R}^4\slash\Gamma,g_e)$.
    
    Moreover, let $u$ be the unique solution of $-\nabla^*\nabla u = 8$ with $u = r_b^2  +o(1)$, then, we actually have
    $$u = r_b^2 +\frac{b}{r_b^2} + \mathcal{O}(r_b^{-3}),$$
    and by the proof of \cite[Theorem B]{bh}, we have the explicit value
    $$b= -4\frac{\mathcal{V}}{|\partial B(0,1)\slash\Gamma|}\geqslant 0.$$
    We also deduce the following development of $\mathbf{o}_1= (2\textup{Hess}_{g_b}u-4g_b)$,
    \begin{equation}
        \mathbf{o}_1(\partial_{r_b},\partial_{r_b}) = \frac{8b}{r_b^4} +\mathcal{O}(r_b^{-5})
    \end{equation}
    which is strictly positive if $g_b$ is not flat.

    Now, for a hyperbolic metric, we have $H_2 = \frac{r_e^4}{3}(\alpha_1^2+\alpha_2^2+\alpha_3^2)$ in geodesic coordinates, and for a spherical metric, $H_2 = -  \frac{r_e^4}{3}(\alpha_1^2+\alpha_2^2+\alpha_3^2)$. Notice moreover that, since $g_e = dr_e^2 +r_e^2(\alpha_1^2+\alpha_2^2+\alpha_3^2)$, we have $0=\textup{tr}_{g_e} O_1^4 = O_1^4(\partial_{r_e},\partial_{r_e}) + \langle r_e^2(\alpha_1^2+\alpha_2^2+\alpha_3^2), O_1^4\rangle$ and therefore
\begin{align}
    \langle r_e^2(\alpha_1^2+\alpha_2^2+\alpha_3^2), O_1^4\rangle =& \textup{tr}_{g_e} O_1^4 - O_1^4(\partial_{r_e},\partial_{r_e})\\
    =& - O_1^4(\partial_{r_e},\partial_{r_e}).\label{control O1drdr}
\end{align}
For the other part of the obstruction, we have $B_e\big(r_e^4(\alpha_1^2+\alpha_2^2+\alpha_3^2)\big) = 6 r_e \partial_{r_e}$. Indeed, $r_e^4(\alpha_1^2+\alpha_2^2+\alpha_3^2) = r_e^2g_e-r_e^2dr_e^2$, and we have
\begin{align*}
    B_e(r_e^2g_e) &= \delta_e(r_e^2g_e) + \frac{1}{2}d\textup{tr}_e(r_e^2g_e)\\
    & = -2r_e g_e(\partial_{r_e},.) + 4rdr_e\\
    &= 2r_e dr_e,
\end{align*}
and
\begin{align*}
    B_e(r_e^2dr_e^2) &= \delta_e(r_e^2dr_e^2) + \frac{1}{2}d\textup{tr}_e(r_e^2dr_e^2)\\
    & = \delta_e\Big(\sum_{ij}x^ix^jdx^idx^j\Big) + r_e dr_e\\
    &= - \sum_{i\neq j} x^jdx^j - 2\sum_{j} x^jdx^j+ r_e dr_e\\
    &=-4r_e dr_e.
\end{align*}
Finally, for $r_e = 1$,
\begin{equation}
    O_1^4(B_e\big(r_e^4(\alpha_1^2+\alpha_2^2+\alpha_3^2)\big),\partial_{r_e}) = 6 O_1^4(\partial_{r_e},\partial_{r_e}).\label{O1Be}
\end{equation}
The obstruction generated by $\mathbf{o}_1$, that is $\int_{\mathbb{S}^3}\big(3\langle H_2, O_1^4\rangle_{g_e} + O_1^4(B_eH_2,\partial_{r_e})\big)dv$ therefore never vanishes by \eqref{control O1drdr} and \eqref{O1Be}.
\end{proof}

\begin{rem}
    It is also possible to extend the deformations given by the Killing vector fields at infinity to generate more obstructions, but it is not clear if a Ricci-flat ALE space can have vanishing terms of order $r_b^{-4}$. Indeed, the quantity $\mathcal{V}$ is global and does not tell anything on the asymptotics of the metrics, but as we just saw, it tells something about their derivatives along the deformation $(\mathcal{L}_{\nabla_b u}g_b)^\circ$.
\end{rem}

We deduce that we get a general obstruction to a Gromov-Hausdorff desingularization if we assume that the Ricci-flat ALE spaces are integrable.

\begin{cor}\label{obst integrable}
    Let $(M_o,g_o)$ be a compact spherical or hyperbolic orbifold. Then, there does not exist any sequence of Einstein manifolds $(M_i,g_i)$ such that $$(M_i,g_i)\xrightarrow{GH} (M_o,g_o),$$
    while the non-flat limits of $\Big(M_i,\frac{g_i}{t_i},p_i\Big)$ for $t_i\to 0$, $t_i>0$ and $p_i\in M_i$ converge to \emph{smooth} and \emph{integrable} Ricci-flat ALE manifolds (which means that there are no trees of singularities forming).
\end{cor}
\begin{proof}
    According to Theorem \ref{obstruction intégrable tout point}, if the quadratic terms of the development of $g_o$ are $H_2$, the obstruction induced by the deformation $\mathbf{o}_1$ is 
$$\int_{\mathbb{S}^3}\big(3\langle H_2, O_1^4\rangle_{g_e} + O_1^4(B_eH_2,\partial_{r_e})\big)dv =0,$$
which is never satisfied according to Theorem \ref{obst generale bulle}.

The obstruction of Theorem \ref{obstruction intégrable tout point} is therefore never satisfied for spherical and hyperbolic metrics, and it is impossible to desingularize it by Ricci-flat ALE manifolds which are integrable.
\end{proof}

The obstruction to the desingularization of spherical and hyperbolic manifolds is therefore identified, but we need the technical integrability assumption to deduce a Gromov-Hausdorff obstruction thanks to it. We believe that this is only a technicality and conjecture the following statement.

\begin{conj}
    Singular spherical and hyperbolic orbifolds cannot be Gromov-Hausdorff limits of smooth Einstein manifolds.
\end{conj}

%
%
%
%

\end{document}